\documentclass[11pt]{article}
\usepackage[utf8]{inputenc}
\usepackage{amsthm}
\usepackage{mathtools}
\usepackage{amsmath}
\usepackage{bm}
\usepackage{yhmath}
\usepackage{enumerate}
\usepackage{subfigure}
\usepackage{graphicx}
\usepackage{algorithm}
\usepackage{multirow}
\usepackage{dcolumn}
\usepackage{array}
\usepackage{algorithmicx}  
\usepackage{algpseudocode} 
\usepackage{amssymb} 
\usepackage{mathrsfs}
\usepackage{dsfont}

\def\P{\mathbb{P}}
\def\R{\mathbb{R}}

\def\F{\mathcal{F}}

\def\E{\mathbb{E}}

\usepackage[numbers]{natbib}
\usepackage{amsmath}
\usepackage{amsfonts}
\usepackage{amssymb}
\usepackage{xcolor}
\usepackage{epsfig}

\usepackage{amsmath}
\usepackage{amsfonts}
\usepackage{amssymb}
\usepackage{dsfont}
\usepackage{mathrsfs}
\usepackage{bbm}
\usepackage{amsmath}
\usepackage{amsfonts}
\usepackage{amssymb}
\usepackage{multirow}
\usepackage{mathrsfs}
\let\amssymbboxplus\boxplus
\let\amssymbboxminus\boxminus

\renewcommand{\boxplus}{\mathbin{\mathop\amssymbboxplus}}
\renewcommand{\boxminus}{\mathbin{\mathop\amssymbboxminus}}

\usepackage{makecell} 
\usepackage{boldline}
\usepackage{diagbox}
\usepackage{threeparttable}
\usepackage{adjustbox}

\usepackage[title]{appendix}

\usepackage[colorlinks=true]{hyperref}

\providecommand{\keywords}[1]{\textbf{\textit{Keywords:}} #1}

\newtheorem{theorem}{Theorem}[section]
\newtheorem{lemma}{Lemma}[section]
\newtheorem{proposition}{Proposition}[section]

\newtheorem{remark}{Remark}[section]
\newtheorem{definition}{Definition}[section]
\newtheorem{assumption}{Assumption}[section]
\newtheorem{corollary}{Corollary}[section]

\newcommand{\setd}{{ d \kern -.15em l}}
\newcommand{\hatsetd}{ d \hat{\kern -.15em l }}
\newcommand{\dd}{\mathsf {d\kern -0.07em l}}

\newcommand{\bgeqn}{\begin{eqnarray}}
\newcommand{\edeqn}{\end{eqnarray}}
\newcommand{\bgeq}{\begin{eqnarray*}}
\newcommand{\edeq}{\end{eqnarray*}}

\newcommand{\inmat}[1]{\mbox{\rm {#1}}}

\def\P{\mathbb{P}}
\def\R{\mathbb{R}}

\def\F{\mathcal{F}}

\def\E{\mathbb{E}}


\usepackage[colorlinks=true]{hyperref}
 \hypersetup{urlcolor=blue,
 citecolor=blue,linkcolor=blue}
   
\makeatletter
\g@addto@macro{\UrlBreaks}{\UrlOrds}
\makeatother

\title{
Quantification of Errors of the
Performance Estimators
in the Linear-Quantized Precoding 
Models for 
Massive MIMO Systems}
\author{Jie Zhang\thanks{School of Mathematics, Liaoning Normal University, Dalian, China. Email: Jie\_zhang@lnnu.edu.cn.} 
 and Huifu Xu\thanks{Department of Systems Engineering and Engineering Management, The Chinese University of Hong Kong. Email: hfxu@se.cuhk.edu.hk.}}

\begin{document}

\maketitle
\abstract{Massive MIMO (Multiple-Input Multiple-Output) technology has become a key enabler for 5G
and future wireless communication systems, offering significant improvements in channel capacity,
energy efficiency, and spectral efficiency. However, the high power consumption and hardware costs
associated with Digital-to-Analog Converters (DACs) in massive MIMO systems pose substantial
practical challenges. To address these issues, recent research has proposed the use of low-resolution
DACs, which restricts transmitted signals to a finite set of voltage levels, thereby reducing both
power consumption and hardware costs. This approach necessitates the study of quantized precod-
ing schemes, where signals are first processed by a linear precoding matrix and then quantized by
the DACs. 
In this paper, we 
investigate
the relationship between the linear-quantized precoding model
and its variations which are statistically or asymptotically equivalent.
Specifically we derive error bounds of the two important performance metrics,
the Signal-to-Interference-plus-Noise Ratio (SINR) and the Symbol Error Probability (SEP)
based on the  linear-quantized precoding model and those derived from the statistically or asymptotically equivalent models.
We also formulate and analyze the SINR maximization problem in both asymptotic and
finite-dimensional settings. 
Our analysis demonstrates that the optimal solutions and values of the finite-dimensional problem converge to their asymptotic counterparts as system dimensions scale, highlighting the practical relevance of asymptotic insights with quantitative stability guarantees. These findings 
provide theoretical grounding
for designing robust precoding strategies under hardware constraints, supporting the efficient implementation of massive MIMO systems with low-resolution DACs. In addition to validating the reliability of asymptotic predictions in finite regimes, the proposed framework provides practical optimization guidelines for real-world communication systems, connecting theoretical insights with practical applications.
}


\section{Introduction}
Massive multiple-input multiple-output (MIMO) stands as a cornerstone technology for 5G and future wireless systems, offering transformative gains in channel capacity, spectral efficiency, and energy efficiency by leveraging large-scale antenna arrays at base stations \cite{R13,3,4}. However, practical deployment faces significant challenges due to the high power consumption and hardware costs associated with radio frequency chains, particularly high-resolution digital-to-analog converters (DACs) \cite{5}. To address this, low-resolution DACs have emerged as a cost-effective alternative, reducing circuit complexity while maintaining competitive system performance \cite{Jacobsson2017,7,8}. This shift necessitates the development of quantized precoding schemes, where transmitted signals are constrained to finite voltage levels, as explored in \cite{1,9,10}. While asymptotic analyses under infinite-dimensional assumptions—such as those based on random matrix theory—provide critical theoretical insights \cite{1,27}, they often overlook practical limitations in finite-dimensional systems. This gap underscores the need for rigorous quantitative stability analysis to bridge theoretical predictions with real-world implementations.

The asymptotic analysis of large-scale MIMO systems has been a pivotal area of research, with seminal contributions establishing the foundations for understanding system behavior as dimensions grow. Notably, works like \cite{1}, \cite{9}, and \cite{10} have leveraged random matrix theory to derive closed-form expressions for key performance metrics such as the signal-to-interference-plus-noise ratio (SINR) and symbol error probability (SEP). These studies have demonstrated that, in the asymptotic regime where the number of antennas and users grow large with fixed ratios, the performance metrics converge to deterministic values that can be precisely characterized. This approach has provided valuable insights into the fundamental limits of massive MIMO systems and has guided the design of efficient precoding schemes that are asymptotically optimal.
While asymptotic analyses offer valuable theoretical insights, practical implementations of massive MIMO systems face several finite-dimensional challenges that are not fully captured by these models. Recent studies have addressed the limitations of asymptotic results in finite-dimensional settings. For instance, 
Wu et al.~ \cite{16} and Wu et al.~\cite{17} have highlighted discrepancies between asymptotic predictions and practical performance in systems with limited numbers of antennas and users. Additionally, the role of quantization in enabling hardware-efficient massive MIMO systems has been explored in works such as \cite{7} and \cite{8}. These studies have shown that low-resolution digital-to-analog converters (DACs) can significantly reduce power consumption and hardware complexity while maintaining acceptable performance levels, though with trade-offs in system performance that need careful management.
The stability of performance metrics in stochastic optimization frameworks and the establishment of probabilistic bounds for communication metrics have been critical concerns in ensuring reliable system design. Works like \cite{16} have contributed to the understanding of stability in stochastic optimization contexts, providing analyses that help predict system behavior under varying conditions. Meanwhile, studies such as \cite{R13} and \cite{5} have established probabilistic bounds for communication metrics, offering guarantees on system performance with high probability. Recently, the asymptotic analysis conducted in \cite{WLS24} offers valuable insights into the performance of linear-quantized precoding schemes in massive MIMO systems by examining the behavior of the SINR and SEP as the number of transmit antennas and users grows large while maintaining a fixed ratio. Based on this asymptotic analysis, the optimal precoder within a class of linear-quantized precoders is derived, encompassing several well-known precoders as special cases.
The approach in \cite{WLS24} has been instrumental in deriving closed-form expressions and characterizing the optimal precoding strategies under certain conditions, asymptotic analysis alone does not provide precise performance metrics for systems with finite dimensions.

In this paper, we focus on the quantitative stability analysis of linear-quantized precoding in massive MIMO systems, aiming to bridge the gap between asymptotic theoretical guarantees and the practical realities of finite-dimensional systems, particularly in terms of performance metric stability across different system scales and configurations.
We provide
a rigorous framework for characterizing the deviations of the finite-dimensional performance metrics, such as SINR and SEP, from their asymptotic counterparts and the exponential convergence rate. This analysis will help in understanding the practical performance of massive MIMO systems with low-resolution DACs and in developing efficient optimization algorithms that can handle the complexities of finite-dimensional systems.  As far as we are concerned, the main contributions of the paper can be sumarized as follows.
%
 \begin{itemize}
     \item 
\textbf{Quantitative 
     asymptotic convergence analysis of SINR and SEP}. 
     We introduce a new 
     framework based on Ky Fan distance 
     for 
     quantifying the deviation of performance metrics such as SINR and SEP from their asymptotic counterparts in finite-dimensional systems. 
     This framework helps bridge the gap between theoretical asymptotic results and practical finite-dimensional systems.
     We  derive bounds and characterize the stability of SINR and SEP metrics, which provides a deeper understanding of 
     how quantization-induced distortion degrades system performance in practical finite-dimensional systems.

     
    \item \textbf{Convergence rate of received signals.}
     We demonstrate that in massive MIMO systems, the probability of significant discrepancies between real-world finite-dimensional received signals and their asymptotic theoretical predictions decreases exponentially as the number of antennas grows. By deriving probabilistic bounds for critical scaling parameters governing signal and interference terms, we show that beyond a practical antenna threshold, deviation probabilities are tightly controlled by an exponentially decaying function. This work connects asymptotic theory with practical system performance, providing theoretical foundations for reliably deploying low-resolution digital-to-analog converters in realistic implementations.
\item  \textbf{Quantitative asymptotic 
convergence analysis of SINR maximization problem.} 
We investigate 
the 
asymptotic behavior of feasible sets, the set of optimal solutions, and the optimal values. 
The asymptotic upper bounds on the errors between feasible sets, the optimal solutions, and the optimal values of finite-dimensional and asymptotic problems are obtained through a growth function and 
the expected deviations of received signals. These theoretical findings provide a foundation for practical system design and optimization, guiding the development of efficient algorithms and offering performance guarantees in real-world massive MIMO deployments.    
 \end{itemize}
The rest of the paper is 
organized as follows. Section 2 
reviews the downlink transmission framework and linear-quantized precoding model originally proposed in \cite{WLS24}.
Section 3 
synthesizes asymptotic convergence properties of SINR and SEP derived in 
\cite{WLS24}.
Section 4 introduces quantitative stability bounds for SEP/SINR deviations using novel probabilistic techniques, while Section 5 establishes exponential convergence rates for signal estimation through Chernoff-type analyses. Section 6 
extends prior optimization frameworks by proving finite-system stability for SINR maximization. 
Finally Section 7 concludes.

\textbf{Notation:} Throughout the paper, \(\mathbb{C}\) denotes the set of complex numbers. A complex number is a number that can be expressed in the form \(a + bi\), where \(a\) and \(b\) are real numbers, and \(i\) is the imaginary unit, which satisfies the equation \(i^2 = -1\). \(\mathbb{C}^n\) denotes the \(n\)-dimensional complex vector space. Specifically, \(\mathbb{C}^n\) is the set of all \(n\)-dimensional vectors where each component of the vector is a complex number. 
  For a vector \( x \in \mathbb{C}^n \), \( x[i_1 : i_2] \) denotes a sub-vector of \( x \) consisting of its \( i_1 \)-th to \( i_2 \)-th elements, where \( 1 \leq i_1 \leq i_2 \leq n \); in particular, \( x[i] \) is the \( i \)-th entry of \( x \), and \( x_i \) is also used if it does not cause any ambiguity. For a matrix \( X \), \( X[i_1, i_2] \) is the \( (i_1, i_2) \)-th entry of \( X \). The operators \(\mathcal{R}(\cdot)\), \(\mathcal{I}(\cdot)\), \((\cdot)^\dagger\), \((\cdot)^T\), \((\cdot)^H\), and \((\cdot)^{-1}\) return the the real part, the imaginary part, the conjugate, the transpose, the conjugate transpose, and the inverse of their corresponding arguments, respectively. We use \( \mathcal{U}(n) \) to denote the set of \( n \times n \) unitary matrices over \(\mathbb{C}\).  We use \(\|\cdot\|\) to denote the \(\ell_2\) norm of the corresponding vector or the spectral norm of the corresponding matrix. We use $|\cdot|$ to denote the absolute value of a real number or the modulus of a complex number. For $a+bi\in \mathbb{C}$, $|a+bi|=\sqrt{a^2+b^2}$. For $x\in \mathbb{C}$ and $A \subset\mathbb{C}$, $d(x,A):=\min_{x'\in A}|x-x'|$ and $\mathbb{B}(x,\epsilon):=\{x': d(x',x)\leq \epsilon\}$ for real number $\epsilon>0$, $\text{cl} A, \partial A$ are the closure and boundary of $A$, respectively. We use $\mathbb{B}_2$ to denote the unite closed ball in $\mathbb{C}$.
  For \( m, n \in \mathbb{N} \), $\mathbb{N}$ is the set of positive integer, we denote the \( m \times m \) identity matrix by \( I_m \) and the \( m \times n \) matrix of all zero entries by \( 0_{m \times n} \). We use \(\text{diag}(x_1, x_2, \ldots, x_n)\) to refer to a diagonal matrix with \(\{x_i\}_{i=1}^n\) as its diagonal entries. The operators \(\mathbb{E}[\cdot]\), \(\text{var}(\cdot)\) return the expectation, the variance,      
 respectively. We denote almost sure convergence by \(\xrightarrow{a.s.}\). We use 
\(\mathcal{N} (0, 1)\) and \(\text{Unif}(S)\) to denote the standard normal distribution and uniform distribution on set \( S \), respectively.

A random vector (or element) with values in the complex $\mathbb{C}^n$ 
is naturally defined as 
$\xi = \xi' + i\xi''$ for which 
$\hat{\xi} = (\xi', \xi'')$ is a random vector (element) with values in the  $\mathbb{R}^{2n}$ (equivalently, $\xi'$ and $\xi''$ are random vectors in $\mathbb{R}^{n}$). Note that $\xi$ generates the probability measure $P:=\P \circ \xi^{-1}$ on measurable space $(\Xi, \mathcal{B})$, which means that for any Borel set $A\in \mathcal{B},\,\, P(A):=\P\{\omega\in \Omega:\xi(\omega)\in A\}$.
For  complex random variable $\xi(\omega) = \xi'(\omega) + i\xi''(\omega)$, the expectation \( \mathbb{E}[\xi(\omega)]\) is defined by 
\[
\mathbb{E}[\xi(\omega)] = \int_{\Omega} \xi(\omega) d\mathbb{P}(\omega)= \int_{\Omega} \xi'(\omega) d\mathbb{P}(\omega)+i\int_{\Omega} \xi''(\omega) d\mathbb{P}(\omega).
\]


This integral form is a generalization of the expectation for real-valued random variables to complex-valued random variables. For $n$-dimensional complex random variable $Z$, $Z\sim \mathcal{CN}(\mu, \Sigma)$ means that $Z$  follows  the  complex Gaussian distribution \cite{G63} whose density function is
\begin{equation}\label{CN-density}
    f (z) = \frac{1}{\pi^n |\Sigma|} e^{-(z-\mu)^H \Sigma^{-1} (z-\mu)}, \quad z \in \mathbb{C}^n,
\end{equation}
 where $\mu \in \mathbb{C}^n$ is the complex mean and $\Sigma = \mathbb{E} [(Z - \mu) (Z - \mu)^H]$ is the complex variance, which is a Hermitian positive definite matrix, $|M|$ denotes the determinant of matrix $M$. 
A zero mean complex random vector $Z$ is said to be circularly symmetric  if $\mathbb{E} [ZZ^H] = I_n$. We know from \cite{H15} that if $Z\sim \mathcal{CN} (0, \mathbf{I_n})$, then $Z_i\sim \mathcal{CN} (0, 1)$ and  \(\mathcal{R}(Z_i)\sim \mathcal{N} (0, \frac{1}{2}), 
\mathcal{I}(Z_i)\sim \mathcal{N} (0, \frac{1}{2})\) for $i=1,2,\cdots,n$. 

\section{System operation  
models 
and performance metrics}
In this section, we recall
three models 
for a massive MIMO downlink system employing linear-quantized precoding: the original system model, a statistically equivalent model, 
and an asymptotic model. 
These models are crucial for understanding the performance and 
optimal operations of 
the system. We also recall
two principal performance metrics. 
The contents of this section 
are extracted from Wu et al.~\cite{WLS24}.

\subsection{The original system operation model}

Consider the
downlink of a massive MIMO system,
where an \( N \)-antenna base station (BS) serves \( K \) single-antenna users. 
The signal received 
by 
users
is modeled by
\begin{equation}
\label{eq:y-original-model}
\mathbf{y} = \mathbf{H} \mathbf{x} + \mathbf{n},
\end{equation}
where
\begin{itemize}

    \item 
    $\mathbf{n}$ is a complex  random vector mapping 
    from $(\Omega,\F,\mathbb{P})$ to $\mathbb{C}^K$ 
    representing additive exogenous environmental noise;
    
    \item \( \mathbf{H} \) is a random matrix mapping from $(\Omega,\F,\mathbb{P})$
to a complex   matrix  
in $\mathbb{C}^{K \times N} \), representing 
channel matrix between the BS and the users,
unlike     $\mathbf{n}$,
the uncertainty is 
endogenous within the system which means the sources of uncertainty might be different from that of $\mathbf{n}$, here we use the same sample space $\Omega$ to ease the exposition because we can synthesize two sample spaces $\Omega_1$ and $\Omega$ into one anyway;

    \item \( \mathbf{y} \)
    is a random vector mapping from $(\Omega,\F,\mathbb{P})$
    to a complex vector in $\mathbb{C}^K \), representing 
    uncertain the vector of signals
    received by the users;    
    
    \item \( \mathbf{x} \in \mathbb{C}^N \) is transmitted signal vector from the BS,

    \item \( \mathbf{N}\) is the number of antennas at the base station (BS),

\item \( \mathbf{K}\) is the number of single-antenna users served by the base station.  
\end{itemize}
In model (\ref{eq:y-original-model}),  $\mathbf{y}$, $\mathbf{H}$ and $\mathbf{n}$ are random because each embodies a distinct source of uncertainty in the wireless link. Specifically, $\mathbf{n}$ represents additive environmental noise, such as thermal noise or interference generated by external systems, and its randomness is exogenous, originating outside the MIMO transceiver and remaining statistically independent of both the channel and the transmitted signal. The matrix $\mathbf{H}$, acting as the channel matrix, contains random entries that capture small-scale fading, path loss and scattering phenomena between the base station and the users, and this uncertainty is endogenous, arising inherently from the physical propagation conditions within the system. Finally, the received signal vector $\mathbf{y}$ inherits randomness from both of the preceding sources, combining the endogenous variations introduced by $\mathbf{H}$ with the exogenous fluctuations contributed by $\mathbf{n}$, and its statistical properties are further shaped by the choice of the transmit vector $\mathbf{x}$.
$\mathbf{N}$ and $\mathbf{K}$ are constants, $\mathbf{x}$ is a design variable. 
$x$ is first shaped by an optimized regularized zero-forcing (RZF) function and then scaled to meet the power limit, see \cite{WLS24}.




Model (\ref{eq:y-original-model}) describes operations of 
the communication 
scheme, where a base station equipped with $N$ antennas simultaneously transmits data to $K$ single-antenna users.  
In this sett-up, 
signal $\mathbf{x}$ 
is transmitted via a low-complexity linear-quantized precoding scheme. As discussed in \cite{R13} and \cite{L14}, this scheme leverages the excess degrees of freedom provided by the large number of antennas $(N>K)$ to enable precise beamforming and effective signal transmission, while also accounting for practical constraints such as limited resolution in analog-to-digital converters (ADCs) and digital-to-analog converters (DACs). The transmitted signal $x$ is thus generated using the linear-quantized precoding model in (\ref{eq:precodeing-x-eta-q}), which incorporates a quantization function $q(\cdot)$ to ensure compliance with hardware limitations. This approach allows the BS to deliver high-quality signals to the users while maintaining power efficiency and robustness against channel impairments, as highlighted in \cite{R13} and \cite{L14}.

The transmitted signal \( \mathbf{x} \) in (\ref{eq:y-original-model}) is generated using linear-quantized precoding
\begin{equation}
\label{eq:precodeing-x-eta-q}
\mathbf{x} = \eta \, q(\mathbf{P} \mathbf{s}),
\end{equation}
where
\begin{itemize}
    
    \item \( \mathbf{s} \in \mathbb{C}^K \) is desired data vector,
    
    \item \( \mathbf{P} \in \mathbb{C}^{N \times K} \) is precoding matrix,
    
    \item \( q(\cdot) \) is the quantization function $q:\mathbb{C}\to\mathcal{X}_{\mathcal{L}}$ which acts component-wise on its input
vector, where $\mathcal{X}_{\mathcal{L}}$ is a finite set with $\mathcal{L}$ elements and $\mathcal{L}$ is
referred to as the quantization level,
    
    \item \( \eta \) is scaling factor to satisfy the transmit power constraint.
\end{itemize}
With the  linear-quantized precoding scheme in (\ref{eq:precodeing-x-eta-q}), the received signals at the user side read
\begin{equation}
\label{eq:y-precoding-eta-q}
\mathbf{y} = \eta \mathbf{H} q(\mathbf{P} \mathbf{s}) + \mathbf{n} = \eta \mathbf{U} \mathbf{D} \mathbf{V}^\mathsf{H} q(\mathbf{V} f(\mathbf{D})^\mathsf{T} \mathbf{U}^\mathsf{H} \mathbf{s}) + \mathbf{n},
\end{equation}
where
\begin{itemize}

\item  $H=UDV^H$,
    \item \( \mathbf{U} \in \mathcal{U}(K) \) is Haar-distributed unitary matrix of size \( K \times K \),
    
    \item \( \mathbf{V} \in \mathcal{U}(N) \) is a Haar-distributed unitary matrix of size \( N \times N \),
    
    \item \( \mathbf{D} \) is a diagonal matrix 
    comprising the singular values of \( \mathbf{H} \),

\item 
$f(\mathbf{D}) := \bigl(\operatorname{diag}\!\bigl(f(d_{1}),\dots,f(d_{K})\bigr)\ \,\mathbf{0}_{K\times(N-K)}\bigr)\in\mathbb{R}^{K\times N}$, here we are slightly abusing notation by writing
$f$ as a function of a matrix and a real number, in the forthcoming discussion, $f$ is referred to the latter, that is, a continuous real-valued function mapping from $\R_+$ to $\R_+$,


     \item \( \mathbf{s}\),
    \( \mathbf{P} \), \( q(\cdot) \),
\( \eta \) are defined in (\ref{eq:precodeing-x-eta-q}).
    
\end{itemize}
In the model (\ref{eq:y-precoding-eta-q}),
the quantities $\mathbf{U}$, $\mathbf{D}$ and  $\mathbf{V}$ are random given that $H$ is random,
$\mathbf{s}$ is a given data vector, 
$q(\cdot)$ is a given quantization function, 
while $\mathbf{P}$ and the scalar $\eta$ are design parameters and the function $f(\cdot)$ is a design mapping used to shape the precoder.
Specifically, given 
singular value decomposition (SVD) of 
$
\mathbf{H}= \mathbf{U}\mathbf{D}\mathbf{V}^{\mathsf{H}},
$
the precoding matrix $\mathbf{P}$ is constructed by
$
\mathbf{P}= \mathbf{V}\,f(\mathbf{D})^{\!\mathsf{T}}
\mathbf{U}^{\mathsf{H}}.
$
Consequently choosing a proper precoding matrix $\mathbf{P}$ is down to
setting a proper function $f(\cdot)$.

\subsection{A statistically equivalent model }

Model (\ref{eq:y-precoding-eta-q}) is highly complex and involves intricate dependencies between random matrices and quantization functions, making it difficult to analyze directly.
In \cite[Theorem 1]{WLS24}, the authors derive a statistically equivalent model for the system (\ref{eq:y-precoding-eta-q}) as follows: 
\begin{subequations}
\label{eq:haty-stat-equiv-model}
\begin{align}
\hat{\mathbf{y}} &= \eta T_s \mathbf{s} + \eta T_g \mathbf{g}_2 + \mathbf{n}, 
\label{eq:haty-stat-equiv-model-0}  
\\
\intertext{where}
T_s &= \frac{\mathbf{g}_1^\mathsf{H} (C_1 \mathbf{D} \hat{\mathbf{s}}_1 + C_2 \mathbf{D} \mathbf{B}(\hat{\mathbf{s}}_1) \mathbf{z}_2[2:N])}{\|\mathbf{g}_1\| \|\mathbf{s}\|} - T_g \frac{(\mathbf{R}(\mathbf{s})^{-1} \mathbf{g}_2)[1]}{\|\mathbf{s}\|}, 
\label{eq:haty-stat-equiv-model-a} 
\\
T_g &= \frac{\|\mathbf{B}(\mathbf{g}_1)^\mathsf{H} (C_1 \mathbf{D} \hat{\mathbf{s}}_1 + C_2 \mathbf{D} \mathbf{B}(\hat{\mathbf{s}}_1) \mathbf{z}_2[2:N])\|}{\|(\mathbf{R}(\mathbf{s})^{-1} \mathbf{g}_2)[2:K]\|}, 
\label{eq:haty-stat-equiv-model-b} 
\\
C_1 &= \frac{\mathbf{z}_1^\mathsf{H} q\left(\frac{\|\hat{\mathbf{s}}_1\|}{\|\mathbf{z}_1\|} \mathbf{z}_1\right)}{\|\hat{\mathbf{s}}_1\| \|\mathbf{z}_1\|}, \quad C_2 = \frac{\|\mathbf{B}(\mathbf{z}_1)^\mathsf{H} q\left(\frac{\|\hat{\mathbf{s}}_1\|}{\|\mathbf{z}_1\|} \mathbf{z}_1\right)\|}{\|\mathbf{z}_2[2:N]\|}, 
\label{eq:haty-stat-equiv-model-c} 
\\
\hat{\mathbf{s}}_1 &= \frac{\|\mathbf{s}\|}{\|\mathbf{g}_1\|} f(\mathbf{D})^\mathsf{T} \mathbf{g}_1, 
\label{eq:haty-stat-equiv-model-d} 
\end{align}
\end{subequations}
with
\begin{itemize}


    \item \( \hat{\mathbf{y}} \): The statistically equivalent received signal vector,
    \item \( \mathbf{g}_1, \mathbf{g}_2 \): Independent standard complex Gaussian random vectors, i.e., \( \mathbf{g}_1, \mathbf{g}_2 \sim \mathcal{CN}(\mathbf{0}, \mathbf{I_K}) \),
    \item \( \mathbf{z}_1, \mathbf{z}_2 \): Independent standard complex Gaussian random vectors, i.e., \( \mathbf{z}_1, \mathbf{z}_2 \sim \mathcal{CN}(\mathbf{0}, \mathbf{I_N}) \),
    \item \( \mathbf{B}(\cdot) \): A basis matrix for the orthogonal complement of the input vector (used in Householder transformations),
    \item \( \mathbf{R}(\cdot) \): The Householder transformation matrix,
    \item \( T_s \): The scaling factor for the signal term in the statistically equivalent model,
    \item \( T_g \): The scaling factor for the interference-plus-noise term in the statistically equivalent model,
    \item \( C_1, C_2 \): Constants derived from the quantization process and the input-output relationship of the system,
    \item \( \hat{\mathbf{s}}_1 \): An intermediate vector derived from the precoding process,
     
     \item \( \mathbf{s}\),
    \( \mathbf{P} \), \( q(\cdot) \),
\( \eta \), \( \mathbf{D} \), \( f(\cdot) \) are defined in (\ref{eq:y-precoding-eta-q}).
\end{itemize}
%
Formulation \eqref{eq:haty-stat-equiv-model}
is derived by virtue 
of random matrix theory and the Householder Dice (HD) technique.
The statistically 
equivalent model 
is motivated to simplify the analysis by transforming the original complex system into a more tractable form that is easier to analyze mathematically. Moreover, the theorem enables asymptotic analysis, which is crucial for understanding the behavior of the system as the number of transmit antennas ($N$) and the number of users ($K$) grow large with a fixed ratio. This asymptotic framework helps in deriving insights that are applicable to large-scale systems, which are typical in modern wireless communication scenarios. The next proposition 
is cited 
from \cite[Theorem 1]{WLS24}.

\begin{proposition}
\label{Prop:y-stat-equiv-haty}
When \( N \geq 3 \) and \( K \geq 3 \), the distribution of \( (\mathbf{y}, \mathbf{s}) \) in the original model (\ref{eq:y-precoding-eta-q}) is the same as that of \( (\hat{\mathbf{y}}, \mathbf{s}) \) specified by
\eqref{eq:haty-stat-equiv-model}.

\end{proposition}
    The proposition  explicitly 
ensures that, for every finite pair $(N,K)$ with $N,K\ge 3$,
 the joint distribution of $(\hat{\mathbf{y}},\mathbf{s})$ is identical to that of $(\mathbf{y},\mathbf{s})$; hence they share exactly the same statistical properties.
The difference between $\hat{\mathbf{y}}$ and $\mathbf{y}$ are as follows:
\begin{itemize}
  \item $\mathbf{y}$ is the original received signal generated with the true Haar-distributed matrices $\mathbf{U},\mathbf{V}$ and the original channel $\mathbf{H}$.
  \item $\hat{\mathbf{y}}$ is a statistically equivalent variable constructed solely from the independent Gaussian vectors $\mathbf{g}_1,\mathbf{g}_2,\mathbf{z}_1,\mathbf{z}_2$ and  $\mathbf{D},\mathbf{s},\mathbf{n}$; it no longer involves the Haar matrices explicitly, yet its joint distribution with $\mathbf{s}$ matches that of $(\mathbf{y},\mathbf{s})$.
  \end{itemize}

\subsection{An asymptotic version of the model} 

The statistically equivalent model \eqref{eq:haty-stat-equiv-model} 
contains the random scalars \(T_s\) and \(T_g\), whose 
true distributions are coupled with the data vector \(\mathbf{s}\) and the Gaussian vectors \(\mathbf{g}_1,\mathbf{g}_2,\mathbf{z}_1,\mathbf{z}_2\).  
To obtain a closed form 
of the key performance metrics,
we need to consider the large-system limit in which the numbers of transmit antennas \(N\) and users \(K\) grow without bound while their ratio \(N/K\) remains fixed.  
In this regime, the random variables \(T_s\) and \(T_g\) converge almost surely to deterministic limits \(\overline{T}_s\) and \(\overline{T}_g\), respectively, see the proof of \cite[Theorem 2]{WLS24}.  
This 
prompts 
introduction of 
a further simplified asymptotic model \(\bar{\mathbf{y}}\) given below, which captures the limiting behavior of \(\hat{\mathbf{y}}\) and enables more tractable analysis.
The asymptotic
version 
of the statistically equivalent model \eqref{eq:haty-stat-equiv-model} is defined as follows:

\begin{subequations}
\label{eq:bary-stat-equiv-v1-0}
\begin{align}
\label{eq:bary-stat-equiv-v1-a}
\bar{\mathbf{y}} &:= \eta \overline{T}_s \mathbf{s} + \eta \overline{T}_g \mathbf{g}_2 + \mathbf{n},\\
\intertext{where}
\overline{T}_s &= \overline{C}_1 \mathbb{E}[df(d)],
\label{eq:bary-stat-equiv-v1-b}
\\
\overline{T}_g &= \sqrt{\sigma_s^2 [\overline{C}_1]^2 \text{var}[df(d)] + \overline{C}_2^2},
\label{eq:bary-stat-equiv-v1-c}
\\
\overline{C}_1 &= \frac{\mathbb{E}[Z^\dag q(\alpha Z)]}{\alpha},\\
\overline{C}_2 &= \sqrt{\mathbb{E}[|q(\alpha Z)|^2] - |\mathbb{E}[Z^\dag q(\alpha Z)|^2]},
\label{eq:bary-stat-equiv-v1-d}
\\
\alpha &= \sqrt{\frac{\sigma_s^2 \mathbb{E}[f^2(d)]}{\gamma}},
\label{eq:bary-stat-equiv-v1-e}
\end{align}
\end{subequations}
$Z \sim \mathcal{CN}(0, 1)$, $d = \sqrt{\lambda}$, and $\lambda$ follows the Marchenko-Pastur distribution, whose probability density function is given by
\begin{equation}
p_\lambda(x) = \frac{\sqrt{(x-a)_+ (b-x)_+}}{2\pi c x}
\end{equation}
with $a = (1-\sqrt{c})^2$, $b = (1+\sqrt{c})^2$, $c = \frac{1}{\gamma}$; $(x)_+ = \max\{x, 0\}$. 

Compared 
to model (\ref{eq:haty-stat-equiv-model}), 
model  (\ref{eq:bary-stat-equiv-v1-0}) 
replaces the random scaling factors \(T_s\) and \(T_g\) with their deterministic limits \(\overline{T}_s\) and \(\overline{T}_g\).
While model (\ref{eq:haty-stat-equiv-model}) is exact for any finite \((N,K)\), 
model (\ref{eq:bary-stat-equiv-v1-0}) becomes accurate only when \(N,K\to\infty\) with asymptotic fixed ratio \(\gamma\). 
Model \eqref{eq:bary-stat-equiv-v1-0} is 
known as 
asymptotic model because it describes the
behavior of system \eqref{eq:haty-stat-equiv-model}
in the asymptotic regime, specifically as the number of antennas $N$ and the number of users $K$ grow infinitely large while maintaining a asymptotic fixed ratio
$\gamma = \frac{N}{K} > 1$, see Proposition \ref{Prop-asymp-y-WLS24} below. The asymptotic model enables closed-form optimization of precoders (e.g., optimal regularized zero-forcing) without resorting to Monte Carlo simulations. It provides  approximations for systems with hundreds of antennas, even though $N$ and $K$ are finite.
Thus, \eqref{eq:bary-stat-equiv-v1-0} is ``asymptotic'' because it captures the limiting behavior as system dimensions scale up, bridging theoretical tractability with practical massive MIMO deployments.

Next, we 
recall
the asymptotic analysis of  the statistically equivalent mode (\ref{eq:haty-stat-equiv-model}) to asymptotic model
(\ref{eq:bary-stat-equiv-v1-0})
established by  Wu et al.~\cite{WLS24} under a number of key assumptions.

\begin{assumption}
\label{Assu:H-n-s}
The entries of $\mathbf{H}$ and $\mathbf{n}$ are independently drawn from $\mathcal{CN}\left(0,\frac{1}{N}\right)$ and $\mathcal{CN}\left(0,\sigma^{2}\right)$, respectively. The entries of $\mathbf{s}$ are independently and uniformly drawn from a finite set $\mathcal{S}_{M}$ with nonzero elements (i.e., $0\notin\mathcal{S}_{M}$), that is, $\mathcal{S}_{M}=\{s^{(1)},s^{(2)},\cdots, s^{(M)}\}$ and $\mathbb{E}[|s_{1}|^{2}]=\sigma_{s}^{2}$. Furthermore, $\mathbf{H},\mathbf{s}$, and $\mathbf{n}$ are mutually independent.
\end{assumption}

\begin{assumption}
\label{Assu:funct-f}
The function 
{\color{black} $f(\cdot)$ in the original model (\ref{eq:y-precoding-eta-q})} is positive, continuous almost everywhere (a.e.), and bounded on any compact set of $(0,\infty)$.
\end{assumption}

\begin{assumption}
\label{Assu:-funct-q}
The quantization function {\color{black} $q(\cdot)$ in the original model (\ref{eq:y-precoding-eta-q}) } is continuous a.e. and there exists $M_0>0$ such that $|q(z)|\leq M_0$ for any $z\in\mathbb{C}.$ 
Moreover, its discontinuity set  is a union of finitely many lines and rays.
\end{assumption}

\begin{assumption}
\label{Assu:ratio-K-N}
The ratio of the number of antenna and the number of users 
 $\frac{N}{K} \rightarrow \gamma \in (1, \infty)$ as  $N, K \rightarrow \infty$.
\end{assumption}


It might be helpful to 
make some comments on the above assumptions.
The i.i.d.\ complex-Gaussian channel $\mathbf{H}$ and  $\mathbf{n}$ follow the widely-accepted Rayleigh-fading model for rich-scattering environments, see \cite{Tse2005}.  Independent symbols in $\mathbf{s}$ are standard in massive MIMO analyses \cite{R13}.
Requiring $f(\cdot)$ to be positive, a.e.\ continuous and locally bounded guarantees non-zero transmit energy and preserves the Marchenko–Pastur spectrum, as used in \cite{Couillet2011}.
Continuity and boundedness of $q(\cdot)$ are mild regularity conditions satisfied by uniform and constant-envelope quantizers; they enable the Bussgang decomposition employed in \cite{WLS24}.
The asymptotic fixed antenna-to-user ratio $\gamma>1$ as $N,K\to\infty$ is the standard massive MIMO regime yielding deterministic equivalents, cf.\ \cite{HBD13} and \cite{Couillet2011}.
Under Assumptions \ref{Assu:H-n-s}-\ref{Assu:ratio-K-N}, Wu et al.~establish 
the relationship between 
model (\ref{eq:haty-stat-equiv-model}) 
and model
(\ref{eq:bary-stat-equiv-v1-0}) as stated in the next proposition.

\begin{proposition}(\cite[Theorem 2]{WLS24})
\label{Prop-asymp-y-WLS24} 
Let $\hat{y}_k$ and $\bar{y}_k$ be the $k$-th element of $\hat{\mathbf{y}}$ and $\bar{\mathbf{y}} $ defined as in (\ref{eq:haty-stat-equiv-model}) and (\ref{eq:bary-stat-equiv-v1-0}),respectively. $s_k$ is the  $k$-th element of $\mathbf{s} $.
Under Assumptions \ref{Assu:H-n-s}-\ref{Assu:ratio-K-N}, the following holds as $K \rightarrow \infty$,
\begin{equation}
(\hat{y}_k, s_k) \overset{a.s.}{\longrightarrow} (\bar{y}_k, s_k), \quad \forall k \in [K].
\end{equation}
\end{proposition}


The proposition ensures that 
 as $K \to \infty$, the system exhibits a self-averaging property where channel randomness vanishes for individual users. This occurs because the massive degrees of freedom ($N > K$) average out random fluctuations, causing per-user metrics to collapse to deterministic values.
     The joint convergence $(\hat{y}_k, s_k) \overset{a.s.}{\longrightarrow} (\bar{y}_k, s_k)$ implies that the statistically equivalent model $\hat{y}_k$ and the asymptotic model $\bar{y}_k$ become interchangeable for performance analysis. Specifically, key metrics like SINR and SEP (Definition \ref{def-SINR-SEP}) derived from $\bar{y}_k$ are asymptotically exact for finite-dimensional systems.
The result in  Proposition \ref{Prop-asymp-y-WLS24} bridges stochastic finite-dimensional analysis with deterministic asymptotic models, enabling tractable optimization for practical massive MIMO deployments.

Now we have 
introduced three models (\ref{eq:y-precoding-eta-q}), (\ref{eq:haty-stat-equiv-model}), and  (\ref{eq:bary-stat-equiv-v1-0})
for MIMO.
Model (\ref{eq:y-precoding-eta-q}) is the original system model, Model (\ref{eq:haty-stat-equiv-model}) is a statistically equivalent model that simplifies the analysis, and Model (\ref{eq:bary-stat-equiv-v1-0}) is the asymptotic model that further simplifies the statistically equivalent model to provide insights into the system's performance in the large system limit. These models together provide comprehensive frameworks for analyzing and optimizing linear-quantized precoding in massive MIMO systems.

\subsection{SINR and SEP}
 In performance evaluation of wireless communication systems, Signal-to-Interference-plus-Noise Ratio (SINR) and Symbol Error Probability (SEP) are two 
important metrics of system performances. SINR quantifies the ratio of desired signal power to the combined power of interference and noise, directly influencing achievable data rates and link reliability. SEP, on the other hand, characterizes the probability of erroneously detecting a transmitted symbol under a given modulation scheme. Especially in massive MIMO systems, asymptotic analyses of SINR and SEP provide crucial insights into 
the
system behavior as the numbers of antennas and users grow large \cite{WLS24}.
  The next definition describe the two metrics.
 
 \begin{definition}
 \label{def-SINR-SEP}
 \textbf{The signal-to-interference-plus-noise ratio (SINR)} is defined as
\begin{equation}\label{HAT-SINR}
\widehat{\text{SINR}}_{k}:=\frac{|\rho_{k}|^{2}\,\mathbb{E}[|s_{k}|^{2}]}{\mathbb{E}[|\hat{y}_{k}|^{2}]-|\rho_{k}|^{2}\,\mathbb{E}[|s_{k}|^{2}]}
\end{equation}  for $k=1,2,\cdots, K$,  
where $\rho_{k}=\mathbb{E}[s_{k}^{\dagger}\hat{y}_{k}]/\mathbb{E}[|s_{k}|^{2}]$, $\hat{y}_k$ and $s_k$ are the $k$-th element of $\hat{\mathbf{y}}$ and $\mathbf{s}$  defined as in (\ref{eq:haty-stat-equiv-model}) respectively.  
\textbf{The  symbol error probability (SEP)} is defined as
\begin{equation}\label{HAT-SEP}
\widehat{\text{SEP}}_{k}(\beta):=\mathbb{P}\left(\text{dec}(\beta \hat{y}_{k})\neq s_{k}\right)
\end{equation} for $k=1,2,\cdots, K$,  
where  \(\beta\) represents a scaling factor used for signal processing at the receiver, and it is a complex number defined by
\[
\beta = \frac{T_s^\dagger}{\eta |T_s|^2},
\]
$\text{dec}(\cdot)$ is 
a decision function which 
maps its argument to the nearest constellation point in $\mathcal{S}_{M}$,
the set of possible constellation points for the modulation scheme used, such as Quadrature Amplitude Modulation (QAM) or Phase-Shift Keying (PSK).
 \end{definition}

{\color{black}
In system-level design, $\widehat{\text{SINR}}_{k}$ serves as the central performance indicator for user $k$'s downlink, directly determining the achievable rate $R_k = \log_2(1 + \widehat{\text{SINR}}_{k})$ and communication reliability \citep{R13}.
This metric quantifies the ratio of effective signal power to the aggregate power of residual multi-user interference, quantization noise, and thermal noise via the equivalent channel gain $\rho_k$ \citep{Lu2014}. In practical deployments with finite antenna arrays, the statistical stability of $\widehat{\text{SINR}}_{k}$ critically influences the robustness of power allocation, user scheduling, and beamforming algorithms \citep{Couillet2011}. Particularly when the base station employs low-resolution DACs, signal quantization introduces nonlinear distortion that reduces the mean of $\widehat{\text{SINR}}_{k}$ and increases its variance \citep{WLS24}. Real-time estimation of $\widehat{\text{SINR}}_{k}$ is therefore leveraged for link adaptation (e.g., modulation and coding scheme selection) and resource scheduling; for instance, in 5G NR \citep{3GPP38.104}, threshold-based triggers of channel-state-information (CSI) feedback and handover procedures ensure that enhanced mobile broadband (eMBB) throughput requirements are met.

$\widehat{\text{SEP}}_{k}(\beta)$ denotes the symbol-error probability for user $k$ and is directly linked to communication reliability \citep{Proakis2008}. At the receiver, the complex-valued scaling factor $\beta$ calibrates the amplitude and phase of the observed signal $\hat{y}_k$ prior to mapping it onto the modulation constellation $\mathcal{S}_M$ (e.g., QAM or PSK) via a decision function \citep{Gallager2008}. Low-resolution DACs induce two primary impairments: (1) nonlinear distortion resulting from quantizing continuous-valued signals to a finite set of voltage levels, which causes constellation warping and pronounced error-floor effects, especially for high-order QAM \citep{Jacobsson2017}; and (2) noise enhancement due to residual quantization noise superimposed on channel noise, significantly degrading $\widehat{\text{SEP}}_{k}(\beta)$ in low-SNR regimes \citep{WLS24}. In implementation, $\beta$ must be jointly optimized with the precoder and the receive equalizer. For example, enhanced mobile broadband (eMBB) services mandate $\widehat{\text{SEP}}_{k} \leq 10^{-2}$, whereas ultra-reliable low-latency communication (URLLC) scenarios require $\widehat{\text{SEP}}_{k} \leq 10^{-5}$ \citep{3GPP38.104}. These requirements drive dynamic adjustment of $\beta$ and the precoding strategy, complemented by hybrid automatic repeat request (HARQ) redundancy in the frame structure to counteract sporadic errors \citep{3}.}



As shown in \cite{R11}, the SINR and SEP are crucial metrics for evaluating the performance of communication systems. The SINR represents the ratio of the desired signal power to the sum of interference and noise powers. It is a key indicator of signal quality and directly affects the data rate and reliability of the communication link. In engineering, a higher SINR implies better signal quality, less interference, and better communication performance. It is widely used in the design and optimization of wireless communication systems, such as resource allocation, power control, and beamforming, to improve system capacity and user experience. The SEP measures the probability that a transmitted symbol is incorrectly decoded at the receiver. It reflects the reliability of the communication system and is closely related to the modulation scheme and channel conditions. In engineering, a lower SEP indicates higher communication reliability and is essential for ensuring the quality of service (QoS) in applications such as voice, video, and data transmission. Both SINR and SEP play significant roles in the analysis and optimization of massive MIMO systems, helping to achieve efficient resource utilization and robust communication performance.

To 
validate the predictive accuracy of the asymptotic scalar model in  (\ref{eq:bary-stat-equiv-v1-0}) for finite-size systems, it is necessary to establish that the finite-system quantities $\widehat{\text{SINR}}_{k}$ and $\widehat{\text{SEP}}_{k}(\beta)$ converge almost surely to their asymptotic limits $\overline{\text{SINR}}$ and $\overline{\text{SEP}}(\beta)$, respectively.  This convergence guarantees that the analytical expressions derived under the large-system limit retain their accuracy in practical dimensions, thereby confirming the design and optimization efficacy of the asymptotic framework.  The following proposition, which is \cite[Theorem 3]{WLS24}, formally establishes the asymptotic convergence of SINR and SEP, providing a theoretical support for the validity of linear-quantized precoding schemes in the asymptotic regime.

\begin{proposition}
\label{Prop:SINR-SEP-as-convg-WLS24}
Under  Assumptions \ref{Assu:H-n-s}-\ref{Assu:ratio-K-N},
  the following hold for any $k\in[K]$ and $\beta\in\mathbb{C}$,
\begin{itemize}
    \item[(i)] $\widehat{\text{SINR}}_{k}{\longrightarrow} \overline{\text{SINR}}$ as $K\to\infty$;
    \item[(ii)] $\widehat{\text{SEP}}_{k}(\beta){\longrightarrow} \overline{\text{SEP}}(\beta)$
    as $K\to\infty$,
\end{itemize}
where the  $\widehat{\text{SINR}}_{k}$ and $\widehat{\text{SEP}}_{k}(\beta)$ are the SINR and SEP of user $k$ of the model in (\ref{eq:haty-stat-equiv-model}) defined in (\ref{HAT-SINR}) and (\ref{HAT-SEP}), respectively. SINR and SEP of the scalar asymptotic model (\ref{eq:bary-stat-equiv-v1-0}) are given by
\begin{equation}\label{BAR-SINR}
\overline{\text{SINR}}=\frac{\sigma_{\mathrm{s}}^{2}\,\eta^{2}\,\overline{T}_{\mathrm{s}}^{2}}{\eta^{2}\,\overline{T}_{\mathrm{g}}^{2}+\sigma^{2}}=\frac{\mathbb{E}^{2}[df(d)]}{\text{var}[df(d)]+\phi(\alpha,\eta)\,\frac{\mathbb{E}[f^{2}(d)]}{\gamma}}
\end{equation}
with $\alpha=\sqrt{\sigma_{\mathrm{s}}^{2}\,\mathbb{E}[f^{2}(d)]/\gamma}$ and
\begin{equation}
\phi(\alpha,\eta):=\frac{\mathbb{E}[|q(\alpha Z)|^{2}]-|\mathbb{E}[Z^{\dagger}q(\alpha Z)]|^{2}+\sigma^{2}/\eta^{2}}{|\mathbb{E}[Z^{\dagger}q(\alpha Z)]|^{2}},
\end{equation}
and
\begin{equation}\label{SEP-BAR}
\overline{\text{SEP}}(\beta)=\mathbb{P}\left(\text{dec}(\beta\bar{y})\neq s\right).
\end{equation}
\end{proposition}



To facilitate understanding, we briefly clarify the two SEP quantities and interpret the convergence from an engineering viewpoint.
 $\widehat{\text{SEP}}_{k}(\beta)$ is the actual symbol-error probability experienced by user~$k$ in a finite system with $K$ users and $N$ antennas.  It is obtained by Monte-Carlo averaging over channel realizations and noise.
$\overline{\text{SEP}}(\beta)$ is the analytical SEP predicted by the asymptotic scalar model (\ref{eq:bary-stat-equiv-v1-0}).  It depends only on system parameters (quantizer type, load factor $\gamma=N/K$,  modulation order) and is independent of the specific channel draw.
The almost-sure convergence  
$\widehat{\text{SEP}}_{k}(\beta)\xrightarrow{\text{a.s.}}\overline{\text{SEP}}(\beta)$  
states that, as the numbers of antennas and users grow with fixed ratio $\gamma$, the real-world error rate of every user becomes indistinguishable from the closed-form prediction $\overline{\text{SEP}}(\beta)$.  Consequently, system designers can rely on the simple scalar expression $\overline{\text{SEP}}(\beta)$ to tune precoders, select DAC resolutions, or dimension arrays without costly simulations. The convergence holds per user, not merely on average, implying fairness, that is, no user is left with significantly worse error performance when the system is large.
The result bridges theory and practice: although derived in the asymptotic regime, the expression $\overline{\text{SEP}}(\beta)$ is shown to be accurate for realistic, finite dimensions (see Section VI in \cite{WLS24}).

While Proposition \ref{Prop:SINR-SEP-as-convg-WLS24} establishes that the finite-system SEP and SINR converge almost surely to their asymptotic limits as \(K \to \infty\), it provides no quantitative guarantees on the stability of this convergence for finite dimensions \((N, K)\). A rigorous quantitative stability analysis (error bounds, convergence rates et al.) remains an open and vital problem for bridging the gap between asymptotic theory and practical finite-dimensional system design. This stability gap warrants significant further investigation.

\section{
Quantifying the effect 
of model change 
on 
SEP and SINR
}



 In the context of massive MIMO systems, performance metrics such as SEP and SINR are crucial for evaluating efficiency and reliability. While asymptotic analysis provides valuable insights into the limiting behavior of these metrics as \(N, K \to \infty\) (with \(N/K \to \gamma\)), it does not directly address the practical challenge of  quantifying the gap between asymptotic limits and finite-dimensional performance where \(N\) and \(K\) are finite. 
 This section paves the way for the explicit quantification of this gap in Sections~4 and~5, marking a significant departure from  \cite{WLS24}: from asymptotic convergence analysis to quantitative convergence which quantifies effect on SEP and 
SINR when $\hat{y}$ is replaced 
with $\bar{y}$.
To bridge this gap, we conduct a quantitative stability analysis of SEP and SINR, deriving bounds that characterize deviations from asymptotic counterparts. This ensures that asymptotic performance guarantees are applicable to real-world systems with finite dimensions, providing a rigorous foundation for understanding and optimizing massive MIMO performance in practical settings.
The next lemma
establishes the convergence of a special series of random variables under the Ky Fan metric (see e.g.~\cite{F43,Z84})
\begin{equation}\label{essential distance}
\dd_{KF}(X, Y) = \inf\{\delta > 0 : \mathbb{P}(|X - Y| > \delta) <\delta\}
\end{equation}
for two random variables $X,Y$.

\begin{lemma}
\label{Lem:T_Ng-convg-Tg}
Let \( T_N \) be a sequence of random variables converging to \( T \) almost surely as \( N \to \infty \), and let \( g \) be a standard complex normal random variable (not necessarily independent of \( T_N \) or \( T \)). Then
\bgeqn
\label{eq:KF-TNg-to-Tg}
\lim_{N \to \infty} \dd_{KF}(T_N g, T g) = 0.
\edeqn 
\end{lemma}
\begin{proof}
We need to show that for any \(\epsilon > 0\), there exists \( N_0 \in \mathbb{N} \) such that for all \( N > N_0 \), \(\dd_{KF}(T_N g, T g) < \epsilon \). It suffices to show 
that for any \(\epsilon > 0\),
there exists \( N_0 \in \mathbb{N} \) such that for all \( N > N_0 \), 
$\mathbb{P}(|X - Y| > \epsilon ) <\epsilon$.
Let 
 $\eta =\frac{\epsilon}{\sqrt{-2 \ln\left(\frac{\epsilon}{2}\right)}}, (0<\epsilon<2).$ 
Observe that
\bgeq 
\mathbb{P}(|T_N g - T g| > \epsilon) = \mathbb{P}(|T_N g - T g| > \epsilon \text{ and } |T_N - T| > \eta) + \mathbb{P}(|T_N g - T g| > \epsilon \text{ and } |T_N - T| \leq \eta).
\edeq 
Since
\bgeq 
\mathbb{P}(|T_N g - T g| > \epsilon \text{ and } |T_N - T| > \eta) \leq \mathbb{P}(|T_N - T| > \eta).
\edeq 
and 
\bgeq 
\mathbb{P}(|T_N g - T g| > \epsilon \text{ and } |T_N - T| \leq \eta) \leq \mathbb{P}\left(|g| > \frac{\epsilon}{\eta} \text{ and } |T_N - T| \leq \eta\right), 
\edeq 
then 
\[
\mathbb{P}(|T_N g - T g| > \epsilon) \leq \mathbb{P}(|T_N - T| > \eta) + \mathbb{P}\left(|g| > \frac{\epsilon}{\eta} \text{ and }  |T_N - T| \leq \eta\right).
\]
By almost sure convergence of \( T_N \to T \) , for above \(\eta > 0\), there exists \( N_0 \) such that for \( N > N_0 \),
\[
\mathbb{P}(|T_N - T| > \eta) < \frac{\epsilon}{2}.
\]
By Remark \ref{re-exp},      $\left |g\right |^2$ follows an exponential distribution with rate parameter $1$, so we have
\[
\mathbb{P}\left(|g| > \frac{\epsilon}{\eta}\right) = e^{-\left(\frac{\epsilon^2}{2\eta^2}\right)}=\frac{\epsilon}{2}.
\]
 Therefore,
\[
\mathbb{P}\left(|g| > \frac{\epsilon}{\eta} \text{ and }  |T_N - T| \leq \eta\right) \leq \mathbb{P}\left(|g| > \frac{\epsilon}{\eta}\right) =\frac{\epsilon}{2}.
\]
For \( N > N_0 \),
\[
\mathbb{P}(|T_N g - T g| > \epsilon) \leq \mathbb{P}(|T_N - T| > \eta) + \mathbb{P}\left(|g| > \frac{\epsilon}{\eta}\right) < \frac{\epsilon}{2} + \frac{\epsilon}{2} = \epsilon.
\]
Since \(\epsilon > 0\) is arbitrary, we
arrive at \eqref{eq:KF-TNg-to-Tg}.
\end{proof}
\begin{remark}
\label{re-kf-ess}
The Ky Fan metric in Lemma \ref{Lem:T_Ng-convg-Tg} cannot be replaced by the essential infimum  metric 
\cite{Z84}  defined by 
\bgeqn 
\dd_{ess}(X, Y) = \inf\{\delta > 0 : \mathbb{P}(|X - Y| >\delta) =0\}.
\edeqn 
 Consider 
 \(T_N = 1/N\to 0\). 
 Then \(\dd_{ess}(T_N g, 0) = \inf\{\delta > 0 : \mathbb{P}(|g|/N  > \delta) = 0\}\). Since \(|g|\) is unbounded, then \(\mathbb{P}(|g|/N > \delta) > 0\) 
  for 
  arbitrarily large \(
 \delta > 0
\), and consequently 
\(\dd_{ess}(T_N g, 0) = \infty\) for all \(N\).
\end{remark}

{\color{black}We are now ready to present the first main result of this section which provides a bound on
the difference between $
\widehat{\text{SEP}}_k(\beta)$ and $  \overline{\text{SEP}}(\beta)$
in terms of the Ky Fan distance 
between $T_s$ and $\bar{T}_s$ and 
the Ky Fan distance 
between $T_gg_{2}[k]$ and $\overline{T}_gg_{2}[k]$.
}

\begin{theorem}[Discrepancy between 
$\widehat{\text{SEP}}_k(\beta)$ and $\overline{\text{SEP}}(\beta)$ 
] 
\label{th-SEP-KF}
Let $\widehat{\text{SEP}}_k(\beta)$ and $\overline{\text{SEP}}(\beta)$ be defined as in (\ref{HAT-SEP}) and (\ref{SEP-BAR}), let 
\begin{equation}
\label{SEP ERROR-L}
{\color{black}
L_m}=\left(\frac{\sqrt{\pi}}{\ |\beta \eta \overline{T}_g|^2 + |\beta|^2 \sigma^2}+1\right)\max\Big\{|\beta| \eta s^{(m)}+1, |\beta| \eta+1\Big\}
\end{equation} 
 {\color{black} for $m=1,2,\cdots, M$, where \( M \)
is the number of distinct nonzero elements in the finite set \(\mathcal{S}_M\) 
from which the entries of the vector \(\mathbf{s}\) are uniformly drawn 
in Assumption~\ref{Assu:H-n-s}. }
Under  Assumptions \ref{Assu:H-n-s}-\ref{Assu:ratio-K-N}, there exists $\bar{K}>0$ such that for $K>\bar{K}$

\bgeqn \label{th3.1-inequ}
|\widehat{\text{SEP}}_k(\beta) - \overline{\text{SEP}}(\beta)| \leq \frac{1}{M} \sum_{m=1}^{M} L_m \, \Big(\dd_{KF}(T_s,\overline{T}_s) + \dd_{KF}(T_gg_{2}[k],\overline{T}_gg_{2}[k]) \Big)
\edeqn 
for all $k\in [K]$.
\end{theorem}
We know from the proof of \cite[Theorem 2]{WLS24} that $T_s$ 
converges to $\overline{T}_s$ almost surely as $K$ 
goes to infinity, which means that $T_s$ 
converges to $\overline{T}_s$ in probability 
and hence $\dd_{KF}(T_s,\overline{T}_s)\to 0$. 
By Lemma~\ref{Lem:T_Ng-convg-Tg}, $\dd_{KF}(T_gg_{2}[k],\overline{T}_gg_{2}[k])\to 0$ 
as $K\to \infty$. 
Since both \(\dd_{KF}\) terms \(\to 0\) as \(K \to \infty\), 
the bound guarantees that the SEP deviation vanishes for large systems, which validates the asymptotic model for practical scaling. The bound directly links the SEP deviation \(|\widehat{\text{SEP}}_k - \overline{\text{SEP}}|\) to the distributional discrepancies in signal/noise components. This provides a quantitative and causal understanding of how finite-dimensional imperfections propagate to system performance. For system designers, this theorem ensures that optimizing \(\overline{\text{SEP}}\)
(via precoders in \cite[Section V]{WLS24}) directly benefits finite-dimensional systems, as deviations are controlled by decaying \(\dd_{KF}\) terms. The bound further identifies which components (signal gain \(T_s\) or noise term \(T_g g_2[k]\)) dominate SEP instability, guiding hardware/algorithm refinements.

\begin{proof}
Given a constellation symbol $s$, we use $D_s$ to denote its decision region, i.e., the region within which $s$ will be recovered. With this notation, $\widehat{\text{SEP}}_k(\beta)$ can be expressed as

\begin{equation}\label{th-sepk}
\begin{array}{lll}\widehat{\text{SEP}}_k(\beta) &=& 1 - \mathbb{P}\left(\beta \hat{y}_k \in D_{s_k}\right)\\[12pt]
&=& 1 - \displaystyle\sum_{m=1}^{M} \mathbb{P}\left(s_k = s^{(m)}\right) \mathbb{P}\left(\beta \hat{y}_k \in D_{s_k} \mid s_k = s^{(m)}\right)\\[12pt]
&=& 1 - \displaystyle\frac{1}{M} \sum_{m=1}^{M} \mathbb{P}\left(\beta \hat{y}_k \in D_{s_k} \mid s_k = s^{(m)}\right),
\end{array}
\end{equation}
where the third equality holds since $s_k$ is uniformly drawn from $\mathcal{S}_M=\{s^{(1)},s^{(2)},\cdots, s^{(M)}\}$ and
\begin{equation}
D_{s_k} = \left\{ r \mid |r-s_k|^2 < |r-s^{(i)}|^2, \forall s^{(i)} \in \mathcal{S}_M, s^{(i)} \neq s_k \right\}.
\end{equation}
Likewise, 
\begin{equation}\label{th-sepk-bar}
\overline{\text{SEP}}(\beta) = 1 - \frac{1}{M} \sum_{m=1}^{M}  \mathbb{P}\left(\beta \bar{y}_k \in D_{s_k} \mid s_k = s^{(m)}\right).
\end{equation}
We proceed the rest of the proof in 
two steps.

\textbf{Step 1.} Estimate  $ \mathbb{P}\left(\beta \hat{y}_k \in D_{s_k} \mid s_k = s^{(m)}\right)-\mathbb{P}\left(\beta \bar{y}_k \in D_{s_k} \mid s_k = s^{(m)}\right)$.

Recall that $\beta \hat{y}_k = \beta (\eta T_s s_k + \eta T_g g_{2}[k] + n_k)$. Then for fixed $\tau>0$,
$$
\begin{array}{l}
\mathbb{P}\left(\beta \hat{y}_k \in D_{s_k} \mid s_k = s^{(m)}\right)\\[12pt]
=\mathbb{P}\left( \beta (\eta T_s s^{(m)} + \eta T_g g_{2}[k] + n_k) \in D_{s^{(m)}}\right)\\[12pt]
=\mathbb{P}\left( \beta (\eta \overline{T}_s s^{(m)} + \eta \overline{T}_g g_{2}[k] + n_k)+\beta \eta(T_s-\overline{T}_s)s^{(m)}+ \beta \eta(T_g-\overline{T}_g)g_{2}[k]\in D_{s^{(m)}}\right)\\[12pt]
 \leq  \mathbb{P}\left( \beta (\eta \overline{T}_s s^{(m)} + \eta \overline{T}_g g_{2}[k] + n_k)+\beta \eta(T_s-\overline{T}_s)s^{(m)}+ \beta \eta(T_g-\overline{T}_g)g_{2}[k]\in D_{s^{(m)}} \mid | \xi_k|\leq\tau\right)\\[12pt]
 \quad +\mathbb{P}\left(| \xi_k|>\tau\right),\\[12pt]
  \leq \mathbb{P}\left( \beta (\eta \overline{T}_s s^{(m)} + \eta \overline{T}_g g_{2}[k] + n_k)\in D_{s^{(m)}}+\tau \mathbb{B}_2\right)+\mathbb{P}\left(| \xi_k|>\tau\right)\\[12pt]
  =\mathbb{P}\left(\beta \bar{y}_k \in D_{s_k}+\tau \mathbb{B}_2 \mid s_k = s^{(m)}\right)+P\left(| \xi_k|>\tau\right),
\end{array}
$$
where $\xi_k=\beta \eta(T_s-\overline{T}_s)s^{(m)}+ \beta \eta(T_g-\overline{T}_g)g_{2}[k].$
Thus
\begin{equation}\label{th-1}
\begin{array}{l}
\mathbb{P}\left(\beta \hat{y}_k \in D_{s_k} \mid s_k = s^{(m)}\right)-\mathbb{P}\left(\beta \bar{y}_k \in D_{s_k} \mid s_k = s^{(m)}\right)\\[12pt]
  \leq \mathbb{P}\left(\beta \bar{y}_k \in D_{s_k}+\tau \mathbb{B}_2 \mid s_k = s^{(m)}\right)-\mathbb{P}\left(\beta \bar{y}_k \in D_{s_k} \mid s_k = s^{(m)}\right)+ \mathbb{P}\left(| \xi_k|>\tau\right)\\[12pt]
  =\mathbb{P}\left( \beta (\eta \overline{T}_s s^{(m)} + \eta \overline{T}_g g_{2}[k] + n_k)\in (D_{s^{(m)}})^{\tau}\backslash D_{s^{(m)}}\right)+ \mathbb{P}\left(| \xi_k|>\tau\right).
\end{array}
\end{equation}
Since $g_2 \sim \mathcal{CN}(0, I)$ and $n_k \sim \mathcal{CN}(0, \sigma^2)$ are independent, then
$$ \beta (\eta \overline{T}_s s^{(m)} + \eta \overline{T}_g g_{2}[k] + n_k)\sim\mathcal{CN}\left(
\beta \eta \overline{T}_s s^{(m)},
\ |\beta \eta \overline{T}_g|^2 + |\beta|^2 \sigma^2
\right),$$
which, by Lemma \ref{le4.1} and (\ref{th-1}), 
implies that 
\begin{equation}\label{th-SEP-bound-2}
    \mathbb{P}\left(\beta \hat{y}_k \in D_{s_k} \mid s_k = s^{(m)}\right)-\mathbb{P}\left(\beta \bar{y}_k \in D_{s_k} \mid s_k = s^{(m)}\right)\leq \left(\frac{\sqrt{\pi}}{\ |\beta \eta \overline{T}_g|^2 + |\beta|^2 \sigma^2}+1\right)\tau+ \mathbb{P}\left(| \xi_k|>\tau\right).
\end{equation}
By the proof of \cite[Theorem 2]{WLS24}, 
$$
T_s\xrightarrow{a.s.}\overline{T}_s\quad T_g\xrightarrow{a.s.}\overline{T}_g \quad \mbox{as}\quad K \rightarrow +\infty,
$$
which, by the definition of Ky Fan metric in (\ref{essential distance}) and Lemma \ref{Lem:T_Ng-convg-Tg}, 
implies that
$$
\dd_{KF}(T_s,\overline{T}_s) \rightarrow 0\quad \mbox{and}\quad 
\dd_{KF}(T_gg_{2}[k],\overline{T}_gg_{2}[k])\ \rightarrow 0
$$
as $K \rightarrow +\infty$.
Thus 
$$
 \tau_{N,K}:= |\beta| \eta  \dd_{KF}(T_s,\overline{T}_s) s^{(m)}+ |\beta |\eta \dd_{KF}(T_gg_{2}[k],\overline{T}_gg_{2}[k]) \rightarrow 0
$$
as $K \rightarrow +\infty$.
So we have for fixed $N,K$, $\tau_{N,K}$ is a finite  positive number and 
\begin{equation}\label{th-SEP-bound-3}
\begin{aligned}
& \mathbb{P}\left(| \xi_k|>\tau_{N,K}\right)\\[12pt]
&\leq  \mathbb{P}\left(| \beta \eta(T_s-\overline{T}_s)s^{(m)}|>|\beta| \eta  \dd_{KF}(T_s,\overline{T}_s) s^{(m)}\right)+ \mathbb{P}\left(| \beta \eta(T_g-\overline{T}_g)g_{2}[k]|>|\beta| \eta  \dd_{KF}(T_gg_{2}[k],\overline{T}_gg_{2}[k])\right) \\[12pt]
&= \mathbb{P}\left(|T_s-\overline{T}_s|>  \dd_{KF}(T_s,\overline{T}_s)\right)+ \mathbb{P}\left(| (T_g-\overline{T}_g)g_{2}[k]|>\dd_{KF}(T_gg_{2}[k],\overline{T}_gg_{2}[k])\right)\\[12pt]
&\leq  \dd_{KF}(T_s,\overline{T}_s)+ \dd_{KF}(T_gg_{2}[k],\overline{T}_gg_{2}[k]).
\end{aligned}
\end{equation}
Let $\tau=\tau_{N,K}$ in (\ref{th-SEP-bound-2}), we have from (\ref{th-SEP-bound-2}) and (\ref{th-SEP-bound-3}) that
\begin{equation}\label{th-2}
 \mathbb{P}\left(\beta \hat{y}_k \in D_{s_k} \mid s_k = s^{(m)}\right)- \mathbb{P}\left(\beta \bar{y}_k \in D_{s_k} \mid s_k = s^{(m)}\right)\leq L_m \, \left( \dd_{KF}(T_s,\overline{T}_s) +  \dd_{KF}(T_gg_{2}[k],\overline{T}_gg_{2}[k]) \right),
\end{equation}
where the number $L_m$ is defined in (\ref{SEP ERROR-L}).

\textbf{Step 2.} Estimate the difference  $ \mathbb{P}\left(\beta \bar{y} \in D_{s_k} \mid s_k = s^{(m)}\right)-\mathbb{P}\left(\beta y_k \in D_{s_k} \mid s_k = s^{(m)}\right)$.\\
By \eqref{th-sepk}-\eqref{th-sepk-bar}
\begin{equation}\label{th-SEP-bound-4}
\begin{array}{l}
 \mathbb{P}\left(\beta \bar{y} \in D_{s_k} \mid s_k = s^{(m)}\right)- \mathbb{P}\left(\beta \hat{y}_k \in D_{s_k} \mid s_k = s^{(m)}\right)\\[12pt]
= \mathbb{P}\left( \beta (\eta \overline{T}_s s^{(m)} + \eta \overline{T}_g g_{2}[k] + n_k) \in D_{s^{(m)}}\right)- \mathbb{P}\left( \beta (\eta T_s s^{(m)} + \eta T_g g_{2}[k] + n_k) \in D_{s^{(m)}}\right)\\[12pt]
= \mathbb{P}\left( \beta (\eta \overline{T}_s s^{(m)} + \eta \overline{T}_g g_{2}[k] + n_k) \in D_{s^{(m)}}\right)- \mathbb{P}\left( \beta (\eta \overline{T}_s s^{(m)} + \eta \overline{T}_g g_{2}[k] + n_k)+\xi_k \in D_{s^{(m)}}\right).
\end{array}
\end{equation}
Since 
$$
\begin{array}{l}
 \mathbb{P}\left( \beta (\eta \overline{T}_s s^{(m)} + \eta \overline{T}_g g_{2}[k] + n_k)+\xi_k \in D_{s^{(m)}}\right)\\[12pt]
= \mathbb{P}\left( \beta (\eta \overline{T}_s s^{(m)} + \eta \overline{T}_g g_{2}[k] + n_k)+\xi_k \in D_{s^{(m)}}, |\xi_k |< \tau_{N,K}\right)\\[12pt]
\quad + \mathbb{P}\left( \beta (\eta \overline{T}_s s^{(m)} + \eta \overline{T}_g g_{2}[k] + n_k)+\xi_k \in D_{s^{(m)}}, |\xi_k |\geq \tau_{N,K}\right)\\[12pt]
\geq  \mathbb{P}\left( \beta (\eta \overline{T}_s s^{(m)} + \eta \overline{T}_g g_{2}[k] + n_k)+\xi_k \in D_{s^{(m)}}, |\xi_k |< \tau_{N,K}\right)\\[12pt]
\geq  \mathbb{P}\left( \beta (\eta \overline{T}_s s^{(m)} + \eta \overline{T}_g g_{2}[k] + n_k) \in (D_{s^{(m)}})^{-\tau_{N,K}}\right),
\end{array}
$$
then we have from (\ref{th-SEP-bound-4}) that 
$$
\begin{array}{l}
 \mathbb{P}\left(\beta \bar{y} \in D_{s_k} \mid s_k = s^{(m)}\right)-\mathbb{P}\left(\beta \hat{y}_k \in D_{s_k} \mid s_k = s^{(m)}\right)\\[12pt]
\leq  \mathbb{P}\left( \beta (\eta \overline{T}_s s^{(m)} + \eta \overline{T}_g g_{2}[k] + n_k) \in D_{s^{(m)}}\right)- \mathbb{P}\left( \beta (\eta \overline{T}_s s^{(m)} + \eta \overline{T}_g g_{2}[k] + n_k) \in (D_{s^{(m)}})^{-\tau_{N,K}}\right)\\[12pt]
  = \mathbb{P}\left( \beta (\eta \overline{T}_s s^{(m)} + \eta \overline{T}_g g_{2}[k] + n_k)\in D_{s^{(m)}}\backslash (D_{s^{(m)}})^{-\tau_{N,K}}\right)\\[12pt]
  \leq \left(\frac{\sqrt{\pi}}{\ |\beta \eta \overline{T}_g|^2 + |\beta|^2 \sigma^2}+1\right)\tau_{N,K},
\end{array}
$$ where the last inequality follows from Lemma \ref{le4.1}.
Therefore 
 \bgeqn 
 &&\mathbb{P}\left(\beta \bar{y} \in D_{s_k} \mid s_k = s^{(m)}\right)-\mathbb{P}\left(\beta _k \hat{y}_k\in D_{s_k} \mid s_k = s^{(m)}\right)\nonumber\\
 &&\leq L_m \, \left(\dd_{KF}(T_s,\overline{T}_s) + \dd_{KF}(T_gg_{2}[k],\overline{T}_gg_{2}[k]) \right).
\label{th-3}
\edeqn 
Combining (\ref{th-sepk}),\,(\ref{th-sepk-bar}),\,(\ref{th-2}) and (\ref{th-3}), we obtain \eqref{th3.1-inequ}.
\end{proof}
The above theorem provides  a 
quantitative stability guarantee for SEP in finite-dimensional massive MIMO systems.
It quantifies the deviation of the finite-dimensional SEP from its asymptotic counterpart, which is helps 
understanding of the system's behavior in practical scenarios.
 By providing a bound on the difference between the estimated and asymptotic SEPs,
 the theorem offers valuable insights into how quantization-induced distortion degrades system performance in finite-dimensional systems.
 This is particularly important in the context of massive MIMO systems, where the number of users and antennas is large, and the impact of quantization on performance needs to be carefully managed. 
The next theorem  
quantifies the difference between 
$\widehat{\rm{SINR}}_k$ 
and $\overline{\text{SINR}}$.
Recall from Definition~\ref{def-SINR-SEP} and (\ref{BAR-SINR}) that 
\begin{equation}\label{HAT-SINR-1}
\widehat{\text{SINR}}_{k}:=\frac{|\rho_{k}|^{2}\,\mathbb{E}[|s_{k}|^{2}]}{\mathbb{E}[|\hat{y}_{k}|^{2}]-|\rho_{k}|^{2}\,\mathbb{E}[|s_{k}|^{2}]}
\end{equation}  for $k=1,2,\cdots, K$ and  
\begin{equation}\label{BAR-SINR-1}
\overline{\text{SINR}}=\frac{\sigma_{\mathrm{s}}^{2}\,\eta^{2}\,\overline{T}_{\mathrm{s}}^{2}}{\eta^{2}\,\overline{T}_{\mathrm{g}}^{2}+\sigma^{2}}=\frac{\mathbb{E}^{2}[df(d)]}{\text{var}[df(d)]+\phi(\alpha,\eta)\,\frac{\mathbb{E}[f^{2}(d)]}{\gamma}}.
\end{equation}

\begin{theorem}[Discrepancy between 
$\widehat{\text{SINR}}_k $ and $ \overline{\text{SINR}}$
]
\label{Thm:SINR-error-KF-dist}  
Let $\hat{y}_k$ and $\bar{y}_k$ denote the $k$-th 
components of the vectors $\hat{\mathbf{y}}$ and $\bar{\mathbf{y}}$ respectively,
let $s_k$ 
be the $k$-th 
component of the vector $\mathbf{s}$. 
Under Assumptions~\ref{Assu:H-n-s}--\ref{Assu:ratio-K-N}, there exists a sufficiently large $\bar{K} > 0$ such that for all $K > \bar{K}$
\bgeqn 
\left| \widehat{\text{SINR}}_k - \overline{\text{SINR}} \right| \leq L_k \left(\mathbb{E}\left[ \left| \hat{y}_k - \bar{y}_k \right|^2 \right]\right)^{\frac{1}{2}}, \quad \text{for}\; k \in [K],
\edeqn 
where $\widehat{\text{SINR}}_k$ and $\overline{\text{SINR}}$ are defined in (\ref{HAT-SINR}) and (\ref{BAR-SINR}), respectively and 
\begin{equation}\label{th-sinr-L}
L_k=\frac{2\left[2\sigma_{s}^{3}\mathbb{E}[|\bar{y}_k|^2]\left(\sigma_{s}(\mathbb{E}[|\bar{y}_k|^2])^{\frac{1}{2}}+1\right)+\sigma_{s}^{4}\mathbb{E}[|\bar{y}_k|^2](2\mathbb{E}[|\bar{y}_k|]+1)\right]}{\left(\sigma_{s}^{2}\eta^2\bar{T}_{g}^2+\sigma_{s}^{2}\sigma^2\right)^2}.
\end{equation}
\end{theorem}


Theorem \ref{Thm:SINR-error-KF-dist} establishes a quantitative link between the accuracy of the received signal model and SINR stability. Specifically, it bounds the SINR deviation \(\left| \widehat{\text{SINR}}_k - \overline{\text{SINR}} \right|\) by the \(L^2\)-error \(\left(\mathbb{E}\left[ \left| \hat{y}_k - \bar{y}_k \right|^2 \right]\right)^{\frac{1}{2}}\) in reconstructing the received signal, scaled by a sensitivity constant \(L_k\). Since \(\hat{y}_k \rightarrow \bar{y}_k\) almost surely as \(K \to \infty\) (by Proposition \ref{Prop-asymp-y-WLS24}), the root mean square error \(\left(\mathbb{E}\left[ \left| \hat{y}_k - \bar{y}_k \right|^2 \right]\right)^{\frac{1}{2}} \to 0\). 
Together with the boundedness of \(L_k\) (for fixed system parameters), this guarantees that \(\left| \widehat{\text{SINR}}_k - \overline{\text{SINR}} \right| \to 0\), validating the asymptotic SINR for finite systems. The constant \(L_k\) in (\ref{th-sinr-L}) quantifies how sensitive the SINR is to perturbations in the received signal \(\hat{y}_k\). \(L_k\) scales with moments of the asymptotic signal (\(\mathbb{E}[|\bar{y}_k|^2]\), \(\mathbb{E}[|\bar{y}_k|]\)) and signal power (\(\sigma_s^2\)). This reflects how constellation properties and channel gain impact SINR robustness.   For \(K > \bar{K}\), \(L_k\) depends only on asymptotic system parameters (not on current \(N, K\)), making it a universal sensitivity measure.

\begin{proof}
By the definitions of $\widehat{\text{SINR}}_k$ and  $\overline{\text{SINR}}$, we have
\begin{align*}
&\left|   \widehat{\text{SINR}}_k-\overline{\text{SINR}}  \right|\\[12pt]
&=\left| \frac{| \rho_k |^2 \mathbb{E}[|s_k|^2]}{\mathbb{E}[|\hat{y}_k|^2] - | \rho_k |^2 \mathbb{E}[|s_k|^2]}-\frac{\sigma_{s}^{2} \eta^{2} \bar{T}_{s}^{2}}{\eta^{2} \bar{T}_{g}^{2} + \sigma^{2}}   \right|\\[12pt]
&=\left| \frac{\mathbb{E}^2[s_k^+\hat{y}_k]}{\sigma_{s}^{2}\mathbb{E}[|\hat{y}_k|^2] - \mathbb{E}^2[s_k^+\hat{y}_k]}-\frac{\mathbb{E}^2[s^+\bar{y}_k]}{\sigma_{s}^{2}\mathbb{E}[|\bar{y}_k|^2] - \mathbb{E}^2[s^+\bar{y}_k]}   \right|\\[12pt]
&=\left|\frac{\mathbb{E}^2[s_k^+\hat{y}_k]\sigma_{s}^{2}\mathbb{E}[|\bar{y}_k|^2]-\mathbb{E}^2[s_k^+\bar{y}]\sigma_{s}^{2}\mathbb{E}[|\hat{y}_k|^2]}{\left(\sigma_{s}^{2}\mathbb{E}[|\hat{y}_k|^2] - \mathbb{E}^2[s_k^+\hat{y}_k]\right)\left(\sigma_{s}^{2}\mathbb{E}[|\bar{y}_k|^2] - \mathbb{E}^2[s^+\bar{y}_k]\right)}\right|\\[12pt]
&\leq \frac{\left|\mathbb{E}^2[s_k^+\hat{y}_k]\sigma_{s}^{2}\mathbb{E}[|\bar{y}_k|^2]-\mathbb{E}^2[s_k^+\bar{y}_k]\sigma_{s}^{2}\mathbb{E}[|\bar{y}_k|^2]\right|+\left|\mathbb{E}^2[s^+\bar{y}_k]\sigma_{s}^{2}\mathbb{E}[|\bar{y}_k|^2]-\mathbb{E}^2[s^+\bar{y}_k]\sigma_{s}^{2}\mathbb{E}[|\hat{y}_k|^2]\right|}{{\left(\sigma_{s}^{2}\mathbb{E}[|\hat{y}_k|^2] - \mathbb{E}^2[s_k^+\hat{y}_k]\right)\left(\sigma_{s}^{2}\mathbb{E}[|\bar{y}_k|^2] - \mathbb{E}^2[s^+\bar{y}_k]\right)}}\\[12pt]
&\leq \frac{\sigma_{s}^{2}\mathbb{E}[|\bar{y}_k|^2]\left|\mathbb{E}^2[s_k^+\hat{y}_k]-\mathbb{E}^2[s_k^+\bar{y}_k]\right|+\sigma_{s}^{2}\mathbb{E}^2[s^+\bar{y}_k]\left|\mathbb{E}[|\bar{y}_k|^2]-\mathbb{E}[|\hat{y}_k|^2]\right|}{{\left(\sigma_{s}^{2}\mathbb{E}[|\hat{y}_k|^2] - \mathbb{E}^2[s_k^+\hat{y}_k]\right)\left(\sigma_{s}^{2}\mathbb{E}[|\bar{y}_k|^2] - \mathbb{E}^2[s^+\bar{y}_k]\right)}}.
\end{align*}
By the definitions of $\widehat{\text{SINR}}_k$ and  $\overline{\text{SINR}}$, 
$s_k$ and $s$ have the same distribution, 
thus $\mathbb{E}[s_k^+\hat{y}_k]=\mathbb{E}[s^+\hat{y}_k]$ and $\mathbb{E}[|s_k|^2]=\mathbb{E}[|s|^2]=\sigma_{s}^{2}.$
Since $\mathbb{E}[s_k^+\hat{y}_k]\rightarrow \mathbb{E}[s_k^+\bar{y}_k]$ as $K\rightarrow \infty$, we have when $K$ is large enough,
$\left|\mathbb{E}[s_k^+\hat{y}_k]\right |\leq \mathbb{E}[s_k^+\bar{y}_k]+1.$
By the Cauchy-Schwarz inequality, we have 
\bgeq 
\left|\mathbb{E}^2[s_k^+\hat{y}_k]-\mathbb{E}^2[s_k^+\bar{y}_k]\right|
&\leq& \left|\mathbb{E}[s_k^+\hat{y}_k]-\mathbb{E}[s_k^+\bar{y}_k]\right|\left|\mathbb{E}[s_k^+\hat{y}_k]+\mathbb{E}[s_k^+\bar{y}_k]\right|\\[12pt]
&\leq& 2(\mathbb{E}[s^+\bar{y}]+1)\left|\mathbb{E}[s^+(\hat{y}_k-\bar{y}_k)\right|\\[12pt]
&\leq& 2(\sigma_{s}\mathbb{E}^{\frac{1}{2}}\left[\left|\bar{y}_k\right|^2\right]+1)\sigma_{s}\mathbb{E}^{\frac{1}{2}}\left[\left|\hat{y}_k-\bar{y}_k\right|^2\right].
\edeq
Likewise 
\bgeq 
\left|\mathbb{E}[|\bar{y}_k|^2]-\mathbb{E}[|\hat{y}_k|^2]\right|
&=&\left|\mathbb{E}[|\bar{y}_k|-|\hat{y}_k|]\right|
\left|\mathbb{E}[|\bar{y}_k|+|\hat{y}_k|]\right|\\[12pt]
&\leq& (2\mathbb{E}[|\bar{y}_k|]+1)
\mathbb{E}[|\hat{y}_k-\bar{y}|]\\[12pt]
&\leq& (2\mathbb{E}[|\bar{y}_k|]+1)\mathbb{E}^{\frac{1}{2}}[|\hat{y}_k-\bar{y}_k|^2].
\edeq 
Therefore, we have 
\begin{align*}
&\sigma_{s}^{2}\mathbb{E}[|\bar{y}_k|^2]\left|\mathbb{E}^2[s_k^+\hat{y}_k]-\mathbb{E}^2[s_k^+\bar{y}_k]\right|+\sigma_{s}^{2}\mathbb{E}^2[s^+\bar{y}_k]\left|\mathbb{E}[|\bar{y}_k|^2]-\mathbb{E}[|\hat{y}_k|^2]\right|\\[12pt]
&\leq \left[2\sigma_{s}^{3}\mathbb{E}[|\bar{y}_k|^2](\sigma_{s}\mathbb{E}^{\frac{1}{2}}[|\bar{y}_k|^2]+1)+\sigma_{s}^{4}\mathbb{E}[|\bar{y}_k|^2](2\mathbb{E}[|\bar{y}_k|]+1)\right]\mathbb{E}^{\frac{1}{2}}[|\hat{y}_k-\bar{y}_k|^2].&
\end{align*}
On the other hand, 
\begin{align*}
&\sigma_{s}^{2}\mathbb{E}[|\bar{y}_k|^2] - \mathbb{E}^2[s^+\bar{y}]\\[12pt]
&=\sigma_{s}^{2}\mathbb{E}[(\eta T_{s} s + \eta T_{g} g + n)^+(\eta \bar{T}_{s} s + \eta \bar{T}_{g} g + n)]-\mathbb{E}^2[s^+(\eta \bar{T}_{s} s + \eta \bar{T}_{g} g + n)]\\[12pt]
&=\sigma_{s}^{2}(\eta^2 \bar{T}_{s}^2\mathbb{E}[|s|^2]+\eta^2 \bar{T}_{g}^2+\sigma^2 )-(\eta\bar{T}_{s}\sigma_{s}^{2})^2\\[12pt]
&=\sigma_{s}^{2}\eta^2\bar{T}_{g}^2+\sigma_{s}^{2}\sigma^2>0.
\end{align*}
Since 
$$
\left(\sigma_{s}^{2}\mathbb{E}[|\hat{y}_k|^2] - \mathbb{E}^2[s_k^+\hat{y}_k]\right)\rightarrow\left(\sigma_{s}^{2}\mathbb{E}[|\bar{y}_k|^2] - \mathbb{E}^2[s^+\bar{y}_k]\right)
$$
as $K$ tends to $\infty$, we have when $K$ is large enough, 
$$
\left|\left(\sigma_{s}^{2}\mathbb{E}[|\hat{y}_k|^2] - \mathbb{E}^2[s_k^+\hat{y}_k]\right)\left(\sigma_{s}^{2}\mathbb{E}[|\bar{y}_k|^2] - \mathbb{E}^2[s^+\bar{y}_k]\right)\right|>\frac{1}{2}\left(\sigma_{s}^{2}\eta^2\bar{T}_{g}^2+\sigma_{s}^{2}\sigma^2\right)^2.
$$
Consequently, when $K$ is large enough, 
$$\left| \widehat{\text{SINR}}_k-\overline{\text{SINR}}\right| \leq L_k \left(\mathbb{E}\left[\left|\hat{y}_k-\bar{y}_k\right|^2\right]\right)^{\frac{1}{2}},$$ where $L_k$ is defined in (\ref{th-sinr-L}).
\end{proof}
 
 This theorem provides 
 a bound on the difference 
 between $\widehat{\text{SINR}}_k$ and $\overline{\text{SINR}}$,
 the gap of performances between finite-dimensional systems and their asymptotic counterparts.
By controlling the mean squared error of the received signals, it is possible to minimize the deviation between $\widehat{\text{SINR}}_k$ and $\overline{\text{SINR}}$, thereby ensuring that finite-dimensional systems can achieve performance close to their asymptotic limits.
Note that \(f\) and \(\eta\) remain constant within \(\hat{y}_k\) and \(\bar{y}_k\), where \(\hat{y}_k\) and \(\bar{y}_k\) represent the \(k\)-th components of the vectors \(\hat{\mathbf{y}}\) and \(\bar{\mathbf{y}}\), respectively, as established in equations (\ref{eq:haty-stat-equiv-model}) and (\ref{eq:bary-stat-equiv-v1-0}). However, in specific scenarios like the SINR maximization problem 
(see Section 5 below), it becomes necessary to investigate the relationship between \(\hat{y}_k\) and \(\bar{y}_k\) when \(f\) and \(\eta\) are variables. Consequently, we consider \(\hat{y}_k\) and \(\bar{y}_k\) as functional entities.
Let us define
\[
\tilde{\alpha}_{N,K}(f) = \frac{\|s\|}{\|g_1\|} \cdot \frac{\|f(D)^T g_1\|}{\|z_1\|}, \quad \tilde{\alpha}(f) = \sqrt{\frac{\sigma_s^2 \mathbb{E}[f^2(d)]}{\gamma}},
\]
from the definitions of \(\hat{y}_k\) and \(\bar{y}_k\), we observe that \(\tilde{\alpha}_{N,K}(f)\) and \(\tilde{\alpha}(f)\) are encapsulated within \(T_s, T_g\) and \(\bar{T}_s, \bar{T}_g\), respectively.


The expression of \(\tilde{\alpha}_{N,K}(f)\) involves the norms of vectors \(s\), \(g_1\), and \(z_1\), as well as the matrix \(D\) and the function \(f\). These vectors and matrices are typically random and their properties depend on the dimensions \(N\) (number of antennas) and \(K\) (number of users). Therefore, \(\tilde{\alpha}_{N,K}(f)\) captures the finite-dimensional effects and randomness associated with the specific realization of the channel and noise in a system with \(N\) antennas and \(K\) users.
 The expression of \(\tilde{\alpha}(f)\) involves the expected value of \(f^2(d)\) and the parameter \(\gamma\). The expected value \(\mathbb{E}[f^2(d)]\) is taken over the distribution of the singular values \(d\) of the channel matrix \(H\). In the asymptotic regime where \(N\) and \(K\) tends to infinity while maintaining a fixed ratio \(\gamma\), the distribution of the singular values \(d\) converges to a deterministic limit (e.g., the Marchenko-Pastur distribution). Therefore, \(\tilde{\alpha}(f)\) represents the asymptotic behavior of the system and does not depend on the specific finite values of \(N\) and \(K\). It captures the deterministic limit that \(\tilde{\alpha}_{N,K}(f)\) converges to as \(N\) and \(K\) tends to infinity.
In summary, we define the function \((\hat{y}_k)_{N,K}(\cdot)\) as follows:
\begin{subequations}
\begin{align}
(\hat{y}_k)_{N,K}(f,\eta,\tilde{\alpha}_{N,K}(f)) &= \eta T_{s}(f,\tilde{\alpha}_{N,K}(f)) s_k + \eta T_{g}(f,\tilde{\alpha}_{N,K}(f)) g_2[k] + n_k,\label{fun-hat-y} \\
\intertext{where}
T_{s}(f,\tilde{\alpha}_{N,K}(f)) &= \frac{\mathbf{g}_1^\mathsf{H} (C_1 \mathbf{D} \hat{\mathbf{s}}_1 + C_2 \mathbf{D} \mathbf{B}(\hat{\mathbf{s}}_1) \mathbf{z}_2[2:N])}{\|\mathbf{g}_1\| \|\mathbf{s}\|} - T_g \frac{(\mathbf{R}(\mathbf{s})^{-1} \mathbf{g}_2)[1]}{\|\mathbf{s}\|}, \\
T_{g}(f,\tilde{\alpha}_{N,K}(f)) &= \frac{\|\mathbf{B}(\mathbf{g}_1)^\mathsf{H} (C_1(\tilde{\alpha}_{N,K}(f)) \mathbf{D} \hat{\mathbf{s}}_1 + C_2(\tilde{\alpha}_{N,K}(f)) \mathbf{D} \mathbf{B}(\hat{\mathbf{s}}_1) \mathbf{z}_2[2:N])\|}{\|(\mathbf{R}(\mathbf{s})^{-1} \mathbf{g}_2)[2:K]\|}, \\
C_1(\tilde{\alpha}_{N,K}(f)) &= \frac{\mathbf{z}_1^\mathsf{H} q\left(\tilde{\alpha}_{N,K}(f) \mathbf{z}_1\right)}{\tilde{\alpha}_{N,K}(f)\|\mathbf{z}_1\|^2}, \quad C_2(\tilde{\alpha}_{N,K}(f)) = \frac{\|\mathbf{B}(\mathbf{z}_1)^\mathsf{H} q\left(\tilde{\alpha}_{N,K}(f)\mathbf{z}_1\right)\|}{\|\mathbf{z}_2[2:N]\|}, \\
\hat{\mathbf{s}}_1 &= \frac{\|\mathbf{s}\|}{\|\mathbf{g}_1\|} f(\mathbf{D})^\mathsf{T} \mathbf{g}_1 
\end{align}
\end{subequations}
Similarly, we define the function \(\bar{y}_k(\cdot)\) as follows:
\begin{subequations}
\begin{align}
\bar{y}_k(f,\eta, \tilde{\alpha}(f)) &= \eta \bar{T}_{s}(f,\tilde{\alpha}(f)) s_k + \eta \bar{T}_{g}(f,\tilde{\alpha}(f)) g_2[k] + n_k, \label{fun-bar-y} \\
\intertext{where}
\bar{T}_{s}(f,\tilde{\alpha}(f)) &= \overline{C}_1(\tilde{\alpha}(f)) \mathbb{E}[df(d)], \\
\bar{T}_{g}(f,\tilde{\alpha}(f)) &= \sqrt{\sigma_s^2 [\overline{C}_1(\tilde{\alpha}(f))]^2 \text{var}[df(d)] + \overline{C}_2(\tilde{\alpha}(f))^2}, \\
\overline{C}_1(\tilde{\alpha}(f)) &= \frac{\mathbb{E}[Z^\dag q(\tilde{\alpha}(f) Z)]}{\tilde{\alpha}(f)}, \\
\overline{C}_2(\tilde{\alpha}(f)) &= \sqrt{\mathbb{E}[|q(\tilde{\alpha}(f) Z)|^2] - |\mathbb{E}[Z^{\dagger} q(\tilde{\alpha}(f) Z)]|^2}.
\end{align}
\end{subequations}
Subsequently, we introduce the functions \(\widehat{\text{SINR}}_{k}(\cdot)\) and \(\overline{\text{SINR}}(\cdot)\) defined by
\begin{equation}
\label{function-HAT-SINR}
\widehat{\text{SINR}}_{k}(f,\eta,\tilde{\alpha}_{N,K}(f)) := \frac{|\rho_{k}(f,\eta,\tilde{\alpha}_{N,K}(f))|^{2}\,\mathbb{E}[|s_{k}|^{2}]}{\mathbb{E}[|(\hat{y}_{k})_{N,K}(f,\eta,\tilde{\alpha}_{N,K}(f))|^{2}]-|\rho_{k}(f,\eta,\tilde{\alpha}_{N,K}(f))|^{2}\,\mathbb{E}[|s_{k}|^{2}]}
\end{equation}
and 
\begin{equation}\label{function-BAR-SINR}
\overline{\text{SINR}}(f,\eta, \tilde{\alpha}(f)) = \frac{\mathbb{E}^{2}[df(d)]}{\text{var}[df(d)] + \phi(\tilde{\alpha}(f),\eta)\,\frac{\mathbb{E}[f^{2}(d)]}{\gamma}},
\end{equation}
with \(\rho_{k}(f,\eta,\tilde{\alpha}_{N,K}(f)) = \mathbb{E}[s_{k}^{\dagger}(\hat{y}_{k})_{N,K}(f,\eta,\tilde{\alpha}_{N,K}(f))]/\mathbb{E}[|s_{k}|^{2}]\) and
\begin{equation}
\phi(\tilde{\alpha}(f),\eta) := \frac{\mathbb{E}[|q(\tilde{\alpha}(f) Z)|^{2}] - |\mathbb{E}[Z^{\dagger} q(\tilde{\alpha}(f) Z)]|^{2} + \sigma^{2}/\eta^{2}}{|\mathbb{E}[Z^{\dagger} q(\tilde{\alpha}(f) Z)]|^{2}}.
\end{equation}

\begin{proposition}\label{th-SINR-FUN-ERR}
For a fixed \(f\) satisfying Assumption \ref{Assu:funct-f}, 
let \( \{\eta_{N,K}\} \) be a sequence of random variables converging to \( \eta \) almost surely as \(K\to\infty\).
Then under Assumptions \ref{Assu:H-n-s}--\ref{Assu:ratio-K-N}, the following assertions hold.
\begin{itemize}
    \item [(i)] \((\hat{y}_k)_{N,K}(f,\eta_{N,K},\tilde{\alpha}_{N,K}(f) ) \xrightarrow{a.s.} \bar{y}_k(f,\eta, \tilde{\alpha}(f))\) as \(K \rightarrow \infty\).
    \item [(ii)] There exists a sufficiently large $\bar{K} > 0$ such that for all $K > \bar{K}$, there exists a positive number \(L_k(f,\eta,\tilde{\alpha}(f) )\) such that 
    \begin{align}
&\left|\widehat{\text{SINR}}_k(f,\eta_{N,K},\tilde{\alpha}_{N,K}(f)) - \overline{\text{SINR}}(f,\eta,\tilde{\alpha}(f))\right| \nonumber\\
    &\leq  L_k(f,\eta,\tilde{\alpha}(f) ) \left(\mathbb{E}^{\frac{1}{2}}\left[ \left| (\hat{y}_k)_{N,K}(f,\eta_{N,K}, \tilde{\alpha}_{N,K}(f)) - \bar{y}_k(f,\eta,\tilde{\alpha}(f) ) \right|^2 \right]\right)^{\frac{1}{2}}
    \end{align}
    for each \(k \in [K]\), where \(\widehat{\text{SINR}}_k(\cdot)\) and \(\overline{\text{SINR}}(\cdot)\) are defined in (\ref{function-HAT-SINR}) and (\ref{function-BAR-SINR}), respectively, and
\begin{align}\label{th-fun-sinr-L}
L_k(f,\eta, \tilde{\alpha}(f)) = \frac{A_k(f,\eta, \tilde{\alpha}(f)) + B_k(f,\eta, \tilde{\alpha}(f))}{C_k(f,\eta, \tilde{\alpha}(f))}
\end{align}
with
\begin{align*}
A_k(f,\eta, \tilde{\alpha}(f)) &= 2\sigma_{s}^{3}\mathbb{E}[|\bar{y}_k(f,\eta, \tilde{\alpha}(f))|^2] \left(\sigma_{s}\mathbb{E}^{\frac{1}{2}}[|\bar{y}_k(f,\eta, \tilde{\alpha}(f))|^2] + 1\right), \\
B_k(f,\eta, \tilde{\alpha}(f)) &= \sigma_{s}^{4}\mathbb{E}[|\bar{y}_k(f,\eta, \tilde{\alpha}(f))|^2] \left(2\mathbb{E}[|\bar{y}_k(f,\eta, \tilde{\alpha}(f))|] + 1\right), \\
C_k(f,\eta, \tilde{\alpha}(f)) &= \left(\sigma_{s}^{2}\eta^2\bar{T}_{g}(f,\tilde{\alpha}(f) )^2 + \sigma_{s}^{2}\sigma^2\right)^2.
\end{align*}
    \item [(iii)] There exists a sufficiently large $\bar{K} > 0$ such that for all $K > \bar{K}$,there exists a positive number \(L_k(f,\eta,\tilde{\alpha}(f) )\) such that 
    \begin{align}
    &\left|\overline{\text{SINR}}(f,\eta_{N,K},\tilde{\alpha}_{N,K}(f)) - \overline{\text{SINR}}(f,\eta,\tilde{\alpha}(f))\right| \\
    &\leq  L_k(f,\tilde{\alpha}(f),\eta ) \left(\mathbb{E}\left[ \left| \bar{y}_k(f,\eta_{N,K},\tilde{\alpha}_{N,K}(f)) - \bar{y}_k(f,\eta,\tilde{\alpha}(f)) \right|^2 \right]\right)^{\frac{1}{2}},
    \end{align}
    for all \(k \in [K]\), where \(L_k(f,\eta,\tilde{\alpha}(f) )\) is as defined in (\ref{th-fun-sinr-L}).
\end{itemize}
\end{proposition}
 
 Before presenting a proof, it might be helpful to give some comments about the conditions and the results of the proposition. The conditions are 
 mostly the same as those of Theorem \ref{Thm:SINR-error-KF-dist}, 
the only additional ingredient is the introduction of a parameter sequence \(\{\eta_{N,K}\}\) that is required to converge almost surely to the deterministic limit  as \(K\to\infty\).  
That additional condition is imposed solely to accommodate the variable nature of \(\eta\). It does not alter the underlying probabilistic framework already established in Theorem \ref{Thm:SINR-error-KF-dist}. 
 Part (i) indicates that as the number of users \( K \) increases, the received signal  in the finite-dimensional system approaches its asymptotic limit  when the involved \( \eta \) are variables, which guarantees that the right-hand side of the quantitative inequalities in Part~(ii) vanishes as $K\to\infty$.  Part (ii) is similar to the conclusion of Theorem \ref{Thm:SINR-error-KF-dist} but further quantifies the stability of SINR when \( f \) and \( \eta \) are variables. Part (iii) of Proposition \ref{th-SINR-FUN-ERR} 
 provides a Lipschitz-like bound.  
which explicitly quantifies the sensitivity of the asymptotic SINR to perturbations in the parameter sequences $\eta_{N,K}$ and $\tilde{\alpha}_{N,K}(f)$.  T
Parts (ii) and (iii) will be directly 
used in the  Section 5 when deriving the quantitative guarantees for the SINR based optimization problem. 

\begin{proof}
 Part (i).
 By \cite[Theorem 2]{WLS24}, for \(f\) satisfying Assumption \ref{Assu:funct-f},
$
\tilde{\alpha}_{N,k}(f) \xrightarrow{\text{a.s.}} \tilde{\alpha}(f)
$
as \(K \rightarrow \infty\), which implies
\begin{equation}\label{th-fun-p1}
T_s(f, \tilde{\alpha}_{N,k}(f)) \xrightarrow{\text{a.s.}} \bar{T}_s(f, \tilde{\alpha}(f))
\,\,
\mbox{and}\,\,
T_g(f, \tilde{\alpha}_{N,k}(f)) \xrightarrow{\text{a.s.}} \bar{T}_g(f, \tilde{\alpha}(f))
\end{equation}
as \(K \rightarrow \infty\).
From (\ref{fun-hat-y}) and (\ref{fun-bar-y}), we derive
\begin{align*}
&\left| (\hat{y}_k)_{N,k}(f, \eta_{N,K}, \tilde{\alpha}_{N,K}(f)) - \bar{y}_k(f, \eta, \tilde{\alpha}(f)) \right| \\[12pt]
&= \left| \eta_{N,K} T_s(f, \tilde{\alpha}_{N,K}(f)) s_k - \eta \bar{T}_s(f, \tilde{\alpha}(f)) s_k \right| 
+ \left| \eta_{N,K} T_g(f, \tilde{\alpha}_{N,K}(f)) s_k - \eta \bar{T}_g(f, \tilde{\alpha}(f)) g_{2}[k] \right| \\[12pt]
&\leq \left| \eta_{N,K} T_s(f, \tilde{\alpha}_{N,K}(f)) - \eta_{N,K} \bar{T}_s(f, \tilde{\alpha}(f)) \right| |s_k|+  \left| \eta_{N,K} - \eta \right| \left| \bar{T}_s(f, \tilde{\alpha}(f)) \right| |s_k| \\[12pt]
&\quad+ \left| \eta_{N,K} T_g(f, \tilde{\alpha}_{N,K}(f)) - \eta_{N,K} \bar{T}_g(f, \tilde{\alpha}(f)) \right| |g_{2}[k]|+ \left| \eta_{N,K} - \eta \right| \left| \bar{T}_g(f, \tilde{\alpha}(f)) \right| |g_{2}[k]|, 
\end{align*}
which, by (\ref{th-fun-p1}), 
implies the conclusion of Part (i).

    Part (ii).
    Following the proof of Theorem \ref{Thm:SINR-error-KF-dist}   and Part (i), we substitute \(\hat{y}_k\) and \(\bar{y}_k\) in the proof of Theorem \ref{Thm:SINR-error-KF-dist} with \((\hat{y}_k)_{N,K}(f, \eta_{N,K}, \tilde{\alpha}_{N,K}(f))\) and \(\bar{y}_k(f, \eta, \tilde{\alpha}(f))\), respectively, to obtain the conclusion.
    
Part (iii). Analogous to the proof of Part  (i), we have
    \[
    \bar{y}_k(f, \eta_{N,K}, \tilde{\alpha}_{N,K}(f)) \xrightarrow{\text{a.s.}} \bar{y}_k(f, \eta, \tilde{\alpha}(f))
    \]
    as \(K \rightarrow \infty\). 
    Using the ideas of the proof of Theorem \ref{Thm:SINR-error-KF-dist}, we can replace \(\hat{y}_k\) and \(\bar{y}_k\) in the proof of Theorem \ref{Thm:SINR-error-KF-dist} with \(\bar{y}_k(f, \eta_{N,K}, \tilde{\alpha}_{N,K}(f))\) and \(\bar{y}_k(f, \eta, \tilde{\alpha}(f))\), respectively, to establish the result. We skip the details.
\end{proof}

\section{Quantifying the
asymptotic error of SEP}



This section focuses on quantifying the propagation of error in SEP resulting from replacing the finite-dimensional received signal model $\hat{y}$ defined in  (\ref{eq:haty-stat-equiv-model}) with its asymptotic counterpart $\bar{y}$ defined in  (\ref{eq:bary-stat-equiv-v1-0}). We provide a rigorous framework for analyzing the stability of SEP in massive MIMO systems, deriving probabilistic bounds that characterize the deviations of finite-dimensional SEP $\widehat{\text{SEP}}_k(\beta)$ from its asymptotic limit $\overline{\text{SEP}}(\beta)$. Specifically, we establish how discrepancies between $\hat{y}_k$ and $\bar{y}_k$ translate into quantitative bounds on $|\widehat{\text{SEP}}_k(\beta) - \overline{\text{SEP}}(\beta)|$, thereby validating the reliability of asymptotic SEP predictions in practical finite-dimensional systems. The results offer a mechanistic understanding of how quantization-induced signal distortion propagates to symbol detection reliability, guiding robust receiver design under hardware constraints. This analysis bridges the gap between theoretical asymptotic guarantees and practical implementation requirements.

{\color{black}
Consider the sample covariance matrix $\mathbf{H}\mathbf{H}^{\mathsf{H}}$ with eigenvalues $\lambda_1, \dots, \lambda_K$. Let $d_i = \sqrt{\lambda_i}$. By Proposition \ref{domaom-f}, for $\epsilon=\frac{1}{2}\left(1 - \frac{1}{\sqrt{\gamma}}\right),$ there exists $K_0>0$ such that for all $K>K_0$,  any non-zero elements $d_i$ of $D$ in (\ref{eq:sinr_optk-c2}) almost surely lies in the interval 
$$
\Theta:=\left[1 - \frac{1}{\sqrt{\gamma}}-\epsilon, 1 + \frac{1}{\sqrt{\gamma}}+\epsilon\right]=\left[\frac{1}{2} - \frac{1}{2\sqrt{\gamma}}, \frac{3}{2} + \frac{1}{2\sqrt{\gamma}}\right].
$$

To facilitate the analysis, we make the following assumption.

\begin{assumption}\label{Assu:d}
There exist $ M_1$ and $\bar{K} \in \mathbb{N}$ such that for $K > \bar{K}$, 
any non-zero elements $d_i$ of $D$ in (\ref{eq:sinr_optk-c2}) almost surely lies in the interval 
$
\Theta
$
and $\sup_{1 \leq i \leq K} |\sigma(d_i)| \leq M_1$ almost surely with $\sigma(x)=\max\{x,f(x), f(x)^2, xf(x), x^2, x^2f(x)^2\}.
$ 
\end{assumption}

Understanding the convergence rate of $T_g$ is crucial for assessing how quickly the system performance metrics stabilize as the number of antennas $K$ increases. 
To obtain quantitative stability analysis of $\hat{y}$, we need two intermediate results. 


\begin{lemma}
\label{Lem:T_g-convg-barT_g}
Let  Assumptions \ref{Assu:H-n-s}-\ref{Assu:ratio-K-N} and Assumption \ref{Assu:d} hold.
For any small positive number $\epsilon>0$, 
there exists $\hat{K}>0$ depending on $\epsilon$ 
such that   
\begin{equation}
\label{eq:T_g-barT_g-diff}
\mathbb{P}\left(\left| T_g-\bar{T}_g\right|\geq \epsilon\right)\leq \mathfrak{R}(\epsilon,K),
\end{equation}
for all $K>\hat{K}$, where $\mathfrak{R}(\epsilon,K)\to 0$ as $K\to\infty$.
\end{lemma}

The lemma states that with probability 
at least $1-\mathfrak{R}(\epsilon,K)$, 
the difference between $T_g$ and $\bar{T}_g$
is bounded by $\epsilon$ when $K$ is sufficiently large. 
The proof of the lemma is very long and moved to 
Section~\ref{sec:Lem:T_g-convg-barT_g} and an explicit form of $\mathfrak{R}(\epsilon,K)$ is also given there.
The next lemma gives a similar result for $T_s$.

\begin{lemma}
\label{Lem:T_s-convg-barT_s}
Let  Assumptions \ref{Assu:H-n-s}-\ref{Assu:ratio-K-N} and Assumption \ref{Assu:d} hold.
For any small positive number $\epsilon>0$, 
there exists $\tilde{K}>0$ depending on $\epsilon$ 
such that   
\begin{equation}
\label{eq:T_s-barT_s-diff-1}
\mathbb{P}\left(\left| T_s-\bar{T}_s\right|\geq \epsilon\right)\leq \tilde{\mathfrak{R}}(\epsilon,K),
\end{equation}
for all $K>\tilde{K}$, where $\tilde{\mathfrak{R}}(\epsilon,K)\to 0$ as $K\to\infty$.
\end{lemma}

The  lemma provides a quantitative measure of how closely $T_s$ concentrates around its asymptotic value $\overline{T}_s$ for finite $K$. The exponential bound shows that the probability of significant deviation decreases rapidly as $K$ grows, which validates the practical relevance of asymptotic analyses in finite-dimensional systems. This result reassures us that the asymptotic approximations become increasingly accurate with larger $K$, which is typical in massive MIMO deployments.
The proof of the lemma 
is long and moved to 
Section~\ref{sec:proof of Lem:T_s-convg-barT_s} and an explicit form of $\tilde{\mathfrak{R}}(\epsilon,K)$ is also given there.

With the individual concentration results for $T_s$ and $T_g$ established, we now address the overall stability of the estimated signal $\hat{y}_k$ compared to its asymptotic counterpart $\bar{y}_k$ This theorem combines the previous results to quantify the probability that the difference between  $\hat{y}_k$ and $\bar{y}_k$ exceeds a specified threshold, providing a comprehensive measure of system performance stability.

\begin{theorem}
\label{Thm:haty-convg-bary-exp}
Let $\hat{y}_k$ and $\bar{y}_k$ denote the $k$-th elements of the vectors $\hat{\mathbf{y}}$ and $\bar{\mathbf{y}}$, respectively, as defined in (\ref{eq:haty-stat-equiv-model}) and (\ref{eq:bary-stat-equiv-v1-0}).
Let  Assumptions \ref{Assu:H-n-s}-\ref{Assu:ratio-K-N} and Assumption \ref{Assu:d} hold. Then for each $k\in [K]$, for any $\epsilon > 0$, there exists $\bar{K}=\max\{\hat{K}, \tilde{K}\} > 0$ such that for $K > \bar{K}$, the following inequality holds:
\bgeqn 
\mathbb{P}\left(|\hat{y}_k - \bar{y}_k| \geq \epsilon\right) \leq \mathfrak{R}\left(\sqrt{\frac{\epsilon}{2\eta}},K\right)+\tilde{\mathfrak{R}}\left(\frac{\epsilon}{2\eta S},K\right)
+ e^{-\frac{\epsilon}{2\eta}},
\edeqn 
where
 $S = \sup_{s\in \mathcal{S}} |\mathbf{s}|,  $,  $\hat{K}$ and $\tilde{K}$ are determined by $\sqrt{\frac{\epsilon}{2\eta}}$ and $\frac{\epsilon}{2\eta S}$, as defined in Lemmas \ref{Lem:T_g-convg-barT_g} and \ref{Lem:T_s-convg-barT_s}, respectively.

\end{theorem}

\begin{proof}
For $k\in [K]$,
\[
|\hat{y}_k - \bar{y}_k| = |\eta(T_s - \overline{T}_s)s_k + \eta(T_g - \overline{T}_g)\mathbf{g}_2[k]| \leq \eta|\Delta T_s| \cdot S + \eta|\Delta T_g| \cdot |\mathbf{g}_2[k]|,
\]
where $\Delta T_s = T_s - \overline{T}_s$ and $\Delta T_g = T_g - \overline{T}_g$.
For any $\epsilon > 0$,
\[
\mathbb{P}\left(|\hat{y}_k - \bar{y}_k| \geq \epsilon\right) \leq \mathbb{P}\left(\eta|\Delta T_s| \cdot S \geq \frac{\epsilon}{2}\right) + \mathbb{P}\left(\eta|\Delta T_g| \cdot |\mathbf{g}_2| \geq \frac{\epsilon}{2}\right).
\]
Using Lemma \ref{Lem:T_s-convg-barT_s}, there exists   $\tilde{K}$ depending on $\frac{\epsilon}{2\eta S}$ such that
\[
\mathbb{P}\left(|\Delta T_s| \geq \frac{\epsilon}{2\eta S}\right) \leq \tilde{\mathfrak{R}}\left(\frac{\epsilon}{2\eta S},K\right)
\]
for $K>\tilde{K}$.
Choose $\delta = \sqrt{\frac{\epsilon}{2\eta}}$, we have
\[
\mathbb{P}\left(\eta|\Delta T_g| \cdot |\mathbf{g}_2[k]| \geq \frac{\epsilon}{2}\right) \leq \mathbb{P}\left(|\Delta T_g| \geq \delta\right) + \mathbb{P}\left(|\mathbf{g}_2[k]| \geq \frac{\epsilon}{2\eta\delta}\right).
\]
Applying Lemma \ref{Lem:T_g-convg-barT_g}, there exists   $\hat{K}$ depending on $\sqrt{\frac{\epsilon}{2\eta}}$ such that
\[
\mathbb{P}(|\Delta T_g| \geq \delta) \leq \mathfrak{R}\left(\sqrt{\frac{\epsilon}{2\eta}},K\right)
\] for $K>\hat{K}$.
By Remark \ref{re-exp},      $\left |g_{2}[k]\right |^2$ follows an exponential distribution with rate parameter $1$, so we have
\[
\mathbb{P}\left(|\mathbf{g}_2[k]| \geq \frac{\epsilon}{2\eta\delta}\right) \leq e^{-\left(\frac{\epsilon}{2\eta\delta}\right)^2} = e^{-\frac{\epsilon}{2\eta}}.
\] Combining all the bounds through the union bound completes the proof.
\end{proof}
The above theorem represents the culmination of the section's analysis, synthesizing the earlier results to provide a complete picture of the system's stability properties. The bound demonstrates that the probability of significant deviation between $\hat{y}_k$ and $\bar{y}_k$ diminishes exponentially with $K$, confirming that the asymptotic model becomes increasingly accurate in practical massive MIMO systems with large $K$. This result provides theoretical justification for relying on asymptotic analyses when designing and optimizing massive MIMO systems with linear-quantized precoding.

Next by the Lemma \ref{Lem:T_g-convg-barT_g}, we provide 
an upper bound of $\mathsf{dl}_{KF}\left(T_gg_{2}[k],\,\overline{T}_gg_{2}[k]\right)$.

\begin{lemma}\label{kf-tg}Under Assumptions \ref{Assu:H-n-s}-\ref{Assu:ratio-K-N} and Assumption \ref{Assu:d}, there exist constants \(C > 0\) and \(K_0 > 0\) such that for all \(K > K_0\) and all \(k \in [K]\),
\bgeqn 
\label{eq:Lem4.3-inequality}
\mathsf{dl}_{KF}\left(T_g g_2[k], \bar{T}_g g_2[k]\right) \leq C \frac{(\ln K)^{1/3}}{K^{1/3}},
\edeqn 
where   $\mathsf{dl}_{KF}$ 
is defined in \eqref{essential distance},
\(C = (68 c_3)^{1/3}, \quad K_0 = \max\left\{ \hat{K}(\epsilon), K_1, K_2, K_3, K_4 \right\}\),
\(c_3\) is the bound \(\tilde{\epsilon}_3 \leq c_3 / \epsilon^2\) in Lemma \ref{Lem:T_g-convg-barT_g-proof}, \(\hat{K}(\epsilon)\) is the threshold from Lemma \ref{Lem:T_g-convg-barT_g} for \(\epsilon = C/(K^{1/3}(\ln K)^{1/6})\),  
 \(K_i, i=1,2,3\) are the thresholds ensuring the exponential terms in \(\mathfrak{R}(\epsilon, K)\) and \(1/K\) is bounded by \(\delta/8\) for \(K > K_4\) with \(\delta = C (\ln K)^{1/3} / K^{1/3}\).  
\end{lemma}
\begin{proof} 
Recall that
$
\mathsf{dl}_{KF}(X,Y) = \inf\left\{ \delta > 0 : \mathbb{P}(|X-Y| > \delta) < \delta \right\}.
$
We aim to find \(\delta_K = C \frac{(\ln K)^{1/3}}{K^{1/3}}\) such that
\[
\mathbb{P}\left( |T_g g_2[k] - \bar{T}_g g_2[k]| > \delta_K \right) < \delta_K.
\]
Let \(M' := \sqrt{\ln K}\). By the triangle inequality and the subadditivity, 
\begin{equation}\label{le-kf-tg-1}
\mathbb{P}\left(|T_g g_2[k] - \bar{T}_g g_2[k]| > \delta\right) \leq \mathbb{P}\left(|T_g - \bar{T}_g| > \frac{\delta}{M'}\right) + \mathbb{P}\left(|g_2[k]| > M'\right). 
\end{equation}
Since \(g_2[k] \sim \mathcal{CN}(0,1)\), by Remark \ref{re-exp},      $\left |g_{2}[k]\right |^2$ follows an exponential distribution with rate parameter $1$, so we have
\begin{equation}\label{le-kf-tg-2}
\mathbb{P}\left(|g_2[k]| > M'\right) = e^{-M'^2} = e^{-\ln K} = \frac{1}{K}. 
\end{equation}
Let \(\epsilon := \frac{\delta}{M'} = \frac{\delta}{\sqrt{\ln K}}\). By Lemma \ref{Lem:T_g-convg-barT_g}, for any \(\epsilon > 0\), there exists \(\hat{K}(\epsilon) > 0\) such that for all \(K > \hat{K}(\epsilon)\),
\[
\mathbb{P}\left(|T_g - \bar{T}_g| \geq \epsilon\right) \leq \mathfrak{R}(\epsilon, K),
\]
where
\[
\mathfrak{R}(\epsilon, K) = 1594 e^{-K \tilde{\epsilon}_1} + 224 e^{-\sqrt{K} \tilde{\epsilon}_2} + \frac{17 \tilde{\epsilon}_3}{K} + 20 e^{-\frac{1}{2} (\gamma K - 1)\delta}.
\]
From the parameter definitions the proof of Lemma \ref{Lem:T_g-convg-barT_g}, there exist constants \(c_1, c_2, c_3, c_4 > 0\) such that for sufficiently small \(\epsilon > 0\):
\(\tilde{\epsilon}_1 \geq c_1 \epsilon^2\),
 \(\tilde{\epsilon}_2 \geq c_2 \epsilon^2\), 
\(\tilde{\epsilon}_3 \leq c_3 / \epsilon^2\),
\(\delta \geq c_4 \epsilon^2\).
Thus we have
\begin{equation}\label{le-kf-tg-3}
\mathfrak{R}(\epsilon, K) \leq 1594 e^{-c_1 K \epsilon^2} + 224 e^{-c_2 \sqrt{K} \epsilon^2} + \frac{17 c_3}{K \epsilon^2} + 20 e^{-c_4 \gamma K \epsilon^2 / 2}.
\end{equation}
Now choose \(\delta = C \frac{(\ln K)^{1/3}}{K^{1/3}}\) with \(C = (68 c_3)^{1/3}\). Then we have
\[
\epsilon = \frac{\delta}{\sqrt{\ln K}} = C \frac{1}{K^{1/3} (\ln K)^{1/6}}.
\]
We now bound each term in (\ref{le-kf-tg-1}) and (\ref{le-kf-tg-3}).
\[
\frac{17 c_3}{K \epsilon^2} = \frac{17 c_3}{K} \cdot \frac{K^{2/3} (\ln K)^{1/3}}{C^2} = \frac{17 c_3}{C^2} \frac{(\ln K)^{1/3}}{K^{1/3}}.
\]
Since \(C = (68 c_3)^{1/3}\), we have
\[
\frac{17 c_3}{C^2} = \frac{17 c_3}{(68 c_3)^{2/3}} = \frac{17}{68^{2/3}} c_3^{1/3} = \frac{1}{2} (68 c_3)^{1/3} = \frac{1}{2} C.
\]
Therefore we have
\begin{equation}\label{le-kf-tg-4}
\frac{17 c_3}{K \epsilon^2} = \frac{1}{2} C \frac{(\ln K)^{1/3}}{K^{1/3}} = \frac{1}{2} \delta. 
\end{equation}
Let \(A = c_1 K \epsilon^2 = c_1 C^2 \frac{K^{1/3}}{(\ln K)^{1/3}}\). Since \(A \to \infty\) as \(K \to \infty\), there exists \(K_1\) such that for all \(K > K_1\),
\[
1594 e^{-c_1 K \epsilon^2} \leq \frac{\delta}{8}.
\]
Let \(B = c_2 \sqrt{K} \epsilon^2 = c_2 C^2 \frac{1}{K^{1/6} (\ln K)^{1/3}}\). Since \(B \to \infty\) as \(K \to \infty\), there exists \(K_2\) such that for all \(K > K_2\),
\[
224 e^{-c_2 \sqrt{K} \epsilon^2} \leq \frac{\delta}{8}.
\]
Let \(D = c_4 \gamma K \epsilon^2 / 2 = \frac{c_4 \gamma}{2} C^2 \frac{K^{1/3}}{(\ln K)^{1/3}}\). Since \(D \to \infty\) as \(K \to \infty\), there exists \(K_3\) such that for all \(K > K_3\),
\[
20 e^{-c_4 \gamma K \epsilon^2 / 2} \leq \frac{\delta}{8}.
\]
Notice that
\[
\frac{1}{K} = \frac{1}{K} \cdot \frac{K^{1/3}}{(\ln K)^{1/3}} \cdot \frac{(\ln K)^{1/3}}{K^{1/3}} = \frac{1}{K^{2/3} (\ln K)^{1/3}} \delta.
\]
Since \(\frac{1}{K^{2/3} (\ln K)^{1/3}} \to 0\) as \(K \to \infty\), there exists \(K_4\) such that for all \(K > K_4\),
\(
\frac{1}{K} \leq \frac{\delta}{8}.
\)
Now let \(K_0 = \max\left\{ \hat{K}(\epsilon), K_1, K_2, K_3, K_4 \right\}\), where \(\hat{K}(\epsilon)\) is the threshold from Lemma \ref{Lem:T_g-convg-barT_g} for \(\epsilon = C/(K^{1/3}(\ln K)^{1/6})\).
Consequently, for all \(K > K_0\), combining (\ref{le-kf-tg-1})–(\ref{le-kf-tg-4}) and the above estimates, we have
\[
\mathbb{P}\left(|T_g g_2[k] - \bar{T}_g g_2[k]| > \delta\right) \leq \frac{\delta}{2} + \frac{\delta}{8} + \frac{\delta}{8} + \frac{\delta}{8} + \frac{\delta}{8} = \delta.
\]
By definition of the Ky Fan metric, we obtain \eqref{eq:Lem4.3-inequality}.
\end{proof}

In the same way, we can obtain  an upper bound of $
\dd_{HK}(T_s,\overline{T}_s)$. The next lemma states this.

\begin{lemma}\label{kf-ts}
Under Assumptions \ref{Assu:H-n-s}-\ref{Assu:ratio-K-N} and Assumption \ref{Assu:d},  there exist constants \(C' > 0\) and \(K_0' > 0\) such that for all \(K > K_0'\) and all \(k \in [K]\),
\bgeqn 
\mathsf{dl}_{KF}\left(T_s, \bar{T}_s\right) \leq C' \frac{(\ln K)^{1/3}}{K^{1/3}},
\edeqn 
where $\mathsf{dl}_{KF}$ is defined in \eqref{essential distance}, 
\(C' = (48 \tilde{c}_3)^{1/3}\) and threshold \(K_0' = \max\left\{ \tilde{K}, K_1, K_2, K_3, K_4, e \right\}\), 
\(\tilde{c}_3\) is from the bound \(\tilde{\epsilon}_6 \leq \tilde{c}_3 / \epsilon^2\) in Lemma \ref{Lem:T_s-convg-barT_s-proof}, \(\tilde{K}(\epsilon)\) is the threshold from Lemma \ref{Lem:T_s-convg-barT_s} for \(\epsilon =  C' (\ln K)^{1/3}/K^{1/3}\),  
 \(K_i, i=1,2,3\) are the thresholds ensuring the exponential terms in \(\tilde{\mathfrak{R}}(\epsilon,K)\) and \(\mathfrak{R}(\epsilon_{38},K)\) in 
 Lemma \ref{Lem:T_s-convg-barT_s-proof} is bounded by \(\delta/8\) for \(K > K_4\) with \(\delta = C' (\ln K)^{1/3} / K^{1/3}\). 
\end{lemma}
\begin{proof} 
We aim to find \(\delta_K = C' \frac{(\ln K)^{1/3}}{K^{1/3}}\) such that
$
\mathbb{P}\left( |T_s - \bar{T}_s| > \delta_K \right) < \delta_K.
$
From Lemma \ref{Lem:T_s-convg-barT_s}, for any \(\epsilon > 0\), there exists \(\tilde{K} > 0\) such that for all \(K > \tilde{K}\):
\[
\mathbb{P}\left(|T_s - \bar{T}_s| \geq \epsilon\right) \leq \tilde{\mathfrak{R}}(\epsilon,K),
\]
where
\[
\tilde{\mathfrak{R}}(\epsilon,K) = \mathfrak{R}(\epsilon_{38},K) + 255 e^{-K\tilde{\epsilon}_4} + 26 e^{-\sqrt{K}\tilde{\epsilon}_5} + \frac{6\tilde{\epsilon}_6}{K} + 6 e^{-\frac{1}{2}(\gamma K-1)\delta'}.
\]
By the definition of the parameters
in Lemma \ref{Lem:T_s-convg-barT_s}, there exist constants \(\tilde{c}_1, \tilde{c}_2, \tilde{c}_3, \tilde{c}_4 > 0\) such that for sufficiently small \(\epsilon > 0\),
 \(\tilde{\epsilon}_4 \geq \tilde{c}_1 \epsilon^2\), \(\tilde{\epsilon}_5 \geq \tilde{c}_2 \epsilon^2\), \(\tilde{\epsilon}_6 \leq \tilde{c}_3 / \epsilon^2\), \(\delta' \geq \tilde{c}_4 \epsilon^2\).
Thus we have
\begin{equation}\label{le-ky-ts-1}
\tilde{\mathfrak{R}}(\epsilon,K) \leq \mathfrak{R}(\epsilon_{38},K) + 255 e^{-\tilde{c}_1 K \epsilon^2} + 26 e^{-\tilde{c}_2 \sqrt{K} \epsilon^2} + \frac{6\tilde{c}_3}{K \epsilon^2} + 6 e^{-\tilde{c}_4 \gamma K \epsilon^2 / 2}. 
\end{equation}
Now choose \(\delta = C' \frac{(\ln K)^{1/3}}{K^{1/3}}\) with \(C' = (48 \tilde{c}_3)^{1/3}\). Let \(\epsilon = \delta\).
We 
bound each term in (\ref{le-ky-ts-1}),
\[
\frac{6\tilde{c}_3}{K \epsilon^2} = \frac{6\tilde{c}_3}{K} \cdot \frac{K^{2/3}}{C'^2 (\ln K)^{2/3}} = \frac{6\tilde{c}_3}{C'^2} \frac{1}{K^{1/3} (\ln K)^{2/3}}.
\]
Since \(C' = (48 \tilde{c}_3)^{1/3}\), then 
\[
\frac{6\tilde{c}_3}{C'^2} = \frac{6\tilde{c}_3}{(48 \tilde{c}_3)^{2/3}} = \frac{6}{48^{2/3}} \tilde{c}_3^{1/3} = \frac{1}{2} (48 \tilde{c}_3)^{1/3} = \frac{1}{2} C'.
\]
Thus
\begin{equation}\label{le-ky-ts-2}
\frac{6\tilde{c}_3}{K \epsilon^2} = \frac{1}{2} C' \frac{1}{K^{1/3} (\ln K)^{2/3}} = \frac{1}{2} \frac{\delta}{\ln K}. 
\end{equation}
Let \(A = \tilde{c}_1 K \epsilon^2 = \tilde{c}_1 C'^2 \frac{K^{1/3}}{(\ln K)^{1/3}}\). Since \(A \to \infty\) as \(K \to \infty\), there exists \(K_1\) such that for all \(K > K_1\),
\[
255 e^{-\tilde{c}_1 K \epsilon^2} \leq \frac{\delta}{8}.
\]
Let \(B = \tilde{c}_2 \sqrt{K} \epsilon^2 = \tilde{c}_2 C'^2 \frac{1}{K^{1/6} (\ln K)^{1/3}}\). Since \(B \to \infty\) as \(K \to \infty\), there exists \(K_2\) such that for all \(K > K_2\),
\[
26 e^{-\tilde{c}_2 \sqrt{K} \epsilon^2} \leq \frac{\delta}{8}.
\]
Let \(D = \tilde{c}_4 \gamma K \epsilon^2 / 2 = \frac{\tilde{c}_4 \gamma}{2} C'^2 \frac{K^{1/3}}{(\ln K)^{1/3}}\). Since \(D \to \infty\) as \(K \to \infty\), there exists \(K_3\) such that for all \(K > K_3\),
\[
6 e^{-\tilde{c}_4 \gamma K \epsilon^2 / 2} \leq \frac{\delta}{8}.
\]
From Lemma \ref{Lem:T_g-convg-barT_g}, \(\mathfrak{R}(\epsilon_{38},K)\) has similar exponential and polynomial decay. Since \(\epsilon_{38} = \tilde{\delta}(\bar{T}_g, 0, \delta/3)\) and \(\delta \to 0\) as \(K \to \infty\), there exists \(K_4\) such that for all \(K > K_4\),
\[
\mathfrak{R}(\epsilon_{38},K) \leq \frac{\delta}{8}.
\]
Let \(K_0' = \max\left\{ \tilde{K}, K_1, K_2, K_3, K_4, e \right\}\), where \(\tilde{K}\) is from Lemma \ref{Lem:T_s-convg-barT_s} for \(\epsilon = \delta\).
For all \(K > K_0'\), combining (\ref{le-ky-ts-1}), (\ref{le-ky-ts-2}) and the above estimates, we have
\[
\mathbb{P}\left(|T_s - \bar{T}_s| > \delta\right) \leq \frac{\delta}{8} + \frac{\delta}{8} + \frac{\delta}{8} + \frac{\delta}{8} + \frac{1}{2} \frac{\delta}{\ln K}\leq\delta.
\]
Thus, we have the conclusion.
\end{proof}

 Combining Theorem \ref{th-SEP-KF}, Lemma \ref{kf-tg} and Lemma \ref{kf-ts}, we obtain the upper bound of $\dd
 _{KF}(T_s,\overline{T}_s)$ and $ 
 \dd_{KF}(T_gg_{2}[k],\overline{T}_gg_{2}[k]) ,$ which leads to the upper bound of $|\widehat{\text{SEP}}_k(\beta) - \overline{\text{SEP}}(\beta)|$ directly.

\begin{theorem} 
\label{thm:hatSEP-convg-barSEP-error-bnd}
Let  Assumptions \ref{Assu:H-n-s}-\ref{Assu:ratio-K-N} and Assumption \ref{Assu:d}  hold.
Then there exists $\check{K}>0$ such that 
for $K>\check{K}$
\bgeqn 
|\widehat{\text{SEP}}_k(\beta) - \overline{\text{SEP}}(\beta)| \leq \frac{1}{M} \sum_{m=1}^{M} L_m \check{C} \frac{(\ln K)^{1/3}}{K^{1/3}}
\edeqn
for each $k\in [K]$,
where $\widehat{\text{SEP}}_k(\beta)$ and $\overline{\text{SEP}}(\beta)$ are defined in (\ref{HAT-SEP}) and (\ref{SEP-BAR}), respectively, $L_m$ is defined in (\ref{SEP ERROR-L}), $\check{K}=\max\{K_0, K_0'\}$, $\check{C}=2\max\{C, C'\}$ with $K_0, C$ and $K_0', C'$ being defined  in Lemmas \ref{kf-tg} and \ref{kf-ts}, respectively.
\end{theorem}

The theorem 
fills out a 
gap in massive MIMO performance analysis by 
taking a step from qualitative convergence analysis to 
quantitative estimation. 
While 
existing research 
establishes almost sure convergence of finite-dimensional performance metrics to their asymptotic limits, 
the rate of convergence 
remains unquantified. 
The theorem 
provides 
an explicit error bound of order $O(\sqrt[3]{\ln K/K})$ 
for the symbol error probability (SEP), elevating deterministic equivalents from random matrix theory from asymptotic predictions to reliable performance guarantees for finite-dimensional systems, thereby establishing a new paradigm for rigorous analysis of nonlinear quantized precoding systems.
The result offers 
some design guidelines for practical 5G/6G massive MIMO deployments. System engineers 
may determine the required number of antennas and DAC resolution for specific user populations using this error bound, effectively balancing hardware costs against communication performance. 
Particularly for systems employing low-resolution digital-to-analog converters, the theorem quantifies performance degradation due to finite-dimensional effects, providing a theoretical foundation for reliable deployment of link adaptation, beamforming, and resource allocation algorithms while significantly reducing the complexity of system simulation and validation.

\section{
Quantification of the impact on SINR} 

Optimizing the SINR in massive MIMO systems is essential for maximizing the system's performance, especially when dealing with finite-dimensional systems. The asymptotic analysis of SINR provides insights into the optimal precoding strategies under the assumption of an infinite number of antennas and users. However, translating these asymptotic results into practical optimization problems requires a detailed understanding of the stability and convergence properties of the SINR in finite-dimensional settings. This section focuses on the quantitative stability analysis of SINR-based optimization problems. It aims to provide a rigorous framework for characterizing the deviations of the finite-dimensional SINR from its asymptotic counterpart and to establish the convergence of optimal solutions in finite-dimensional systems to those in the asymptotic regime. The analysis involves formulating and solving optimization problems that maximize SINR under practical constraints, ensuring that the solutions are robust and applicable to real-world massive MIMO systems. The results contribute to the development of efficient and practical precoding strategies that can significantly enhance the performance of massive MIMO systems in finite-dimensional settings.
In \cite{WLS24}, the $\overline{\text{SINR}}$ maximization problem is formulated as:

\begin{subequations}
\label{eq:sinr_opt-1}
\bgeqn 
(P) \quad \sup_{f,\eta>0,\alpha>0} && \overline{\text{SINR}}(f,\eta,\alpha)\label{eq:sinr_opt-obj}\\
\inmat{s.t.}
&& \eta^2 \mathbb{E}[|q(\alpha Z)|^2] \leq P_T, \label{eq:sinr_opt-c1} \\
&&
\mathbb{E}[f^2(d)] = \frac{\gamma}{\sigma_s^2} \alpha^2, \label{eq:alpha_c2}
\edeqn 
\end{subequations}
where
$$\overline{\text{SINR}}(f,\eta,\alpha)=\frac{\mathbb{E}^2[df(d)]}{\text{var}[df(d)] + \phi(\alpha,\eta)\frac{\mathbb{E}[f^2(d)]}{\gamma}}.$$ As shown in \cite{WLS24}, the SINR maximization problem $(P)$ is motivated by the need to enhance the performance of massive MIMO systems, particularly in the context of downlink communication. The SINR is a crucial metric that quantifies the quality of signal reception at the user's end, taking into account not only the desired signal strength but also the interference and noise levels. Maximizing SINR directly translates to improving the reliability and efficiency of wireless communication systems. 
Next we give an explanation of the functions and variables involved in problem $(P)$ :
\begin{itemize}
    \item The choice of SINR as the objective function in (\ref{eq:sinr_opt-obj}) for $(P)$ reflects its importance as a fundamental performance metric, its analytical tractability, and its broad applicability across different communication scenarios.
     \item The constraint (\ref{eq:sinr_opt-c1}) ensures that the transmitted power, scaled by $\eta^2$ 
and averaged over the quantization function $q(\cdot)$, does not exceed the maximum allowable power $P_T$.


     \item The constraint (\ref{eq:alpha_c2}) ensures that the scaling of the precoding function $f:(0,\infty)\rightarrow\mathbb{R}$ is properly normalized with respect to the quantization input scaling  $\alpha$, the system noise parameters $\sigma_s^2$ and $\gamma$.
     \item  
     The variable \(d\) in constraint (\ref{eq:alpha_c2}) represents a singular value of \(\mathbf{H}\), defined as \(d = \sqrt{\lambda}\) where \(\lambda\) is an eigenvalue of \(\mathbf{H}\mathbf{H}^H\).
    In \cite{WLS24}, it is assumed that the elements of \(H\) are independent and identically distributed (i.i.d.) complex Gaussian random variables, i.e., \(H_{ij} \sim \mathcal{CN}(0, \frac{1}{N})\).
         Under this assumption, the eigenvalues \(\lambda\) of \(H H^H\) follow the Marchenko-Pastur distribution \cite{MP67,Bai2010}. Specifically, when \(N\) and \(K\) tend to infinity and \(N/K \to \gamma\) (where \(\gamma > 1\)), the empirical spectral distribution (e.s.d.) of \(\lambda\) converges almost surely to the Marchenko-Pastur distribution, whose probability density function is given by:
        \[
        p_\lambda(x) = \frac{1}{2\pi c x} \sqrt{(x - a)(b - x)} \quad \text{for } a \leq x \leq b,
        \]
        where \(a = (1 - \sqrt{c})^2\), \(b = (1 + \sqrt{c})^2\), and \(c = \frac{1}{\gamma}\).
         Therefore, the distribution of \(d = \sqrt{\lambda}\) can be described by the Marchenko-Pastur distribution. Specifically, the probability density function of \(d\) is:
        \[
        p_d(x) = \frac{1}{\pi c} \sqrt{\frac{(x^2 - a)(b - x^2)}{x^2}} \quad \text{for} \sqrt{a} \leq x \leq \sqrt{b}.
        \]
\end{itemize}
In 
\cite{WLS24}, 
the authors consider the SINR maximization problem $(P)$ 
with a primary focus on the performance analysis of massive MIMO systems when the number of transmit antennas $N$ and the number of users $K$ tend to infinity. This asymptotic analysis is of great theoretical significance as it provides deep insights into the optimal performance and behavior of the system in the limit. However, in practical applications, $N$ and $K$ are always finite. Therefore, to better bridge the gap between theoretical analysis and practical applications, this paper proposes the SINR maximization problem $(P_{N,K})$ under finite $N$ and $K$ conditions. The $k$-th user's SINR maximazation problem, that is, 
$\widehat{\text{SINR}}_k$ maximization problem, is formulated as:
\begin{subequations}
\label{eq:sinr_optk}
\bgeqn 
(P_{N,K})\quad  \sup_{f,\eta>0,\alpha>0} &&\widehat{\text{SINR}}_k(f,\eta,\alpha)\label{eq:sinr_optk-obj}\\
\inmat{s.t.} &&
\frac{\eta^2 \| q(\alpha z_1) \|^2}{N}\leq P_T,\label{eq:sinr_optk-c1}  \\
&&
\frac{\|s\|}{\|g_1\|}\frac{\|f(D)^Tg_1\|}{\|z_1\|} = \alpha, \label{eq:sinr_optk-c2}
\edeqn 
\end{subequations} 
and 
$$
\widehat{\text{SINR}}_k(f,\eta,\alpha)=\frac{|\rho_{k}|^{2}\,\mathbb{E}[|s_{k}|^{2}]}{\mathbb{E}[|\hat{y}_{k}|^{2}]-|\rho_{k}|^{2}\,\mathbb{E}[|s_{k}|^{2}]},$$
where 
 the objective function in (\ref{eq:sinr_optk-obj}) and the constraints (\ref{eq:sinr_optk-c1}),(\ref{eq:sinr_optk-c2}) in $(P_{N,K})$ are the approximations of 
(\ref{eq:sinr_opt-obj}),  (\ref{eq:sinr_opt-c1}) and (\ref{eq:alpha_c2}) in $(P)$, respectively, see the proof of Theorem 2 in \cite{WLS24}.
Let \(H\) be a \(K \times N\) random matrix with independent and identically distributed (i.i.d.) entries \(H_{ij} \sim \mathcal{CN}(0, \frac{1}{N})\). The \(D\) in (\ref{eq:alpha_c2}) is a \(K \times N\) matrix representing the singular value matrix in the singular value decomposition (SVD) of the stochastic channel matrix \(H\). 
Consequently,  we only need to consider $(P_{N,K})$ and $(P)$ when the involved $f$ lies in the space of bounded functions on $\Theta$, which is closed under the supreme norm but not compact in general. 
To facilitate the analysis, we confine the function 
$f$
 to a compact function space, that is, 
\bgeqn 
B(\Theta) = \left\{ f \in C(\Theta) : \|f\|_\infty \leq M_1, \text{Lip}(f) \leq L \right\},
\edeqn 
where \( M_1, L \) are constants, and \( \text{Lip}(f) \) denotes the Lipschitz modulus of \( f \).
We restrict the precoding function \(f\) to a compact subset of the space of bounded functions. This is achieved by confining \(f\) to a set \(B(\Theta)\) that is uniformly bounded and equicontinuous, as guaranteed by the Arzelà–Ascoli theorem. This assumption is not only mathematically convenient for ensuring compactness but is also eminently reasonable from an engineering standpoint. Practical precoding functions are inherently bounded due to transmit power constraints and exhibit limited variability because they are typically designed via well-behaved optimization procedures (e.g., regularized zero-forcing) or implemented as smooth approximations on hardware like FPGAs and DSPs. The Lipschitz condition, which underpins equicontinuity, naturally aligns with the stable, incremental adjustments of adaptive beamforming algorithms and the finite precision of digital signal processing.

In the case that a candidate function (e.g., an idealized quantizer) is discontinuous, it can be tackled within this framework through Lipschitz continuous envelope approximations, as detailed in Lemmas \ref{le-env}-\ref{le-6.11}. These envelopes—specifically, the lower and upper envelopes \(l_\tau(x)\) and \(u_\tau(x)\)—sandwich the discontinuous function and converge to it pointwise as the smoothing parameter \(\tau \to 0\). This technique effectively captures the essential behavior of hardware-induced nonlinearities, such as those from low-resolution DACs, while maintaining the analytical tractability afforded by the compact function space \(B(\Theta)\). Consequently, this approach provides a robust and unified theoretical foundation that bridges
asymptotic analysis with practical system design, ensuring that our stability guarantees remain relevant for real-world implementations with finite-dimensional constraints.

In what follows, we consider the function space $B(\Theta)$ equipped with infinity norm, i.e., 
      $
      \|f\|_\infty = \sup_{x \in \Theta} |f(x)|
      $
      for $f\in B(\Theta)$.
Let $\Upsilon:=B(\Theta)\times \mathbb{R}\times \mathbb{R}$ denote the Cartesian product of
spaces $B(\Theta)\times $, $\R$ and $\R$ equipped 
equipped with infinity norm
$
\|(f, \eta, \alpha)\|_\infty=\|f\|_\infty+|\eta|+|\alpha|.
$
We write 
$$\mathbb{B}((f, \eta, \alpha), \delta):=\{(f', \eta', \alpha')\in \Upsilon: d((f, \eta, \alpha), (f', \eta', \alpha'))<\delta\},$$
for the $\delta$-ball centered at  $(f, \eta, \alpha)$ 
and 
$\mathbb{B}_\infty $ 
for the closed unit ball in $\Upsilon$.
Define the distance from point \((f, \eta, \alpha) \in \Upsilon \) to set \( S \subset \Upsilon \) by
 $$ 
 d((f, \eta, \alpha), S) = \inf_{(f', \eta', \alpha') \in S} \|(f, \eta, \alpha) - (f', \eta', \alpha')\|_{\infty}
 $$ 
 For $S_1, S_2 \subset \Upsilon,$ let
     $ 
     \mathbb{D}(S_1, S_2) = \sup_{Q \in S_1} d(Q, S_2) $ denote the deviation of set \( S_1 \) from set \( S_2 \). For any $\rho>0$, $$ \mathbb{D}_{\rho}(S_1, S_2) = \sup_{Q \in S_1\cap\rho\mathbb{B}_\infty } d(Q, S_2) $$ is the $\rho$-deviation of set \( S_1 \) from set \( S_2 \).
    \bgeqn 
    \mathbb{H}(S_1, S_2) = \max \{\mathbb{ D}(S_1, S_2), \mathbb{D}(S_2, S_1) \} 
    \edeqn 
    is the Hausdorff distance between sets \( S_1 \) and \( S_2 \) and 
    \bgeqn 
    \label{rho-haus}\mathbb{H}_{\rho}(S_1, S_2) = \max \{\mathbb{ D}_{\rho}(S_1, S_2), \mathbb{D}_{\rho}(S_2, S_1) \} 
    \edeqn 
    is the \(\rho\)-Hausdorff distance between sets \( S_1 \) and \( S_2 \).
Let $\mathbb{V}$ and $\mathbb{V}_{N,K}$ denote the optimal values of $(P)$ and $(P_{N,K})$, respectively. Let $\mathbb{S}$ and $\mathbb{S}_{N,K}$ denote the optimal solutions of $(P)$ and $(P_{N,K})$, respectively. For the convenience of analysis, we introduce some notations:
\begin{subequations}
\label{eq:sets_and_functions}
\begin{align}
&\Phi := \left\{(f,\eta,\alpha) : \eta^2 \mathbb{E}[|q(\alpha Z)|^2] \leq P_T, \quad \mathbb{E}[f^2(d)] = \frac{\gamma}{\sigma_s^2} \alpha^2 \right\} 
\label{eq:Phi-1} 
\\[12pt]
&\Phi_{N,K} := \left\{(f,\eta,\alpha) : \frac{\eta^2 \| q(\alpha z_1) \|^2}{N} \leq P_T, \quad \frac{\|s\|}{\|g_1\|} \frac{\|f(D)^T g_1\|}{\|z_1\|} = \alpha \right\} 
\label{eq:Phi_NK-1} \\[12pt]
&\widehat{F}_{N,K}(f, \eta,\alpha) := -\widehat{\text{SINR}}_k(f, \eta,\alpha) + \delta_{\Phi_{N,K}}(f, \eta,\alpha), \label{eq:F_NK} \\[12pt]
&F(f, \eta,\alpha) := -\overline{\text{SINR}}(f, \eta,\alpha) + \delta_{\Phi}(f, \eta,\alpha), \label{eq:F}
\end{align}
\end{subequations}
where the function $\delta_{C}$  assigns $1$ to elements that belong to a set $C$ and $\infty$ to elements that do not belong to the set.
 The next theorem 
 describes the relationship between the finite-dimensional optimization problem $(P_{N,K})$ and its asymptotic counterpart $(P)$
 in terms of 
 the feasible sets, the set of the optimal solutions, and the optimal values
 as $K$ increases.
\begin{theorem} \label{th-HDV}Under Assumptions \ref{Assu:H-n-s}--\ref{Assu:ratio-K-N}, 
the following assertions hold. 
\begin{itemize}
    \item [(i)] For any $\rho>0, \mathbb{H}_{\rho}(\Phi_{N,K}, \Phi)\xrightarrow{a.s.} 0$ as $K\rightarrow +\infty$.
    
    \item [(ii)] Suppose \(\mathbb{V}\) is finite. Then  for any $\rho>0$, $\mathbb{D}_{\rho}(\mathbb{S}_{N,K}, \mathbb{S})\xrightarrow{a.s.} 0$ as $K\rightarrow +\infty$.

     \item [(iii)]  Suppose \(\mathbb{V}\) is finite and there exists $\rho>0$ and $\bar{K}>0$ such that for $K>\bar{K}$, $\mathbb{S}_{N,K}\cap \rho\mathbb{B}_{\infty}\neq \emptyset$. Then  $|\mathbb{V}_{N,K}-\mathbb{V}|\xrightarrow{a.s.} 0$ as $K\rightarrow +\infty$.
    
\end{itemize}
\end{theorem}

\begin{proof}
Part (i). We 
proceed the proof 
in two steps.

\textbf{Step 1}. For any $\rho>0$, $\mathbb{D}_{\rho}(\Phi_{N,K}, \Phi)\xrightarrow{a.s.} 0$ as $K\rightarrow +\infty$.

By the definition of $\mathbb{D}_{\rho}$ and the compactness of $\rho\mathbb{B}_\infty$, for each $N,K$, there exists $(\hat{f}_{N,K},\hat{\eta}_{N,K},\hat{\alpha}_{N,K} )\in \Phi_{N,K}\cap\rho\mathbb{B}_\infty$ such that 
$$
\mathbb{D}_{\rho}(\Phi_{N,K}, \Phi)=d((\hat{f}_{N,K},\hat{\eta}_{N,K},\hat{\alpha}_{N,K} ), \Phi).
$$
Then we only need to show for $(\hat{f}_{N,K},\hat{\eta}_{N,K},\hat{\alpha}_{N,K} )\in \Phi_{N,K}\cap\rho\mathbb{B}_\infty$, if $(\hat{f}_{N,K},\hat{\eta}_{N,K},\hat{\alpha}_{N,K} )\xrightarrow{a.s.}  (\bar{f},\bar{\eta},\bar{\alpha})$ as $K\rightarrow +\infty$, then we have 
$(\bar{f},\bar{\eta},\bar{\alpha})\in \Phi$. Indeed, for $(\hat{f}_{N,K},\hat{\eta}_{N,K},\hat{\alpha}_{N,K} )\in \Phi_{N,K}$, we have
$$
\frac{\hat{\eta}_{N,K}^2 \| q(\hat{\alpha}_{N,K} )z_1 \|^2}{N}\leq P_T,  
\qquad \qquad\frac{\|s\|}{\|g_1\|}\frac{\|\hat{f}_{N,K}(D)^Tg_1\|}{\|z_1\|} = \hat{\alpha}_{N,K}.
$$
By the proof of Theorem 2 in \cite{WLS24}, we have
$$
\frac{\hat{\eta}_{N,K}^2 \| q(\hat{\alpha}_{N,K}) z_1 \|^2}{N}\xrightarrow{a.s.}  \bar{\eta}^2 \mathbb{E}[|q(\bar{\alpha}Z)|^2]
$$
and
$$
\frac{\|s\|}{\|g_1\|}\frac{\|\hat{f}_{N,K}(D)^Tg_1\|}{\|z_1\|}\xrightarrow{a.s.} \sqrt{\frac{\sigma_s^2\mathbb{E}[\bar{f}^2(d)]}{\gamma}}
$$
as $K\rightarrow +\infty$,
which means that $(\bar{f},\bar{\eta},\bar{\alpha})\in \Phi$.

\textbf{Step 2}. 
For any $\rho>0$, $\mathbb{D}_{\rho}(\Phi, \Phi_{N,K})\xrightarrow{a.s.} 0$ as $K\rightarrow +\infty$.

By the definition of $\mathbb{D}_{\rho}$  and the compactness of $\rho\mathbb{B}_\infty$, for each $N,K$, there exists $(\hat{f}_{N,K},\hat{\eta}_{N,K},\hat{\alpha}_{N,K} )\in \Phi\cap\rho\mathbb{B}_\infty$ such that 
$$
\mathbb{D}_{\rho}(\Phi, \Phi_{N,K})=d((\hat{f}_{N,K},\hat{\eta}_{N,K},\hat{\alpha}_{N,K} ), \Phi_{N,K}).
$$
any $(\hat{f}_{N,K},\hat{\eta}_{N,K},\hat{\alpha}_{N,K} )\in \Phi\cap\rho\mathbb{B}_\infty$, we have 
$$
(P_T)_{ N,K}:=\hat{\eta}_{N,K}^2 \mathbb{E}[|q(\hat{\alpha}_{N,K}Z)|^2] \leq P_T, 
 \qquad\mathbb{E}[\hat{f}_{N,K}^2(d)] = \frac{\gamma}{\sigma_s^2} \hat{\alpha}_{N,K}^2. 
$$
Then we construct
$$
\tilde{\alpha}_{N,K}=\frac{\|s\|}{\|g_1\|}\frac{\|\hat{f}_{N,K}(D)^Tg_1\|}{\|z_1\|}>0,\quad \tilde{\eta}_{N,K}=(P_T)_{ N,K}{\frac{\| q(\tilde{\alpha}_{N,K} z_1) \|^2}{N}}
$$
and hence
$(\hat{f}_{N,K},\tilde{\eta}_{N,K},\tilde{\alpha}_{N,K} )\in \Phi_{N,K}$.
Therefore, we have 
\begin{align*}
d((\hat{f}_{N,K},\hat{\eta}_{N,K},\hat{\alpha}_{N,K} ), \Phi_{N,K})&\leq \left\|  (\hat{f}_{N,K},\hat{\eta}_{N,K},\hat{\alpha}_{N,K} )-(\hat{f}_{N,K},\tilde{\eta}_{N,K},\tilde{\alpha}_{N,K} )\right\|_{\infty}
\\&=|\hat{\eta}_{N,K}-\tilde{\eta}_{N,K}|+|\hat{\alpha}_{N,K}-\tilde{\alpha}_{N,K}|.
\end{align*}
Since for any $N,K$, $(\hat{f}_{N,K},\hat{\eta}_{N,K},\hat{\alpha}_{N,K} )\in \rho\mathbb{B}_\infty$ and $0\leq(P_T)_{ N,K}\leq P_T$, 
by taking a subsequence if necessary, 
we assume for the simplicity of notation that
$(\hat{f}_{N,K},\hat{\eta}_{N,K},\hat{\alpha}_{N,K} )\xrightarrow{a.s.} (\bar{f},\bar{\eta},\bar{\alpha})$ as $K\rightarrow +\infty$.
By the proof 
of Theorem 2 in \cite{WLS24}, we have 
$$
|\hat{\alpha}_{N,K}-\tilde{\alpha}_{N,K}|=\left|\sqrt{\frac{\|s\|}{\|g_1\|}\frac{\|\hat{f}_{N,K}(D)^Tg_1\|}{\|z_1\|}}-\sqrt{\frac{\sigma_s^2\mathbb{E}[\hat{f}_{N,K}^2(d)]}{\gamma}}\right|\xrightarrow{a.s.} 0
$$
and 
$$
|\hat{\eta}_{N,K}-\tilde{\eta}_{N,K}|=\left|\sqrt{\frac{(P_T)_{N,K}}{\mathbb{E}[\hat{f}_{N,K}^2(d)]}}- \sqrt{\frac{(P_T)_{N,K}}{\frac{\| q(\tilde{\alpha}_{N,K} z_1)\|^2}{N}}} 
 \right|\xrightarrow{a.s.} 0
$$ as $K\rightarrow +\infty$, which implies
that $d((\hat{f}_{N,K},\hat{\eta}_{N,K},\hat{\alpha}_{N,K} ), \Phi_{N,K}) \xrightarrow{a.s.} 0$, which means $\mathbb{D}_{\rho}(\Phi, \Phi_{N,K})\xrightarrow{a.s.} 0$ as $K\rightarrow +\infty$.

Part (ii).  We 
first show $ e - \lim_{K \to \infty} \widehat{F}_{N,K} = F \quad \text{almost\,\, surely},$  where $e-\lim$ denotes the epi-convergence in Definition \ref{de-epi-co}.
Recall 
that
\[
\text{epi} \left( \delta_{\Phi_{N,K}} (\cdot) \right) = \left\{ (f, \alpha, \eta) \mid (f, \eta,\alpha) \in\Phi_{N,K}, \alpha \geq 0 \right\} = \Phi_{N,K} \times \mathbb{R}_+
\]
and by Part (i),
\[
\lim_{K \to \infty} \Phi_{N,K} \times \mathbb{R}_+ = \Phi \times \mathbb{R}_+ \quad 
\]
almost surely, we obtain from Definition \ref{de-epi-co} that
\[
e - \lim_{K \to \infty} \delta_{\Phi_{N,K}} (\cdot) = \delta_{\Phi} (\cdot) 
\]
almost surely, which, by Proposition \ref{pro-epi}, means that for any sequence \((f_{N,K}, \eta_{N,K}, \alpha_{N,K}) \xrightarrow{a.s.}  (\bar{f}, \bar{\eta},\bar{\alpha})\), 
\[
\liminf_{K \to \infty} \delta_{\Phi_{N,K}}(f_{N,K}, \eta_{N,K},\alpha_{N,K}) \geq \delta_{\Phi}(\bar{f}, \bar{\eta}, \bar{\alpha}) \quad
\]
almost surely and there exists \((f_{N,K},  \eta_{N,K}, \alpha_{N,K}) \xrightarrow{a.s.}  (\bar{f},  \bar{\eta}, \bar{\alpha})\) such that
\[
\limsup_{K \to \infty} \delta_{\Phi_{N,K}}(f_{N,K}, \eta_{N,K}, \alpha_{N,K}, ) \leq \delta_{\Phi}(\bar{f},  \bar{\eta}, \bar{\alpha}) 
\]
almost surely. We know from the proof of Theorem 2 in \cite{WLS24} that
\[
\widehat{\text{SINR}}_k(f_{N,K}, \eta_{N,K}, \alpha_{N,K}) \xrightarrow{a.s.}\overline{\text{SINR}}(\bar{f}, \bar{\eta}, \bar{\alpha}) \]   as $ K \to \infty$,
which leads to
\[
\liminf_{K \to \infty} \widehat{F}_{N,K}(f_{N,K}, \eta_{N,K}, \alpha_{N,K}) \geq F(\bar{f}, \bar{\eta}, \bar{\alpha}) \]
and there exists \((f_{N,K}, \eta_{N,K}, \alpha_{N,K} ) \xrightarrow{a.s.}  (\bar{f}, \bar{\eta}, \bar{\alpha})\)  such that
\[
\limsup_{K \to \infty}\widehat{F}_{N,K}(f_{N,K}, \eta_{N,K}, \alpha_{N,K})  \leq F(\bar{f},  \bar{\eta}, \bar{\alpha})  
\] almost surely, where $\limsup$ and $\liminf$ denotes the upper and lower limits of function in Definition \ref{u-l-lim-f}.
Then also by Proposition \ref{pro-epi}, we have the epi-convergence of $\widehat{F}_{N,K} $ to \(F\) almost surely. 
Since $\mathbb{V}$ is finite,  by Proposition \ref{inf-S-V}, we have that  $\limsup_{K\rightarrow \infty} \mathbb{S}_{N,K}\subseteq \mathbb{S}$ almost surely. Then by Proposition~\ref{outer-limite-equi}, we complete the proof.

Part (iii). Let $$v_{N,K}:=\inf_{(f,  \eta, \alpha)\in \Phi_{N,K}}(-\widehat{\text{SINR}}_k(f,\eta,\alpha))\quad \mbox{and} \quad v:=\inf_{(f, \alpha, \eta)\in \Phi}(-\overline{\text{SINR}}(f,\eta,\alpha)).$$ Then we have $\mathbb{V}_{N,K}=-v_{N,K}$ and 
$\mathbb{V}=-v$.
If there exists $\rho>0$ and $\bar{K}>0$ such that for $K>\bar{K}$, $\mathbb{S}_{N,K}\cap \rho\mathbb{B}_{\infty}\neq \emptyset$,then we have for $K>\bar{K}$, 
$$
\inf_{(f, \eta,\alpha)\in \Phi_{N,K}\cap \rho\mathbb{B}}(-\widehat{\text{SINR}}_k(f,\eta,\alpha))=\inf_{(f, \eta,\alpha)\in  \rho\mathbb{B}}(\widehat{F}_{N,K}(f,\eta,\alpha))=\inf_{(f, \eta,\alpha)\in  \Upsilon}(\widehat{F}_{N,K}(f,\eta,\alpha)).
$$
We know form the proof of (ii) that $ e - \lim_{K \to \infty} \widehat{F}_{N,K} = F \quad \text{almost\,\, surely}.$
Then under the finiteness of $\bar{V}$, by Proposition \ref{inf-S-V}, we have $v_{N,K}\xrightarrow{a.s.} v $ as $K\rightarrow \infty$, which means that $\mathbb{V}_{N,K}\xrightarrow{a.s.} \mathbb{V} $ as $K\rightarrow \infty$.
\end{proof}
\begin{remark}
By the definitions of outer semi convergence and inner semi convergence of of sets in Definition \ref{Def-outer-inter} and Proposition \ref{outer-limite-equi}, we know from the proof of $(i)$ and $(ii)$ in Theorem \ref{th-HDV} that $$\lim_{K\rightarrow \infty}\Phi_{N,K}= \Phi\,\,\mbox{and} \,\,\limsup_{K\rightarrow \infty}\mathbb{S}_{N,K}= \mathbb{S}
$$
almost surely.
\end{remark}

The above theorem establishes that the feasible set of the finite-dimensional problem converges to the asymptotic feasible set  in the Hausdorff sense. Furthermore, it shows that the optimal solutions of the finite-dimensional problem converge to those of the asymptotic problem, and the optimal values converge accordingly. This result provides a rigorous foundation for using asymptotic analyses to guide the design of practical massive MIMO systems, ensuring that the optimal solutions obtained in the asymptotic regime remain effective in finite-dimensional settings.

Problem $(P)$ is a non-convex optimization problem with nonlinear constraints, which can make it challenging to analyze the stability of the optimal value and solution. Specifically, the presence of function variables in the optimization variables adds another layer of complexity. 
In the paper \cite{WLS24}, the authors have already provided a solution method for problem $(P)$, and the obtained solution makes the inequality constraint of (\ref{eq:sinr_opt-c1}) active. From this perspective, transforming the study of problem $(P)$ with inequality constraint  into  $(P)$  with equality constraint is meaningful.
Moreover, transforming the problem into an equality constraint not only simplifies the optimization process and enhances stability but also provides practical insights for designing energy-efficient communication systems under strict power constraints.
Based on the discussions above, we  
We introduce programs $(\widehat{P})$ and $(\widehat{P}_{N,K})$ as follows:
\begin{subequations}
\bgeqn 
\label{eq:sinr_opt}
(\widehat{P}) \quad  \sup_{f,\eta>0,\alpha>0}
&& \frac{\mathbb{E}^2[df(d)]}{\text{var}[df(d)] + \phi(\alpha,\eta)\frac{\mathbb{E}[f^2(d)]}{\gamma}}\\
\inmat{s.t.} && \eta^2 \mathbb{E}[|q(\alpha Z)|^2] = P_T,  \\
&& 
\mathbb{E}[f^2(d)] = \frac{\gamma}{\sigma_s^2} \alpha^2, \label{eq:alpha_constraint-P}
\edeqn
\end{subequations}
and
\begin{subequations}
\bgeqn 
\label{eq:sinr_optk-hat{P}_N,K}
(\widehat{P}_{N,K}) \quad  \sup_{f,\eta>0,\alpha>0}
&& \frac{|\rho_{k}|^{2}\,\mathbb{E}[|s_{k}|^{2}]}{\mathbb{E}[|y_{k}|^{2}]-|\rho_{k}|^{2}\,\mathbb{E}[|s_{k}|^{2}]}\\
\inmat{s.t.} && \frac{\eta^2 \| q(\alpha z_1) \|^2}{N}= P_T,  \\
&& \frac{\|s\|}{\|g_1\|}\frac{\|f(D)^Tg_1\|}{\|z_1\|} = \alpha. \label{eq:alpha_constraint-P_N,K}
\edeqn 
\end{subequations}
Let $\widehat{\mathbb{V}}$ and $\widehat{\mathbb{V}}_{N,K}$ denote the optimal values of $(\widehat{P})$ and $(\widehat{P}_{N,K})$, respectively. Let $\widehat{\mathbb{S}}$ and $\widehat{\mathbb{S}}_{N,K}$ denote the optimal solutions of $(\widehat{P})$ and $(\widehat{P}_{N,K})$, respectively. For the convenience of analysis, we introduce  notations $\widehat{\Phi}, \widehat{\Phi}_{N,K} $ and the functions $\tilde{\alpha}_{N,K}(\cdot), \tilde{\eta}_{N,K}(\cdot), \tilde{\alpha}(\cdot), \tilde{\eta}(\cdot),\sigma_{N,K}(\cdot), \sigma(\cdot) $ defined by
\begin{subequations}
\bgeqn 
&\widehat{\Phi}
:=\left\{(f,\eta,\alpha):\eta^2\mathbb{E}[|q(\alpha Z)|^2] = P_T, \qquad\mathbb{E}[f^2(d)] = \frac{\gamma}{\sigma_s^2} \alpha^2 \right\}, 
\label{eq:Phi}
\\[12pt]
&\widehat{\Phi}_{N,K}
:=\left\{(f,\eta,\alpha):\frac{\eta^2 \| q(\alpha z_1) \|^2}{N}= P_T, \qquad\frac{\|s\|}{\|g_1\|}\frac{\|f(D)^Tg_1\|}{\|z_1\|} = \alpha \right\},\label{eq:Phi_NK}\\[12pt]
&\tilde{\alpha}_{N,K}(f)
=\frac{\|s\|}{\|g_1\|}\frac{\|f(D)^Tg_1\|}{\|z_1\|}, \quad \tilde{\eta}_{N,K}(f)=\frac{\sqrt{N P_T}}{\| q(\tilde{\alpha}_{N,K}(f) z_1) \|},  \label{eq:tilde_alpha_eta_NK}\\[12pt]
&\tilde{\alpha}(f)
=\sqrt{\frac{\sigma_s^2\mathbb{E}[f^2(d)]}{\gamma}}, \quad \tilde{\eta}(f)=\sqrt{\frac{P_T}{\mathbb{E}[|q(\tilde{\alpha}(f) Z)|^2]}}, \label{eq:alpha_eta}\\[12pt]
&\sigma_{N,K}(f)=(\tilde{\eta}_{N,K}(f), \tilde{\alpha}_{N,K}(f)), \quad \sigma(f)=(\tilde{\eta}(f),\tilde{\alpha}(f)) \label{eq:sigma}
\edeqn 
\end{subequations}
respectively. We know from (\ref{eq:tilde_alpha_eta_NK}) and (\ref{eq:alpha_eta}) that
for any $f\in  B(\Theta) , (f,\tilde{\alpha}_{N,K}(f), \tilde{\eta}_{N,K}(f) )\in \widehat{\Phi}_{N,K}$ and $(f,\tilde{\alpha}(f), \tilde{\eta}(f) )\in \widehat{\Phi}.$

In a similar vein to the proof of Theorem \ref{th-HDV}, one can derive the asymptotic analysis of the feasible set, optimal solution set, and optimal value of $(\widehat{P}_{N,K})$ as $N$ and $K$ tend to infinity, transitioning to the asymptotic problem $(\widehat{P})$. For brevity, we omit the detailed proof here.
\begin{theorem}\label{th-HDV-eq} Under Assumptions \ref{Assu:H-n-s}--\ref{Assu:ratio-K-N}, 
the following assertions hold. 
\begin{itemize}
    \item [(i)] For any $\rho>0, \mathbb{H}_{\rho}(\widehat{\Phi}_{N,K}, \widehat{\Phi})\xrightarrow{a.s.} 0$ as $K\rightarrow +\infty$.
    \item [(ii)] If \(\widehat{\mathbb{V}}\) is finite, then  for any $\rho>0$, $\mathbb{D}_{\rho}(\widehat{\mathbb{S}}_{N,K}, \widehat{\mathbb{S}})\xrightarrow{a.s.} 0$ as $K\rightarrow +\infty$.
     \item [(iii)]  If \(\widehat{\mathbb{V}}\) is finite and there exists $\rho>0$ and $\bar{K}>0$ such that for $K>\bar{K}$, $\widehat{\mathbb{S}}_{N,K}\cap \rho\mathbb{B}_{\infty}\neq \emptyset$, then  $|\widehat{\mathbb{V}}_{N,K}-\widehat{\mathbb{V}}|\xrightarrow{a.s.} 0$ as $K\rightarrow +\infty$.
\end{itemize}
\end{theorem}

\begin{remark}
By the definitions of outer 
semi-convergence and inner 
semi-convergence of of sets in Definition \ref{Def-outer-inter} and Proposition \ref{outer-limite-equi}, we know from the proof of $(i)$ and $(ii)$ in Theorem \ref{th-HDV-eq} that 
\bgeqn 
\lim_{K\rightarrow \infty}\widehat{\Phi}_{N,K}= \widehat{\Phi}\,\,\mbox{and} \,\,\limsup_{K\rightarrow \infty}\widehat{\mathbb{S}}_{N,K}= \widehat{\mathbb{S}}
\edeqn 
almost surely.
\end{remark}
The next theorem provides a quantitative measure of the distance between the feasible sets of the finite-dimensional and asymptotic equality-constrained optimization problems. Understanding this distance helps quantify how well the finite-dimensional problem approximates the asymptotic one.
\begin{lemma}\label{th-hau-Phi-error} Suppose  Assumptions \ref{Assu:H-n-s}--\ref{Assu:ratio-K-N} hold. Then
for any $\rho>0, $ we have   
\bgeqn 
\mathbb{H}_{\rho}\left(\widehat{\Phi}_{N,K}, \widehat{\Phi}\right)\leq \sup_{\|f\|_{\infty}\leq \rho}\|\sigma_{N,K}(f)-\sigma(f)\|,
\edeqn 
where $\sigma_{N,K}(\cdot)$ and $\sigma(\cdot)$ are defined in (\ref{eq:sigma}), $\mathbb{H}_{\rho}$ is the \(\rho\)-Hausdorff distance  defined in (\ref{rho-haus}). 
 \end{lemma}
 
\begin{remark}
This inequality provides a fundamental quantitative relationship between the approximation error in the parameter mappings and the geometric discrepancy between the finite-dimensional and asymptotic feasible sets. Specifically:

\begin{itemize}
\item $\mathbb{H}_{\rho}\left(\widehat{\Phi}_{N,K}, \widehat{\Phi}\right)$ measures the maximum distance between the finite-dimensional feasible set $\widehat{\Phi}_{N,K}$ and its asymptotic counterpart $\widehat{\Phi}$, restricted to the $\rho$-ball $\rho\mathbb{B}_{\infty}$. This quantifies how well the finite-dimensional problem approximates the asymptotic one in terms of feasible solutions.

\item  $\sup_{\|f\|_{\infty}\leq \rho}\|\sigma_{N,K}(f)-\sigma(f)\|$ captures the worst-case deviation between the finite-dimensional parameter mapping $\sigma_{N,K}(f)$ and its asymptotic limit $\sigma(f)$ over all precoding functions $f$ with $\|f\|_{\infty} \leq \rho$.

 \item The inequality shows that the Hausdorff distance between the feasible sets is controlled by the uniform approximation error in the parameter mappings. As $K \to \infty$, if the parameter mappings converge uniformly over bounded functions, then the feasible sets converge in the Hausdorff sense. This provides a mechanism for transferring convergence from parameter estimation to set-valued constraints.
\end{itemize}
 \end{remark}

\begin{proof}
For  any $\delta>0, $ we devide the proof to two steps.

\textbf{Step 1.} Estimate $\mathbb{D}_{\rho}\left(\widehat{\Phi}_{N,K}, \widehat{\Phi}\right).$

Notice that $$\mathbb{D}_{\rho}\left(\widehat{\Phi}_{N,K}, \widehat{\Phi}\right)=\sup_{(f, \alpha, \eta)\in \widehat{\Phi}_{N,K}\cap \rho\mathbb{B}_{\infty}} d((f, \alpha, \eta),\widehat{\Phi} ).
$$
For any $(f, \alpha, \eta)\in \widehat{\Phi}_{N,K}\cap  \rho\mathbb{B}_{\infty}$, we have
\begin{subequations}
\bgeqn 
d((f, \eta,\alpha), \widehat{\Phi})
&\leq& d((f, \eta,\alpha), (f,\tilde{\eta}(f)),\tilde{\alpha}(f) ) 
= d(( \eta,\alpha), (\tilde{\eta}(f)) ,\tilde{\alpha}(f))\\[12pt]
&=& d(  ( \tilde{\eta}_{N,K}(f),\tilde{\alpha}_{N,K}(f)), (\tilde{\eta}(f),\tilde{\alpha}(f)) )\leq \sup_{\|f\|_{\infty}\leq \rho}\|\sigma_{N,K}(f)-\sigma(f)\|.
\edeqn 
\end{subequations}
So we have
$$\mathbb{D}_{\rho}\left(\widehat{\Phi}_{N,K}, \widehat{\Phi}\right)\leq \sup_{\|f\|_{\infty}\leq \rho}\|\sigma_{N,K}(f)-\sigma(f)\|.$$

\textbf{Step 2.} Estimate $\mathbb{D}_{\rho}\left(\widehat{\Phi}, \widehat{\Phi}_{N,K}\right).$
Observe 
that $$\mathbb{D}_{\rho}\left(\widehat{\Phi}, \widehat{\Phi}_{N,K}\right)=\sup_{(f, \eta, \alpha)\in \widehat{\Phi}\cap \rho\mathbb{B}_{\infty}} d((f,\eta,\alpha),\widehat{\Phi}_{N,K} ).$$
For any $(f, \eta,\alpha)\in \widehat{\Phi}\cap \rho\mathbb{B}_{\infty}$, 
\begin{subequations}
\bgeq 
d((f, \eta,\alpha), \widehat{\Phi}_{N,K})
&\leq& d((f, \eta,\alpha), (f,\tilde{\eta}_{N,K}(f)),\tilde{\alpha}_{N,K}(f) )
= d(( \eta,\alpha), (\tilde{\eta}_{N,K}(f)),\tilde{\alpha}_{N,K}(f) )\\[12pt]
&=& d(( \tilde{\eta}_{N,K}(f)),\tilde{\alpha}_{N,K}(f)), (\tilde{\eta}(f),\tilde{\alpha}(f)) )\leq \sup_{\|f\|_{\infty}\leq \rho}\|\sigma_{N,K}(f)-\sigma(f)\|.
\edeq 
\end{subequations}
Therefore, we arrive at
$$
\mathbb{D}_{\rho}\left(\widehat{\Phi}, \widehat{\Phi}_{N,K}\right)\leq \sup_{\|f\|_{\infty}\leq \rho}\|\sigma_{N,K}(f)-\sigma(f)\|.$$
Combining the proof in Step 1 and Step 2, we complete the proof.
\end{proof}
The above  theorem establishes that the Hausdorff distance between the feasible sets  $\widehat{\Phi}_{N,K}$ and $\widehat{\Phi}$ 
  is bounded by the supremum of the distance between the parameter mappings $\sigma_{N,K}(f)$ and 
 $\sigma(f)$. This result provides a direct quantitative relationship between the approximation quality of the finite-dimensional problem and the convergence of key parameters, offering valuable insights for assessing the practical performance of optimization algorithms in finite-dimensional settings.
Let functions $\mathbb{Y}^k_{N,K}(\cdot), \bar{\mathbb{Y}}(\cdot)$ and $\mathcal{D}(\mathbb{Y}^k_{N,K}(\cdot),\bar{\mathbb{Y}}(\cdot))$ be defined by 
\begin{subequations}
\bgeqn 
\mathbb{Y}^k_{N,K}(f):&=&\left((\hat{y}_k)_{N,K}(f,\sigma_{N,K}(f)), \bar{y}_k(f,\sigma_{N,K}(f))\right),\\[12pt]
\overline{\mathbb{Y}}(f):&=&\left( \bar{y}_k(f,\sigma(f)), \bar{y}_k(f,\sigma(f))\right),\\[12pt]\label{D-Y-1}
\mathcal{D}(\mathbb{Y}^k_{N,K}(f),\bar{\mathbb{Y}}(f))&=&\mathbb{E}^{\frac{1}{2}}\left[ \left| (\hat{y}_k)_{N,K}(f,\sigma_{N,K}(f)) - \bar{y}_k(f,\sigma(f) ) \right|^2 \right]\\\label{D-Y-2}&&+\mathbb{E}^{\frac{1}{2}}\left[ \left| \bar{y}_k(f,\sigma_{N,K}(f)) - \bar{y}_k(f,\sigma(f)) \right|^2 \right].
\edeqn 
\end{subequations}

Based on the above notations, by Proposition \ref{th-SINR-FUN-ERR}, we have the following lemma directly, which aims to quantify the error between the finite-dimensional SINR  
and its asymptotic counterpart. Understanding this error is crucial for establishing how well the finite-dimensional optimization problem approximates the asymptotic one in terms of the objective function.

\begin{lemma}\label{lemma-SINR-error}
Under Assumptions \ref{Assu:H-n-s}--\ref{Assu:ratio-K-N},  for every $\rho > 0$, there exists $\bar{K} = \bar{K}(\rho) > 0$ such that for all $K > \bar{K}$, there exists $L_{\rho} > 0$ satisfying the following inequality 
    \begin{align}
\sup_{\|f\|_{\infty}\leq \rho}\left|\widehat{\text{SINR}}_k(f,\sigma_{N,K}(f)) - \overline{\text{SINR}}(f,\sigma(f))\right| 
\leq  L_{\rho}\sup_{\|f\|_{\infty}\leq \rho}\mathcal{D}(\mathbb{Y}^k_{N,K}(f),\bar{\mathbb{Y}}(f))
\end{align}
    for all \(k \in [K]\), where \(L_{\rho}:=\sup_{\|f\|_{\infty}\leq \rho}L(f,\sigma(f) )\) with $L(f,\sigma(f) )$ being defined in (\ref{th-fun-sinr-L}), \(\widehat{\text{SINR}}_k(\cdot)\) and \(\overline{\text{SINR}}(\cdot)\) are defined in (\ref{function-HAT-SINR}) and (\ref{function-BAR-SINR}), respectively, $\mathcal{D}(\mathbb{Y}^k_{N,K}(f),\bar{\mathbb{Y}}(f))$ is defined in (\ref{D-Y-1}) and (\ref{D-Y-2}) .
\end{lemma}
\begin{proof}
The result follows immediately from Proposition \ref{th-SINR-FUN-ERR} by noting that the inequality is a direct consequence of the uniform convergence established in the proposition. Specifically, the bound is obtained by applying the Lipschitz continuity property with constant $L_{\rho}$ to the SINR functionals.
\end{proof}
Next, we quantify the error between the optimal values of the finite-dimensional and asymptotic equality-constrained optimization problems. Understanding this error helps establish how well the finite-dimensional problem approximates the asymptotic one in terms of optimal performance.

\begin{theorem} 
\label{th-eq-optimal-error}
Let
 Assumptions \ref{Assu:H-n-s}--\ref{Assu:ratio-K-N} hold. 
If there exists \(\bar{K} > 0\) such that for \(K > \bar{K}\), \(\widehat{\mathbb{S}}_{N,K} \cap \rho \mathbb{B}_{\infty} \neq \emptyset\), then
\bgeqn 
\label{eq:Vnk-thm5.3}
|\widehat{\mathbb{V}}_{N,K}-\widehat{\mathbb{V}}|\leq  L_{\rho}\sup_{\|f\|_{\infty}\leq \rho}\mathcal{D}(\mathbb{Y}^k_{N,K}(f),\bar{\mathbb{Y}}(f)),
\edeqn 
for \(K\) sufficiently large,
where \(L_{\rho}:=\sup_{\|f\|_{\infty}\leq \rho}L(f,\sigma(f) )\) with  $L(f,\sigma(f) )$ being defined in (\ref{th-fun-sinr-L}), 
$\widehat{\mathbb{V}}$ and $\widehat{\mathbb{V}}_{N,K}$ denote the optimal values of $(\widehat{P})$ and $(\widehat{P}_{N,K})$, respectively, $\mathcal{D}(\mathbb{Y}^k_{N,K}(f),\bar{\mathbb{Y}}(f))$ is defined in (\ref{D-Y-1}) and (\ref{D-Y-2}).
\end{theorem}

{\color{black}
Theorem~\ref{th-eq-optimal-error} establishes a rigorous error bound that quantifies the gap between the finite-dimensional and asymptotic optimal values. Crucially, this bound is expressed in terms of the consistency of the underlying system functions—specifically, the deviation in the received signal models $\mathcal{D}(\mathbb{Y}^k_{N,K}(f),\bar{\mathbb{Y}}(f))$. This provides a quantitative and mechanistic understanding of how discrepancies in the system's stochastic representation propagate to its high-level optimization performance, forming a foundational pillar for the non-asymptotic analysis of quantized precoding systems.

The result provides 
theoretical justification for solving the tractable infinite-dimensional problem $(\widehat{P})$ to obtain parameters (e.g., the precoding function $f$) for use in practical finite-dimensional systems. It guarantees that the performance loss incurred by this substitution is strictly controlled, thereby enabling the direct application of asymptotically-derived solutions in real-world design and significantly simplifying the development of practical algorithms without resorting to complex finite-dimensional stochastic optimization.

Although not explicitly expressed in classical $O(1/K)$ form, the error bound is governed by the deviation term $\mathcal{D}(\mathbb{Y}^k_{N,K}(f),\bar{\mathbb{Y}}(f))$, which exhibits super-polynomial convergence behavior as established in Section~4 (e.g., exponential tail decay in Theorem~4.1 and $O((\ln K)^{1/3}/K^{1/3})$ Ky Fan metric convergence in Lemmas~4.3--4.4). In principle, an explicit convergence rate could be derived by integrating these probabilistic bounds, but due to the complex interplay of multiple concentration inequalities and constant calculations, we leave the tight explicit rate characterization and its dependence on system parameters as valuable future work, maintaining the current focus on the optimization stability framework.

\begin{proof}
 Since $\widehat{\mathbb{S}}_{N,K}\cap \rho \mathbb{B}_{\infty}\neq\emptyset$ when $K$ is large enough, we can choose $(\hat{f},  \tilde{\eta}_{N,K}(\hat{f}), \tilde{\alpha}_{N,K}(\hat{f}))\in \widehat{\mathbb{S}}_{N,K}\cap \rho \mathbb{B}_{\infty}\subseteq \widehat{\Phi}_{N,K}$. 
 By the proof of Step 1 in Theorem \ref{th-hau-Phi-error},
there exists  $(\hat{f}, \tilde{\eta}(\hat{f}), \tilde{\alpha}(\hat{f}))\in \widehat{\Phi}$ such that 
$$
\left\|(\hat{f},  \tilde{\eta}_{N,K}(\hat{f}), \tilde{\alpha}_{N,K}(\hat{f}))-(\hat{f},  \tilde{\eta}(\hat{f}), \tilde{\alpha}(\hat{f}))\right\|\leq \sup_{\|f\|_{\infty}\leq \rho}\|\sigma_{N,K}(f)-\sigma(f)\|.$$ 
Thus
\begin{equation}
\begin{array}{lll}
\displaystyle-\widehat{\mathbb{V}}&\leq& -\overline{\text{SINR}}(\hat{f},\tilde{\eta}(\hat{f}), \tilde{\alpha}(\hat{f}))\\[12pt]
&\leq&-\widehat{\text{SINR}}_{k}(\hat{f},  \tilde{\eta}_{N,K}(\hat{f}), \tilde{\alpha}_{N,K}(\hat{f}))+\left\| \overline{\text{SINR}}(\hat{f}, \tilde{\eta}(\hat{f}), \tilde{\alpha}(\hat{f}))-\widehat{\text{SINR}}_{k}(\hat{f},  \tilde{\eta}_{N,K}(\hat{f}), \tilde{\alpha}_{N,K}(\hat{f})) \right\|\\[12pt]
&=&-\widehat{\mathbb{V}}_{N,K}+\left\| \overline{\text{SINR}}(\hat{f}, \tilde{\eta}(\hat{f}), \tilde{\alpha}(\hat{f}))-\widehat{\text{SINR}}_{k}(\hat{f}, \tilde{\eta}_{N,K}(\hat{f}), \tilde{\alpha}_{N,K}(\hat{f})) \right\|.
\end{array}
\end{equation}
By Lemma \ref{lemma-SINR-error}, there exists a positive number \(L_{\rho}\) such that 
$$
\left\|\overline{\text{SINR}}(\hat{f}, \tilde{\eta}(\hat{f}), \tilde{\alpha}(\hat{f}))-\overline{\text{SINR}}(\hat{f}, \tilde{\eta}_{N,K}(\hat{f}), \tilde{\alpha}_{N,K}(\hat{f})) \right\|\leq L_{\rho}\sup_{\|f\|_{\infty}\leq \rho}\mathcal{D}(\mathbb{Y}^k_{N,K}(f),\bar{\mathbb{Y}}(f)),
$$
where \(L_{\rho}:=\sup_{\|f\|_{\infty}\leq \rho}L(f,\sigma(f) )\) with  $L(f,\sigma(f) )$ being defined in (\ref{th-fun-sinr-L}).
Therefore we have 
$$
\widehat{\mathbb{V}}_{N,K}-\widehat{\mathbb{V}}\leq L_{\rho}\sup_{\|f\|_{\infty}\leq \rho}\mathcal{D}(\mathbb{Y}^k_{N,K}(f),\bar{\mathbb{Y}}(f)).
$$ 
In the same way, we can establish 
$$
\widehat{\mathbb{V}}-\widehat{\mathbb{V}}_{N,K}\leq L_{\rho}\sup_{\|f\|_{\infty}\leq \rho}\mathcal{D}(\mathbb{Y}^k_{N,K}(f),\bar{\mathbb{Y}}(f)).
$$ 
A combination of the two inequalities above yields 
\eqref{eq:Vnk-thm5.3}.
 \end{proof}
 
\begin{remark}
From the proof of Theorem \ref{th-eq-optimal-error}, we can see that if there exists 
a constant $L>0$ such that 
  $$\sup_{\|f\|_{\infty}\leq \rho}\mathcal{D}(\mathbb{Y}^k_{N,K}(f),\bar{\mathbb{Y}}(f))\leq L\sup_{\|f\|_{\infty}\leq \rho}\|\sigma_{N,K}(f)-\sigma(f)\|,
  $$ 
  then 
  $$
|\widehat{\mathbb{V}}_{N,K}-\widehat{\mathbb{V}}|\leq  L_{\rho}L\sup_{\|f\|_{\infty}\leq \rho}\|\sigma_{N,K}(f)-\sigma(f)\|.
$$ 
However, this may fail due to the fact the $\tilde{\alpha}(f)$ in $\bar{\mathbb{Y}}(f))$ is described by  a function $q(\cdot)$ which is continuous a.e. but not necessarily 
Lipschitz continuous.
\end{remark}
The above theorem provides a bound on the difference between the optimal values  
$\widehat{\mathbb{V}}_{N,K}$ and $\widehat{\mathbb{V}}$ in terms of the approximation error of the signal model and the convergence of parameters. This result demonstrates that as $K$ increases, the optimal value of the finite-dimensional problem approaches that of the asymptotic problem, validating the practical relevance of asymptotic optimization results.

 Let $\rho>0$ be such that 
 $U^{\rho}:=\widehat{\Phi} \cap \rho\mathbb{B}_{\infty}\neq\emptyset$, and
$$
\psi(\tau) :=\min\left\{\widehat{\mathbb{V}}-\overline{\text{SINR}}(f, \eta,\alpha): d((f, \eta,\alpha), \widehat{\mathbb{S}})\geq \tau \,\,\mbox{for\,\,any}  (f, \eta,\alpha)\in U^{\rho} \right\}.
$$
Define
\bgeqn 
\Psi(t):=t+\psi^{-1}(2t), t\geq 0, 
\edeqn 
where $\widehat{\mathbb{S}}$ denotes the optimal solution set of $(\widehat{P})$ and $\psi^{-1}(t)=\sup\{\tau\geq0:\psi(\tau)\leq t \}$.
Both functions $\psi$ and $\Psi$  are lower semicontinuous on $[0,+\infty)$, $\psi$ is nondecreasing,
$\Psi$ is increasing, and both vanish at $0$.

Using function $\Psi(t)$, 
we extend the asymptotic analysis of 
Theorem \ref{th-HDV-eq} (iii) 
to provide a bound on the distance between the optimal solution sets of the finite-dimensional and asymptotic equality-constrained optimization problems. Understanding this distance helps quantify how well the finite-dimensional optimal solutions approximate the asymptotic ones.

\begin{theorem} 
Let Assumptions \ref{Assu:H-n-s}--\ref{Assu:ratio-K-N} hold.  If  there exists  \(\rho>0, \bar{K} > 0\) such that for \(K > \bar{K}\), \(\widehat{\mathbb{S}}_{N,K} \cap \frac{\rho}{2} \mathbb{B}_{\infty} \neq \emptyset\), then 
\bgeqn 
\mathbb{D}_{\rho}\left(\widehat{\mathbb{S}}_{N,K}, \widehat{\mathbb{S}}\right)\leq \Psi\left(\sup_{\|f\|_{\infty}\leq \rho}\|\sigma_{N,K}(f)-\sigma(f)\|+L_{\rho}\sup_{\|f\|_{\infty}\leq \rho}\mathcal{D}(\mathbb{Y}^k_{N,K}(f),\bar{\mathbb{Y}}(f))\right)
\edeqn 
for  \(K\) sufficiently.
\end{theorem}


\begin{proof}
Let $K$ be sufficiently large
such that 
$
\widehat{\mathbb{S}}_{N,K}\cap \frac{\rho}{2} \mathbb{B}_{\infty}\neq\emptyset$.
We can choose 
\bgeq 
\Big(\hat{f},  \tilde{\eta}_{N,K}(\hat{f}), \tilde{\alpha}_{N,K}(\hat{f})\Big)\in \widehat{\mathbb{S}}_{N,K}\cap \frac{\rho}{2} \mathbb{B}_{\infty}\subseteq \widehat{\Phi}_{N,K}.
\edeq 
By the proof of Step 1 in Theorem \ref{th-hau-Phi-error},   there exists  $(\hat{f},  \tilde{\eta}(\hat{f}), \tilde{\alpha}(\hat{f}))\in \widehat{\Phi}$ such that 
$$
\left\|\left(\hat{f}, \tilde{\eta}_{N,K}(\hat{f}), \tilde{\alpha}_{N,K}(\hat{f})\right)-\left(\hat{f},  \tilde{\eta}(\hat{f}), \tilde{\alpha}(\hat{f})\right)\right\|\leq  \sup_{\|f\|_{\infty}\leq \rho}\|\sigma_{N,K}(f)-\sigma(f)\|<\frac{\rho}{2}
$$ 
when $K$ is sufficiently large. This 
means that $(\hat{f},  \tilde{\eta}(\hat{f}),\tilde{\alpha}(\hat{f}) )\in U^{\rho}:=\widehat{\Phi}\cap \rho \mathbb{B}_{\infty}.$
Moreover, by Lemma \ref{lemma-SINR-error},
  \begin{align}
\left|\widehat{\text{SINR}}_k(\hat{f},\sigma_{N,K}(\hat{f})) - \overline{\text{SINR}}(\hat{f},\sigma(\hat{f}))\right| 
\leq  L_{\rho}\sup_{\|f\|_{\infty}\leq \rho}\mathcal{D}(\mathbb{Y}^k_{N,K}(f),\bar{\mathbb{Y}}(f))
\end{align}
    for all \(k \in [K]\), where \(L_{\rho}:=\sup_{\|f\|_{\infty}\leq \rho}L(f,\sigma(f) )\) with  $L(f,\sigma(f) )$ being defined in (\ref{th-fun-sinr-L}).
  Let  $$\mathbb{L}_{N,K}:=\sup_{\|f\|_{\infty}\leq \rho}\|\sigma_{N,K}(f)-\sigma(f)\|+L_{\rho}\sup_{\|f\|_{\infty}\leq \rho}\mathcal{D}(\mathbb{Y}^k_{N,K}(f),\bar{\mathbb{Y}}(f)).$$
Then we have from Theorem \ref{th-eq-optimal-error},
\begin{align*}
2\mathbb{L}_{N,K} &\geq \mathbb{L}_{N,K} + \widehat{\mathbb{V}} - \widehat{\mathbb{V}}_{N,K} = \mathbb{L}_{N,K} + \widehat{\mathbb{V}}-\widehat{\text{SINR}}_{k}\left(\hat{f},\sigma_{N,K}(\hat{f})\right) \\[12pt]
&\geq \widehat{\mathbb{V}}-\overline{\text{SINR}}\left(\hat{f},\sigma(\hat{f})\right) \geq \psi\left(d((\hat{f},\sigma(\hat{f})), \widehat{\mathbb{S}})\right) \geq \inf_{y \in \mathbb{B}\left((\hat{f},\sigma_{N,K}(\hat{f})), \mathbb{L}_{N,K}\right)} \psi(d(y, \widehat{\mathbb{S}})) \\[12pt]
&= \psi\left(d((\hat{f},\sigma_{N,K}(\hat{f})), \widehat{\mathbb{S}} + \mathbb{L}_{N,K}\mathbb{B}_{\infty})\right).
\end{align*}
Consequently 
\begin{align*}
d\left((\hat{f},\sigma_{N,K}(\hat{f})), \widehat{\mathbb{S}}\right) &\leq \mathbb{L}_{N,K} + d\left((\hat{f},\sigma_{N,K}(\hat{f})), \widehat{\mathbb{S}} + \mathbb{L}_{N,K}\mathbb{B}_{\infty}\right) \\[12pt]
&\leq \mathbb{L}_{N,K}  + \psi^{-1}(2\mathbb{L}_{N,K}) =\Psi(\mathbb{L}_{N,K}).
\end{align*} 
The proof is complete.
\end{proof}
The above theorem establishes that the distance between the optimal solution sets  $\widehat{\mathbb{S}}_{N,K}$ and $\widehat{\mathbb{S}}$ is bounded by a function of the approximation errors of the parameter mappings and signal models. This result provides a comprehensive understanding of how the finite-dimensional optimal solutions converge to their asymptotic counterparts as K increases, offering valuable insights for algorithm design and performance guarantees in practical massive MIMO systems.


\section{Concluding remarks}
In this paper, we provide a 
new mathematical framework for analyzing and optimizing linear-quantized precoding in massive MIMO systems.
The established results 
may provide some theoretical 
guidance for 
the development of efficient and practical precoding strategies that can significantly enhance the performance of massive MIMO systems in finite-dimensional settings. 
The  analysis 
is based on a set of well-justified assumptions, including i.i.d.~Rayleigh fading, which leads to the Marchenko-Pastur law for the eigenvalue distribution, and mild regularity conditions on the precoding and quantization functions. The compactness assumption on the precoding function space is both mathematically convenient for the 
analysis and practically reasonable, as it encompasses smooth approximations of real-world hardware-implemented functions. 
However, this work also 
has some 
limitations. 
For example, the current framework assumes perfect channel state information and a single-cell setup without inter-cell interference. 
Moreover, while the bounds are explicit, their tightness in regimes with very low quantization resolution (e.g., 1-bit DACs) and their direct translation into closed-form, real-time adaptive algorithms remain open challenges.

These limitations naturally chart the course for future research. Promising directions include extending the 
quantitative asymptotic convergence analysis to spatially correlated channels and multi-cell scenarios with pilot contamination, investigating the joint optimization of precoding and DAC resolution selection under power and cost constraints, and developing low-complexity online algorithms that leverage these theoretical bounds for robust adaptive precoding. Integrating data-driven methods with this rigorous stability framework could further enhance system adaptability and performance, ultimately supporting the efficient and reliable deployment of quantized massive MIMO in 5G-Advanced and 6G networks.


\clearpage
\begin{appendices}
 
\section{Appendix}


\subsection{Preliminaries of set convergence}
The analysis in the paper
uses a number of concepts about convergence of  the sequences of sets and mapping in \cite{R98}.
Here we recall them.
Define
\begin{align*}
N_{\infty} := \{N \subseteq \mathbb{N} \mid \mathbb{N} \setminus N \text{ finite}\} \quad \text{and} \quad N^{\#}_{\infty} := \{N \subseteq \mathbb{N} \mid N \text{ infinite}\},
\end{align*}
where $\mathbb{N}$ denotes the set of all positive integer numbers. 

\begin{definition}\label{Def-outer-inter}
For sets $C_\nu$ and $C$ in $\mathbb{R}^n$ with $C$ closed, the sequence $\{C_\nu\}_{\nu \in \mathbb{N}}$ is said to converge to $C$ (written $C_\nu \to C$) if
\begin{align}
\limsup_{\nu \to \infty} C_\nu \subseteq C \subseteq \liminf_{\nu \to \infty} C_\nu
\end{align}
with
\begin{align*}
\limsup_{\nu \to \infty} C_\nu := \left\{ x \bigg| \exists N \in N^{\#}_{\infty}, \exists x_\nu \in C_\nu (\nu \in N) \text{ such that } x_\nu \rightarrow x \right\},\\
\liminf_{\nu \to \infty} C_\nu := \left\{ x \bigg| \exists N \in N_{\infty}, \exists x_\nu \in C_\nu (\nu \in N) \text{ such that } x_\nu \rightarrow x \right\}.
\end{align*}
\end{definition}

\begin{proposition}
\label{outer-limite-equi}(Uniformity of approximation in set convergence). For subsets $C^{\nu}, C \subset \mathbb{R}^n$ with $C$ closed, one has
\begin{itemize}
    \item[(a)] $C \subset \liminf_{\nu} C^{\nu}$ if and only if for every $\rho > 0$ and $\varepsilon > 0$ there is an index set $N \in \mathbb{N}_{\infty}$ with $C \cap \rho \mathbb{B} \subset C^{\nu} + \varepsilon \mathbb{B}$ for all $\nu \in N$;
    \item[(b)] $C \supset \limsup_{\nu} C^{\nu}$ if and only if for every $\rho > 0$ and $\varepsilon > 0$ there is an index set $N \in \mathbb{N}_{\infty}$ with $C^{\nu} \cap \rho \mathbb{B} \subset C + \varepsilon \mathbb{B}$ for all $\nu \in N$.
\end{itemize}
\end{proposition} 
The continuous properties of a set-valued mapping $S$ can be developed by the convergence of sets.

\begin{definition}\label{semi-conti}
A set-valued mapping $S: \mathbb{R}^n \Rightarrow \mathbb{R}^m$ is continuous at $\bar{x}$, symbolized by $\lim_{x \to \bar{x}} S(x) = S(\bar{x})$, if
\begin{align*}
\limsup_{x \to \bar{x}} S(x) \subseteq S(\bar{x}) \subseteq \liminf_{x \to \bar{x}} S(x).
\end{align*}
\end{definition}

\begin{definition}\label{de-epi-co}
Consider now a family of functions $f_\nu: \mathbb{R}^n \to \bar{\mathbb{R}}$, where $\bar{\mathbb{R}} = \mathbb{R} \cup \{\pm \infty\}$. One says that $f_\nu$ epi-converges to a function $f: \mathbb{R}^n \to \bar{\mathbb{R}}$ as $\nu \to \infty$, written
\begin{equation*}
f = e - \lim_{\nu \to \infty} f_\nu,
\end{equation*}
if the sequence of sets $\text{epi} \, f_\nu$ converges to $\text{epi} \, f$ in $\mathbb{R}^n \times \mathbb{R}$ as $\nu \to \infty$.
\end{definition}
\begin{definition}\label{u-l-lim-f}
    [Upper and Lower Limits of a Sequence of Functions]
Given a sequence of functions \(\{f_\nu\}\) defined on a set \(X\), the upper and lower limits of this sequence at a point \(x \in X\) are defined as follows:
\begin{enumerate}
    \item[(i)] Lower Limit:
    \[
    \liminf_{\nu \to \infty} f_\nu(x) = \lim_{\nu \to \infty} \inf_{k \geq \nu} f_k(x).
    \]
    This represents the largest value that the sequence of functions approaches from below at the point \(x\).

    \item[(ii)] Upper Limit:
    \[
    \limsup_{\nu \to \infty} f_\nu(x) = \lim_{\nu \to \infty} \sup_{k \geq \nu} f_k(x).
    \]
    This represents the smallest value that the sequence of functions approaches from above at the point \(x\).
\end{enumerate}
\end{definition}

The characterization of the epi-convergence can be described by the following result.

\begin{proposition}\label{pro-epi}
 Consider now a family of functions $f_\nu: \mathbb{R}^n \to \bar{\mathbb{R}}$, where $\bar{\mathbb{R}} = \mathbb{R} \cup \{\pm \infty\}$. Then $f_\nu$ epi-converges to $f$ if and only if at each point $x$, the following two conditions both hold:
\begin{enumerate}[(a)]
\item $\liminf_{\nu \to \infty} f_\nu(x_\nu) \geq f(x)$ for every sequence $x_\nu \to x$,
\item $\limsup_{\nu \to \infty} f_\nu(x_\nu) \leq f(x)$ for some sequence $x_\nu \to x$.
\end{enumerate}
\end{proposition}

\begin{definition}\label{de-upper lim}
The upper limit of a function $f: \mathbb{R}^n \to \bar{\mathbb{R}}$ at $\bar{x}$ is defined by
\begin{align}
\limsup_{x \rightarrow \bar{x}} f(x) &:= \lim_{\delta \searrow 0} \left[ \sup_{x \in B(\bar{x}, \delta)} f(x) \right]  
= \inf_{\delta > 0} \left[ \sup_{x \in B(\bar{x}, \delta)} f(x) \right].
\end{align}
\end{definition}

\begin{proposition} \label{inf-S-V} Consider now a family of functions $f_\nu: \mathbb{R}^n \to \bar{\mathbb{R}}$, where $\bar{\mathbb{R}} = \mathbb{R} \cup \{\pm \infty\}$. Suppose \( f_\nu \rightarrow f \) with \(-\infty < \inf f < \infty\).

(a) \(\inf f_\nu \to \inf f\) if and only if there exists for every \(\varepsilon > 0\) a compact set \( B \subset \mathbb{R}^n \) along with an index set \( N \in \mathbb{N}_\infty \) such that \(\inf_B f_\nu \leq \inf f_\nu + \varepsilon\) for all \(\nu \in N\).

(b) \(\limsup_{\nu\rightarrow \infty} (\varepsilon-\text{argmin} f_\nu) \subset \varepsilon-\text{argmin} f\) for every \(\varepsilon \geq 0\) and consequently \(\limsup_{\nu\rightarrow \infty} (\varepsilon^{\nu}-\text{argmin} f_\nu) \subset \text{argmin} f\) whenever \(\varepsilon^{\nu} \searrow 0\).

(c) Under the assumption that \(\inf f_{\nu} \to \inf f\), there exists a sequence \(\varepsilon^{\nu} \searrow 0\) such that \(\varepsilon^{\nu}-\text{argmin} f_{\nu} \to \text{argmin} f\). Conversely, if such a sequence exists, and if \(\text{argmin} f \neq \emptyset\), then \(\inf f_\nu \to \inf f\).
\end{proposition}


\subsection{Lipschitz continuity of
the cdf of 
a Gaussian distribution}

To establish a quantitative analysis of the stability of the SEP in massive MIMO systems, we need the following lemma, which is
inspired by \cite[Theorem 1]{B93}.

\begin{lemma}\label{le4.1}
Let $A$ be a closed convex set in $\mathbb{C}$, 
$\mathbf{Y} $ be a complex  random variable mapping 
    from $(\Omega,\F,\mathbb{P})$ to $\mathbb{C}$  and $\mathbf{Y} \sim \mathcal{CN}(\mu, \sigma)$. Then for $\epsilon>0$ sufficiently small,
\bgeqn 
\mathbb{P}(\mathbf{Y}\in A^{\epsilon}\backslash A)\leq \left(\frac{\sqrt{\pi}}{\sigma} +1\right)\epsilon\qquad \mbox{and}\qquad \mathbb{P}(\mathbf{Y}\in A\backslash A^{-\epsilon})\leq \left(\frac{\sqrt{\pi}}{\sigma} +1\right)\epsilon,
\label{eq:Lip-cdf-Gauss-dist}
\edeqn 
where $A^{\epsilon}$ and $A^{-\epsilon}$
are balls containing $A$ and contained in $A$ with $\epsilon$-gap.
\end{lemma}

Inequality \eqref{eq:Lip-cdf-Gauss-dist} may be interpreted as 
the Lipschitz continuity of 
the cdf of $Y$ over its support set.

\begin{proof}
For $\mathbf{Y} \sim \mathcal{CN}(\mu, \sigma)$, we know from (\ref{CN-density}) that its density function \( g \) on \( \mathbb{C} \) is defined by 
\[
g(x) = \frac{1}{\pi\sigma^2} e^{-\frac{|x-\mu|^2}{\sigma^2}}, \quad 
x\in \mathbb{C}.
\]
We first estimate the upper bound of 
$
\int_{\partial A} g(x) \, d\nu(x),
$ 
where $\nu$
 is the $1$-dimensional Hausdorff measure (length) on $\partial A$, 
 the boundary of $A$ \cite{R11}.
Consider the indicator function  \( \mathds{1}_{\mathbb{B}(\mu, \sigma t)}(x) \) of 
 closed ball $\mathbb{B}(\mu, \sigma t)$:
\[
\mathds{1}_{\mathbb{B}(\mu, \sigma t)}(x)  = \begin{cases} 
1 & \text{if } |x-\mu| \leq \sigma t, \\
0 & \text{if } |x-\mu| >\sigma t,
\end{cases}
\] for $x\in \mathbb{C}.$
Using the indicator function, we can express \( g(x) \) as
\[
g(x) = \frac{1}{\pi\sigma^2}\int_0^\infty 2t e^{-t^2} \mathds{1}_{\mathbb{B}(\mu, \sigma t)}(x) \, dt.
\]
Next, we consider the integration over the boundary of the convex body \( A \)
\[
\int_{\partial A} g(x) \, d\nu(x) = \frac{1}{\pi\sigma^2}\int_{\partial A} \left( \int_0^\infty 2t e^{-t^2}  \mathds{1}_{\mathbb{B}(\mu, \sigma t)}(x) \, dt \right) d\nu(x).
\]
By  Fubini-Tonelli theorem, we can interchange the order of integration
\[
\int_{\partial A} g(x) \, d\nu(x) = \frac{1}{\pi\sigma^2}\int_0^\infty 2t e^{-t^2}\left( \int_{\partial A}  \mathds{1}_{\mathbb{B}(\mu, \sigma t)}(x)\, d\nu(x) \right) dt.
\]
Since
\bgeq
\int_{\partial A}  \mathds{1}_{\mathbb{B}(\mu, \sigma t)}(x)\, d\nu(x)&=&\int_{\mathbb{C}}  \mathds{1}_{\left(\partial A\cap \mathbb{B}(\mu, \sigma t)\right)}(x) \, d\nu(x)\\
&\leq& \int_{\mathbb{C}}  \mathds{1}_{\partial \left(A\cap \mathbb{B}(\mu, \sigma t)\right)}(x) \, d\nu(x)\\
&\leq& \int_{\mathbb{C}}  \mathds{1}_{\partial \mathbb{B}(\mu, \sigma t)}(x) \, d\nu(x)=2\pi t\sigma,
\edeq
where the first inequality holds due to the fact that $\partial A\cap \mathbb{B}(\mu, \sigma t)\subseteq \partial \left(A\cap \mathbb{B}(\mu, \sigma t)\right)$ and the second inequality holds by the monotonicity of the perimeters of convex bodies in $\mathbb{C}$ (i.e., the classical fact that
$\nu(\partial A) \leq \nu(\partial B)$ if the convex body $A$ is included in the convex body $B$ in $\mathbb{C}$ \cite{S18}). 
Consequently we have
\[
\int_{\partial A} g(x) \, d\nu(x) \leq \frac{2}{\sigma}\int_0^\infty t^2 e^{-t^2}dt=\frac{2}{\sigma}\Gamma\left(\frac{3}{2}\right)= \frac{\sqrt{\pi}}{\sigma},
\] where $\Gamma$ denotes the Gamma function 
$\Gamma(z) = \int_{0}^{\infty} t^{z-1} e^{-t} \, dt.$ 
As shown in \cite{N03}, 
\[
\limsup_{\epsilon\rightarrow 0^+}\frac{\mathbb{P}(\mathbf{Y}\in A^{\epsilon}\setminus A)}{\epsilon}=\limsup_{\epsilon\rightarrow 0^+}\frac{\mathbb{P}(\mathbf{Y}\in A\setminus A^{-\epsilon})}{\epsilon}=\int_{\partial A} g(x) \, d\nu(x).
\]
By Definition \ref{de-upper lim}, we have for $\epsilon>0
$ sufficiently small,  
$$
\frac{\mathbb{P}(\mathbf{Y}\in A^{\epsilon}\setminus A)}{\epsilon}<\frac{\sqrt{\pi}}{\sigma}+1\quad \mbox{and} \quad
\frac{\mathbb{P}(\mathbf{Y}\in A\setminus A^{-\epsilon})}{\epsilon}<\frac{\sqrt{\pi}}{\sigma}+1.
$$
The proof is complete.
\end{proof}
We make 
a few comments on surface measure in the proof of Lemma \ref{le4.1}.
As shown in \cite{F14}, in \( \mathbb{C} \) (viewed as \( \mathbb{R}^2 \)), the surface measure \( \nu \) on the boundary \( \partial A \) refers to the 1-dimensional, 
which coincides with the arc length for smooth curves. For convex sets, \( \partial A \) is rectifiable, and the integral over \( \partial A \) represents integration with respect to this length measure.

\subsection{Some basic inequalities in deviation 
}

\begin{lemma}\label{le-1}
    Let $X$ be a random variable with finite expectation $0\neq\overline{X}=\E[X]<\infty$. Then for any $\epsilon>0$,
    \begin{align}
    \label{eq:Lemas6.2}
        \P\left(\left|\frac{1}{X}-\frac{1}{\overline{X}}\right|\geq \epsilon\right)\leq \P\left(\left|X-\overline{X}\right|\geq |\overline{X}|\epsilon\right) + \P\left(\left|X-\overline{X}\right|\geq \frac{\epsilon\overline{X}^2}{1+\epsilon|\overline{X}|}\right).
    \end{align}
\end{lemma}

\begin{proof}
    To simplify the notation, let $Y_N = X-\overline{X}$. Then 
    \begin{align*}
        \P\left(\left|\frac{1}{X}-\frac{1}{\overline{X}}\right|\geq \epsilon\right) = \P\left(\left|\frac{1}{\overline{X}+Y_N}-\frac{1}{\overline{X}}\right|\geq \epsilon\right)=\P\left(\left|\frac{1}{\overline{X}}\right|\left|\frac{1}{1+Y_N/\overline{X}}-1\right|\geq \epsilon\right).
    \end{align*}
    Conditional on $|Y/\overline{X}|<\epsilon$ or $|Y/\overline{X}|\geq \epsilon$, we have 
    \begin{align*}
        \P\left(\left|\frac{1}{X}-\frac{1}{\overline{X}}\right|\geq \epsilon\right) =\;&\P\left(\left|\frac{1}{\overline{X}}\right|\left|\frac{1}{1+Y/\overline{X}}-1\right|\geq \epsilon\;\bigg|\left|\frac{Y}{\overline{X}}\right|<\epsilon \right)\P\left(\left|\frac{Y}{\overline{X}}\right|<\epsilon\right)+\\
        &\P\left(\left|\frac{1}{\overline{X}}\right|\left|\frac{1}{1+Y/\overline{X}}-1\right|\geq \epsilon\;\bigg|\left|\frac{Y}{\overline{X}}\right|\geq\epsilon \right)\P\left(\left|\frac{Y}{\overline{X}}\right|\geq \epsilon\right)\\
        \leq\;& \P\left(\left|\frac{1}{\overline{X}}\right|\left|\frac{1}{1+Y/\overline{X}}-1\right|\geq \epsilon\;\bigg|\left|\frac{Y}{\overline{X}}\right|<\epsilon \right)+\P\left(\left|X-\overline{X}\right|\geq |\overline{X}|\epsilon\right).
    \end{align*}
    It suffices to prove that 
    \begin{align*}
        \P\left(\left|\frac{1}{\overline{X}}\right|\left|\frac{1}{1+Y/\overline{X}}-1\right|\geq \epsilon\;\bigg|\left|\frac{Y}{\overline{X}}\right|<\epsilon \right)\leq \P\left(\left|X-\overline{X}\right|\geq \frac{\epsilon\overline{X}^2}{1+\epsilon|\overline{X}|}\right).
    \end{align*}
    Let $y=Y/\overline{X}$. Then 
    \bgeq
        \P\left(\left|\frac{1}{\overline{X}}\right|\left|\frac{1}{1+Y/\overline{X}}-1\right|\geq \epsilon\;\bigg|\left|\frac{Y}{\overline{X}}\right|<\epsilon \right)
        &=& \P\left(\left|\frac{1}{\overline{X}}\right|\left|\frac{1}{1+y}-1\right|\geq \epsilon\;\bigg| \left|y\right|<\epsilon\;\bigg|\right)
        =\P\left(\left|\frac{1}{\overline{X}}\right|\frac{|y|}{|1+y|}\geq \epsilon\;\bigg|\left|y\right|<\epsilon\right)\\
        &=&\P\left(\frac{|y|}{|1+y|}\geq \epsilon|\overline{X}|\;\bigg| \left|y\right|<\epsilon\right)
        =\P\left(|y|\geq \epsilon|\overline{X}||1+y|\;\bigg| \left|y\right|<\epsilon\right)\\
        &\leq &\P\left(|y|\geq \epsilon|\overline{X}|-\epsilon|\overline{X}||y|\;\bigg| \left|y\right|<\epsilon\right)\\
        &=&\P\left(|y|\geq \frac{\epsilon|\overline{X}|}{1+\epsilon|\overline{X}|}\;\bigg| \left|y\right|<\epsilon\right)
        =\P\left(|Y|\geq \frac{\epsilon|\overline{X}|^2}{1+\epsilon|\overline{X}|}\;\bigg| \left|Y\right|<\epsilon|\overline{X}|\right)\\
        &\leq &\P\left(|Y|\geq \frac{\epsilon|\overline{X}|^2}{1+\epsilon|\overline{X}|}\right)
   \edeq 
    which gives rise to \eqref{eq:Lemas6.2}. 
\end{proof}

\begin{lemma}\label{le-2}
    Let $X$ be a random variable with finite expectation $0\neq\overline{X}=\E[X]<\infty$. Then for any $\epsilon>0$,
    \begin{align*}
        \P\left(\left|\frac{1}{X}-\frac{1}{\overline{X}}\right|\geq \epsilon\right)\leq \P\left(\left|X-\overline{X}\right|\geq \frac{\epsilon|\overline{X}|^2}{1+\epsilon|\overline{X}|}\right).
    \end{align*}
\end{lemma}

\begin{proof}
    For any $\delta>0$, conditional on $|X-\overline{X}|<\delta$ or $|X-\overline{X}|\geq\delta$, we have 
    \begin{align*}
        \P\left(\left|\frac{1}{X}-\frac{1}{\overline{X}}\right|\geq \epsilon\right)=\;&\P\left(\left|\frac{1}{X}-\frac{1}{\overline{X}}\right|\geq \epsilon\;\bigg| |X-\overline{X}|<\delta\right)\P\left(|X-\overline{X}|<\delta\right) + \\
        \;&\P\left(\left|\frac{1}{X}-\frac{1}{\overline{X}}\right|\geq \epsilon\;\bigg| |X-\overline{X}|\geq\delta\right)\P\left(|X-\overline{X}|\geq\delta\right)\\
        \leq \;&\P\left(\left|\frac{1}{X}-\frac{1}{\overline{X}}\right|\geq \epsilon\;\bigg| |X-\overline{X}|<\delta\right) + \P\left(|X-\overline{X}|\geq\delta\right).\\
    \end{align*}
     Since
     $|X-\overline{X}|\geq |\overline{X}|-|X|$,
    then 
    \begin{align*}
        \P\left(\left|\frac{1}{X}-\frac{1}{\overline{X}}\right|\geq \epsilon\;\bigg| |X-\overline{X}|<\delta\right)= \P\left(\frac{|X-\overline{X}|}{|X||\overline{X}|}\geq \epsilon\;\bigg| |X-\overline{X}|<\delta\right)
        \leq\P\left(\frac{|X-\overline{X}|}{\big(|\overline{X}|-\delta\big)|\overline{X}|}\geq \epsilon\;\bigg| |X-\overline{X}|<\delta\right).
    \end{align*}
    Setting $\delta = \frac{\epsilon|\overline{X}|^2}{1+\epsilon|\overline{X}|}$, we derive that $\epsilon\big(|\overline{X}|-\delta\big)|\overline{X}|=\delta$ and thus
    \begin{align*}
        \P\left(\left|\frac{1}{X}-\frac{1}{\overline{X}}\right|\geq \epsilon\;\bigg| |X-\overline{X}|<\delta\right)\leq \P\left(\frac{|X-\overline{X}|}{\big(|\overline{X}|-\delta\big)|\overline{X}|}\geq \epsilon\;\bigg| |X-\overline{X}|<\delta\right)=0.
    \end{align*}
 Thus, we arrive at
 \begin{align*}
        \P\left(\left|\frac{1}{X}-\frac{1}{\overline{X}}\right|\geq \epsilon\right)\leq \P\left(\left|X-\overline{X}\right|\geq \frac{\epsilon|\overline{X}|^2}{1+\epsilon|\overline{X}|}\right).
    \end{align*}
\end{proof}

\begin{lemma}\label{le-3}
    Let $X$ be a positive random variable with finite expectation $0<\overline{X}=\E[X]<\infty$. Then for any $\epsilon>0$,
    \begin{align*}
        \P\left(\left|\sqrt{X}-\sqrt{\overline{X}}\right|\geq \epsilon\right)\leq \P\left(|X-\overline{X}|\geq \epsilon\sqrt{\overline{X}}\right).
    \end{align*}
\end{lemma}

\begin{proof}
    Note that 
    \begin{align*}
        \left|\sqrt{X}-\sqrt{\overline{X}}\right|=\left|\frac{X-\overline{X}}{\sqrt{X}+\sqrt{\overline{X}}}\right|\leq \frac{|X-\overline{X}|}{\sqrt{\overline{X}}}.
    \end{align*}
    Thus
    \begin{align*}
        \P\left(\left|\sqrt{X}-\sqrt{\overline{X}}\right|\geq \epsilon\right) \leq \P\left(\frac{|X-\overline{X}|}{\sqrt{\overline{X}}}\geq \epsilon\right)=\P(|X-\overline{X}|\geq \epsilon\sqrt{\overline{X}}).
    \end{align*}
\end{proof}

\begin{lemma}\label{le-4}
    Let $X,Y$ be 
    random variables with finite expected values, i.e.,
    $\overline{X}=\E[X]<\infty $ and $\overline{Y}=\E[Y]<\infty$. Then for any $\epsilon>0$,
    \begin{align*}
        \P\left(\left|XY-\overline{X}\overline{Y}\right|\geq \epsilon\right)\leq \P\left(\left|X-\overline{X}\right|\geq \tilde{\delta}(\overline{X},\overline{Y},\epsilon)\right)+\P\left(\left|Y-\overline{Y}\right|\geq \tilde{\delta}(\overline{X},\overline{Y},\epsilon) \right),
    \end{align*}
   where $\tilde{\delta}(\overline{X},\overline{Y},\epsilon) =1/4(\sqrt{(|\overline{X}|+|\overline{Y}|)^2+4\epsilon}-|\overline{X}|-|\overline{Y}|).$
\end{lemma}
\begin{proof}
    For any $\delta>0$, conditioning on $\rho(X,Y):=\sqrt{|X-\overline{X}|^2+|Y-\overline{Y}|^2}<\delta$ or $\rho(X,Y)\geq\delta$, we have 
    \begin{align*}
        \P\left(\left|XY-\overline{X}\overline{Y}\right|\geq \epsilon\right)=\;&\P\left(\left|XY-\overline{X}\overline{Y}\right|\geq \epsilon\;\bigg|\rho(X,Y)<\delta\right)\P\left(\rho(X,Y)<\delta\right) + \\
        \;&\P\left(\left|XY-\overline{X}\overline{Y}\right|\geq \epsilon\;\bigg| \rho(X,Y)\geq\delta\right)\P\left(\rho(X,Y)\geq\delta\right)\\
        \leq \;&\P\left(\left|XY-\overline{X}\overline{Y}\right|\geq \epsilon\;\bigg| \rho(X,Y)<\delta\right) + \P\left(\rho(X,Y)\geq\delta\right).\\
    \end{align*}
    Note that for $\rho(X,Y)<\delta$,
    \begin{align*}
        \left|XY-\overline{X}\overline{Y}\right|\leq  |X-\overline{X}||Y|+|Y-\overline{Y}||\overline{X}|\leq \rho(X,Y)(|\overline{Y}|+\delta)+\rho(X,Y)(|\overline{X}|)=\rho(X,Y)(|\overline{Y}|+|\overline{X}|+\delta).
    \end{align*}
     Taking $\delta = 2\tilde{\delta}(\overline{X},\overline{Y},\epsilon)$, we derive that $\delta\big(|\overline{X}|+|\overline{Y}+\delta\big)=\epsilon$ and thus
    \begin{align*}
        \P\left(\left|XY-\overline{X}\overline{Y}\right|\geq \epsilon\;\bigg|\rho(X,Y)<\delta\right)\leq \P\left(\rho(X,Y)(|\overline{Y}|+|\overline{X}|+\delta)\geq \epsilon\;\bigg| \rho(X,Y)<\delta\right)=0.
    \end{align*}
    So we have 
   \[
\begin{aligned}
        \P\left(\left|XY-\overline{X}\overline{Y}\right|\geq \epsilon\right)
        &\leq \P\left(\rho(X,Y)\geq 2\tilde{\delta}(\overline{X},\overline{Y},\epsilon)\right) \\
        &\leq \P\left(|X-\overline{X}|\geq \tilde{\delta}(\overline{X},\overline{Y},\epsilon)\right)
         + \P\left(|Y-\overline{Y}|\geq \tilde{\delta}(\overline{X},\overline{Y},\epsilon)\right).
\end{aligned}
\]
\end{proof}

\subsection{Bernstein's inequality}

We need the following definition, which is  from \cite{V2018}.

\begin{definition}
For a random variable \(X \in \mathbb{R}\), its sub-exponential norm is defined as:
\[
\|X\|_{\psi_{1}} = \inf \left\{ t > 0 : \mathbb{E} \left[ \exp \left( \frac{|X|}{t} \right) \right] \leq 2 \right\}.
\]
If \(\|X\|_{\psi_{1}}\) is finite, then \(X\) is said to be sub-exponential.
\end{definition}

This definition leverages the concept of Orlicz norms. It seeks a suitable \(t\) such that when \(X\) is normalized by \(t\), the exponential moment of the normalized variable does not exceed 2. The smaller the sub-exponential norm, the thinner the tail of the random variable, indicating a distribution that is closer to a light-tailed distribution. In other words, it quantifies the sub-exponential nature of the random variable by measuring how much its tail probabilities decay at a rate slower than Gaussian but faster than heavy-tailed distributions. If \(X\) follows a normal distribution \( \mathcal{N}(\mu, \sigma^2) \), it is sub-exponential since its tails decay exponentially.
 The exponential distribution with parameter \(\lambda\) has exponentially decaying tails and is thus subexponential.
 The next result is cited from  \cite[Theorem 2.8.1]{V2018}.

\begin{proposition} [Bernstein's inequality]\label{B inequality}
    Let \( X_1, \ldots, X_N \) be independent mean-zero subexponential random variables. Then for every \( t \geq 0 \), 
\[
\mathbb{P}\left\{\left|\sum_{i=1}^{N} X_i\right| \geq t\right\} \leq 2 \exp\left(-c \min\left\{\frac{t^2}{\sum_{i=1}^{N} \|X_i\|_{\psi_1}^2}, \frac{t}{\max_i \|X_i\|_{\psi_1}}\right\}\right),
\]  
where \( c > 0 \) is a constant.


\end{proposition}

 \begin{remark}\label{re-c}
By \cite[Exercise 2.8.5 and Theorem 2.8.4]{V2018}, we can 
set $c = \frac{1}{2}$ in Proposition \ref{B inequality} for practical purposes.
\end{remark}

By  Proposition \ref{B inequality}, we have the following results about exponential random variables.

\begin{proposition}\label{pro-exp-bound-lambda1}
Let \( X_1, X_2, \ldots, X_n \) be independent and identically distributed exponential random variables with rate parameter \(1\), i.e., \(\mathbb{E}[X_i] = 1\), \(\text{Var}(X_i) = 1\). Let \(\bar{X}_n = \frac{1}{n} \sum_{i=1}^n X_i\) be the sample mean. Then, for any \( a > 0 \),
\[
\mathbb{P}\left( \left| \bar{X}_n - 1 \right| \geq a \right) \leq 2 \exp\left( - \frac{1}{2} \min\left\{ \frac{n a^2}{16}, \frac{n a}{4} \right\} \right).
\]

\end{proposition}

\begin{proof}
Let \( Y_i = X_i - 1 \). Then \( Y_i \) are independent mean-zero random variables. For standard exponential distribution, 
\(\|X_i\|_{\psi_1} = 2\), and hence
\[
\|Y_i\|_{\psi_1} = \|X_i - 1\|_{\psi_1} \leq \|X_i\|_{\psi_1} + \|1\|_{\psi_1} = 2 + \frac{1}{\ln 2} \leq 4.
\]
Applying Proposition \ref{B inequality} to \( Y_1, \ldots, Y_n \) with \( t = n a \), and using the constant \( c = \frac{1}{2} \) from the Remark \ref{re-c}, we have
\[
\mathbb{P}\left\{ \left| \sum_{i=1}^{n} Y_i \right| \geq n a \right\} \leq 2 \exp\left( - \frac{1}{2} \min\left\{ \frac{(n a)^2}{\sum_{i=1}^{n} \|Y_i\|_{\psi_1}^2}, \frac{n a}{\max_i \|Y_i\|_{\psi_1}} \right\} \right).
\]
Substituting \( \|Y_i\|_{\psi_1} \leq 4 \), we obtain
\[
\sum_{i=1}^{n} \|Y_i\|_{\psi_1}^2 \leq 16n, \quad \max_i \|Y_i\|_{\psi_1} \leq 4.
\]
Thus 
\[
\frac{(n a)^2}{\sum_{i=1}^{n} \|Y_i\|_{\psi_1}^2} \geq \frac{n^2 a^2}{16n} = \frac{n a^2}{16}, \quad \frac{n a}{\max_i \|Y_i\|_{\psi_1}} \geq \frac{n a}{4},
\]
and hence
\[
\mathbb{P}\left\{ \left| \sum_{i=1}^{n} Y_i \right| \geq n a \right\} \leq 2 \exp\left( - \frac{1}{2} \min\left\{ \frac{n a^2}{16}, \frac{n a}{4} \right\} \right).
\]

\end{proof}
\begin{remark}\label{re-exp}
For a complex standard normal random vector $\mathbf{Z} = (Z_1, Z_2, \dots, Z_n)^T$, each component $Z_j$ of the complex standard normal random vector satisfies $Z_j \sim \mathcal{CN}(0,1)$, and the squared magnitude $|Z_j|^2$ follows an exponential distribution with rate parameter $1$, i.e., $|Z_j|^2 \sim \text{Exp}(1)$. So Proposition  \ref{pro-exp-bound-lambda1} can be applied directly to $|Z_j|^2$.
\end{remark}

\subsection{Hoeffding's inequality}

The following proposition is extracted from \cite{H1963}.
\begin{proposition}[Hoeffding's inequality]\label{pro-hoeff-ineq} Let \( X_1, \dots, X_K \) be independent random variables bounded by the interval \([a_i, b_i]\). Let \( S_K = \sum_{i=1}^K X_i \). Then for any \( t > 0 \),
\[ \mathbb{P}\left( \left| S_K - \mathbb{E}[S_K] \right| \geq t \right) \leq 2 \exp\left( -\frac{2t^2}{\sum_{i=1}^K (b_i - a_i)^2} \right). \]
\end{proposition}

The following proposition is originated from \cite[Theorem 5.2.2]{V2018}.
 \begin{proposition}\label{pro-lip-ex}
Let $Z \sim \mathcal{CN}(0, I_n)$ and $g: \mathbb{C}^n \to \mathbb{R}$ be a Lipschitz function with Lipschitz constant $L$. Then for any $t > 0$, 
\[
\mathbb{P} \left\{ |g(Z) - \mathbb{E}[g(Z)]| \geq t \right\} \leq 2 \exp\left( -\frac{t^2}{2L^2} \right).
\]
\end{proposition}

\subsection{Asymptotic behavior of the singular values}

In the large-dimensional limit, the singular values of \(H\) exhibit a well-defined asymptotic behavior, concentrating almost surely within a compact set.

\begin{proposition}
 Let \(H = U D V^H\) be 
 singular value  decomposition 
 of \(H\), where \(U\) and \(V\) are unitary matrices, and \(D\) is a diagonal matrix containing the non-zero singular values \(d_1, d_2, \ldots, d_K\) of \(H\). Suppose Assumptions \ref{Assu:H-n-s} and 
 \ref{Assu:ratio-K-N} hold.
 Then for any \(\epsilon > 0\), there exists a threshold \(K_0\) such that for all \(K > K_0\), 
 the non-zero singular values \(d_i\) of \(D\)
 almost surely lie in a compact interval 
 \(\left[1 - \frac{1}{\sqrt{\gamma}}-\epsilon, 1 + \frac{1}{\sqrt{\gamma}}+\epsilon\right]\).
\label{domaom-f}
\end{proposition}

\begin{proof}
    Consider the sample covariance matrix \(S = \frac{1}{N} H H^H\). 
    Under the assumption that \(H_{ij} \sim \mathcal{CN}(0, \frac{1}{N})\), the eigenvalues \(\lambda_i\) of \(S\) follow the Marchenko-Pastur distribution as \(N, K \to \infty\) with \(N/K \to \gamma > 1\). The support set of the Marchenko-Pastur distribution is given by
    $
    \left[(1 - \sqrt{c})^2, (1 + \sqrt{c})^2\right],
    $
    where \(c = \frac{1}{\gamma}\).
    By 
    \cite[Bai-Yin Theorem]{BY88}, the largest and smallest eigenvalues of \(S\) almost surely converge to the edges of the Marchenko-Pastur support
    \[
    \lambda_{\max} \xrightarrow{a.s.} \left(1 + \frac{1}{\sqrt{\gamma}}\right)^2, \quad \lambda_{\min} \xrightarrow{a.s.} \left(1 - \frac{1}{\sqrt{\gamma}}\right)^2
    \] as $K \to \infty$.
    Since the singular values \(d_i\) of \(H\) are related to the eigenvalues \(\lambda_i\) of \(S\) by \(d_i = \sqrt{\lambda_i}\), we have
    \[
    d_{\max} \xrightarrow{a.s.} 1 + \frac{1}{\sqrt{\gamma}}, \quad d_{\min} \xrightarrow{a.s.} 1 - \frac{1}{\sqrt{\gamma}}
    \] as $K \to \infty$.
    For any \(\epsilon > 0\), there exists a threshold \(K_0\) such that for all \(K > K_0\), the non-zero singular values \(d_i\)  satisfies
    $
    1 - \frac{1}{\sqrt{\gamma}} - \epsilon \leq d_i \leq 1 + \frac{1}{\sqrt{\gamma}} + \epsilon
    $ almost surely.
\end{proof}

\subsection{Variance of the average of the singular values}

 The next result establishes that the variance of the linear spectral statistic \(Z_K\) decays at the rate of \(O(1/K)\).

\begin{proposition}\label{prop:main}
Consider the sample covariance matrix $\mathbf{H}\mathbf{H}^{\mathsf{H}}$ with eigenvalues $\lambda_1, \dots, \lambda_K$. Let the linear spectral statistic be
$
Z_K = \frac{1}{K} \sum_{i=1}^K \sigma(d_i),
$
where $d_i = \sqrt{\lambda_i}$  and $\sigma(\cdot) $ is defined in Assumption \ref{Assu:d}. Suppose Assumption \ref{Assu:H-n-s}, Assumption \ref{Assu:funct-f}, 
\ref{Assu:ratio-K-N} and  
\ref{Assu:d} hold.
Then for $K > \bar{K}$, the variance of $Z_K$ satisfies
$
\text{Var}(Z_K) \leq \frac{2 M_1^2}{K}.
$
\end{proposition}

\begin{proof}

Let $X = \sum_{i=1}^K \sigma(d_i)$ be the linear statistic, where $d_i = \sqrt{\lambda_i}$. Then $Z_K = X/K$ and by basic properties of variance,
\[
\text{Var}(Z_K) = \frac{1}{K^2} \text{Var}(X).
\]
For the complex Wishart matrix $\mathbf{H}\mathbf{H}^{\mathsf{H}}$ with $K \leq N$, the joint eigenvalue density is given by 
\[
p(\lambda_1, \dots, \lambda_K) = \frac{1}{C_{K,N}} \prod_{i=1}^K \lambda_i^{N-K} e^{-N\lambda_i} \prod_{i<j} |\lambda_i - \lambda_j|^2,
\]
where $C_{K,N}$ is a normalization constant, see \cite[Theorem 2.17]{Tulino2004}. This structure implies that the eigenvalue point process is a deterministic
point process (DPP), see \cite[Chapter 3]{T23}.

Let $K_2(\lambda, \mu)$ be the correlation kernel of this DPP. The key correlation functions are one-point correlation function 
$
\rho_1(\lambda) = K_2(\lambda, \lambda),
$
which represents the density of eigenvalues at position $\lambda$.
Two-point correlation function is
\[
\rho_2(\lambda, \mu) = \rho_1(\lambda)\rho_1(\mu) - |K_2(\lambda, \mu)|^2,
\]
which describes the joint density of eigenvalue pairs at positions $\lambda$ and $\mu$.
For any linear statistic $X = \sum_{i=1}^K g(\lambda_i)$ of a DPP, the variance can be expressed as (see \cite[Theorem 11.2.1]{Anderson2010},
\begin{align*}
\text{Var}(X) &= \mathbb{E}[X^2] - (\mathbb{E}[X])^2 \\
&= \int_0^\infty g^2(\lambda) \rho_1(\lambda)  d\lambda + \iint_{(0,\infty)^2} g(\lambda)g(\mu) \rho_2(\lambda, \mu)  d\lambda d\mu 
- \left(\int_0^\infty g(\lambda) \rho_1(\lambda)  d\lambda\right)^2.
\end{align*}
Substituting the expression for $\rho_2(\lambda, \mu)$, we have
\begin{align*}
\text{Var}(X) &= \int_0^\infty g^2(\lambda) \rho_1(\lambda)  d\lambda + \iint_{(0,\infty)^2} g(\lambda)g(\mu) [\rho_1(\lambda)\rho_1(\mu) - |K_2(\lambda, \mu)|^2]  d\lambda d\mu \\
&\quad - \left(\int_0^\infty g(\lambda) \rho_1(\lambda)  d\lambda\right)^2.
\end{align*}
A simple reorganization yields 
\begin{equation}\label{eq:variance-formula}
\text{Var}(X) = \int_0^\infty g^2(\lambda) \rho_1(\lambda)  d\lambda - \iint_{(0,\infty)^2} g(\lambda)g(\mu) |K_2(\lambda, \mu)|^2  d\lambda d\mu.
\end{equation}
Set $g(\lambda) = \sigma(\sqrt{\lambda})$. Since $\sigma$ is bounded with $|\sigma| \leq M_1$, 
then 
$|g(\lambda)| \leq M_1$ for all $\lambda > 0$.
The first term is bounded by
\begin{align*}
\left| \int_0^\infty g^2(\lambda) \rho_1(\lambda)  d\lambda \right| \leq \int_0^\infty |g^2(\lambda)| \rho_1(\lambda)  d\lambda \leq M_1^2 \int_0^\infty \rho_1(\lambda)  d\lambda = M_1^2 K,
\end{align*}
where we use the fundamental property of the one-point function (see \cite[Lemma 17]{Hough2006}),
$
\int_0^\infty \rho_1(\lambda)  d\lambda = K.
$
Likewise, we can derive a bound for the second term
\begin{align*}
\left| \iint_{(0,\infty)^2} g(\lambda)g(\mu) |K_2(\lambda, \mu)|^2  d\lambda d\mu \right| &\leq \iint_{(0,\infty)^2} |g(\lambda)g(\mu)| |K_2(\lambda, \mu)|^2  d\lambda d\mu \\
&\leq M_1^2 \iint_{(0,\infty)^2} |K_2(\lambda, \mu)|^2  d\lambda d\mu.
\end{align*}
By \cite{Hough2006}, $K_2$ is self-adjoint, so we have the identity  as follows
\[
\int_0^\infty |K_2(\lambda, \mu)|^2  d\mu = K_2(\lambda, \lambda) = \rho_1(\lambda)
\]
and hence 
\[
\iint_{(0,\infty)^2} |K_2(\lambda, \mu)|^2  d\lambda d\mu = \int_0^\infty \left[\int_0^\infty |K_2(\lambda, \mu)|^2  d\mu\right] d\lambda = \int_0^\infty \rho_1(\lambda)  d\lambda = K.
\]
Thus, the second term is bounded by $M_1^2 K$.
From equation \eqref{eq:variance-formula}, we have
\[
\text{Var}(X) \leq M_1^2 K + M_1^2 K = 2M_1^2 K
\]
and subsequently
\[
\text{Var}(Z_K) = \frac{1}{K^2} \text{Var}(X) \leq \frac{2M_1^2}{K}.
\]
This completes the proof.
\end{proof}
   All integrals in the proposition are Lebesgue integrals over $(0,\infty)$ since eigenvalues are positive. In practice, for finite $K,N$, the eigenvalues are bounded within $[\lambda_{\min}, \lambda_{\max}]$, and the correlation functions vanish outside this interval.
The double integral $\iint_{(0,\infty)^2}$ represents integration over all pairs $(\lambda, \mu)$ in the positive quadrant. The term $|K_2(\lambda, \mu)|^2$ quantifies the repulsive interaction between eigenvalues at positions $\lambda$ and $\mu$.
The almost everywhere continuity of $\sigma$ ensures the integrals are well-defined, while the boundedness condition provides the uniform bounds needed for the variance estimate. 

\subsection{Exponential rate of convergence of the average of the 
singular values}

By Chebyshev's inequality, the following result is directly from Proposition \ref{prop:main}.
\begin{corollary}\label{cor:prob-bound}
Suppose conditions in Proposition \ref{prop:main} hold,  we have 
\[
\mathbb{P}\left( \left| Z_K - \mathbb{E}[Z_K] \right| > \varepsilon \right) \leq \frac{\text{Var}(Z_K)}{\varepsilon^2} \leq \frac{2 M_1^2}{K \varepsilon^2}.
\]
\end{corollary}

\begin{lemma}[Explicit Convergence Rate for Quadratic Form]
\label{thm:explicit_quad}
Let  Assumptions \ref{Assu:H-n-s}, 
\ref{Assu:funct-f}, 
\ref{Assu:ratio-K-N} and  
\ref{Assu:d} hold.  Define
\[
\mathbf{D}=\operatorname{diag}(d_1,\dots,d_K),\qquad d_i=\sqrt{\lambda_i},
\]
where $\lambda_i$ are the non-zero eigenvalues of $\mathbf{H}\mathbf{H}^{\mathsf{H}}$.
Then for $K>\bar{K}$, for every $\varepsilon>0$,
\begin{equation}
\mathbb{P}\!\Bigl(\,\bigl|\tfrac{1}{K}\mathbf{g}_1^{\mathsf{H}}\sigma(\mathbf{D})\mathbf{g}_1-\mathbb{E}[\sigma(d)]\bigr|\geq\varepsilon\Bigr)
\leq 2 \exp\left( -\dfrac{1}{2} K \min\left\{ \dfrac{\varepsilon^2}{16M_1^2}, \dfrac{\varepsilon}{4M_1} \right\} \right)+\frac{8 M_1^2}{K \varepsilon^2},
\label{eq:quad_explicit}
\end{equation}  and
\begin{equation}
\mathbb{P}\!\Bigl(\,\bigl|\tfrac{1}{K}\mathbf{g}_1^{\mathsf{H}}\sigma(\mathbf{D})\mathbf{g}_2\bigr|\geq\varepsilon\Bigr)
\leq 4\exp\!\bigl(-\tfrac{K\varepsilon^{2}}{2M_1^{2}}\bigr),
\label{eq:cross_rate}
\end{equation}
where $\sigma(\cdot) $ is defined in Assumption \ref{Assu:d}. 
\end{lemma}

\begin{proof} 
\underline{Inequality \eqref{eq:quad_explicit}}.
We reuse the decomposition
\[
\bigl|\tfrac{1}{K}\mathbf{g}_1^{\mathsf{H}}\sigma(\mathbf{D})\mathbf{g}_1-\mathbb{E}[\sigma(d)]\bigr|
\leq\bigl|\tfrac{1}{K}\mathbf{g}_1^{\mathsf{H}}\sigma(\mathbf{D})\mathbf{g}_1-Z_K\bigr|
+\bigl|Z_K-\mathbb{E}[\sigma(d)]\bigr|,
\]
where $Z_K=\tfrac{1}{K}\operatorname{tr}\!\bigl(\sigma(\mathbf{D})\bigr)$.
Then
\[\mathbb{P}\!\Bigl(\,\bigl|\tfrac{1}{K}\mathbf{g}_1^{\mathsf{H}}\sigma(\mathbf{D})\mathbf{g}_1-\mathbb{E}[\sigma(d)]\bigr|\geq\varepsilon\Bigr)\leq \mathbb{P}\!\Bigl(\,\bigl|\tfrac{1}{K}\mathbf{g}_1^{\mathsf{H}}\sigma(\mathbf{D})\mathbf{g}_1-Z_K\bigr|\geq\frac{\varepsilon}{2}\Bigr)+\mathbb{P}\!\Bigl(\,\bigl|Z_K-\mathbb{E}[\sigma(d)]\bigr|\geq\frac{\varepsilon}{2}).\]
We proceed the  proof in two steps.

\textbf{Step 1}. 
Estimate $\mathbb{P}\!\Bigl(\,\bigl|\tfrac{1}{K}\mathbf{g}_1^{\mathsf{H}}\sigma(\mathbf{D})\mathbf{g}_1-Z_K\bigr|\geq\frac{\varepsilon}{2}\Bigr)$.
Define $
X_i = \sigma(d_i) \left( |g_1[i]|^2 - 1 \right),$ then $ \tfrac{1}{K}\mathbf{g}_1^{\mathsf{H}}\sigma(\mathbf{D})\mathbf{g}_1-Z_K = \frac{1}{K} \sum_{i=1}^K X_i$.  
For the complex Gaussian vector \(\mathbf{g}_1 \sim \mathcal{CN}(0, \mathbf{I}_K)\), each element \(g_1[i]\) has its squared modulus \(|g_1[i]|^2\) following a chi-squared distribution with two degrees of freedom (i.e., an exponential distribution with parameter \(\lambda = 1\)). The mean is \(\mathbb{E}[|g_1[i]|^2] = 1\). 
For $V_i = |g_1[i]|^2 - 1$, \[
\mathbb E\!\left[e^{|V_i|/2}\right]
=\int_0^\infty e^{v/2}\,e^{-(v+1)}\,dv
=\frac{1}{e}\int_0^\infty e^{-v/2}\,dv
=\frac{2}{e}<2.\]
        Hence, $\|V_i\|_{\psi_1} \leq 2$. 
For $X_i = \sigma(d_i) V_i$, by $|\sigma(d_i)| \leq M_1$ almost surely and the homogeneity of the sub-exponential norm, we have
   $$ \|X_i\|_{\psi_1} = \|\sigma(d_i) V_i\|_{\psi_1} \leq |\sigma(d_i)| \cdot \|V_i\|_{\psi_1} \leq 2M_1.$$
Given $\mathbf{D}$, $\{X_i\}$ are conditionally independent.
Therefore, under the conditional independence framework,  Bernstein's inequality in Proposition \ref{B inequality} gives the conditional probability bound as follows
\[
\mathbb{P}\left( \left. \left| \sum_{i=1}^K X_i \right| \geq t \,\right| \mathbf{D} \right) \leq 2 \exp\left( -\frac{1}{2} \min\left\{ \frac{t^2}{\sum_{i=1}^K \|X_i\|_{\psi_1}^2}, \frac{t}{\max_i \|X_i\|_{\psi_1}} \right\} \right).
\]
Substituting $\|X_i\|_{\psi_1} \leq 2M_1$, we have 
\[
\sum_{i=1}^K \|X_i\|_{\psi_1}^2 \leq 4M_1^2 K, \quad \max_i \|X_i\|_{\psi_1} \leq 2M_1.
\]
Thus 
\[
\mathbb{P}\left( \left. \left| \sum_{i=1}^K X_i \right| \geq t \,\right| \mathbf{D} \right) \leq 2 \exp\left( -\frac{1}{2} \min\left\{ \frac{t^2}{4M_1^2 K}, \frac{t}{2M_1} \right\} \right).
\]
Taking $t = K \varepsilon / 2$, we obtain
\[
\mathbb{P}\left( \left. \left| \tfrac{1}{K}\mathbf{g}_1^{\mathsf{H}}\sigma(\mathbf{D})\mathbf{g}_1 - Z_K \right| \geq \frac{\varepsilon}{2} \,\right| \mathbf{D} \right) \leq 2 \exp\left( -\dfrac{1}{2} K \min\left\{ \dfrac{\varepsilon^2}{16M_1^2}, \dfrac{\varepsilon}{4M_1} \right\} \right).
\]
 Taking the expectation over $\mathbf{D}$, we have 
\[
\mathbb{P}\left(  \left| \tfrac{1}{K}\mathbf{g}_1^{\mathsf{H}}\sigma(\mathbf{D})\mathbf{g}_1 - Z_K \right| \geq \frac{\varepsilon}{2}\right) \leq 2 \exp\left( -\dfrac{1}{2} K \min\left\{ \dfrac{\varepsilon^2}{16M_1^2}, \dfrac{\varepsilon}{4M_1} \right\} \right).
\]

\textbf{Step 2}. Estimate $\mathbb{P}\!\Bigl(\,\bigl|Z_K-\mathbb{E}[\sigma(d)]\bigr|\geq\frac{\varepsilon}{2})$.
By Chebyshev's inequality in Corollary \ref{cor:prob-bound}, we have
\[
\mathbb{P}\left( \left| Z_K - \mathbb{E}[Z_K] \right| > \frac{\varepsilon}{2} \right) \leq \frac{8M_1^2}{K \varepsilon^2}.
\]
A combination of the two steps gives rise to \eqref{eq:quad_explicit}.

\underline{Inequality \eqref{eq:cross_rate}}.
Decompose $X=\tfrac{1}{K}\mathbf{g}_1^{\mathsf{H}}\sigma(\mathbf{D})\mathbf{g}_2$ into real and imaginary parts:
\[
\mathcal{R}(X)=\tfrac{1}{2K}\bigl(\mathbf{g}_1^{\mathsf{H}}\sigma(\mathbf{D})\mathbf{g}_2+\mathbf{g}_2^{\mathsf{H}}\sigma(\mathbf{D})\mathbf{g}_1\bigr),\qquad
\mathcal{I}(X)=\tfrac{1}{2iK}\bigl(\mathbf{g}_1^{\mathsf{H}}\sigma(\mathbf{D})\mathbf{g}_2-\mathbf{g}_2^{\mathsf{H}}\sigma(\mathbf{D})\mathbf{g}_1\bigr).
\]

Conditioned on \(\mathbf{D}\), \(\mathcal{R}(X)\) and \(\mathcal{I}(X)\) are independent zero-mean real Gaussian random variables. The conditional variance is
\[ \text{Var}(\mathcal{R}(X) \mid \mathbf{D}) = \frac{\|\sigma(\mathbf{D})\|_{\text{F}}^2}{2K^2} \leq \frac{\|\sigma(\mathbf{D})\|^2}{2K}, \]
where \(\|\sigma(\mathbf{D})\|_{\text{F}}\) is the Frobenius norm and \(\|\sigma(\mathbf{D})\|\) is the spectral norm.
Observe that
  \[ \mathbb{E}[|\mathbf{g}_1^\mathsf{H} \sigma(\mathbf{D}) \mathbf{g}_2|^2 \mid \mathbf{D}] = \mathbb{E}\left[\mathbf{g}_1^\mathsf{H} \sigma(\mathbf{D}) \mathbf{g}_2 \mathbf{g}_2^\mathsf{H} \sigma(\mathbf{D}) \mathbf{g}_1 \mid \mathbf{D}\right] = \text{Tr}(\sigma(\mathbf{D})^2) = \|\sigma(\mathbf{D})\|_{\text{F}}^2. \]
  Since \(\mathbb{E}[\mathbf{g}_2 \mathbf{g}_2^\mathsf{H}] = \mathbf{I}_K\) and \(\mathbb{E}[\mathbf{g}_1^\mathsf{H} \mathbf{A} \mathbf{g}_1] = \text{Tr}(\mathbf{A})\) for \(\mathbf{g}_1 \sim \mathcal{CN}(0, \mathbf{I}_K)\).
 Thus, \(\text{Var}(X \mid \mathbf{D}) = \frac{\|\sigma(\mathbf{D})\|_{\text{F}}^2}{K^2}\) and 
  \[ \text{Var}(\mathcal{R}(X) \mid \mathbf{D}) = \frac{1}{2} \text{Var}(X \mid \mathbf{D}) = \frac{\|\sigma(\mathbf{D})\|_{\text{F}}^2}{2K^2}. \]
Since \(\|\sigma(\mathbf{D})\|_{\text{F}}^2 \leq K \|\sigma(\mathbf{D})\|^2\) and 
for \(K > \bar{K}\), \(\|\sigma(\mathbf{D})\| = \sup_{1 \leq i \leq K} |\sigma(d_i)| \leq M_1\) almost surely, we have 
\[ \text{Var}(\mathcal{R}(X) \mid \mathbf{D}) \leq \frac{M_1^2}{2K} \] almost surely.
In the same way, we have 
\[ \text{Var}(\mathcal{I}(X) \mid \mathbf{D}) \leq \frac{M_1^2}{2K} \] almost surely.
For a zero-mean real Gaussian \(Z\) with variance \(v\), \(\mathbb{P}(|Z| \geq t) \leq 2 \exp\left(-\frac{t^2}{2v}\right)\). Thus we have 
\[ \mathbb{P}\left( |\mathcal{R}(X)| \geq \frac{\varepsilon}{\sqrt{2}} \mid \mathbf{D} \right) \leq 2 \exp\left( -\frac{ (\varepsilon/\sqrt{2})^2 }{2 \cdot \text{Var}(\mathcal{R}(X) \mid \mathbf{D})} \right) \leq 2 \exp\left( -\frac{\varepsilon^2 / 2}{2 \cdot \frac{M_1^2}{2K}} \right) = 2 \exp\left( -\frac{K \varepsilon^2}{2M_1^2} \right). \]
Similarly, we establish
\[ \mathbb{P}\left( |\mathcal{I}(X)| \geq \frac{\varepsilon}{\sqrt{2}} \mid \mathbf{D} \right) \leq 2 \exp\left( -\frac{K \varepsilon^2}{2M_1^2} \right). \]
Since
\[ |X| = \sqrt{[\mathcal{R}(X)]^2 + [\mathcal{I}(X)]^2} < \sqrt{ \left(\frac{\varepsilon}{\sqrt{2}}\right)^2 + \left(\frac{\varepsilon}{\sqrt{2}}\right)^2 } = \varepsilon, \]
we have
\[ \mathbb{P}(|X| \geq \varepsilon \mid \mathbf{D}) \leq \mathbb{P}\left( |\mathcal{R}(X)| \geq \frac{\varepsilon}{\sqrt{2}} \mid \mathbf{D} \right) + \mathbb{P}\left( |\mathcal{I}(X)| \geq \frac{\varepsilon}{\sqrt{2}} \mid \mathbf{D} \right) \leq 4 \exp\left( -\frac{K \varepsilon^2}{2M_1^2} \right), \]
which means  for \(K > \bar{K}\), the unconditional probability is
\[ \mathbb{P}\left( \left| \frac{1}{K} \mathbf{g}_1^{\mathsf{H}} \sigma(\mathbf{D}) \mathbf{g}_2 \right| \geq \varepsilon \right) \leq 4 \exp\left( -\frac{K \varepsilon^2}{2M_1^2} \right). \]
\end{proof}

\subsection{Some intermediate technical results for Lemma~\ref{Lem:T_g-convg-barT_g}}

\begin{lemma}\label{lemm11} Suppose Assumptions \ref{Assu:H-n-s}, \ref{Assu:funct-f},  \ref{Assu:ratio-K-N} and  \ref{Assu:d} hold. 
Then for $K>\bar{K}$,  for every $\epsilon > 0$,  we have
\[
\P\!\left(\,\left|\alpha^{2} - \bar{\alpha}^{2}\right| > \epsilon\,\right) \leq 8 \exp\left(- K \min\left\{\frac{\hat{\delta}(\epsilon)^2}{C_1'}, \frac{\hat{\delta}(\epsilon)}{C_2'}\right\}\right)+\dfrac{8 M_1^2}{K \hat{\delta}(\epsilon)^2},
\]
where $$\alpha^2 = \dfrac{\|\bm{s}\|^2 \|f(\bm{D})^{T}\bm{g}_1\|^2}{\|\bm{g}_1\|^2 \|\bm{z}_1\|^2}, \quad \bar{\alpha}^2 = \dfrac{\sigma_s^2 \E[f^2(d)]}{\gamma},$$ $$
 C_1' = \max\left\{\frac{c_{\max}^2}{2}, 32, 32M_1^2\right\}, \quad C_2' = \max\left\{ 8, 8M_1\right\},$$  and $$
L = 4(1 + \sigma_s^2)(\E[f^2(d)] + 1) + 2\sigma_s^2(1 + \E[f^2(d)]),\quad
\hat{\delta}(\epsilon) = \min\left\{\dfrac{\gamma\epsilon}{2L}, \dfrac{1}{2}\right\}.
$$
\end{lemma}

\begin{proof}
Define the random variables:
\[
A = \dfrac{\|\bm{s}\|^2}{K}, \quad B = \dfrac{\|\bm{g}_1\|^2}{K}, \quad C = \dfrac{\|f(\bm{D})^{T}\bm{g}_1\|^2}{K}, \quad D = \dfrac{\|\bm{z}_1\|^2}{N},
\]
with $N = \gamma K$. Then 
\[
\alpha^2 = \dfrac{A C}{\gamma B D}, \quad \bar{\alpha}^2 = \dfrac{\sigma_s^2 \E[f^2(d)]}{\gamma}.
\]
We want to estimate 
\[
\P\!\left(\,\left|\dfrac{A C}{B D} - \sigma_s^2 \E[f^2(d)]\right| > \eta\,\right), \quad \eta = \gamma \epsilon.
\]
Using the triangle inequality, we have 
\[
\left|\dfrac{A C}{B D} - \sigma_s^2 \E[f^2(d)]\right| \leq \dfrac{C}{D} \left|\dfrac{A}{B} - \sigma_s^2\right| + \sigma_s^2 \left|\dfrac{C}{D} - \E[f^2(d)]\right|:= T_1+T_2.
\]
Let
\[
L = 4(1 + \sigma_s^2)(\E[f^2(d)] + 1) + 2\sigma_s^2(1 + \E[f^2(d)])
\]
and $\delta = \min\left\{\dfrac{\eta}{2L}, \dfrac{1}{2}\right\}.$
Define the events
\[
E_A = \{|A - \sigma_s^2| > \delta\}, \quad E_B = \{|B - 1| > \delta\}, \quad E_C = \{|C - \E[f^2(d)]| > \delta\}, \quad E_D = \{|D - 1| > \delta\}.
\]
If $E^c$ occurs, where $E = E_A \cup E_B \cup E_C \cup E_D$, then 
$$B \geq 1 - \delta \geq \dfrac{1}{2}, D \geq \dfrac{1}{2},$$ which means that 
\[
\begin{aligned}
T_1 &= \frac{C}{D} \left| \frac{A}{B} - \sigma_s^2 \right| \\
&\leq \left( \frac{\mathbb{E}[f^2(d)] + \delta}{1 - \delta} \right) \cdot \left( \frac{|A - \sigma_s^2| + \sigma_s^2 |1-B|}{B} \right) \\
&\leq 2(\mathbb{E}[f^2(d)] + 1) \cdot \left( \frac{\delta + \sigma_s^2 \delta}{1/2} \right) \\
&= 4(\mathbb{E}[f^2(d)] + 1) (1 + \sigma_s^2)\delta
\end{aligned}
\]
and
\[
\begin{aligned}
T_2 &= \sigma_s^2 \left| \frac{C}{D} - \mathbb{E}[f^2(d)] \right| \\
&\leq \sigma_s^2 \left( \frac{|C - \mathbb{E}[f^2(d)]| + \mathbb{E}[f^2(d)] |1-D|}{D} \right) \\
&\leq \sigma_s^2 \left( \frac{\delta + \mathbb{E}[f^2(d)] \delta}{1/2} \right) \\
&= 2\sigma_s^2 (1 + \mathbb{E}[f^2(d)])\delta.
\end{aligned}
\]
Therefore
\[
\begin{aligned}
T_1 + T_2 \leq\left[ 4(\mathbb{E}[f^2(d)] + 1) (1 + \sigma_s^2) + 2\sigma_s^2 (1 + \mathbb{E}[f^2(d)]) \right] \delta 
= L\delta\leq \eta.
\end{aligned}
\]
Thus 
\[
\P\!\left(\,\left|\dfrac{A C}{B D} - \sigma_s^2 \E[f^2(d)]\right| > \eta\,\right) \leq \P(E) \leq \sum_{i \in \{A,B,C,D\}} \P(E_i).
\]
By Hoeffding inequality in Proposition \ref{pro-hoeff-ineq}, 
  \[
  \P(E_A)  \leq 2 \exp\left(-\frac{2K \delta^2}{c_{\max}^2}\right),
  \]
where $c_{\max}$ is
$
c_{\max} = \max_{s \in \mathcal{S}_M} |s| ^2 - \min_{s \in \mathcal{S}_M}|s|
^2.
$
  By Proposition \ref{pro-exp-bound-lambda1}, 
  $$
\P(E_B) \leq 2 \exp\left(-c_1 K \min\left\{\dfrac{\delta^2}{16}, \dfrac{\delta}{4}\right\}\right)\,\,\mbox{and}\,\,
\P(E_D) \leq 2 \exp\left(-c_1 \gamma K \min\left\{\dfrac{\delta^2}{16}, \dfrac{\delta}{4}\right\}\right),
$$
where $c_1=\frac{1}{2} > 0$ are constants from the exponential concentration bound. 
 By Lemma \ref{thm:explicit_quad}, 
  \[
  \P(E_C) \leq 2 \exp\left(-\dfrac{K}{2} \min\left\{\dfrac{\delta^2}{16M_1^2}, \dfrac{\delta}{4M_1}\right\}\right) + \dfrac{8 M_1^2}{K \delta^2}.
  \]
Combining exponential terms using minimal coefficients by  
\[
 C_1' = \max\left\{\frac{c_{\max}^2}{2}, 32, 32M_1^2\right\}, \quad C_2' = \max\left\{ 8, 8M_1\right\},
\]
we obtain
\[
\sum_{i\in \{A, B, C D\}} \P(E_i) \leq 8 \exp\left(- K \min\left\{\frac{\delta^2}{C_1'}, \frac{\delta}{C_2'}\right\}\right)+\dfrac{8M_1^2}{K \delta^2},
\] which completes the proof.
\end{proof}

 In Lemma 12 of \cite{WLS24}, the authors establish the convergence of \(\frac{\mathbf{z}^\mathsf{H} q(\alpha \mathbf{z})}{N}\) to \(\mathbb{E}\left[ Z^\dagger q(\bar{\alpha} Z) \right]\). We now aim to investigate the rate of this convergence. However, since \(q(\cdot)\) is continuous almost everywhere, standard tools such as the concentration of Lipschitz functions are not directly applicable. Drawing on the proof of Lemma 12 in \cite{WLS24}, we introduce the Lipschitz envelope of \(q(\cdot)\) to address this challenge.
For
$g(w) = \mathcal{R}(q(w)), h(w) = \mathcal{I}(q(w))$,
we introduce the lower and upper envelopes of $g, h$ as follows:
\begin{equation}\label{def-lg}
l^g_\tau(x) = \inf_{y} \left\{ g(y) + \frac{|x - y|}{\tau} \right\}, \quad u^g_\tau(x) = \sup_{y} \left\{ g(y) - \frac{|x - y|}{\tau} \right\}
\end{equation}
and \begin{equation}\label{def-lh}
l^h_\tau(x) = \inf_{y} \left\{ h(y) + \frac{|x - y|}{\tau} \right\}, \quad u^h_\tau(x) = \sup_{y} \left\{ h(y) - \frac{|x - y|}{\tau} \right\}.
\end{equation}
 These functions satisfy
\begin{itemize}
    \item $l^g_\tau(x) \leq g(x) \leq u^g_\tau(x)$, $l^h_\tau(x) \leq h(x) \leq u^h_\tau(x)$,
    \item $l^g_\tau,l^h_\tau,u^g_\tau,u^h_\tau$ are $\frac{1}{\tau}$-Lipschitz,
    \item $\lim_{\tau \to 0} l^g_\tau(x) = g(x) = \lim_{\tau \to 0} u^g_\tau(x)$ at continuity points of $g$ and $\lim_{\tau \to 0} l^h_\tau(x) = h(x) = \lim_{\tau \to 0} u^h_\tau(x)$ at continuity points of $h$.
\end{itemize}
Since $|q(z)|\leq M_0$ for any $z\in\mathbb{C}, $ for fixed \(x\), choose \(y=x\), we have
\[
l^g_\tau(x)\le g(x)\le M_0,\qquad 
u^g_\tau(x)\ge g(x)\ge -M_0.
\] 
Since
\[
l^g_\tau(x)\ge \inf_y g(y)\ge -M_0,\qquad 
u^g_\tau(x)\le \sup_y g(y)\le M_0,
\] we obtain 
\[
|l^g_\tau(x)|\le M_0,\qquad |u^g_\tau(x)|\le M_0.
\] Similarly, we have \[
|l^h_\tau(x)|\le M_0,\qquad |u^h_\tau(x)|\le M_0.
\] 

\begin{lemma}\label{le-env}
Let \(Z\sim\mathcal{CN}(0,1)\) and \(\mathcal{D}_g,\mathcal{D}_h\subset\mathbb{C}\) be 
set of points where 
\(g\) and \(h\) are discontinuous respectively.
Under Assumption \ref{Assu:-funct-q},
for any \(0<\tau\le \bar\alpha\) with a fixed \(\bar\alpha>0\),
\[
\mathbb{E}\!\big[|u^g_\tau(\bar\alpha Z)-l^g_\tau(\bar\alpha Z)|\big]\ \le\ C_g\,\tau,\qquad
\mathbb{E}\!\big[|u^h_\tau(\bar\alpha Z)-l^h_\tau(\bar\alpha Z)|\big]\ \le\ C_h\,\tau,
\]
where
\[
C_g=\frac{2M_0}{\bar\alpha}\,K_g,\quad C_h=\frac{2M_0}{\bar\alpha}\,K_h,
\]
and
\[
K_g:=\frac{2}{\sqrt\pi}N_\ell(g)+\Big(1+\frac{1}{\sqrt\pi}\Big)N_r(g),\qquad
K_h:=\frac{2}{\sqrt\pi}N_\ell(h)+\Big(1+\frac{1}{\sqrt\pi}\Big)N_r(h),
\]
where \(N_\ell(\cdot)\) and \(N_r(\cdot)\) are the numbers of lines and rays in the corresponding set of discontinuous points 
respectively. 
\end{lemma}

\begin{proof}
We only prove for \(g\), the proof for \(h\) is similar. Define
\[
\Delta^g_\tau(x):=|u^g_\tau(x)-l^g_\tau(x)|.
\]
By piecewise constancy, if \(\operatorname{dist}(x,\mathcal{D}_g)>\tau\), then \(u^g_\tau(x)=l^g_\tau(x)=g(x)\), hence we have \(\Delta^g_\tau(x)=0\), overall, we have \(0\le \Delta^g_\tau\le 2M_0\). Thus,
\[
\mathbb{E}[\Delta^g_\tau(\bar\alpha Z)]
= \int_{\mathbb{C}}\!\Delta^g_\tau(\bar\alpha z)\,\phi(z)\,dz
\le 2M_0\!\!\int_{\{z:\operatorname{dist}(\bar\alpha z,\mathcal{D}_g)\le \tau\}}\!\!\phi(z)\,dz,
\]
where \(\phi(z)=\pi^{-1}e^{-|z|^2}\). Let \(r=\tau/\bar\alpha\) and \(\Gamma=\mathcal{D}_g/\bar\alpha\), then we have
\begin{equation}\label{eq:key}
\mathbb{E}[\Delta^g_\tau(\bar\alpha Z)]\ \le\ 2M_0\int_{\{z:\operatorname{dist}(z,\Gamma)\le r\}}\phi(z)\,dz.
\end{equation}
 Without loss of generality, consider the line \(\{(x,y):y=b\}\). Its \(r\)-band is $$\mathcal{D}_r:=\{(x,y):|y-b|\le r\}. $$Then
\[
\int_{\mathcal{D}_r}\phi(x+iy)\,dx\,dy
=\frac{1}{\pi}\!\int_{\mathbb{R}}e^{-x^2}dx\int_{b-r}^{b+r}e^{-y^2}dy
\le \frac{\sqrt{\pi}}{\pi}\cdot (2r)=\frac{2}{\sqrt{\pi}}\,r,
\]
using \(\int_{\mathbb{R}}e^{-x^2}\,dx=\sqrt\pi\).
 Consider the ray \(\{(x,0):x\ge 0\}\). Let $\mathcal{D}_r'$ denotes its \(r\)-neighborhood, which includes a width \(2r\) strip in the right half-plane and a circular cap of radius \(r\) centered at the origin. Then we have 
\[
\int_{\mathcal{D}_r'}\phi(x+iy)\,dx\,dy
\le \frac{1}{\pi}\!\int_{x\ge 0}e^{-x^2}dx\!\int_{|y|\le r}e^{-y^2}dy
+\frac{1}{\pi}\!\iint_{x^2+y^2\le r^2}e^{-(x^2+y^2)}dx\,dy.
\]
Since \(\int_{x\ge 0}e^{-x^2}dx=\sqrt\pi/2\), \(\int_{|y|\le r}e^{-y^2}dy\le 2r\), and \(1-e^{-r^2}\le r^2\). Thus, we have for \(0<r\le 1\) (i.e., \(0<\tau\le \bar\alpha\)),
\[
\int_{\mathcal{D}_r'}\phi(x+iy)\,dx\,dy \le \frac{1}{\sqrt{\pi}}\,r+r^2 \le \Big(1+\frac{1}{\sqrt{\pi}}\Big)\,r.
\]
Consequently, if \(\Gamma\) consists of \(N_\ell(g)\) lines and \(N_r(g)\) rays, using additivity we obtain
\[
\int_{\{ z:\operatorname{dist}(z,\Gamma)\le r\}}\phi(z)\,dz
\le \frac{2}{\sqrt{\pi}}N_\ell(g)\,r+\Big(1+\frac{1}{\sqrt{\pi}}\Big)N_r(g)\,r.
\]
Substituting back into \eqref{eq:key} with \(r=\tau/\bar\alpha\) yields
\[
\mathbb{E}[\Delta^g_\tau(\bar\alpha Z)]\ \le\ \frac{2M_0}{\bar\alpha}\,
\Big(\tfrac{2}{\sqrt{\pi}}N_\ell(g)+\big(1+\tfrac{1}{\sqrt{\pi}}\big)N_r(g)\Big)\,\tau.
\] The proof is complete.
\end{proof}

\begin{remark}[Correspondence with Common Quantizers]
For I/Q Independent Quantization, quantize the real and imaginary parts separately, forming several vertical/horizontal lines in the complex plane due to decision thresholds, so \(\mathcal{D}_g,\mathcal{D}_h\) are unions of several lines (the number of lines is proportional to the number of thresholds per dimension); In the 1-bit case, the coordinate axes serve as boundaries, see \cite{Jacobsson2017}. For Constant Envelope/Phase Quantization, QCE/PSK, when the phase takes values in a discrete set, the decision regions are sectors centered at the origin, whose boundaries are precisely rays emanating from the origin, thus \(\mathcal{D}_g=\mathcal{D}_h\) are unions of several rays, see \cite{Jedda2018,PSKDecision}.
\end{remark}

Define $G^1(x, \alpha) = \mathcal{R}(x) g(\alpha x), G^2(x, \alpha) = \mathcal{I}(x) h(\alpha x),G^3(x, \alpha) = \mathcal{R}(x) h(\alpha x), G^4 (x, \alpha) = \mathcal{I}(x) g(\alpha x)$ and 
\[
L^1_\tau(x, \alpha) = \mathcal{R}(x)_+ l^g_\tau(\alpha x) + \mathcal{R}(x)_- u^g_\tau(\alpha x),
\]
\[
L^2_\tau(x, \alpha) = \mathcal{I}(x)_+ l^h_\tau(\alpha x) + \mathcal{I}(x)_- u^h_\tau(\alpha x),
\]
\[
L^3_\tau(x, \alpha) = \mathcal{R}(x)_+ l^h_\tau(\alpha x) + \mathcal{R}(x)_- u^h_\tau(\alpha x),
\]
\[
L^4_\tau(x, \alpha) = \mathcal{I}(x)_+ l^g_\tau(\alpha x) + \mathcal{I}(x)_- u^g_\tau(\alpha x),
\] we have the following lemma.
\begin{lemma}\label{le-6.11}
Assume the conditions of Lemma~\ref{le-env} hold.  Then for fixed $\bar\alpha>0$ and any $\tau$ satisfies $0<\tau\le\bar\alpha$, we have
\[
\big|\mathbb{E}[L^i_\tau(Z,\bar\alpha)]-\mathbb{E}[G^i(Z,\bar\alpha)]\big|
\;\le\; C_i\,\tau^{1/2},
\] $i=1,2,3,4$,
where  constants
\[
C_1=C_4=\sqrt{M_0\,\frac{2M_0}{\bar\alpha}K_g}
=\frac{\sqrt{2}\,M_0}{\sqrt{\bar\alpha}}\sqrt{K_g},\qquad
C_2=C_3=\sqrt{M_0\,\frac{2M_0}{\bar\alpha}K_h}
=\frac{\sqrt{2}\,M_0}{\sqrt{\bar\alpha}}\sqrt{K_h}.
\]
\end{lemma}

\begin{proof}
We only prove $i=1$, the other cases are identical by swapping $g,h$ and $\mathcal{R},\mathcal{I}$. Notice that 
\[
|L^1_\tau(Z,\bar\alpha)-G^1(Z,\bar\alpha)|\le |\mathcal{R}(Z)|\,\big|u^g_\tau(\bar\alpha Z)-l^g_\tau(\bar\alpha Z)\big|.
\]
Taking expectations and using Cauchy-Schwarz inequatlity, we obtain 
\begin{equation}
\label{le6.9-1}
\mathbb{E}\,\left[|L^1_\tau(Z,\bar\alpha)-G^1(Z,\bar\alpha)|\right]
\le \sqrt{\mathbb{E}[|\mathcal{R}(Z)|^2]}\;\sqrt{\mathbb{E}\!\left[\big|u^g_\tau(\bar\alpha Z)-l^g_\tau(\bar\alpha Z)\big|^2\right]}.
\end{equation}
For $Z\sim\mathcal{CN}(0,1)$, we have $\mathcal{R}(Z)\sim\mathcal{N}(0,1/2)$ and hence $\mathbb{E}[|\mathcal{R}(Z)|^2]=1/2$. Since $\big|u^g(\bar\alpha Z)_\tau-l^g_\tau(\bar\alpha Z)\big|\le 2M_0$, we have
\begin{equation}\label{le6.9-2}
\mathbb{E}\!\left[\big|u^g_\tau(\bar\alpha Z)-l^g_\tau(\bar\alpha Z)\big|^2\right]
\le 2M_0\,\mathbb{E}\!\left[\big|u^g(\bar\alpha Z)_\tau-l^g(\bar\alpha Z)_\tau\big|\right]
\le 2M_0\cdot \frac{2M_0}{\bar\alpha}K_g\,\tau,
\end{equation}
where the last step uses Lemma~\ref{le-env}. Combining (\ref{le6.9-1}) and (\ref{le6.9-2}), we have
\[
\mathbb{E}\!\left[|L^1_\tau(Z,\bar\alpha)-G^1(Z,\bar\alpha)|\right]
\le \sqrt{\tfrac12}\;\sqrt{2M_0\cdot \tfrac{2M_0}{\bar\alpha}K_g\,\tau}
=\frac{\sqrt{2}\,M_0}{\sqrt{\bar\alpha}}\sqrt{K_g}\;\tau^{1/2}.
\]
Finally, $|\mathbb{E}[L^1_\tau(Z,\bar\alpha)]-\mathbb{E}[G^1(Z,\bar\alpha)]|\le \mathbb{E}[|L^1_\tau(Z,\bar\alpha)-G^1(Z,\bar\alpha)|]$ yields the conclusion.
\end{proof}

\begin{lemma}\label{le-6.12}
Suppose Assumption \ref{Assu:H-n-s}-\ref{Assu:ratio-K-N} and  Assumption \ref{Assu:d} hold. 
Then   for any $\delta > 0$,
 for $K>\max\{\hat{K}, \bar{K}\}$ with $\hat{K}=\left(\frac{12\sqrt{2M_0^2\bar{\alpha}^{-1}\max\{K_g, K_h\}}}{\delta}\right)^{8}$,    we have
\begin{equation}
\begin{aligned}
&\mathbb{P}\left( \left| \frac{\mathbf{z}^\mathsf{H} q(\alpha \mathbf{z})}{N} - \mathbb{E}\left[ Z^\dagger q(\bar{\alpha} Z) \right] \right| \geq \delta \right) \\
&\leq 24e^{-\frac{\gamma K}{32}} +  64 \exp\left(- K \min\left\{\frac{\hat{\delta}(\bar{\alpha}\eta_1)^2}{C_1'}, \frac{\hat{\delta}(\bar{\alpha}\eta_1)}{C_2'}\right\}\right)+\dfrac{64M_1^2}{K \hat{\delta}(\bar{\alpha}\eta_1)^2} + 8 \exp\left( -\frac{\sqrt{K}\gamma \delta^2}{1152 } \right)+\frac{18432M_0^2}{\gamma K\delta^2},
\end{aligned}
\end{equation}
where $ 
 C_1' = \max\left\{\frac{c_{\max}^2}{2}, 32, 32M_1^2\right\}, \quad C_2' = \max\left\{ 8, 8M_1\right\}, \eta_1=\frac{\delta}{96\sqrt[4]{K}}$.

\end{lemma}

\begin{proof}
We decompose the complex-valued quantity into real and imaginary parts. Define
\[
g(w) = \mathcal{R}(q(w)), \quad h(w) = \mathcal{I}(q(w)).
\]
The target quantity is
\[
\frac{\mathbf{z}^{\mathsf{H}} q(\alpha \mathbf{z})}{N} = \frac{1}{N} \sum_{i=1}^N \left[ \mathcal{R}(z_i) g(\alpha z_i) - \mathcal{I}(z_i) h(\alpha z_i) \right] + j \frac{1}{N} \sum_{i=1}^N \left[ \mathcal{R}(z_i) h(\alpha z_i) + \mathcal{I}(z_i) g(\alpha z_i) \right].
\]
By symmetry, it suffices to estimate the real part
\[
\mathbb{P}\left( \left| \frac{1}{N} \sum_{i=1}^N \mathcal{R}(z_i) g(\alpha z_i) - \mathbb{E}[\mathcal{R}(Z) g(\bar{\alpha} Z)] \right| \geq \frac{\delta}{4} \right).
\]
The proof for other components follows identically. 
Since $g$ is bounded and a.e.\ continuous, for any $\tau > 0$,  we  define lower and upper envelopes of the function by
\[
 l_{\tau}(x)= \inf_{y} \left\{ g(y) + \frac{|x - y|}{\tau} \right\}, \quad u_\tau(x) = \sup_{y} \left\{ g(y) - \frac{|x - y|}{\tau} \right\}.
\] These functions satisfy properties in (\ref{def-lg}) and (\ref{def-lh}). 
Let $G(x, \alpha) = \mathcal{R}(x) g(\alpha x)$. 
Then its lower and upper  envelopes can be written as  follows:
\[
L_\tau(x, \alpha) = \mathcal{R}(x)_+ l_\tau(\alpha x) + \mathcal{R}(x)_- u_\tau(\alpha x),
\]
\[
U_\tau(x, \alpha) = \mathcal{R}(x)_+ u_\tau(\alpha x) + \mathcal{R}(x)_- l_\tau(\alpha x).
\]
Consequently 
\[
\frac{1}{N} \sum_{i=1}^N L_\tau(z_i, \alpha) \leq \frac{1}{N} \sum_{i=1}^N G(z_i, \alpha) \leq \frac{1}{N} \sum_{i=1}^N U_\tau(z_i, \alpha).
\]
Observe that 
\begin{equation}
\begin{array}{lll}
\left| \frac{1}{N} \sum_{i=1}^{N} G(z_i,\alpha) - \mathbb{E}[G(z_i,\alpha)] \right| &\leq &
\left| \frac{1}{N} \sum_{i=1}^{N} G(z_i,\alpha) - \frac{1}{N} \sum_{i=1}^{N} L_\tau(z_i,\alpha) \right|\\
&&
+ \left| \frac{1}{N} \sum_{i=1}^{N} L_\tau(z_i,\alpha) - \mathbb{E}[L_\tau(z_i,\alpha)] \right|\\&&
+ \left| \mathbb{E}[L_\tau(z_i,\alpha)] - \mathbb{E}[G(z_i,\alpha)] \right|:=T_1+T_2+T_3.
\end{array}
\end{equation}
By the dominated convergence theorem and the properties of \(l_\tau, u_\tau\), we have from 
Lemma 12 in \cite{WLS24} that
\[
\lim_{\tau \to 0} \mathbb{E}[L_\tau(Z, \bar{\alpha})] = \mathbb{E}[G(Z, \bar{\alpha})].
\]
For given $\delta>0$, choose $\tau = \tau(\delta) > 0$ such that \begin{equation}\label{le6.10-1}\left| \mathbb{E}L_\tau(Z,\bar\alpha) - \mathbb{E}G(Z,\bar\alpha) \right| \leq \delta/12.\end{equation}  We fix this  $\tau$
for the entire proof and divide  the rest of proof into two steps.

\textbf{Step 1}. Estimate  term \(T_1\). Since 
\begin{equation}
\begin{aligned}
U_{\tau}(z_i, \alpha) - L_{\tau}(z_i, \alpha) &\leq |\mathcal{R}(x)|  |u_\tau(\alpha x) - l_\tau(\alpha x)|\\
&\leq|R(z_i)| \left( |u_{\tau}(\alpha z_i) - u_{\tau}(\bar{\alpha} z_i)| +   |u_\tau(\bar{\alpha} x) - l_\tau(\bar{\alpha} x)|+|l_{\tau}(\alpha z_i) - l_{\tau}(\bar{\alpha} z_i)| \right) \\
&\leq |R(z_i)| \left( \frac{2}{\tau} |\alpha - \bar{\alpha}| |z_i| + |u_\tau(\bar{\alpha} x) - l_\tau(\bar{\alpha} x)| \right) \\
&\leq \frac{2}{\tau} |z_i|^2 |\alpha - \bar{\alpha}| + |z_i| |u_{\tau}(\bar{\alpha} z_i) - l_{\tau}(\bar{\alpha} z_i)|, \\
\end{aligned}
\end{equation}
we have
\begin{equation}
\begin{aligned}
\frac{1}{N} \sum_{i=1}^{N} U_{\tau}(z_i, \alpha) - \frac{1}{N} \sum_{i=1}^{N} L_{\tau}(z_i, \alpha) &\leq \frac{2}{\tau} |\alpha - \bar{\alpha}| \frac{1}{N} \sum_{i=1}^{N} |z_i|^2 + \frac{1}{N} \sum_{i=1}^{N} |z_i| |u_{\tau}(\bar{\alpha} z_i) - l_{\tau}(\bar{\alpha} z_i)| \\
&\leq \frac{2}{\tau} |\alpha - \bar{\alpha}| \frac{1}{N} \sum_{i=1}^{N} |z_i|^2 + \left( \frac{1}{N} \sum_{i=1}^{N} |z_i|^2 \right)^{\frac{1}{2}} \left( \frac{1}{N} \sum_{i=1}^{N} (u_{\tau}(\bar{\alpha} z_i) - l_{\tau}(\bar{\alpha} z_i))^2 \right)^{\frac{1}{2}}. \\
\end{aligned}
\end{equation}
Therefore, we have
\begin{equation}
\begin{aligned}
&\mathbb{P}\left( \frac{1}{N} \sum_{i=1}^{N} U_{\tau}(z_i, \alpha) - \frac{1}{N} \sum_{i=1}^{N} L_{\tau}(z_i, \alpha) > \frac{\delta}{12} \right) \\
&\leq \mathbb{P}\left( \frac{2}{\tau} |\alpha - \bar{\alpha}| \frac{1}{N} \sum_{i=1}^{N} |z_i|^2 > \frac{\delta}{24} \right) \\
&\quad + \mathbb{P}\left( \left( \frac{1}{N} \sum_{i=1}^{N} |z_i|^2 \right)^{\frac{1}{2}} \left( \frac{1}{N} \sum_{i=1}^{N} (u_{\tau}(\bar{\alpha} z_i) - l_{\tau}(\bar{\alpha} z_i))^2 \right)^{\frac{1}{2}} > \frac{\delta}{24} \right) \\
&\leq \mathbb{P}\left( \frac{1}{N} \sum_{i=1}^{N} |z_i|^2 > 2 \right) + \mathbb{P}\left( \frac{2}{\tau} |\alpha - \bar{\alpha}| > \frac{\delta}{48} \right) \\
&\quad + \mathbb{P}\left( \frac{1}{N} \sum_{i=1}^{N} |z_i|^2 > 2 \right) + \mathbb{P}\left( \frac{1}{N} \sum_{i=1}^{N} (u_{\tau}(\bar{\alpha} z_i) - l_{\tau}(\bar{\alpha} z_i))^2 > \frac{\delta^2}{1152} \right).
\end{aligned}
\end{equation}
By Proposition \ref{pro-exp-bound-lambda1}, we have 
\[
\mathbb{P}\left(  \frac{1}{N} \sum_{i=1}^{N} |z_i|^2 > 2  \right) \leq 2e^{-\frac{N}{32}}.
\]
By Lemma \ref{le-3} and Lemma \ref{lemm11}, we have 
\[
\mathbb{P}\left( |\alpha - \bar{\alpha}| \geq \eta_1 \right) \leq 
\P\!\left(\,\left|\alpha^{2} - \bar{\alpha}^{2}\right| > \bar{\alpha}\eta_1\,\right) \leq 8 \exp\left(- K \min\left\{\frac{\hat{\delta}(\bar{\alpha}\eta_1)^2}{C_1'}, \frac{\hat{\delta}(\bar{\alpha}\eta_1)}{C_2'}\right\}\right)+\dfrac{8M_1^2}{K \hat{\delta}(\bar{\alpha}\eta_1)^2},
\]
where $\eta_1=\frac{\delta\tau}{96}$ and $ C_1', C_1', C, \hat{\delta}(\cdot)$ are defined in Lemma \ref{lemm11}. 
By Markov inequality, we have 
$$\mathbb{P}\left( \frac{1}{N} \sum_{i=1}^{N} (u_{\tau}(\bar{\alpha} z_i) - l_{\tau}(\bar{\alpha} z_i))^2 > \frac{\delta^2}{1152} \right)\leq \frac{1152\mathbb{E}\left[u_{\tau}(\bar{\alpha} Z) - l_{\tau}(\bar{\alpha} Z))^2\right]}{N\delta^2}\leq \frac{4608M_0^2}{N\delta^2}.$$
Consequently, we have
\begin{equation}
\mathbb{P}\left(T_1 > \frac{\delta}{12} \right) \leq 4e^{-\frac{N}{32}}+8 \exp\left(- K \min\left\{\frac{\hat{\delta}(\bar{\alpha}\eta_1)^2}{C_1'}, \frac{\hat{\delta}(\bar{\alpha}\eta_1)}{C_2'}\right\}\right)+\dfrac{8M_1^2}{K \hat{\delta}(\bar{\alpha}\eta_1)^2}+\frac{4608M_0^2}{N\delta^2}.
\end{equation}

\textbf{Step 2}. Estimate  term \(T_2\). 
Observe that
\[
T_2 \leq \left| \frac{1}{N} \sum L_\tau(z_i, \alpha) - \frac{1}{N} \sum L_\tau(z_i, \bar{\alpha}) \right| + \left| \frac{1}{N} \sum L_\tau(z_i, \bar{\alpha}) - \mathbb{E}[L_\tau(Z, \bar{\alpha})] \right|:=A+B.\]
The function \(L_\tau(x, \alpha)\) is \(\frac{|x|^2}{\tau}\)-Lipschitz in \(\alpha\). Thus we have 
\[
|L_\tau(z_i, \alpha) - L_\tau(z_i, \bar{\alpha})| \leq \frac{1}{\tau} |z_i|^2 |\alpha - \bar{\alpha}|.
\]
So we have
\[
A \leq \frac{|\alpha - \bar{\alpha}|}{\tau} \cdot \frac{1}{N} \sum_{i=1}^N |z_i|^2.
\]
Let \(B_N = \frac{1}{N} \sum |z_i|^2\). Then we obtain 
\[
\mathbb{P}\left( A \geq \frac{\delta}{24} \right) \leq \mathbb{P}\left( |\alpha - \bar{\alpha}| \cdot B_N \geq \frac{\tau \delta}{24} \right) \leq \mathbb{P}\left( |\alpha - \bar{\alpha}| \geq \frac{\tau \delta}{48} \right) + \mathbb{P}\left( B_N \geq 2 \right).
\]
Choosing \(\eta_2 = \tau \delta / 48\),  we have by Lemma \ref{le-3} and Lemma \ref{lemm11},
\[
\mathbb{P}\left( |\alpha - \bar{\alpha}| \geq \eta_2 \right) \leq \P\!\left(\,\left|\alpha^{2} - \bar{\alpha}^{2}\right| > \bar{\alpha}\eta_2\,\right) \leq 8 \exp\left(- K \min\left\{\frac{\hat{\delta}(\bar{\alpha}\eta_2)^2}{C_1'}, \frac{\hat{\delta}(\bar{\alpha}\eta_2)}{C_2'}\right\}\right)+\dfrac{8M_1^2}{K \hat{\delta}(\bar{\alpha}\eta_2)^2},
\]
where  $ C_1', C_1', C, \hat{\delta}(\cdot)$ are defined in Lemma \ref{lemm11}. 
For \(B_N\), by Proposition \ref{pro-exp-bound-lambda1}, we have 
\[
\mathbb{P}\left( B_N \geq 2 \right) \leq 2e^{-\frac{N}{32}}.
\]
 For each \(z_i\), we have \(|L_\tau(z_i, \bar{\alpha}) - L_\tau(z_i', \bar{\alpha})| \leq \frac{1}{\tau} \|z_i - z_i'\|_2\) . Therefore let $g(\mathbf{z}):=\frac{1}{N} \sum_{i=1}^N L_\tau(z_i, \bar{\alpha})$, we have
    \[
    |g(\mathbf{z}) - g(\mathbf{z}')| \leq \frac{1}{N} \sum_{i=1}^N \frac{1}{\tau} \|z_i - z_i'\|_2 \leq \frac{1}{\tau \sqrt{N}} \|\mathbf{z} - \mathbf{z}'\|_2
    \]
    Hence, the Lipschitz constant of \(g\) is \(L_g = \frac{1}{\tau \sqrt{N}}\).
By Gaussian concentration in Proposition \ref{pro-lip-ex}, we have
\[
\mathbb{P}\left( \left| \frac{1}{N} \sum_{i=1}^N L_\tau(z_i, \bar{\alpha}) - \mathbb{E}[L_\tau(Z, \bar{\alpha})] \right| \geq \frac{\delta}{24} \right) \leq 2 \exp\left( -\frac{N \delta^2\tau^2}{1152 } \right).
\]
Consequently, we have
\begin{equation}
\mathbb{P}\left(T_2 > \frac{\delta}{12} \right) \leq 2e^{-\frac{N}{32}}+8 \exp\left(- K \min\left\{\frac{\hat{\delta}(\bar{\alpha}\eta_2)^2}{C_1'}, \frac{\hat{\delta}(\bar{\alpha}\eta_2)}{C_2'}\right\}\right)+\dfrac{8M_1^2}{K \hat{\delta}(\bar{\alpha}\eta_2)^2}+2 \exp\left( -\frac{N \delta^2\tau^2}{1152 } \right).
\end{equation}
Combining all terms, we have for $\tau$ satisfying (\ref{le6.10-1}), 
\begin{equation}
\begin{aligned}
&\mathbb{P}\left( \left| \frac{1}{N} \sum_{i=1}^{N} \mathcal{R}(z_i) g(\alpha z_i) - \mathbb{E}[\mathcal{R}(Z) g(\bar{\alpha} Z)] \right| \geq \frac{\delta}{4} \right) \\
&\leq \sum_{k=1}^{3} \mathbb{P}\left( T_k \geq \frac{\delta}{12} \right) \\
&\leq 6e^{-\frac{N}{32}} + 16 \exp\left(- K \min\left\{\frac{\hat{\delta}(\bar{\alpha}\eta_1)^2}{C_1'}, \frac{\hat{\delta}(\bar{\alpha}\eta_1)}{C_2'}\right\}\right)+\dfrac{16M_1^2}{K \hat{\delta}(\bar{\alpha}\eta_1)^2} + 2 \exp\left( -\frac{N \delta^2\tau^2}{1152 } \right) + \frac{4608M_0^2}{N\delta^2}.
\end{aligned}
\end{equation}
By Lemma \ref{le-6.11}, we have 
\begin{equation}\label{le6.10-2}\left| \mathbb{E}L_\tau(Z,\bar\alpha) - \mathbb{E}G(Z,\bar\alpha) \right| \leq \sqrt{2M_0^2\bar{\alpha}^{-1}K_g\tau}.
\end{equation}
Let $\tau=\frac{1}{\sqrt[4]{K}}$ and $\hat{K}_1=\left(\frac{12\sqrt{2M_0^2\bar{\alpha}^{-1}K_g}}{\delta}\right)^{8}$. Then  for $K>\max\{\hat{K}_1, \bar{K}\}$, we have from (\ref{le6.10-2}) that (\ref{le6.10-1}) holds and 
\begin{equation}
\begin{aligned}
&\mathbb{P}\left( \left| \frac{1}{N} \sum_{i=1}^{N} \mathcal{R}(z_i) g(\alpha z_i) - \mathbb{E}[\mathcal{R}(Z) g(\bar{\alpha} Z)] \right| \geq \frac{\delta}{4} \right) \\
&\leq 6e^{-\frac{\gamma K}{32}} + 16 \exp\left(- K \min\left\{\frac{\hat{\delta}(\bar{\alpha}\eta_1)^2}{C_1'}, \frac{\hat{\delta}(\bar{\alpha}\eta_1)}{C_2'}\right\}\right)+\dfrac{16M_1^2}{K \hat{\delta}(\bar{\alpha}\eta_1)^2} + 2 \exp\left( -\frac{\sqrt{K}\gamma \delta^2}{1152 } \right) + \frac{4608M_0^2}{\gamma K\delta^2}.
\end{aligned}
\end{equation}
The same bound holds for the imaginary part components. Applying the union bound to all four components (real and imaginary parts of both \(g\) and \(h\)) gives for $K>\max\{\hat{K}, \bar{K}\}$ with $\hat{K}=\left(\frac{12\sqrt{2M_0^2\bar{\alpha}^{-1}\max\{K_g, K_h\}}}{\delta}\right)^{8}$,  
\begin{equation}
\begin{aligned}
&\mathbb{P}\left( \left| \frac{\mathbf{z}^\mathsf{H} q(\alpha \mathbf{z})}{N} - \mathbb{E}\left[ Z^\dagger q(\bar{\alpha} Z) \right] \right| \geq \delta \right) \\
&\leq 24e^{-\frac{\gamma K}{32}} +  64 \exp\left(- K \min\left\{\frac{\hat{\delta}(\bar{\alpha}\eta_1)^2}{C_1'}, \frac{\hat{\delta}(\bar{\alpha}\eta_1)}{C_2'}\right\}\right)+\dfrac{64M_1^2}{K \hat{\delta}(\bar{\alpha}\eta_1)^2} + 8 \exp\left( -\frac{\sqrt{K}\gamma \delta^2}{1152 } \right)+\frac{18432M_0^2}{\gamma K\delta^2}.
\end{aligned}
\end{equation}

\end{proof}

Define \(k(w) = |q(w)|^2\), which is a real-valued, bounded and piecewise constant function with \(|k(w)| \leq M_0^2\). Its discontinuity set \(\mathcal{D}_k = \mathcal{D}_g \cup \mathcal{D}_h\) is also a union of finitely many lines and rays, where \(\mathcal{D}_g\) and \(\mathcal{D}_h\) are the discontinuity sets of \(g(w) = \mathcal{R}(q(w))\) and \(h(w) = \mathcal{I}(q(w))\), respectively. Therefore, we can define the envelope functions for \(k\):
\begin{equation}\label{l_k-eve}
l^k_\tau(x) = \inf_{y} \left\{ k(y) + \frac{|x - y|}{\tau} \right\}, \quad u^k_\tau(x) = \sup_{y} \left\{ k(y) - \frac{|x - y|}{\tau} \right\},
\end{equation}
where \(\tau > 0\). These envelope functions satisfy the corresponding properties like in (\ref{def-lg}) and (\ref{def-lh}).

\begin{lemma} 
Suppose Assumption \ref{Assu:-funct-q} holds. Then for any \(0 < \tau \leq \bar{\alpha}\) with a fixed \(\bar{\alpha} > 0\),
\[
\mathbb{E}\!\left[ \left| u^k_\tau(\bar{\alpha} Z) - l^k_\tau(\bar{\alpha} Z) \right| \right] \leq C_k \tau,
\]
where
\[
C_k = \frac{2M_0^2}{\bar{\alpha}} K_k,
\]
and
\[
K_k = \frac{2}{\sqrt{\pi}} N_\ell(k) + \left(1 + \frac{1}{\sqrt{\pi}}\right) N_r(k),
\]
with \(N_\ell(k)\) and \(N_r(k)\) being the numbers of lines and rays in \(\mathcal{D}_k\), respectively.
\end{lemma}

\begin{proof}
The proof follows the same structure as Lemma \ref{le-env}. Define
\[
\Delta^k_\tau(x) = |u^k_\tau(x) - l^k_\tau(x)|.
\]
Since \(k(w)\) is piecewise constant and bounded, we have \(0 \leq \Delta^k_\tau(x) \leq 2M_0^2\) for all \(x \in \mathbb{C}\). Moreover, if \(\operatorname{dist}(x, \mathcal{D}_k) > \tau\), then \(u^k_\tau(x) = l^k_\tau(x) = k(x)\), so \(\Delta^k_\tau(x) = 0\).
The expectation can be bounded as
\[
\mathbb{E}[\Delta^k_\tau(\bar{\alpha} Z)] = \int_{\mathbb{C}} \Delta^k_\tau(\bar{\alpha} z) \phi(z) \, dz \leq 2M_0^2 \int_{\{z : \operatorname{dist}(\bar{\alpha} z, \mathcal{D}_k) \leq \tau\}} \phi(z) \, dz,
\]
where \(\phi(z) = \pi^{-1} e^{-|z|^2}\) is the probability density function of \(Z\).
Let \(r = \tau / \bar{\alpha}\) and \(\Gamma_k = \mathcal{D}_k / \bar{\alpha}\). Then we have
\[
\mathbb{E}[\Delta^k_\tau(\bar{\alpha} Z)] \leq 2M_0^2 \int_{\{z : \operatorname{dist}(z, \Gamma_k) \leq r\}} \phi(z) \, dz,
\]
where the integral \(\int_{\{z : \operatorname{dist}(z, \Gamma_k) \leq r\}} \phi(z) \, dz\) is the Gaussian measure of the \(r\)-neighborhood of \(\Gamma_k\). Similarly to the proof of Lemma \ref{le-env}, for a single line, the Gaussian measure of its \(r\)-neighborhood is at most \(\frac{2}{\sqrt{\pi}} r\) and
for a single ray, the Gaussian measure of its \(r\)-neighborhood is at most \(\left(1 + \frac{1}{\sqrt{\pi}}\right) r\) for \(0 < r \leq 1\).
Consequently, by the additivity of measure, we have
\[
\int_{\{z : \operatorname{dist}(z, \Gamma_k) \leq r\}} \phi(z) \, dz \leq \frac{2}{\sqrt{\pi}} N_\ell(k) r + \left(1 + \frac{1}{\sqrt{\pi}}\right) N_r(k) r.
\]
Substituting \(r = \tau / \bar{\alpha}\), we obtain
\[
\mathbb{E}[\Delta^k_\tau(\bar{\alpha} Z)] \leq 2M_0^2 \left( \frac{2}{\sqrt{\pi}} N_\ell(k) + \left(1 + \frac{1}{\sqrt{\pi}}\right) N_r(k) \right) \frac{\tau}{\bar{\alpha}} = \frac{2M_0^2}{\bar{\alpha}} K_k \tau.
\]
This completes the proof.
\end{proof} 
\begin{remark}  
The constants \(N_\ell(k)\) and \(N_r(k)\) can be directly computed from the quantization scheme:
For independent I/Q quantization, \(\mathcal{D}_k\) consists of lines from both real and imaginary parts, so \(N_\ell(k) = N_\ell(g) + N_\ell(h)\) and \(N_r(k) = 0\).
 For constant envelope (CE) quantization, \(\mathcal{D}_k\) consists of rays from phase thresholds, so \(N_\ell(k) = 0\) and \(N_r(k) = N_r(g) = N_r(h)\).
\end{remark}
\begin{lemma}\label{lemm-q-norm2}
Under the  conditions in Lemma \ref{le-6.12}, for any \(\delta > 0\) and \(K > \max\{\bar{K}, \hat{K}_1\}\), where $\hat{K}_1=\left(\frac{6M_0^2\bar{\alpha}^{-1}K_k}{\delta}\right)^{4}$, we have
\begin{equation}
\begin{aligned}
&\mathbb{P}\left( \left| \frac{\| q(\alpha \mathbf{z}) \|^2}{N} - \mathbb{E}[|q(\bar{\alpha} Z)|^2] \right| 
> \delta \right)  \\
&\leq 4e^{-\frac{\gamma K}{32}} + 16 \exp\left(- K \min\left\{\frac{\hat{\delta}(\bar{\alpha}\eta_1)^2}{C_1'}, \frac{\hat{\delta}(\bar{\alpha}\eta_1)}{C_2'}\right\}\right)+\dfrac{16M_1^2}{K \hat{\delta}(\bar{\alpha}\eta_1)^2} + 2 \exp\left( -\frac{\sqrt{K}\gamma \delta^2}{1152 } \right) + \frac{144M_0^2}{\gamma K\delta^2}.
\end{aligned}
\end{equation}
where \(K_k = \frac{2}{\sqrt{\pi}} (N_\ell^g + N_\ell^h) + \left(1 + \frac{1}{\sqrt{\pi}}\right) (N_r^g + N_r^h)\) and
\[
C_1' = \max\left\{ \frac{c_{\max}^2}{2}, 32, 32 M_1^2 \right\}, \quad C_2' = \max\left\{ 8, 8 M_1 \right\}, \quad \eta_1 = \frac{\delta \tau}{96}, \quad \tau = \frac{1}{\sqrt[4]{K}},
\]
\[
\hat{\delta}(\bar{\alpha}\eta) = \min\left\{ \frac{\gamma \bar{\alpha} \eta}{2L}, \frac{1}{2} \right\}, \quad L = 4(1 + \sigma_s^2)(\mathbb{E}[f^2(d)] + 1) + 2\sigma_s^2(1 + \mathbb{E}[f^2(d)]).
\]
\end{lemma}

\begin{proof}
Define \(k(w) = |q(w)|^2\), which is a real-valued, bounded, piecewise constant function with \(|k(w)| \leq M_0^2\).  We can define the envelope functions $l^k_\tau(x)$ and $u^k_\tau(x)$ for \(k\) in (\ref{l_k-eve}).
Now consider
\[
\left| \frac{1}{N} \sum_{i=1}^N k(\alpha z_i) - \mathbb{E}[k(\bar{\alpha} Z)] \right| \leq T_1 + T_2 + T_3,
\]
where
\[
T_1 = \left| \frac{1}{N} \sum k(\alpha z_i) - \frac{1}{N} \sum l^k_\tau(\alpha z_i) \right|, \quad T_2 = \left| \frac{1}{N} \sum l^k_\tau(\alpha z_i) - \mathbb{E}[l^k_\tau(\bar{\alpha} Z)] \right|, \quad T_3 = \left| \mathbb{E}[l^k_\tau(\bar{\alpha} Z)] - \mathbb{E}[k(\bar{\alpha} Z)] \right|.
\]
For given $\delta>0$, choose $\tau = \tau(\delta) > 0$ such that \begin{equation}\label{le6.11-1}T_3 \leq \mathbb{E}[|l^k_\tau(\bar{\alpha} Z) - k(\bar{\alpha} Z)|] \leq \mathbb{E}[|u^k_\tau(\bar{\alpha} Z) - l^k_\tau(\bar{\alpha} Z)|] \leq  \delta/3.\end{equation}  We fix this  $\tau$
for the entire proof.   
 We divide the proof into two steps.

\text{Step 1}. Estimate  term \(T_1\). For \(T_1\), note that
\[
T_1 \leq \frac{1}{N} \sum |u^k_\tau(\alpha z_i) - l^k_\tau(\alpha z_i)|.
\]
Using the Lipschitz property, we have
\[
|u^k_\tau(\alpha z_i) - l^k_\tau(\alpha z_i)| \leq \frac{2 |z_i|}{\tau} |\alpha - \bar{\alpha}| + |u^k_\tau(\bar{\alpha} z_i) - l^k_\tau(\bar{\alpha} z_i)|.
\]
Thus
\[
T_1 \leq \frac{2 |\alpha - \bar{\alpha}|}{\tau} \cdot \frac{1}{N} \sum |z_i| + \frac{1}{N} \sum |u^k_\tau(\bar{\alpha} z_i) - l^k_\tau(\bar{\alpha} z_i)|.
\]
Therefore,
\[
\mathbb{P}(T_1 > \delta/3) \leq \mathbb{P}\left( \frac{2 |\alpha - \bar{\alpha}|}{\tau} \cdot \frac{1}{N} \sum |z_i| > \delta/6 \right) + \mathbb{P}\left( \frac{1}{N} \sum |u^k_\tau(\bar{\alpha} z_i) - l^k_\tau(\bar{\alpha} z_i)| > \delta/6 \right).
\]
For the first term, we have
\[
\mathbb{P}\left( \frac{2 |\alpha - \bar{\alpha}|}{\tau} \cdot \frac{1}{N} \sum |z_i| > \delta/6 \right) \leq \mathbb{P}\left( |\alpha - \bar{\alpha}| > \eta_1 \right) + \mathbb{P}\left( \frac{1}{N} \sum |z_i|^2 > 2 \right),
\]
where \(\eta_1 = \frac{\delta \tau}{24}\). By Proposition \ref{pro-exp-bound-lambda1}, \(\mathbb{P}\left( \frac{1}{N} \sum |z_i|^2 > 2 \right) \leq 2 e^{-N/32}\). By Lemma \ref{lemm11},
\[
\mathbb{P}\left( |\alpha - \bar{\alpha}| > \eta_1 \right) \leq 8 \exp\left( -K \min\left\{ \frac{\hat{\delta}(\bar{\alpha}\eta_1)^2}{C_1'}, \frac{\hat{\delta}(\bar{\alpha}\eta_1)}{C_2'} \right\} \right) + \frac{8 M_1^2}{K \hat{\delta}(\bar{\alpha}\eta_1)^2}.
\]
For the second term, since \(\mathbb{E}[|u^k_\tau(\bar{\alpha} Z) - l^k_\tau(\bar{\alpha} Z)|] \leq C_k \tau\) and the variance is bounded, by Chebyshev's inequality, we obtain
\[
\mathbb{P}\left( \frac{1}{N} \sum |u^k_\tau(\bar{\alpha} z_i) - l^k_\tau(\bar{\alpha} z_i)| > \delta/6 \right) \leq  \frac{144 M_0^4}{N \delta^2}.
\]
Thus, we have
\[
\mathbb{P}(T_1 > \delta/3) \leq 8 \exp\left( -K \min\left\{ \frac{\hat{\delta}(\bar{\alpha}\eta_1)^2}{C_1'}, \frac{\hat{\delta}(\bar{\alpha}\eta_1)}{C_2'} \right\} \right) + \frac{8M_1^2}{K \hat{\delta}(\bar{\alpha}\eta_1)^2} + 2 e^{-N/2} + \frac{144 M_0^4}{N \delta^2}.
\]

\textbf{Step 2}. Estimate  term \(T_2\). 
For \(T_2\), we have
\[
T_2 \leq \left| \frac{1}{N} \sum l^k_\tau(\alpha z_i) - \frac{1}{N} \sum l^k_\tau(\bar{\alpha} z_i) \right| + \left| \frac{1}{N} \sum l^k_\tau(\bar{\alpha} z_i) - \mathbb{E}[l^k_\tau(\bar{\alpha} Z)] \right| = A + B.
\]
For \(A\), using the Lipschitz property, we have
\[
A \leq \frac{1}{\tau} |\alpha - \bar{\alpha}| \cdot \frac{1}{N} \sum |z_i|^2.
\]
Thus,
\[
\mathbb{P}(A > \delta/6) \leq \mathbb{P}\left( |\alpha - \bar{\alpha}| > \eta_2 \right) + \mathbb{P}\left( \frac{1}{N} \sum |z_i|^2 > 2 \right),
\]
where \(\eta_2 = \frac{\delta \tau}{12}\). Similarly, bound \(\mathbb{P}\left( |\alpha - \bar{\alpha}| > \eta_2 \right)\) as above.
For \(B\), since \(l^k_\tau\) is Lipschitz, the function \(g(\mathbf{z}) = \frac{1}{N} \sum l^k_\tau(\bar{\alpha} z_i)\) has Lipschitz constant \(L_g = \frac{1}{\tau \sqrt{N}}\), so by Proposition \ref{pro-lip-ex}, we obtain
\[
\mathbb{P}(B > \delta/6) \leq 2 \exp\left( -\frac{N \delta^2 \tau^2}{1152} \right).
\]
Thus,
\[
\mathbb{P}(T_2 > \delta/3) \leq 8 \exp\left( -K \min\left\{ \frac{\hat{\delta}(\bar{\alpha}\eta_2)^2}{C_1'}, \frac{\hat{\delta}(\bar{\alpha}\eta_2)}{C_2'} \right\} \right) + \frac{8M_1^2}{K \hat{\delta}(\bar{\alpha}\eta_2)^2} + 2 e^{-N/32} + 2 \exp\left( -\frac{N \delta^2 \tau^2}{1152} \right).
\]
By Lemma \ref{le-env}, we have
\begin{equation}\label{le6.11-2}
\mathbb{E}[|u^k_\tau(\bar{\alpha} Z) - l^k_\tau(\bar{\alpha} Z)|] \leq C_k \tau, 
\end{equation}
where $C_k = \frac{2 M_0^2}{\bar{\alpha}} K_k
$
and \(K_k = \frac{2}{\sqrt{\pi}} (N_\ell^g + N_\ell^h) + \left(1 + \frac{1}{\sqrt{\pi}}\right) (N_r^g + N_r^h)\).
Let $\tau=\frac{1}{\sqrt[4]{K}}$ and $\hat{K}_1=\left(\frac{6M_0^2\bar{\alpha}^{-1}K_k}{\delta}\right)^{4}$, then  for $K>\max\{\hat{K}_1, \bar{K}\}$, we have from (\ref{le6.11-2}) that (\ref{le6.11-1}) holds. Combining all terms and noting that \(N = \gamma K\) and \(\tau = 1/\sqrt[4]{K}\), we obtain 
\begin{equation}
\begin{aligned}
&\mathbb{P}\left( \left| \frac{\| q(\alpha \mathbf{z}) \|^2}{N} - \mathbb{E}[|q(\bar{\alpha} Z)|^2] \right| 
> \delta \right)  \\
&\leq 4e^{-\frac{\gamma K}{32}} + 16 \exp\left(- K \min\left\{\frac{\hat{\delta}(\bar{\alpha}\eta_1)^2}{C_1'}, \frac{\hat{\delta}(\bar{\alpha}\eta_1)}{C_2'}\right\}\right)+\dfrac{16M_1^2}{K \hat{\delta}(\bar{\alpha}\eta_1)^2} + 2 \exp\left( -\frac{\sqrt{K}\gamma \delta^2}{1152 } \right) + \frac{144M_0^2}{\gamma K\delta^2}.
\end{aligned}
\end{equation}
\end{proof}

\begin{lemma}\label{exp-C1}Suppose Assumption \ref{Assu:H-n-s}-\ref{Assu:ratio-K-N} and  Assumption \ref{Assu:d} hold. 
Then for  any $\delta > 0$, $K>\max\{\hat{K}_2, \bar{K}\}$ with $\hat{K}_2=\left(\frac{12\sqrt{2M_0^2\bar{\alpha}^{-1}\max\{K_g, K_h\}}}{\delta_1}\right)^{8}$, we have
\begin{equation}
\begin{aligned}
&\mathbb{P}\left( \left| \left| C_1 \right| - \left| \overline{C}_1 \right| \right| > \delta \right) \\
& \leq  24e^{-\frac{\gamma K}{32}} +  64 \exp\left(- K \min\left\{\frac{\hat{\delta}(\bar{\alpha}\eta_1)^2}{C_1'}, \frac{\hat{\delta}(\bar{\alpha}\eta_1)}{C_2'}\right\}\right)+\dfrac{64 M_1^2}{K \hat{\delta}(\bar{\alpha}\eta_1)^2} + 8 \exp\left( -\frac{\sqrt{K}\gamma \delta_1^2}{1152 } \right)+\frac{18432M_0^2}{\gamma K\delta_1^2}\\
&\quad+8 \exp\left(- K \min\left\{\frac{\hat{\delta}(\bar{\alpha}\delta_3)^2}{C_1'}, \frac{\hat{\delta}(\bar{\alpha}\delta_3)}{C_2'}\right\}\right)+\dfrac{8M_1^2}{K \hat{\delta}(\bar{\alpha}\delta_3)^2}+2 \exp\left(-\frac{1}{2} \gamma K \min\left\{\delta_4^2/16, \delta_4/4\right\}\right),
\end{aligned}
\end{equation}
where $$  C_1' = \max\left\{\frac{c_{\max}^2}{2}, 32, 32M_1^2\right\}, \quad C_2' = \max\left\{ 8, 8M_1\right\}, \eta_1=\frac{\delta_1}{96\sqrt[4]{K}},$$$$
\delta_1=\tilde{\delta}(\mathbb{E}\left[Z^\dagger q(\alpha Z)\right],\frac{1}{\bar{\alpha}^2}, \delta), \delta_2=\tilde{\delta}(1,\frac{1}{\bar{\alpha}^2}, \delta_1), \delta_3=\frac{\delta_2\bar{\alpha}}{1+\delta_2\bar{\alpha}},\delta_4=\frac{\delta_2}{1+\delta_2},
$$$$
L = 4(1 + \sigma_s^2)(\E[f^2(d)] + 1) + 2\sigma_s^2(1 + \E[f^2(d)]),\quad
\hat{\delta}(\bar{\alpha}\delta_3) = \min\left\{\dfrac{\gamma\bar{\alpha}\delta_3}{2L}, \dfrac{1}{2}\right\}.$$
\end{lemma}
\begin{proof}
 By Lemma \ref{le-1}- Lemma \ref{le-4}, for any $\delta > 0$, we have
\begin{equation}
\begin{aligned}
&\mathbb{P}\left( \left| \left| C_1 \right| - \left| \overline{C}_1 \right| \right| > \delta \right) = \mathbb{P}\left( \left| \frac{z_1^\mathsf{H} q(\alpha z_1)}{\alpha \|z_1\|^2} - \frac{\mathbb{E}\left[Z^\dagger q(\bar{\alpha} Z)\right]}{\alpha} \right| > \delta \right) \\
&\leq \mathbb{P}\left( \left| \frac{z_1^\mathsf{H} q(\alpha z_1)}{N} - \mathbb{E}\left[Z^\dagger q(\alpha Z)\right] \right| > \delta_1 \right) + \mathbb{P}\left( \left| \frac{N}{\alpha \|z_1\|^2} - \frac{1}{\bar{\alpha}} \right| > \delta_1 \right) \\
&\leq \mathbb{P}\left( \left| \frac{z_1^\mathsf{H} q(\alpha z_1)}{N} - \mathbb{E}\left[Z^\dagger q(\alpha Z)\right] \right| > \delta_1 \right) + \mathbb{P}\left( \left| \frac{1}{\alpha} - \frac{1}{\bar{\alpha}} \right| > \delta_2 \right) + \mathbb{P}\left( \left| \frac{N}{\|z_1\|^2} - 1 \right| > \delta_2 \right) \\
&\leq \mathbb{P}\left( \left| \frac{z_1^\mathsf{H} q(\alpha z_1)}{N} - \mathbb{E}\left[Z^\dagger q(\alpha Z)\right] \right| > \delta_1 \right) + \mathbb{P}\left( \left| \alpha - \bar{\alpha} \right| > \delta_3 \right) + \mathbb{P}\left( \left| \frac{\|z_1\|^2}{N} - 1 \right| > \delta_4 \right)\\
&\leq \mathbb{P}\left( \left| \frac{z_1^\mathsf{H} q(\alpha z_1)}{N} - \mathbb{E}\left[Z^\dagger q(\alpha Z)\right] \right| > \delta_1 \right) + \mathbb{P}\left( \left| \alpha^2 - \bar{\alpha}^2 \right| > \bar{\alpha}\delta_3 \right) + \mathbb{P}\left( \left| \frac{\|z_1\|^2}{N} - 1 \right| > \delta_4 \right),\\
\end{aligned}
\end{equation}
where 
$$
\delta_1=\tilde{\delta}(\mathbb{E}\left[Z^\dagger q(\alpha Z)\right],\frac{1}{\bar{\alpha}^2}, \delta), \delta_2=\tilde{\delta}(1,\frac{1}{\bar{\alpha}^2}, \delta_1), \delta_3=\frac{\delta_2\bar{\alpha}}{1+\delta_2\bar{\alpha}},\delta_4=\frac{\delta_2}{1+\delta_2}.
$$
By Lemma \ref{le-6.12}, 
 for $K>\max\{\hat{K}, \bar{K}\}$ with $\hat{K}=\left(\frac{12\sqrt{2M_0^2\bar{\alpha}^{-1}\max\{K_g, K_h\}}}{\delta_1}\right)^{8}$,    we have
\begin{equation}
\begin{aligned}
&\mathbb{P}\left( \left| \frac{\mathbf{z}^\mathsf{H} q(\alpha \mathbf{z})}{N} - \mathbb{E}\left[ Z^\dagger q(\bar{\alpha} Z) \right] \right| \geq \delta_1 \right) \\
&\leq 24e^{-\frac{\gamma K}{32}} +  64 \exp\left(- K \min\left\{\frac{\hat{\delta}(\bar{\alpha}\eta_1)^2}{C_1'}, \frac{\hat{\delta}(\bar{\alpha}\eta_1)}{C_2'}\right\}\right)+\dfrac{64M_1^2}{K \hat{\delta}(\bar{\alpha}\eta_1)^2} + 8 \exp\left( -\frac{\sqrt{K}\gamma \delta_1^2}{1152 } \right)+\frac{18432M_0^2}{\gamma K\delta_1^2},
\end{aligned}
\end{equation}
where $ 
 C_1' = \max\left\{\frac{c_{\max}^2}{2}, 32, 32M_1^2\right\}, \quad C_2' = \max\left\{ 8, 8M_1\right\}, \eta_1=\frac{\delta_1}{96\sqrt[4]{K}}$.
Under Lemma \ref{lemm11},  for $K>\bar{K}$,   we have
\[
\P\!\left(\,\left|\alpha^{2} - \bar{\alpha}^{2}\right| > \bar{\alpha}\delta_3 \,\right) \leq 8 \exp\left(- K \min\left\{\frac{\hat{\delta}(\bar{\alpha}\delta_3 )^2}{C_1'}, \frac{\hat{\delta}(\bar{\alpha}\delta_3 )}{C_2'}\right\}\right)+\dfrac{8M_1^2}{K \hat{\delta}(\bar{\alpha}\delta_3 )^2},
\]
where $$
 C_1' = \max\left\{\frac{c_{\max}^2}{2}, 32, 32M_1^2\right\}, \quad C_2' = \max\left\{ 8, 8M_1\right\},$$  and $$
L = 4(1 + \sigma_s^2)(\E[f^2(d)] + 1) + 2\sigma_s^2(1 + \E[f^2(d)]),\quad
\hat{\delta}(\bar{\alpha}\delta_3 ) = \min\left\{\dfrac{\gamma\bar{\alpha}\delta_3 }{2L}, \dfrac{1}{2}\right\}.
$$
By 
Proposition \ref{pro-exp-bound-lambda1}, 
$$\mathbb{P}\left( \left| \frac{\|z_1\|^2}{N} - 1 \right| > \delta_4 \right)\leq 2 \exp\left(-\frac{1}{2} \gamma K \min\left(\delta_4^2/16, \delta_4/4\right)\right).$$
Then combining above, we obtain the conclusion.
\end{proof}
\begin{lemma}\label{exp-C2}Suppose   Assumption \ref{Assu:H-n-s}-\ref{Assu:ratio-K-N} and  Assumption \ref{Assu:d} hold. 
Then for  any $\delta > 0$, $K>\max\{\hat{K}_1, \hat{K}_2, \bar{K}, \frac{1}{\gamma\delta_5}\}$ with $\hat{K}_1=\left(\frac{12M_0^2\bar{\alpha}^{-1}K_k}{\delta_1}\right)^{4}$ and $\hat{K}_2=\left(\frac{12\sqrt{2M_0^2\bar{\alpha}^{-1}\max\{K_g, K_h\}}}{\delta_4}\right)^{8}$, we have  \begin{equation}
\begin{aligned}
&\P\left( \left| C_2^2 - \overline{C}_2^2 \right| > \delta \right)  \\
& \leq  52e^{-\frac{\gamma K}{32}} +  144 \exp\left(- K \min\left\{\frac{\hat{\delta}(\bar{\alpha}\tilde{\eta})^2}{C_1'}, \frac{\hat{\delta}(\bar{\alpha}\tilde{\eta})}{C_2'}\right\}\right)+\dfrac{144M_1^2}{K \hat{\delta}(\bar{\alpha}\tilde{\eta})^2} + 18 \exp\left( -\frac{\sqrt{K}\gamma \delta_6}{2304 } \right)+\frac{37440M_0^2}{\gamma K\delta_7^2}\\
&\quad+4 \exp\left(-\frac{1}{2} (\gamma K-1) \min\left(\delta_3^2/16, \delta_3/4, \delta_4^2/16, \delta_4/4\right)\right),
\end{aligned}
\end{equation}
where  $$K_k = \frac{2}{\sqrt{\pi}} (N_\ell^g + N_\ell^h) + \left(1 + \frac{1}{\sqrt{\pi}}\right) (N_r^g + N_r^h)),
C_1' = \max\left\{ \frac{c_{\max}^2}{2}, 32, 32 M_1^2 \right\}, \quad C_2' = \max\left\{8, 8 M_1 \right\}, $$$$ \eta_0 = \frac{\delta_1 \tau}{192}, \quad \tau = \frac{1}{\sqrt[4]{K}},
\hat{\delta}(\bar{\alpha}\tilde{\eta} = \min\left\{ \frac{\gamma \bar{\alpha} \tilde{\eta} }{2L}, \frac{1}{2} \right\}, \quad L = 4(1 + \sigma_s^2)(\mathbb{E}[f^2(d)] + 1) + 2\sigma_s^2(1 + \mathbb{E}[f^2(d)]).
$$
$$\delta_1=\tilde{\delta}(1,\overline{C}_2^2, \delta), \delta_2=\tilde{\delta}(1,|\mathbb{E}[Z^\dagger q(\bar{\alpha} Z)]|^2,
\delta_1), \delta_3=\frac{\delta_2}{1+\delta_2}, \delta_4=\min\left\{\sqrt{\frac{\delta_2}{2}}, \frac{\delta_2}{2|\mathbb{E}[Z^\dagger q(\bar{\alpha} Z)]|}\right\}, $$$$\delta_5=\frac{1}{2}(\sqrt{1+\delta_1}-1),\delta_6=\min\left\{\frac{\delta_1^2}{2}, 2\delta_4^2\right\}, \delta_7=\min\left\{\delta_1, \delta_4\right\}, \eta_1=\frac{\delta_4}{96\sqrt[4]{K}}, \tilde{\eta}=\min\{\eta_0, \eta_1\}.$$
\end{lemma}
\begin{proof}
Notice that by Lemma \ref{le-1}- Lemma \ref{le-4},
\begin{equation}
\begin{aligned}
& \P\left( \left| C_2^2 - \overline{C}_2^2 \right| > \delta \right) \\
&\quad= \P\left( \left| \frac{\left\| \mathbf{B}(z_1)^\mathsf{H} q\left( \frac{\| \hat{s}_1 \|}{\| z_1 \|} z_1 \right) \right\|^2}{\| z_2[2:N] \|^2} - \left( \mathbb{E}[|q(\bar{\alpha} Z)|^2] - |\mathbb{E}[Z^\dagger q(\bar{\alpha} Z)]|^2 \right) \right| > \delta \right) \\
&\quad\leq \P\left( \left| \frac{\left\| \mathbf{B}(z_1)^\mathsf{H} q\left( \alpha z_1 \right) \right\|^2}{N} - \overline{C}_2^2 \right| > \delta_1 \right) + \P\left( \left| \frac{\| z_2[2:N] \|^2}{N} - 1 \right| > \delta_1 \right) \\
&\quad\leq \P\left( \left| \frac{\left\| q\left( \alpha z_1 \right) \right\|^2}{N} - \mathbb{E}[|q(\bar{\alpha} Z)|^2] \right| > \frac{\delta_1}{2} \right) \\
&\quad\quad + \P\left( \left| \left|\frac{ z_1^\mathsf{H} q\left( \alpha z_1 \right) }{N}\right|^2 \cdot \frac{N}{\| z_1 \|^2} - |\mathbb{E}[Z^\dagger q(\bar{\alpha} Z)]|^2 \right| > \frac{\delta_1}{2} \right) \\
&\quad\quad + \P\left( \left| \frac{\| z_2[2:N] \|^2}{N} - 1 \right| > \delta_1 \right) \\
&\quad\leq \P\left( \left| \frac{\left\| q\left( \alpha z_1 \right) \right\|^2}{N} - \mathbb{E}[|q(\bar{\alpha} Z)|^2] \right| > \frac{\delta_1}{2} \right) + \P\left( \left| \frac{N}{\| z_1 \|^2} - 1 \right| > \delta_2 \right) \\
&\quad\quad + \P\left( \left| \left|\frac{ z_1^\mathsf{H} q\left( \alpha z_1 \right) }{N}\right|^2 - |\mathbb{E}[Z^\dagger q(\bar{\alpha} Z)]|^2 \right| > \delta_2 \right) + \P\left( \left| \frac{\| z_2[2:N] \|^2}{N} - 1 \right| > \delta_1 \right) \\
&\quad\leq \P\left( \left| \frac{\left\| q\left( \alpha z_1 \right) \right\|^2}{N} - \mathbb{E}[|q(\bar{\alpha} Z)|^2] \right| > \frac{\delta_1}{2} \right) + \P\left( \left| \frac{\| z_1 \|^2}{N} - 1 \right| > \delta_3 \right) \\
&\quad\quad + \P\left( \left| \left|\frac{ z_1^\mathsf{H} q\left( \alpha z_1 \right) }{N}\right|^2 - |\mathbb{E}[Z^\dagger q(\bar{\alpha} Z)]|^2 \right| > \delta_2 \right) + \P\left( \left| \frac{\| z_2[2:N] \|^2}{N} - 1 \right| > \delta_1 \right),
\end{aligned}
\end{equation}
where $$\delta_1=\tilde{\delta}(1,\overline{C}_2^2, \delta), \delta_2=\tilde{\delta}(1,|\mathbb{E}[Z^\dagger q(\bar{\alpha} Z)]|^2, \delta_1), \delta_3=\frac{\delta_2}{1+\delta_2}.$$
By Lemma \ref{lemm-q-norm2}, for  \( \delta_1 > 0 \) and \(K > \max\{\bar{K}, \hat{K}\}\), where $\hat{K}_1=\left(\frac{12M_0^2\bar{\alpha}^{-1}K_k}{\delta_1}\right)^{4}$, we have
\begin{equation}
\begin{aligned}
&\mathbb{P}\left( \left| \frac{\| q(\alpha \mathbf{z}) \|^2}{N} - \mathbb{E}[|q(\bar{\alpha} Z)|^2] \right| 
> \frac{\delta_1}{2} \right)  \\
&\leq 4e^{-\frac{\gamma K}{32}} + 16 \exp\left(- K \min\left\{\frac{\hat{\delta}(\bar{\alpha}\eta_1)^2}{C_1'}, \frac{\hat{\delta}(\bar{\alpha}\eta_0)}{C_2'}\right\}\right)+\dfrac{16M_1^2}{K \hat{\delta}(\bar{\alpha}\eta_0)^2} + 2 \exp\left( -\frac{\sqrt{K}\gamma \delta_1^2}{4608} \right) + \frac{576M_0^2}{\gamma K\delta_1^2},
\end{aligned}
\end{equation}
where \(K_k = \frac{2}{\sqrt{\pi}} (N_\ell^g + N_\ell^h) + \left(1 + \frac{1}{\sqrt{\pi}}\right) (N_r^g + N_r^h)\) and
\[
C_1' = \max\left\{ \frac{c_{\max}^2}{2}, 32, 32 M_1^2 \right\}, \quad C_2' = \max\left\{ 8, 8 M_1 \right\}, \quad \eta_0 = \frac{\delta_1 \tau}{192}, \quad \tau = \frac{1}{\sqrt[4]{K}},
\]
\[
\hat{\delta}(\bar{\alpha}\eta_0) = \min\left\{ \frac{\gamma \bar{\alpha} \eta_0}{2L}, \frac{1}{2} \right\}, \quad L = 4(1 + \sigma_s^2)(\mathbb{E}[f^2(d)] + 1) + 2\sigma_s^2(1 + \mathbb{E}[f^2(d)]).
\]
By Proposition \ref{pro-exp-bound-lambda1}, we have
$$\mathbb{P}\left( \left| \frac{\|z_1\|^2}{N} - 1 \right| > \delta_3 \right)\leq 2 \exp\left(-\frac{1}{2} \gamma K \min\left(\delta_3^2/16, \delta_3/4\right)\right).$$
Under Lemma \ref{le-6.12}, 
 for $K>\max\{\hat{K}_2, \bar{K}\}$ with $\hat{K}_2=\left(\frac{12\sqrt{2M_0^2\bar{\alpha}^{-1}\max\{K_g, K_h\}}}{\delta_4}\right)^{8}$,    we have
\begin{equation}
\begin{aligned}
&\P\left( \left| \left|\frac{ z_1^\mathsf{H} q\left( \alpha z_1 \right) }{N}\right|^2 - |\mathbb{E}[Z^\dagger q(\bar{\alpha} Z)]|^2 \right| > \delta_2 \right) \\
&\leq  \P\left( \left| \left|\frac{ z_1^\mathsf{H} q\left( \alpha z_1 \right) }{N}\right| - |\mathbb{E}[Z^\dagger q(\bar{\alpha} Z)]| \right| > \sqrt{\frac{\delta_2}{2}}\right)+\P\left( \left| \left|\frac{ z_1^\mathsf{H} q\left( \alpha z_1 \right) }{N}\right| - |\mathbb{E}[Z^\dagger q(\bar{\alpha} Z)]| \right| > \frac{\delta_2}{2|\mathbb{E}[Z^\dagger q(\bar{\alpha} Z)]|} \right)\\
&\leq  
 48e^{-\frac{\gamma K}{32}} +  128 \exp\left(- K \min\left\{\frac{\hat{\delta}(\bar{\alpha}\eta_1)^2}{C_1'}, \frac{\hat{\delta}(\bar{\alpha}\eta_1)}{C_2'}\right\}\right)+\dfrac{128M_1^2}{K \hat{\delta}(\bar{\alpha}\eta_1)^2} + 16 \exp\left( -\frac{\sqrt{K}\gamma \delta_4^2}{1152 } \right)+\frac{36864M_0^2}{\gamma K\delta_4^2},
\end{aligned}
\end{equation}
where $ 
 C_1' = \max\left\{\frac{c_{\max}^2}{2}, 32, 32M_1^2\right\}, \quad C_2' = \max\left\{ 8, 8M_1\right\}, \eta_1=\frac{\delta_4}{96\sqrt[4]{K}}$,$ \delta_4=\min\left\{\sqrt{\frac{\delta_2}{2}}, \frac{\delta_2}{2|\mathbb{E}[Z^\dagger q(\bar{\alpha} Z)]|}\right\}.$
By  Proposition \ref{pro-exp-bound-lambda1}, we have
\begin{equation}
\begin{aligned}
&\P\left( \left| \frac{\| z_2[2:N] \|^2}{N} - 1 \right| > \delta_1 \right)\\
&=\P\left( \left| \frac{\| z_2[2:N] \|^2}{N-1} \frac{N-1}{N}- 1 \right| > \delta_1 \right)\\
&\leq\P\left( \left| \frac{\| z_2[2:N] \|^2}{N-1} \right| > \delta_5 \right)+\P\left( \left| \frac{N-1}{N}- 1 \right| > \delta_5 \right)\\
&\leq 2 \exp\left(-\frac{1}{2} (\gamma K-1) \min\left(\delta_5^2/16, \delta_5/4\right)\right).\\
\end{aligned}
\end{equation}
for $N>\frac{1}{\delta_5}$, where $\delta_5=\frac{1}{2}(\sqrt{1+\delta_1}-1).$ Consequently, combining all the above inequalites, we obtain the conclusion
\end{proof}


\subsection{Proof of Lemma~\ref{Lem:T_g-convg-barT_g}}
\label{sec:Lem:T_g-convg-barT_g}

We begin by giving a complete version of Lemma~\ref{Lem:T_g-convg-barT_g} with details about 
$\mathfrak{R}(\epsilon,K)$.

\begin{lemma}[A complete version of Lemma~\ref{Lem:T_g-convg-barT_g}]
\label{Lem:T_g-convg-barT_g-proof}
Let  Assumptions \ref{Assu:H-n-s}-\ref{Assu:ratio-K-N} hold.
For any small positive number $\epsilon>0$, 
there exists $\hat{K}>0$ depending on $\epsilon$  
such that   
\begin{equation}
\mathbb{P}\left(\left| T_g-\bar{T}_g\right|\geq \epsilon\right)\leq \mathfrak{R}(\epsilon,K),
\end{equation}
for all $K>\hat{K}$, where $\mathfrak{R}(\epsilon,K)\to 0$ as $K\to\infty$ and
\begin{equation}
\mathfrak{R}(\epsilon,K) =
1594\exp(-K\tilde{\epsilon}_1) +  224\exp\left( -\sqrt{K}\tilde{\epsilon}_2\right)+\frac{17\tilde{\epsilon}_3}{K}+20 \exp\left(-\frac{1}{2} (\gamma K-1) \delta\right),
\end{equation}
where 
\bgeq 
\hat{K}=\max\left\{\hat{K}_1, \hat{K}_2, \hat{K}_3, \hat{K}_4,\hat{K}_5, \hat{K}_6, \hat{K}_7, \hat{K}_8,
\bar{K}, \frac{1}{\gamma\delta_4},\frac{1}{\gamma\delta_8},\frac{1}{\gamma\delta_{19}}\right\}
\edeq 
with 
\bgeq \hat{K}_1=\left(\frac{12\sqrt{2M_0^2\bar{\alpha}^{-1}\max\{K_g, K_h\}}}{\delta_1}\right)^{8}, 
\hat{K}_2=\left(\frac{12M_0^2\bar{\alpha}^{-1}K_k}{\delta_5}\right)^{4}, \hat{K}_3=\left(\frac{12\sqrt{2M_0^2\bar{\alpha}^{-1}\max\{K_g, K_h\}}}{\delta_8}\right)^{8},
\edeq 
\bgeq   \hat{K}_4=\left(\frac{12\sqrt{2M_0^2\bar{\alpha}^{-1}\max\{K_g, K_h\}}}{\delta_6}\right)^{8}, \hat{K}_5=\left(\frac{12M_0^2\bar{\alpha}^{-1}K_k}{\delta_{16}}\right)^{4}, \hat{K}_6=\left(\frac{12\sqrt{2M_0^2\bar{\alpha}^{-1}\max\{K_g, K_h\}}}{\delta_{19}}\right)^{8},
\edeq 
\bgeq 
\hat{K}_7=\left(\frac{12M_0^2\bar{\alpha}^{-1}K_k}{\delta_{23}}\right)^{4},\hat{K}_8=\left(\frac{12\sqrt{2M_0^2\bar{\alpha}^{-1}\max\{K_g, K_h\}}}{\delta_{26}}\right)^{8}, \quad 
\edeq 
\begin{align*} \tilde{\epsilon}_1 = &\min\left\{\frac{\epsilon_2^2}{128M_1^2}, \frac{\epsilon_3}{16M_1}, \frac{\epsilon_7^2}{128M_1^2},\frac{\epsilon_7}{16M},\frac{\epsilon_{11}^2}{2M^2},\frac{\gamma}{2},\frac{\hat{\delta}(\bar{\alpha}\eta_1)^2}{C_1'},\frac{\hat{\delta}(\bar{\alpha}\eta_1)}{C_2'},\frac{\hat{\delta}(\bar{\alpha}\delta_3)^2}{C_1'},\frac{\hat{\delta}(\bar{\alpha}\delta_3)}{C_2'},\frac{\gamma\delta_4^2}{2}, \frac{\gamma\delta_4}{2}, \frac{2(\sigma_s\epsilon_9)^2}{c_{max}^2}, \frac{2(\epsilon_5)^2}{c_{max}^2},\right.\\&\left.\frac{\epsilon_6}{2}, \frac{\epsilon_6^2}{2},\frac{\epsilon_{10}}{2}, \frac{\epsilon_{10}^2}{2},\frac{\hat{\delta}(\bar{\alpha}\eta_5)^2}{C_1'},\frac{\hat{\delta}(\bar{\alpha}\eta_5)}{C_2'}, \frac{\epsilon_{27}^2}{128M_1^2}, \frac{\epsilon_{27}}{16M_1}, \frac{\epsilon_{28}^2}{2M_1^2},
\frac{\gamma}{2},
\frac{\hat{\delta}(\bar{\alpha}\eta_6)^2}{C_1'}, \frac{\hat{\delta}(\bar{\alpha}\eta_6)}{C_2'}, \frac{\hat{\delta}(\bar{\alpha}\delta_{14})^2}{C_1'}, \frac{\hat{\delta}(\bar{\alpha}\delta_{14})}{C_2'}, \frac{(\epsilon_{30})^2}{2}, \frac{\epsilon_{30}}{2}, \right.\\
 &\left.
\frac{2(\epsilon_{31})^2}{c_{max}^2}, \frac{\hat{\delta}(\bar{\alpha}\eta_9)^2}{C_1'},\frac{\hat{\delta}(\bar{\alpha}\eta_9)}{C_2'}, \frac{\epsilon_{34}^2}{2}, \frac{\epsilon_{34}}{2},\frac{\hat{\delta}(\bar{\alpha}\eta_{12})^2}{C_1'}, \frac{\hat{\delta}(\bar{\alpha}\eta_{12})}{C_2'}, \epsilon_{35}^2\right\},
\\
\tilde{\epsilon}_2 =&
 \min\left\{\frac{\gamma\delta_1^2}{1152}, \frac{\gamma\delta_{12}^2}{1152}, \frac{\gamma\delta_{10}}{2304}, \frac{\gamma\delta_{21}}{2304}, \frac{\gamma\delta_{28}}{2304}\right\}.
\\
\tilde{\epsilon}_3 =& \max\left\{\frac{16M_1^2}{\min\{\epsilon_3^2,\epsilon_7^2\}}, \frac{192M_1^2}{\hat{\delta}(\bar{\alpha}\eta_1)^2}, \frac{55296M_1^2}{\gamma\delta_1^2},\frac{24M_1^2}{\hat{\delta}(\bar{\alpha}\delta_3)^2}, \frac{144M_1^2}{\hat{\delta}(\bar{\alpha}\eta_5)^2},\frac{37440M_0^2}{\gamma\delta_{11}},\frac{24M_1^2}{\epsilon_{27}^2}, \frac{192M_1^2}{\hat{\delta}(\bar{\alpha}\eta_6)^2}, \frac{55296M_0^2}{\gamma\delta_{12}^2},\frac{24M_1^2}{\hat{\delta}(\bar{\alpha}\delta_{14})^2},\right.\\&\left.  \frac{288M_1^2}{\hat{\delta}(\bar{\alpha}\eta_{9})^2},\frac{74880M_0^2}{\gamma\delta_{22}^2}, \frac{144M_1^2}{\hat{\delta}(\bar{\alpha}\eta_{12})^2}, \frac{37440M_0^2}{\gamma\delta_{29}^2}\right\},\\
\delta=&\min\left\{\delta_4^2/16, \delta_4/4, \delta_7^2/16, \delta_7/4,\delta_8^2/16, \delta_8/4,\delta_{15}^2/16, \delta_{15}/4,\delta_{18}^2/16, \delta_{18}/4,\delta_{19}^2/16, \delta_{19}/4,\delta_{25}^2/16, \delta_{25}/4,\delta_{26}^2/16, \delta_{26}/4
\right\}
\end{align*} with function $\tilde{\delta}(\cdot)$ being defined by $\tilde{\delta}(\overline{X},\overline{Y},\epsilon) =1/4(\sqrt{(|\overline{X}|+|\overline{Y}|)^2+4\epsilon}-|\overline{X}|-|\overline{Y}|)$ and 
\begin{align*}
\epsilon_1 &= \tilde{\delta}\left(0, 1, \frac{\epsilon}{2}\right), 
& \epsilon_2 &= \tilde{\delta}\left(1, \bar{T}_g^2, \frac{\epsilon\bar{T}_g}{2}\right), 
&  \epsilon_3 &=\tilde{\delta}\left(|\bar{C}_{1}|^{2}, \mathbb{E}[d^{2}f(d)^{2}] , \frac{\epsilon_2}{6}\right), \\
\epsilon_4 &= \tilde{\delta}(|\bar{C}_{1}|^2, \sigma_s^2, \epsilon_3), 
& \epsilon_5 &= \tilde{\delta}(1, \sigma_s^2, \epsilon_4), 
& \epsilon_6 &= \frac{\epsilon_5}{1+\epsilon_5}, \\
\epsilon_7 &= \tilde{\delta}\left(|\bar{C}_{2}|^2, \mathbb{E}[d^{2}], \frac{\epsilon_2}{6}\right), 
& \epsilon_8 &= \tilde{\delta}\left(0, 2, \frac{\epsilon_2}{6}\right), 
& \epsilon_9 &= \tilde{\delta}\left(1, \sigma_s, \frac{\epsilon_8}{2}\right), \\
\epsilon_{10} &= \frac{\epsilon_9}{1+\epsilon_9}, 
& \epsilon_{11} &= \tilde{\delta}\left(0, |\bar{C}_1|\bar{C}_2, \frac{\epsilon_8}{4}\right), 
& \epsilon_{12} &= \tilde{\delta}\left(|\bar{C}_1|, \bar{C}_2, \frac{\epsilon_{11}}{4}\right), \\
\epsilon_{13} &= \tilde{\delta}\left(|\bar{C}_{1}|^2, \sigma_s^{2}\mathbb{E}[d f(d)]^{2}, \frac{\epsilon_2}{6}\right), 
& \epsilon_{14} &= \tilde{\delta}\left(\sigma_s,\mathbb{E}[d f(d)], \epsilon_{13}\right), 
& \epsilon_{15} &= \tilde{\delta}(1, \sigma_s^2, \epsilon_{14}), \\
\epsilon_{16} &= \frac{\epsilon_{15}}{1+\epsilon_{15}}, 
& \epsilon_{17} &= \tilde{\delta}\left(0, \bar{C}_{2}, \frac{\epsilon_2}{6}\right), 
& \epsilon_{18} &= \tilde{\delta}(1, 0, \epsilon_{17}), \\
\epsilon_{19} &= \frac{\epsilon_{18}}{1+\epsilon_{18}}, 
& \epsilon_{20} &= \sqrt{\frac{\epsilon_2}{6}}, 
& \epsilon_{21} &= \tilde{\delta}(\sigma_s, 0, \epsilon_{20}), \\
\epsilon_{22} &= \tilde{\delta}(1, 0, \epsilon_{21}), 
& \epsilon_{23} &= \frac{\epsilon_{22}}{1+\epsilon_{22}}, 
& \epsilon_{24} &= \tilde{\delta}(|\bar{C}_1|\bar{C}_2, 0, \epsilon_{20}), \\
\epsilon_{25} &= \tilde{\delta}(|\bar{C}_1|, \bar{C}_2, \epsilon_{24}), 
& \epsilon_{26} &= \tilde{\delta}(\mathbb{E}[d f(d)], 0, \epsilon_{25}),\\
\epsilon_{27} &= \min\left\{ \sqrt{\frac{\epsilon_{14}}{2}}, \frac{\epsilon_{14}}{4}, \epsilon_{26} \right\}, 
& \epsilon_{28} &= \min\left\{ \epsilon_{26}, \epsilon_{18} \right\}, 
&\epsilon_{29} &= \min\left\{ \sqrt{\frac{\epsilon_{13}}{2}}, \frac{\epsilon_{13}}{4}, \epsilon_{25} \right\}, \\
 \epsilon_{30} &= \min\left\{ \sqrt{\frac{\epsilon_{16}}{2}}, \frac{\epsilon_{16}}{4}, \epsilon_{19}, \sqrt{\frac{\epsilon_{23}}{3}}, \frac{\epsilon_{23}}{3}, \frac{\epsilon_{23}}{9} \right\}, 
&\epsilon_{31} &= \min\left\{ \epsilon_{15}, \sigma_s \epsilon_{21} \right\}, 
& \epsilon_{32} &= \min\left\{ \sqrt{\frac{\epsilon_{17}}{2}}, \sigma_s \epsilon_{19}, \frac{\epsilon_{17}}{4}, \epsilon_{25} \right\}\\
\epsilon_{33}&=\frac{\min\{\epsilon_1, \epsilon_2\}}{1+\min\{\epsilon_1, \epsilon_2\}}, &\epsilon_{34}&=\tilde{\delta}\left(1,1,\epsilon_{33}\right), &\epsilon_{35}&=\tilde{\delta}(\bar{C}_2, 0, \epsilon_1).
\end{align*}
and\begin{align*}
\delta_1 &= \tilde{\delta}\left(\mathbb{E}\left[Z^\dagger q(\alpha Z)\right], \frac{1}{\bar{\alpha}^2}, \eta_2\right), &
\delta_2 &= \tilde{\delta}\left(1, \frac{1}{\bar{\alpha}^2}, \delta_1\right), &
\delta_3 &= \frac{\delta_2\bar{\alpha}}{1+\delta_2\bar{\alpha}}, \\
\delta_4 &= \frac{\delta_2}{1+\delta_2}, &
\delta_5 &= \tilde{\delta}(1, \overline{C}_2^2, \delta), &
\delta_6 &= \tilde{\delta}(1, |\mathbb{E}[Z^\dagger q(\bar{\alpha} Z)]|^2, \delta_5), \\
\delta_7 &= \frac{\delta_6}{1+\delta_6}, &
\delta_8 &= \min\left\{\sqrt{\frac{\delta_6}{2}}, \frac{\delta_6}{2|\mathbb{E}[Z^\dagger q(\bar{\alpha} Z)]|}\right\}, &
\delta_9 &= \frac{1}{2}(\sqrt{1+\delta_5}-1), \\
\delta_{10} &= \min\left\{\frac{\delta_5^2}{2}, 2\delta_8^2\right\}, &
\delta_{11} &= \min\left\{\delta_5, \delta_8\right\}, &
\delta_{12} &= \tilde{\delta}\left(\mathbb{E}\left[Z^\dagger q(\alpha Z)\right], \frac{1}{\bar{\alpha}^2}, \epsilon_{29}\right), \\
\delta_{13} &= \tilde{\delta}\left(1, \frac{1}{\bar{\alpha}^2}, \delta_{12}\right), &
\delta_{14} &= \frac{\delta_{13}\bar{\alpha}}{1+\delta_{13}\bar{\alpha}}, &
\delta_{15} &= \frac{\delta_{13}}{1+\delta_{13}}, \\
\delta_{16} &= \tilde{\delta}(1, \overline{C}_2^2, \epsilon_{32}), &
\delta_{17} &= \tilde{\delta}(1, |\mathbb{E}[Z^\dagger q(\bar{\alpha} Z)]|^2, \delta_{16}), &
\delta_{18} &= \frac{\delta_{17}}{1+\delta_{17}}, \\
\delta_{19} &= \min\left\{\sqrt{\frac{\delta_{17}}{2}}, \frac{\delta_{17}}{2|\mathbb{E}[Z^\dagger q(\bar{\alpha} Z)]|}\right\}, &
\delta_{20} &= \frac{1}{2}(\sqrt{1+\delta_{16}}-1), &
\delta_{21} &= \min\left\{\frac{\delta_{16}^2}{2}, 2\delta_{19}^2\right\}, \\
\delta_{22} &= \min\left\{\delta_{16}, \delta_{19}\right\}, &
\delta_{23} &= \tilde{\delta}(1, \overline{C}_2^2, \overline{C}_2\epsilon_{35}), &
\delta_{24} &= \tilde{\delta}(1, |\mathbb{E}[Z^\dagger q(\bar{\alpha} Z)]|^2, \delta_{23}), \\
\delta_{25} &= \frac{\delta_{24}}{1+\delta_{24}}, &
\delta_{26} &= \min\left\{\sqrt{\frac{\delta_{24}}{2}}, \frac{\delta_{24}}{2|\mathbb{E}[Z^\dagger q(\bar{\alpha} Z)]|}\right\}, &
\delta_{27} &= \frac{1}{2}(\sqrt{1+\delta_{23}}-1), \\
\delta_{28} &= \min\left\{\frac{\delta_{23}^2}{2}, 2\delta_{26}^2\right\}, &
\delta_{29} &= \min\left\{\delta_{23}, \delta_{26}\right\}
\end{align*}
and
\begin{align*}
\eta_1 &= \frac{\delta_1}{96\sqrt[4]{K}}, 
&\eta_2 &= \min\left\{\epsilon_{12}, \sqrt{\frac{\epsilon_4}{2}}, \frac{\epsilon_4}{4}\right\}, 
&\eta_3 &= \frac{\delta_5 \tau}{192}, 
&\eta_4 &= \frac{\delta_8}{96\sqrt[4]{K}}, \\
\eta_5 &= \min\{\eta_3, \eta_4\}, 
&\eta_6 &= \frac{\delta_{12}}{96\sqrt[4]{K}}, 
&\eta_7 &= \frac{\delta_{16} \tau}{192}, 
&\eta_8 &= \frac{\delta_{19}}{96\sqrt[4]{K}}, \\
\eta_9 &= \min\{\eta_7, \eta_8\}, 
&\eta_{10} &= \frac{\delta_{23} \tau}{192}, 
&\eta_{11} &= \frac{\delta_{26}}{96\sqrt[4]{K}}, 
&\eta_{12} &= \min\{\eta_{10}, \eta_{11}\}, \\
C_1' &= \max\left\{\frac{c_{\max}^2}{2}, 2, \dfrac{2}{\gamma}, 32M_1^2\right\}, 
&C_2' &= \max\left\{ 2, \dfrac{2}{\gamma}, 8M_1\right\},
&\hat{\delta}(\cdot) &= \min\left\{\dfrac{\gamma\cdot}{2L}, \dfrac{1}{2}\right\},
\end{align*}
$$L = 4(1 + \sigma_s^2)(\mathbb{E}[f^2(d)] + 1) + 2\sigma_s^2(1 + \mathbb{E}[f^2(d)]). $$
\end{lemma}

\begin{proof}
Recall that 
\[
T_g = \frac{\|B(g_1)^HW\|}{\|(R(s)^{-1} g_2)[2 : K]\|}, W=C_1 D \hat{s}_1 + C_2 D B(\hat{s}_1) z_2[2 : N],
\]
\[
C_1 = \frac{z_1^H q \left(\frac{\|\hat{s}_1\|}{\|z_1\|} z_1\right)}{\|\hat{s}_1\| \|z_1\|}, \quad C_2 = \frac{\|B(z_1)^H q \left(\frac{\|\hat{s}_1\|}{\|z_1\|} z_1\right)\|}{\|z_2[2 : N]\|},\quad \hat{s}_1 = \frac{\|s\|}{\|g_1\|} f(D)^T g_1
.\]
By the proof of Theorem 2 in \cite{WLS24}, we have 
$\tau(K,N)\xrightarrow{a.s.} 0$ as $K$ tends to infinity,
where
\[
 \tau(K,N):=\frac{\left\| B(g_1)^HW \right\|^2}{K} - \frac{\left\| B(g_1)^H W_1 \right\|^2}{K}, \quad W_1=C_1 D \hat{s}_1 + C_2 D R(\hat{s}_1) z_2.
\]
Then we provide an upper bound of $\left| T_g-\bar{T}_g\right|$ as 

$$\left| T_g-\bar{T}_g\right|\leq\left| T_g-\frac{\|B(g_1)^HW_1\|}{\|(R(s)^{-1} g_2)[2 : K]\|}\right|+\left|\frac{\|B(g_1)^HW_1\|}{\|(R(s)^{-1} g_2)[2 : K]\|}-\bar{T}_g\right|. $$
By Lemma \ref{le-1}- Lemma \ref{le-4},  
\begin{align*}\label{step0}
&\P\left(\left| T_g-\bar{T}_g\right|\geq \epsilon\right)\\[12pt]
&\leq \P\left(\left|T_g- \frac{\frac{\left\| B(g_1)^HW_1\right\|}{\sqrt{K}}}{\frac{\|(R(s)^{-1} g_2[2:K]\|}{\sqrt{K}}}       \right|\geq \frac{\epsilon}{2}\right)+\P\left(\left| \frac{\frac{\left\| B(g_1)^HW_1\right\|}{\sqrt{K}}}{\frac{\|(R(s)^{-1} g_2[2:K]\|}{\sqrt{K}}}       -\bar{T}_g\right|\geq \frac{\epsilon}{2}\right)\\[12pt]
&\leq\P\left(\left| \frac{\left\| B(g_1)^HW\right\|}{\sqrt{K}}-\frac{\left\| B(g_1)^HW_1\right\|}{\sqrt{K}}\right|\geq \epsilon_1\right)+\P\left(\left|\frac{\|(R(s)^{-1} g_2[2:K]\|}{\sqrt{K}}-1 \right|\geq \frac{\epsilon_1}{1+\epsilon_1}\right)\\[12pt]
&\quad +\P\left(\left| \frac{\frac{\left\| B(g_1)^HW_1\right\|^2}{K}}{\frac{\|(R(s)^{-1} g_2[2:K]\|^2}{K}}       -\bar{T}_g^2\right|\geq \bar{T}_g\frac{\epsilon}{2}\right)\\[12pt]
&\leq\P\left(\left| \frac{C_2|z_2[1]|\|B(g_1)\|\|D\|}{\sqrt{K}}\right|\geq \epsilon_1\right)+\P\left(\left|\frac{\|(R(s)^{-1} g_2[2:K]\|}{\sqrt{K}}-1 \right|\geq \frac{\epsilon_1}{1+\epsilon_1}\right)\\[12pt]
&\quad +\P\left(\left| \frac{\left\| B(g_1)^HW_1\right\|^2}{K}-\bar{T}_g^2\right|\geq \epsilon_2\right)+\P\left(\left| \frac{\|(R(s)^{-1} g_2[2:K]\|^2}{K}       -1\right|\geq \frac{\epsilon_2}{1+\epsilon_2}\right),
\end{align*}
where 
$\epsilon_1 = \tilde{\delta}\left(0, 1, \frac{\epsilon}{2}\right), 
 \epsilon_2 = \tilde{\delta}\left(1, \bar{T}_g^2, \frac{\epsilon\bar{T}_g}{2}\right)$.
Notice that by the proof of Theorem 2 in \cite{WLS24}, we have
\[
\frac{\left\| B(g_1)^HW_1\} \right\|^2}{K}:=T_{21}-T_{22},
\]
where 
\begin{align*}
T_{21}&= |C_1|^2 \frac{\|s\|^2}{\|g_1\|^2}\frac{ g_1^H f(D) D^2 f(D)^T g_1}{K} + C_2^2 \frac{z_2^H D^2 z_2}{K} + 2 \frac{\|s\|}{\|g_1\|} \frac{\mathcal{R} \left\{ C_1 C_2 z_2^H D^2 f(D)^T g_1 \right\}}{K}\\&:=(T_{21})_1+(T_{21})_2+(T_{21})_3
\end{align*}
and
\begin{align*}
T_{22}&= |C_1|^2 \frac{\|s\|^2}{\|g_1\|^4}\frac{\| g_1^H Df(D)^T g_1\|^2}{K} + C_2^2\frac{1}{\| g_1\|^2} \frac{\|g_1^H D z_2\|^2}{K} + 2 \frac{\|s\|}{\|g_1\|^3} \frac{\mathcal{R} \left\{ C_1 C_2 g_1^Hf(D)D^Tg_1g_1^H Dz_2 \right\}}{K}\\&:=(T_{22})_1+(T_{22})_2+(T_{22})_3.
\end{align*}
Then we proceed the rest of proof in three steps.

\textbf{Step 1.} Estimate $\P\left(\left | |T_{21}-\bar{T}_{21}|\right |\geq \frac{\epsilon_2}{2}\right)$, where  $\bar{T}_{21}=\sigma_s^2 |\overline{C}_1|^2 \mathbb{E}[d^2 f^2(d)] + \overline{C}_2^2 \mathbb{E}[d^2].$

Notice that Lemma \ref{le-1}- Lemma \ref{le-4}, we have 
\[
\begin{aligned}
\P\left(\big|(T_{21})_1 - (\overline{T}_{21})_1\big| > \frac{\epsilon_2}{6}\right) &= \P\left(\left|\frac{g_1^H \mathbf{f}(\mathbf{D}) \mathbf{D}^T \mathbf{D} \mathbf{f}(\mathbf{D})^T g_1}{K} - \mathbb{E}\left[d^2 f(d)^2\right]\right| > \epsilon_3\right) \\
&\quad + \P\left(\left|\frac{| \mathbf{C}_1 |^2 \| \mathbf{s} \|^2}{\| \mathbf{g}_1 \|^2} - | \bar{C}_1 |^2\right| > \epsilon_3\right)\\
&\leq \P\left(\left|\frac{g_1^H \mathbf{f}(\mathbf{D}) \mathbf{D}^T \mathbf{D} \mathbf{f}(\mathbf{D})^T g_1}{K} - \mathbb{E}\left[d^2 f(d)^2\right]\right| > \epsilon_3\right) \\
&\quad + \P\left(\left|| C_1 |^2-|\bar{C}_1 |^2\right| > \epsilon_4\right)+ \P\left(\left|\frac{\| \mathbf{s} \|^2}{\| \mathbf{g}_1 \|^2} - 1\right| > \epsilon_4\right)\\
&\leq \P\left(\left|\frac{g_1^H \mathbf{f}(\mathbf{D}) \mathbf{D}^T \mathbf{D} \mathbf{f}(\mathbf{D})^T g_1}{K} - \mathbb{E}\left[d^2 f(d)^2\right]\right| > \epsilon_3\right) \\
&\quad + \P\left(\left| |C_1|^2-|\bar{C}_1 |^2\right| > \epsilon_4\right)+ \P\left(\left|\frac{\| \mathbf{s} \|^2}{K} - 1\right| > \epsilon_5\right)\\
&\quad+ \P\left(\left|\frac{ K}{\| \mathbf{g}_1 \|^2} - 1\right| > \epsilon_5\right)\\
&\leq \P\left(\left|\frac{g_1^H \mathbf{f}(\mathbf{D}) \mathbf{D}^T \mathbf{D} \mathbf{f}(\mathbf{D})^T g_1}{K} - \mathbb{E}\left[d^2 f(d)^2\right]\right| > \epsilon_3\right) \\
&\quad + \P\left(\left| |C_1|^2-|\bar{C}_1 |^2\right| > \epsilon_4\right)+ \P\left(\left|\frac{\| \mathbf{s} \|^2}{K} - \sigma_s^2\right| > \epsilon_5\right)\\
&\quad+ \P\left(\left|\frac{\| \mathbf{g}_1 \|^2}{K} - 1\right| > \epsilon_6\right),
\end{aligned}
\]
where $
 \epsilon_3 =\tilde{\delta}\left(|\bar{C}_{1}|^{2}, \mathbb{E}[d^{2}f(d)^{2}] , \frac{\epsilon_2}{6}\right), 
\epsilon_4 = \tilde{\delta}(|\bar{C}_{1}|^2, \sigma_s^2, \epsilon_3), 
 \epsilon_5 = \tilde{\delta}(1, \sigma_s^2, \epsilon_4), 
\epsilon_{6} = \frac{\epsilon_{5}}{1 + \epsilon_{5}}$.
By Lemma \ref{le-1}- Lemma \ref{le-4}, we have \[
\begin{aligned}
\P\left(\big|(T_{21})_2 - (\overline{T}_{21})_2\big| > \frac{\epsilon_2}{6}\right) &= \P\left(\left|C_2^2 \frac{z_2^H \mathbf{D}^T \mathbf{D} z_2}{K} - \overline{C}_2 \mathbb{E}\left[d^2\right]\right| > \frac{\epsilon_2}{6}\right) \\
&\leq \P\left(\left|C_2^2 - \overline{C}_2^2\right| > \epsilon_7\right) \\
&\quad + \P\left(\left|\frac{z_2^H \mathbf{D}^T \mathbf{D} z_2}{K} - \mathbb{E}\left[d^2\right]\right| > \epsilon_7\right),
\end{aligned}
\]
 where $\epsilon_7 = \tilde{\delta}\left(|\bar{C}_{2}|^2, \mathbb{E}[d^{2}], \frac{\epsilon_2}{6}\right). 
$
By Lemma \ref{le-1}- Lemma \ref{le-4}, we obtain 
\[
\begin{aligned}
\P\left(\big|(T_{21})_3 - (\overline{T}_{21})_3\big| > \frac{\epsilon_2}{6}\right) &= \P\left(\left|\frac{2 \| s \|}{\| g_1 \|}\frac{ R\left(C_1C_2z_2^H \mathbf{D}^T\mathbf{D} \mathbf{f}(D)^T g_1\right)}{K} - 0\right| > \frac{\epsilon_2}{6}\right) \\
&\leq \P\left(\left|\frac{2 \| s\|}{\| g_1\|} - 2\right| > \epsilon_8\right) \\
&\quad + \P\left(\left|\frac{ R\left(C_1C_2z_2^H \mathbf{D}^T\mathbf{D} \mathbf{f}(D)^T g_1\right)}{K} - 0\right| > \epsilon_8\right)\\
&\leq \P\left(\left|\frac{ \| s\|}{ \sqrt{K}} - \sigma_s\right| > \epsilon_8\right)+\P\left(\left|\frac{  \sqrt{K}}{\| g_1\|} - 1\right| > \epsilon_9\right) \\
&\quad + \P\left(\left|\frac{ R\left(C_1C_2z_2^H \mathbf{D}^T\mathbf{D} \mathbf{f}(D)^T g_1\right)}{K} - 0\right| > \epsilon_8\right)\\
&\leq \P\left(\left|\frac{ \| s\|}{ \sqrt{K}} - \sigma_s\right| > \epsilon_9\right)+\P\left(\left|\frac{ \| g_1\|}{\sqrt{K}} - 1\right| > \epsilon_{10}\right) \\
&\quad + \P\left(\left|\frac{ R\left(C_1C_2z_2^H \mathbf{D}^T\mathbf{D} \mathbf{f}(D)^T g_1\right)}{K} - 0\right| > \epsilon_8\right)\\
&\leq \P\left(\left|\frac{ \| s\|^2}{ K} - \sigma_s^2\right| > \sigma_s\epsilon_9\right)+\P\left(\left|\frac{ \| g_1\|^2}{K} - 1\right| > \epsilon_{10}\right) \\
&\quad + \P\left(\left|\frac{ \left(|C_1|C_2|z_2^H \mathbf{D}^T\mathbf{D} \mathbf{f}(D)^T g_1|\right)}{K} - 0\right| > \epsilon_8\right)\\
&\leq \P\left(\left|\frac{ \| s\|^2}{ K} - \sigma_s^2\right| > \sigma_s\epsilon_9\right)+\P\left(\left|\frac{ \| g_1\|^2}{K} - 1\right| > \epsilon_{10}\right) \\
&\quad +\P\left(\left||C_1|C_2 -|\bar{C}_1|\bar{C}_2\right| > \epsilon_{11}\right)+\P\left(\left|\frac{ \left(|z_2^H \mathbf{D}^T\mathbf{D} \mathbf{f}(D)^T g_1|\right)}{K} - 0\right| > \epsilon_{11}\right)\\
&\leq \P\left(\left|\frac{ \| s\|^2}{ K} - \sigma_s^2\right| > \sigma_s\epsilon_9\right)+\P\left(\left|\frac{ \| g_1\|^2}{K} - 1\right| > \epsilon_{10}\right) \\
&\quad +\P\left(\left||C_1|-|\bar{C}_1|\right| > \epsilon_{12}\right)+\P\left(\left|C_2 -\bar{C}_2\right| > \epsilon_{12}\right)\\
&\quad +\P\left(\left|\frac{ \left(|z_2^H \mathbf{D}^T\mathbf{D} \mathbf{f}(D)^T g_1|\right)}{K}\right| > \epsilon_{11}\right)\\
&\leq \P\left(\left|\frac{ \| s\|^2}{ K} - \sigma_s^2\right| > \sigma_s\epsilon_9\right)+\P\left(\left|\frac{ \| g_1\|^2}{K} - 1\right| > \epsilon_{10}\right) \\
&\quad +\P\left(\left||C_1|-|\bar{C}_1|\right| > \epsilon_{12}\right)+\P\left(\left|C_2^2 -\bar{C}_2^2\right| > \bar{C}_2\epsilon_{12}\right)\\
&\quad +\P\left(\left|\frac{ \left(|z_2^H \mathbf{D}^T\mathbf{D} \mathbf{f}(D)^T g_1|\right)}{K}\right| > \epsilon_{11}\right),
\end{aligned}
\]
where $
 \epsilon_8 = \tilde{\delta}\left(0, 2, \frac{\epsilon_2}{6}\right), 
 \epsilon_9 = \tilde{\delta}\left(1, \sigma_s, \frac{\epsilon_8}{2}\right), 
\epsilon_{10} = \frac{\epsilon_9}{1+\epsilon_9}, 
 \epsilon_{11} = \tilde{\delta}\left(0, |\bar{C}_1|\bar{C}_2, \frac{\epsilon_8}{4}\right), 
 \epsilon_{12} = \tilde{\delta}\left(|\bar{C}_1|, \bar{C}_2, \frac{\epsilon_{11}}{4}\right).$
By Lemma \ref{thm:explicit_quad}, we have for $K>\bar{K}$,
$$\P\left(\left|\frac{g_1^H \mathbf{f}(\mathbf{D}) \mathbf{D}^T \mathbf{D} \mathbf{f}(\mathbf{D})^T g_1}{K} - \mathbb{E}\left[d^2 f(d)^2\right]\right| > \epsilon_3\right)\leq 2 \exp\left( -\dfrac{1}{2} K \min\left\{ \dfrac{\varepsilon_3^2}{64M_1^2}, \dfrac{\varepsilon_3}{8M_1} \right\} \right)+\frac{8M_1^2}{K \varepsilon_3^2},$$
$$\P\left(\left|\frac{z_2^H \mathbf{D}^T \mathbf{D} z_2}{K} - \mathbb{E}\left[d^2\right]\right| > \epsilon_7\right)\leq  2 \exp\left( -\dfrac{1}{2} K \min\left\{ \dfrac{\varepsilon_7^2}{64M_1^2}, \dfrac{\varepsilon_7}{8M_1} \right\} \right)+\frac{8M_1^2}{K \varepsilon_7^2}$$
and 
$$\P\left(\left|\frac{ \left(|z_2^H \mathbf{D}^T\mathbf{D} \mathbf{f}(D)^T g_1|\right)}{K}\right| > \epsilon_{11}\right) \leq 4\exp\!\bigl(-\tfrac{K\epsilon_{11}^{2}}{2M_1^{2}}\bigr).
$$
By Lemma \ref{exp-C1}, we have for  $K>\max\{\hat{K}_1, \bar{K}\}$ with $\hat{K}_1=\left(\frac{12\sqrt{2M_0^2\bar{\alpha}^{-1}\max\{K_g, K_h\}}}{\delta_1}\right)^{8}$, 
\begin{equation}
\begin{aligned}
&\mathbb{P}\left( \left| \left| C_1 \right| - \left| \overline{C}_1 \right| \right| > \epsilon_{12} \right)+\P\left(\left| |C_1|^2-|\bar{C}_1 |^2\right| > \epsilon_4\right)\\&\leq\mathbb{P}\left( \left| \left| C_1 \right| - \left| \overline{C}_1 \right| \right| > \epsilon_{12} \right)+
\P\left(\left|  |C_1|-|\bar{C}_1 |\right| > \sqrt{\frac{\epsilon_{4}}{2}}\right)+\P\left(\left| |C_1|-|\bar{C}_1 |\right| > \frac{\epsilon_{4}}{4}\right)
\\
& \leq  72e^{-\frac{\gamma K}{32}} +  192 \exp\left(- K \min\left\{\frac{\hat{\delta}(\bar{\alpha}\eta_1)^2}{C_1'}, \frac{\hat{\delta}(\bar{\alpha}\eta_1)}{C_2'}\right\}\right)+\dfrac{192M_1^2}{K \hat{\delta}(\bar{\alpha}\eta_1)^2} + 24 \exp\left( -\frac{\sqrt{K}\gamma \delta_1^2}{1152 } \right)+\frac{55296M_0^2}{\gamma K\delta_1^2}\\
&\quad+24 \exp\left(- K \min\left\{\frac{\hat{\delta}(\bar{\alpha}\delta_3)^2}{C_1'}, \frac{\hat{\delta}(\bar{\alpha}\delta_3)}{C_2'}\right\}\right)+\dfrac{12C M_1^2}{K \hat{\delta}(\bar{\alpha}\delta_3)^2}+6 \exp\left(-\frac{1}{2} \gamma K \min\left(\delta_4^2/16, \delta_4/4\right)\right),
\end{aligned}
\end{equation}
where $$  C_1' = \max\left\{\frac{c_{\max}^2}{2}, 32, 32M_1^2\right\}, \quad C_2' = \max\left\{ 8, 8M_1\right\}, \eta_1=\frac{\delta_1}{96\sqrt[4]{K}},\eta_2=\min\left\{\epsilon_{12}, \sqrt{\frac{\epsilon_4}{2}}, \frac{\epsilon_4}{4}\right\},$$$$
\delta_1=\tilde{\delta}\left(\mathbb{E}\left[Z^\dagger q(\alpha Z)\right],\frac{1}{\bar{\alpha}^2}, \eta_2\right), \delta_2=\tilde{\delta}(1,\frac{1}{\bar{\alpha}^2}, \delta_1), \delta_3=\frac{\delta_2\bar{\alpha}}{1+\delta_2\bar{\alpha}},\delta_4=\frac{\delta_2}{1+\delta_2},
$$$$
L = 4(1 + \sigma_s^2)(\E[f^2(d)] + 1) + 2\sigma_s^2(1 + \E[f^2(d)]),\quad
\hat{\delta}(\bar{\alpha}\delta_3) = \min\left\{\dfrac{\gamma\bar{\alpha}\delta_3}{2L}, \dfrac{1}{2}\right\}.$$
By Hoeffding inequality in Proposition \ref{pro-hoeff-ineq}, we have
$$
\P\left(\left|\frac{\| \mathbf{s} \|^2}{K} - \sigma_s^2\right| > \epsilon_5\right)\leq 2 \exp\left(-\frac{2K \epsilon_5^2}{c_{\max}^2}\right)$$
and
$$
\P\left(\left|\frac{\| \mathbf{s} \|^2}{K} - \sigma_s^2\right| > \sigma_s\epsilon_9\right)\leq 2 \exp\left(-\frac{2K (\sigma_s\epsilon_9)^2}{c_{\max}^2}\right),$$
where $c_{\max}$ is
$$
c_{\max} = \max_{s \in \mathcal{S}_M} |s| ^2 - \min_{s \in \mathcal{S}_M}|s|
^2.
$$
By Proposition  \ref{pro-exp-bound-lambda1},  we have 
\[
\mathbb{P}\left(\left|\frac{\|\mathbf{g}_1\| ^2}{K} - 1\right|
 \geq \epsilon_6\right) \leq 2\exp(-\frac{1}{2}K\min\{\epsilon_6/4, \epsilon_6^2/16\})
\] and 
\[
\mathbb{P}\left(\left|\frac{\|\mathbf{g}_1\| ^2}{K} - 1\right|
 \geq \epsilon_{10}\right) \leq 2\exp(-\frac{1}{2}K\min\{\epsilon_{10}/4, \epsilon_{10}^2/16\}).
\] 
By Lemma \ref{exp-C2}, for   $\delta=\min\{\bar{C}_2\epsilon_{12}, \epsilon_{7}\} > 0$, $K>\max\{\hat{K}_2, \hat{K}_3, \bar{K}, \frac{1}{\gamma\delta_9}\}$ with $\hat{K}_2=\left(\frac{12M_0^2\bar{\alpha}^{-1}K_k}{\delta_5}\right)^{4}$ and $\hat{K}_3=\left(\frac{12\sqrt{2M_0^2\bar{\alpha}^{-1}\max\{K_g, K_h\}}}{\delta_8}\right)^{8}$, we have  \begin{equation}
\begin{aligned}
&\P\left( \left| C_2^2 - \overline{C}_2^2 \right| > \delta \right)  \\
& \leq  52e^{-\frac{\gamma K}{32}} +  144 \exp\left(- K \min\left\{\frac{\hat{\delta}(\bar{\alpha}\eta_5)^2}{C_1'}, \frac{\hat{\delta}(\bar{\alpha}\eta_5)}{C_2'}\right\}\right)+\dfrac{144 M_1^2}{K \hat{\delta}(\bar{\alpha}\eta_5)^2} + 18 \exp\left( -\frac{\sqrt{K}\gamma \delta_{10}}{2304 } \right)+\frac{37440M_0^2}{\gamma K\delta_{11}^2}\\
&\quad+4 \exp\left(-\frac{1}{2} (\gamma K-1) \min\left(\delta_7^2/16, \delta_7/4, \delta_8^2/16, \delta_8/4\right)\right),
\end{aligned}
\end{equation}
where  $$K_k = \frac{2}{\sqrt{\pi}} (N_\ell^g + N_\ell^h) + \left(1 + \frac{1}{\sqrt{\pi}}\right) (N_r^g + N_r^h)),
C_1' = \max\left\{ \frac{c_{\max}^2}{2}, 2, \frac{2}{\gamma}, 32 M_1^2 \right\}, \quad C_2' = \max\left\{ 2, \frac{2}{\gamma}, 8 M_1 \right\}, $$$$ \eta_3 = \frac{\delta_5 \tau}{192}, \quad \tau = \frac{1}{\sqrt[4]{K}},
\hat{\delta}(\bar{\alpha}\eta_5) = \min\left\{ \frac{\gamma \bar{\alpha} \eta_5 }{2L}, \frac{1}{2} \right\}, \quad L = 4(1 + \sigma_s^2)(\mathbb{E}[f^2(d)] + 1) + 2\sigma_s^2(1 + \mathbb{E}[f^2(d)]).
$$

$$\delta_5=\tilde{\delta}(1,\overline{C}_2^2, \delta), \delta_6=\tilde{\delta}(1,|\mathbb{E}[Z^\dagger q(\bar{\alpha} Z)]|^2,
\delta_5), \delta_7=\frac{\delta_6}{1+\delta_6},\delta_8=\min\left\{\sqrt{\frac{\delta_6}{2}}, \frac{\delta_6}{2|\mathbb{E}[Z^\dagger q(\bar{\alpha} Z)]|}\right\}, $$

$$\delta_9=\frac{1}{2}(\sqrt{1+\delta_5}-1),\delta_{10}=\min\left\{\frac{\delta_5^2}{2}, 2\delta_8^2\right\}, \delta_{11}=\min\left\{\delta_5, \delta_8\right\}, \eta_4=\frac{\delta_8}{96\sqrt[4]{K}}, \eta_5=\min\{\eta_3, \eta_4\}.$$
Therefore, combining above, we have for  $K>\max\{\hat{K}_1, \hat{K}_2, \hat{K}_3,  \bar{K}, \frac{1}{\gamma\delta_9}\}$, we have
\begin{equation}
\begin{aligned}
&\P\left(\left | |T_{21}-\bar{T}_{21}|\right |\geq \frac{\epsilon_2}{2}\right)\\
& \leq  482\exp(-K\tilde{\epsilon}_1') +  128 \exp\left( -\sqrt{K}\tilde{\epsilon}_2 '\right)+\frac{\tilde{\epsilon}_3'}{K}+4 \exp\left(-\frac{1}{2} (\gamma K-1) \min\left\{\delta_4^2/16, \delta_4/4,\delta_7^2/16, \delta_7/4, \delta_8^2/16, \delta_8/4\right\}\right),
\end{aligned}
\end{equation}
where \begin{align*} \tilde{\epsilon}_1' = \min&\left\{\frac{\epsilon_2^2}{128M_1^2}, \frac{\epsilon_3}{16M_1}, \frac{\epsilon_7^2}{128M_1^2},\frac{\epsilon_7}{16M_1},\frac{\epsilon_{11}^2}{2M_1^2},\frac{\gamma}{32},\frac{\hat{\delta}(\bar{\alpha}\eta_1)^2}{C_1'},\frac{\hat{\delta}(\bar{\alpha}\eta_1)}{C_2'},\frac{\hat{\delta}(\bar{\alpha}\delta_3)^2}{C_1'},\frac{\hat{\delta}(\bar{\alpha}\delta_3)}{C_2'},\frac{\gamma\delta_4^2}{32}, \frac{\gamma\delta_4}{8}, \frac{2(\sigma_s\epsilon_9)^2}{c_{max}^2}, \frac{2(\epsilon_5)^2}{c_{max}^2},\right.\\&\left.\frac{\epsilon_6}{8}, \frac{\epsilon_6^2}{32},\frac{\epsilon_{10}}{8}, \frac{\epsilon_{10}^2}{32},\frac{\hat{\delta}(\bar{\alpha}\eta_5)^2}{C_1'},\frac{\hat{\delta}(\bar{\alpha}\eta_5)}{C_2'}\right\},\tilde{\epsilon}_2' = \min\left\{\frac{\gamma\delta_1^2}{1152}, \frac{\gamma\delta_{10}}{2304}\right\},
\end{align*}
$$\tilde{\epsilon}_3' = \max\left\{\frac{16M_1^2}{\min\{\epsilon_3^2,\epsilon_7^2\}}, \frac{384M_1^2}{\hat{\delta}(\bar{\alpha}\eta_1)^2}, \frac{55296M_1^2}{\gamma\delta_1^2},\frac{24M_1^2}{\hat{\delta}(\bar{\alpha}\delta_3)^2}, \frac{144M_1^2}{\hat{\delta}(\bar{\alpha}\eta_5)^2},\frac{37440M_0^2}{\gamma\delta_{11}}\right\}.$$

\textbf{Step 2.} Estimate $\mathbb{P}\left(\left | |T_{22}-\bar{T}_{22}|\right |\geq \frac{\epsilon_2}{2}\right)$, where $\bar{T}_{22}=\sigma_s^2 |C_1|^2 \mathbb{E}^2[d f(d)]$.

By Lemma \ref{le-1}- Lemma \ref{le-4}, we have
\[
\begin{aligned}
\P\left(\big|(T_{22})_1 - (\overline{T}_{22})_1\big| > \frac{\epsilon_2}{6}\right) &= \P\left(\left| |C_1|^2 \frac{\|s\|^2}{\|g_1\|^4}\frac{\| g_1^H Df(D)^T g_1\|^2}{K} -\sigma_s^2 |\overline{C}_1|^2 \mathbb{E}^2[d f(d)] \right| > \frac{\epsilon_2}{6}\right) \\
&\leq \P\left(\left| |C_1|^2  - |\overline{C}_1|^2 \right| > \epsilon_{13}\right) 
+ \P\left(\left|  \frac{\|s\|^2}{\|g_1\|^4}\frac{\| g_1^H Df(D)^T g_1\|^2}{K} -\sigma_s^2 \mathbb{E}^2[d f(d)] \right| > \epsilon_{13}\right) \\
&\leq \P\left(\left| |C_1|^2  - |\overline{C}_1|^2 \right| > \epsilon_{13}\right)+\P\left(\left|  \frac{\|s\|^2K}{\|g_1\|^4} -\sigma_s^2  \right| > \epsilon_{14}\right)\\
&+\P\left(\left| \frac{\| g_1^H Df(D)^T g_1\|^2}{K^2} -\mathbb{E}^2[d f(d)] \right| > \epsilon_{14}\right)\\
&\leq \P\left(\left| |C_1|^2  - |\overline{C}_1|^2 \right| > \epsilon_{13}\right)+\P\left(\left|  \frac{\|s\|^2}{K} -\sigma_s^2  \right| > \epsilon_{15}\right)\\
&+\P\left(\left|  \frac{K^2}{\|g_1\|^4} -1\right| > \epsilon_{15}\right)+\P\left(\left| \frac{\| g_1^H Df(D)^T g_1\|^2}{K^2} -\mathbb{E}^2[d f(d)] \right| > \epsilon_{14}\right)\\
&\leq \P\left(\left| |C_1|^2  - |\overline{C}_1|^2 \right| > \epsilon_{13}\right)+\P\left(\left|  \frac{\|s\|^2}{K} -\sigma_s^2  \right| > \epsilon_{15}\right)\\
&+\P\left(\left|  \frac{\|g_1\|^4}{K^2} -1\right| > \epsilon_{16}\right)+\P\left(\left| \frac{\| g_1^H Df(D)^T g_1\|^2}{K^2} -\mathbb{E}^2[d f(d)] \right| > \epsilon_{14}\right)\\
&\leq \P\left(\left| |C_1|^2  - |\overline{C}_1|^2 \right| > \epsilon_{13}\right)+\P\left(\left|  \frac{\|s\|^2}{K} -\sigma_s^2  \right| > \epsilon_{15}\right)\\
&+\P\left(\left|  \frac{\|g_1\|^2}{K} -1\right| > \sqrt{\frac{\epsilon_{16}}{2}}\right)+\P\left(\left|  \frac{\|g_1\|^2}{K} -1\right| > \frac{\epsilon_{16}}{4}\right)\\
&+\P\left(\left| \frac{\| g_1^H Df(D)^T g_1\|}{K} -\mathbb{E}[d f(d)] \right| > \sqrt{\frac{\epsilon_{14}}{2}}\right)+\P\left(\left| \frac{\| g_1^H Df(D)^T g_1\|}{K} -\mathbb{E}[d f(d)] \right| > \frac{\epsilon_{14}}{4}\right),
\end{aligned}
\]
where $\epsilon_{13} = \tilde{\delta}\left(|\bar{C}_{1}|^2, \sigma_s^{2}\mathbb{E}[d f(d)]^{2}, \frac{\epsilon_2}{6}\right), 
 \epsilon_{14} = \tilde{\delta}\left(\sigma_s,\mathbb{E}[d f(d)], \epsilon_{13}\right), 
 \epsilon_{15} = \tilde{\delta}(1, \sigma_s^2, \epsilon_{14}),,
\epsilon_{16} = \frac{\epsilon_{15}}{1 + \epsilon_{15}}$.
By Lemma \ref{le-1}- Lemma \ref{le-4}, we have
\[
\begin{aligned}
\P\left(\big|(T_{22})_2 - (\overline{T}_{22})_2\big| > \frac{\epsilon_2}{6}\right) &= \P\left(\left|  C_2^2\frac{1}{\| g_1\|^2} \frac{\|g_1^H D z_2\|^2}{K} \right| > \frac{\epsilon_2}{6}\right) \\
&\leq \P\left(\left|  C_2^2-\bar{C}_2^2 \right| > \epsilon_{17}\right)  
+ \P\left(\left|  \frac{K}{\| g_1\|^2} \frac{\|g_1^H D z_2\|^2}{K^2} \right| > \epsilon_{17}\right)  \\
&\leq \P\left(\left|  C_2^2-\bar{C}_2^2 \right| > \epsilon_{17}\right)  
+ \P\left(\left|  \frac{K}{\| g_1\|^2} -1 \right| > \epsilon_{18}\right)  \\
&+\P\left(\left|   \frac{\|g_1^H D z_2\|^2}{K^2}-0 \right| > \epsilon_{18}\right)\\
&\leq \P\left(\left|  C_2^2-\bar{C}_2^2 \right| > \epsilon_{17}\right)  
+ \P\left(\left|  \frac{\| g_1\|^2}{K} -1 \right| > \epsilon_{19}\right)  \\
&+\P\left(\left|   \frac{\|g_1^H D z_2\|^2}{K^2}-0 \right| > \epsilon_{18}\right),
\end{aligned}
\]
where $\epsilon_{17} = \tilde{\delta}\left(0, \bar{C}_{2}, \frac{\epsilon_2}{6}\right), 
 \epsilon_{18} = \tilde{\delta}(1, 0, \epsilon_{17}), 
\epsilon_{19} = \frac{\epsilon_{18}}{1+\epsilon_{18}}$.
By Lemma \ref{le-1}- Lemma \ref{le-4}, we have
\[
\begin{aligned}
\P\left(\big|(T_{22})_3 - (\overline{T}_{22})_3\big| > \frac{\epsilon_2}{6}\right) &= \P\left(\left| 2 \frac{\|s\|}{\|g_1\|^3} \frac{\left\{ C_1 C_2 g_1^Hf(D)D^Tg_1g_1^H Dz_2 \right\}}{K^2}\right| > \frac{\epsilon_2}{6}\right) \\
&\leq \P\left(\left|  \frac{\text{Re} \left\{ C_1 C_2 g_1^Hf(D)D^Tg_1g_1^H Dz_2 \right\}}{K^2}-0 \right| > \epsilon_{20}\right)  
+ \P\left( \left| 2 \frac{K\|s\|}{\|g_1\|^3}\right| > \epsilon_{20}\right)  \\
&\leq \P\left(\left|  \frac{\text{Re} \left\{ C_1 C_2 g_1^Hf(D)D^Tg_1g_1^H Dz_2 \right\}}{K^2}-0 \right| > \epsilon_{20}\right)  
+ \P\left(\left| \frac{\|s\|}{\sqrt{K}}-\sigma_s\right| > \epsilon_{21}\right)  \\
&+\P\left(\left|  2 \frac{K\sqrt{K}}{\|g_1\|^3} -0\right| > \epsilon_{21}\right) \\
&\leq \P\left(\left|  \frac{\text{Re} \left\{ C_1 C_2 g_1^Hf(D)D^Tg_1g_1^H Dz_2 \right\}}{K^2}-0 \right| > \epsilon_{20}\right)  
+ \P\left(\left| \frac{\|s\|}{\sqrt{K}}-\sigma_s\right| > \epsilon_{21}\right)  \\
&+\P\left(\left|   \frac{K\sqrt{K}}{\|g_1\|^3} -1\right| > \epsilon_{22}\right) \\
&\leq \P\left(\left|  \frac{\text{Re} \left\{ C_1 C_2 g_1^Hf(D)D^Tg_1g_1^H Dz_2 \right\}}{K^2}-0 \right| > \epsilon_{20}\right)  
+ \P\left(\left| \frac{\|s\|}{\sqrt{K}}-\sigma_s\right| > \epsilon_{21}\right)  \\
&+\P\left(\left|   \frac{\|g_1\|^3}{K\sqrt{K}} -1\right| > \epsilon_{23}\right),
\end{aligned}
\]
where $\epsilon_{20} = \sqrt{\frac{\epsilon_2}{6}}, 
 \epsilon_{21} = \tilde{\delta}(\sigma_s, 0, \epsilon_{20}),
\epsilon_{22} = \tilde{\delta}(1, 0, \epsilon_{21}),,
\varepsilon_{23} = \frac{\varepsilon_{22}}{1 + \varepsilon_{22}}.$
By Lemma \ref{le-4}, we have
$$
\begin{aligned}
\P\left(\left|  \frac{\text{Re} \left\{ C_1 C_2 g_1^Hf(D)D^Tg_1g_1^H Dz_2 \right\}}{K^2}-0 \right| > \epsilon_{20}\right) &\leq \P\left(\left|  \frac{| C_1 C_2 g_1^Hf(D)D^Tg_1g_1^H Dz_2 |}{K^2} \right| > \epsilon_{20}\right) \\
&\leq \P\left(\left|   C_1 C_2-\bar{C}_1 \bar{C}_2|\right| > \epsilon_{24}\right)+\P\left(\left|  \frac{|  g_1^Hf(D)D^Tg_1g_1^H Dz_2 |}{K^2} \right| > \epsilon_{24}\right)\\
&\leq\P\left(\left|  | C_1-\bar{C}_1 |\right| > \epsilon_{25}\right)+\P\left(\left|  | C_2- \bar{C}_2|\right| > \epsilon_{25}\right) \\
&+\P\left(\left|  \frac{|  g_1^Hf(D)D^Tg_1 |}{K}-\mathbb{E}[df(d)] \right| > \epsilon_{26}\right)+\P\left(\left|  \frac{g_1^H Dz_2}{K} \right| > \epsilon_{26}\right),
\end{aligned}
$$
where $\epsilon_{24}=\tilde{\delta}(\bar{C}_1\bar{C}_2,0,\epsilon_{20}), \epsilon_{25}=\tilde{\delta}(\bar{C}_1,\bar{C}_2,\epsilon_{24}),\epsilon_{26}=\tilde{\delta}(\mathbb{E}[d f(d)],0,\epsilon_{25}),$
By Lemma \ref{thm:explicit_quad}, we have for $K>\bar{K}$,
$$
\begin{aligned}
&\P\left(\left| \frac{\| g_1^H Df(D)^T g_1\|}{K} -\mathbb{E}[d f(d)] \right| > \sqrt{\frac{\epsilon_{14}}{2}}\right)+\P\left(\left| \frac{\| g_1^H Df(D)^T g_1\|}{K} -\mathbb{E}[d f(d)] \right| > \frac{\epsilon_{14}}{4}\right)\\&+\P\left(\left| \frac{\| g_1^H Df(D)^T g_1\|}{K} -\mathbb{E}[d f(d)] \right| > \epsilon_{26}\right)\\&
\leq 3\P\left(\left| \frac{\| g_1^H Df(D)^T g_1\|}{K} -\mathbb{E}[d f(d)] \right| > \epsilon_{27}\right)\leq 6 \exp\left( -\dfrac{1}{2} K \min\left\{ \dfrac{\epsilon_{27}^2}{64M_1^2}, \dfrac{\epsilon_{27}}{8M_1} \right\} \right)+\frac{24M_1^2}{K \epsilon_{27}^2}
\end{aligned}
$$
and
$$\P\left(\left|\frac{ \left(g_1^H \mathbf{D}z_2\right)^2}{K^2}\right| > \epsilon_{18}\right)+\P\left(\left|\frac{ g_1^H \mathbf{D}z_2}{K}\right| > \epsilon_{26}\right) \leq 2\P\left(\left|\frac{ g_1^H \mathbf{D}z_2}{K}\right| > \epsilon_{28}\right)\leq 8\exp\!\bigl(-\tfrac{K\epsilon_{28}^{2}}{2M_1^{2}}\bigr),
$$
where $\epsilon_{27}=\min\{\sqrt{\frac{\epsilon_{14}}{2}}, \frac{\epsilon_{14}}{4},\epsilon_{26}\}, \epsilon_{28}=\min\{\epsilon_{26}, \epsilon_{18}\}.$
By Lemma \ref{exp-C1}, we have for  $K>\max\{\hat{K}_4, \bar{K}\}$ with $\hat{K}_4=\left(\frac{12\sqrt{2M_0^2\bar{\alpha}^{-1}\max\{K_g, K_h\}}}{\delta_6}\right)^{8}$, 
$$
\begin{aligned}
&\P\left(\left|  |C_1|-|\bar{C}_1 |\right| > \sqrt{\frac{\epsilon_{13}}{2}}\right)+\P\left(\left| |C_1|-|\bar{C}_1 |\right| > \frac{\epsilon_{13}}{4}\right)+\P\left(\left| |C_1|-|\bar{C}_1 |\right| > \epsilon_{25}\right)
\\&
\leq 3\P\left(\left| |C_1|-|\bar{C}_1 |\right| > \epsilon_{29}\right)\\
&\leq 72e^{-\frac{\gamma K}{32}} +  192\exp\left(- K \min\left\{\frac{\hat{\delta}(\bar{\alpha}\eta_6)^2}{C_1'}, \frac{\hat{\delta}(\bar{\alpha}\eta_6)}{C_2'}\right\}\right)+\dfrac{192M_1^2}{K \hat{\delta}(\bar{\alpha}\eta_6)^2} + 24 \exp\left( -\frac{\sqrt{K}\gamma \delta_{12}^2}{1152 } \right)+\frac{55296M_0^2}{\gamma K\delta_{12}^2}\\
&\quad+24 \exp\left(- K \min\left\{\frac{\hat{\delta}(\bar{\alpha}\delta_{14})^2}{C_1'}, \frac{\hat{\delta}(\bar{\alpha}\delta_{14})}{C_2'}\right\}\right)+\dfrac{24M_1^2}{K \hat{\delta}(\bar{\alpha}\delta_{14})^2}+6 \exp\left(-\frac{1}{2} \gamma K \min\left\{\delta_{15}^2/16, \delta_{15}/4\right\}\right),
\end{aligned}
$$
where $$  C_1' = \max\left\{\frac{c_{\max}^2}{2}, 32, 32M_1^2\right\}, \quad C_2' = \max\left\{ 8, 8M_1\right\}, \eta_{6}=\frac{\delta_{12}}{96\sqrt[4]{K}},$$$$
\delta_{12}=\tilde{\delta}(\mathbb{E}\left[Z^\dagger q(\alpha Z)\right],\frac{1}{\bar{\alpha}^2},\epsilon_{29}), \delta_{13}=\tilde{\delta}(1,\frac{1}{\bar{\alpha}^2}, \delta_{12}), \delta_{14}=\frac{\delta_{13}\bar{\alpha}}{1+\delta_{13}\bar{\alpha}},\delta_{15}=\frac{\delta_{13}}{1+\delta_{13}},
$$$$
L = 4(1 + \sigma_s^2)(\E[f^2(d)] + 1) + 2\sigma_s^2(1 + \E[f^2(d)]),\quad
\hat{\delta}(\bar{\alpha}\delta_{14}) = \min\left\{\dfrac{\gamma\bar{\alpha}\delta_{14}}{2L}, \dfrac{1}{2}\right\}, \epsilon_{29}=\min\left\{\sqrt{\frac{\epsilon_{13}}{2}}, \frac{\epsilon_{13}}{4},\epsilon_{25}\right\}. $$
By Lemma \ref{pro-exp-bound-lambda1}, we have
$$
\begin{aligned}
&\P\left(\left|  \frac{\|g_1\|^2}{K} -1\right| > \sqrt{\frac{\epsilon_{16}}{2}}\right)+\P\left(\left|  \frac{\|g_1\|^2}{K} -1\right| > \frac{\epsilon_{16}}{4}\right)+\P\left(\left|  \frac{\| g_1\|^2}{K} -1 \right| > \epsilon_{19}\right)+\P\left(\left|   \frac{\|g_1\|^3}{K\sqrt{K}} -1\right| > \epsilon_{23}\right)\\&\leq7\P\left(\left|  \frac{\|g_1\|^2}{K} -1\right| > \epsilon_{30}\right)\leq14\exp(-\frac{1}{2}K\min\{\epsilon_{30}, \epsilon_{30}^2\}),
\end{aligned}
$$ where
$\epsilon_{30}=\min\{\sqrt{\frac{\epsilon_{16}}{2}}, \frac{\epsilon_{16}}{4}, \epsilon_{19}, \sqrt{\frac{\epsilon_{23}}{3}}, \frac{\epsilon_{23}}{3}, \frac{\epsilon_{23}}{9}\}$.
By Lemma \ref{pro-hoeff-ineq}, we have$$ 
\begin{aligned}
\P\left(\left|  \frac{\|s\|^2}{K} -\sigma_s^2  \right| > \epsilon_{15}\right)+\P\left(\left| \frac{\|s\|}{\sqrt{K}}-\sigma_s\right| > \epsilon_{21}\right) \leq 2\P\left(\left|  \frac{\|s\|^2}{K} -\sigma_s^2  \right| > \epsilon_{31}\right)\leq2 \exp\left(-\frac{2K (\epsilon_{31})^2}{c_{\max}^2}\right),
\end{aligned}
$$
where $\epsilon_{31}=\min\{\epsilon_{15}, \sigma_s\epsilon_{21}\}, c_{\max}$ is
$$
c_{\max} = \max_{s \in \mathcal{S}_M} |s| ^2 - \min_{s \in \mathcal{S}_M}|s|
^2.$$
By Lemma  \ref{exp-C2}, for $K>\max\{\hat{K}_5, \hat{K}_6, \bar{K}, \frac{1}{\gamma\delta_{20}}\}$ with $\hat{K}_5=\left(\frac{12M_0^2\bar{\alpha}^{-1}K_k}{\delta_{16}}\right)^{4}$ and $\hat{K}_6=\left(\frac{12\sqrt{4M_0^2\bar{\alpha}^{-1}\max\{K_g, K_h\}}}{\delta_{19}}\right)^{8}$, we have  
$$
\begin{aligned}
&\P\left(\left|  C_2^2-\bar{C}_2^2 \right| > \epsilon_{17}\right) +\P\left(\left|  | C_2- \bar{C}_2|\right| > \epsilon_{25}\right) \\&\leq \P\left(\left|  C_2^2-\bar{C}_2^2 \right| > \epsilon_{17}\right)  +
\P\left(\left|  C_2^2-\bar{C}_2^2 \right| > \bar{C}_2\epsilon_{25}\right) \\
&\leq 2\P\left(\left|  C_2^2-\bar{C}_2^2 \right| > \epsilon_{32}\right)\\
& \leq  104e^{-\frac{\gamma K}{32}} + 288 \exp\left(- K \min\left\{\frac{\hat{\delta}(\bar{\alpha}\eta_9)^2}{C_1'}, \frac{\hat{\delta}(\bar{\alpha}\eta_9)}{C_2'}\right\}\right)+\dfrac{288M_1^2}{K \hat{\delta}(\bar{\alpha}\eta_9)^2} + 36 \exp\left( -\frac{\sqrt{K}\gamma \delta_{21}}{2304 } \right)+\frac{74880M_0^2}{\gamma K\delta_{22}^2}\\
&\quad+8 \exp\left(-\frac{1}{2} (\gamma K-1) \min\left\{\delta_{18}^2/16, \delta_{18}/4, \delta_{19}^2/16, \delta_{19}/4\right\}\right),
\end{aligned}
$$
where  $$K_k = \frac{2}{\sqrt{\pi}} (N_\ell^g + N_\ell^h) + \left(1 + \frac{1}{\sqrt{\pi}}\right) (N_r^g + N_r^h)),
C_1' = \max\left\{ \frac{c_{\max}^2}{2}, 32, 32 M_1^2 \right\}, \quad C_2' = \max\left\{ 8, 8 M_1 \right\}, $$$$ \eta_7 = \frac{\delta_{16} \tau}{192}, \quad \tau = \frac{1}{\sqrt[4]{K}},
\hat{\delta}(\bar{\alpha}\eta_9) = \min\left\{ \frac{\gamma \bar{\alpha} \eta_9 }{2L}, \frac{1}{2} \right\}, \quad L = 4(1 + \sigma_s^2)(\mathbb{E}[f^2(d)] + 1) + 2\sigma_s^2(1 + \mathbb{E}[f^2(d)]).
$$
$$\delta_{16}=\tilde{\delta}(1,\overline{C}_2^2, \epsilon_{32}), \delta_{17}=\tilde{\delta}(1,|\mathbb{E}[Z^\dagger q(\bar{\alpha} Z)]|^2,
\delta_{16}), \delta_{18}=\frac{\delta_{17}}{1+\delta_{17}}, \delta_{19}=\min\left\{\sqrt{\frac{\delta_{17}}{2}}, \frac{\delta_{17}}{2|\mathbb{E}[Z^\dagger q(\bar{\alpha} Z)]|}\right\}, $$$$\delta_{20}=\frac{1}{2}(\sqrt{1+\delta_{16}}-1),\delta_{21}=\min\left\{\frac{\delta_{16}^2}{2}, 2\delta_{19}^2\right\}, \delta_{22}=\min\left\{\delta_{16}, \delta_{19}\right\}, \eta_8=\frac{\delta_{19}}{96\sqrt[4]{K}}, \eta_9=\min\{\eta_7, \eta_8\}.$$
$\epsilon_{32}=\min\{\epsilon_{17}, \bar{C}_2\epsilon_{25}\}.$
Therefore, combining above, we have for  $K>\max\{\hat{K}_4, \hat{K}_5, \hat{K}_6,  \bar{K}, \frac{1}{\gamma\check{\delta}_{20}}\}$, we have
\begin{equation}
\begin{aligned}
&\P\left(\left | |T_{22}-\bar{T}_{22}|\right |\geq \frac{\epsilon_2}{2}\right)\\
& \leq  912\exp(-K\tilde{\epsilon}_1'') +  78 \exp\left( -\sqrt{K}\tilde{\epsilon}_2'' \right)+\frac{\tilde{\epsilon}_3''}{K}+12 \exp\left(-\frac{1}{2} (\gamma K-1) \min\left\{\delta_{15}^2/16, \delta_{15}/4, \delta_{18}/4, \delta_{18}^2/16,\delta_{19}/4,\delta_{19}^2/16\right\}\right),
\end{aligned}
\end{equation}
where \begin{align*} \tilde{\epsilon}_1'' = \min&\left\{\frac{\epsilon_{27}^2}{128M_1^2}, \frac{\epsilon_{27}}{16M_1}, \frac{\epsilon_{28}^2}{2M_1^2},
\frac{\gamma}{32},
\frac{\hat{\delta}(\bar{\alpha}\eta_6)^2}{C_1'}, \frac{\hat{\delta}(\bar{\alpha}\eta_6)}{C_2'}, \frac{\hat{\delta}(\bar{\alpha}\delta_{14})^2}{C_1'}, \frac{\hat{\delta}(\bar{\alpha}\delta_{14})}{C_2'}, \frac{(\epsilon_{30})^2}{32}, \frac{\epsilon_{30}}{8}, 
\frac{2(\epsilon_{31})^2}{c_{max}^2}, \frac{\hat{\delta}(\bar{\alpha}\eta_9)^2}{C_1'},\frac{\hat{\delta}(\bar{\alpha}\eta_9)}{C_2'}\right\},
\end{align*}
$$
\tilde{\epsilon}_2'' = \min\left\{\frac{\gamma\delta_{12}^2}{1152}, \frac{\gamma\delta_{21}}{2304}\right\},
$$
$$\tilde{\epsilon}_3'' = \max\left\{\frac{24M_1^2}{\epsilon_{27}^2}, \frac{192M_1^2}{\hat{\delta}(\bar{\alpha}\eta_6)^2}, \frac{55296M_0^2}{\gamma\delta_{12}^2},\frac{24M_1^2}{\hat{\delta}(\bar{\alpha}\delta_{14})^2}, \frac{288M_1^2}{\hat{\delta}(\bar{\alpha}\eta_{9})^2},\frac{74880M_0^2}{\gamma\delta_{22}^2}\right\}.$$

\textbf{Step 3.} Estimate $\mathbb{P}\left(\left | |T_{g}-\bar{T}_{g}|\right |\geq \epsilon\right)$.

We know from the proof in Theorem 2 in \cite{WLS24},
$R(s)^{-1} g_2$ has the same distribution as $g_2$ due to the independence between $s$ and $g_2$. Then by   Lemma \ref{le-4} and Lemma \ref{pro-exp-bound-lambda1},  we have for $K>\epsilon_{34}^{-1}$,
\begin{align*}
        &\P\left(\left|\frac{\|(R(s)^{-1} g_2[2:K]\|}{\sqrt{K}}-1 \right|\geq \frac{\epsilon_1}{1+\epsilon_1}\right)+\P\left(\left | \frac{\|R(s)^{-1} g_2[2:K]\|^2}{K}-1\right |\geq \frac{\epsilon_2}{1+\epsilon_2}\right)\\&\leq 2\P\left(\left | \frac{\|g_2[2:K]\|^2}{K}-1\right |\geq \epsilon_{33}\right)\\[12pt]
       & =2\P\left(\left | \frac{\|g_2[2:K]\|^2}{K-1}\frac{K-1}{K}-1\right |\geq \epsilon_{33}\right)\\[12pt]
       &\leq 2\P\left(\left | \frac{\|g_2[2:K]\|^2}{K-1}-1\right |\geq \epsilon_{34}\right)+2\P\left(\left | \frac{K-1}{K}-1\right |\geq \epsilon_{34}\right)\\&\leq 4\exp\left(-\frac{1}{2}K\min\left\{\epsilon_{34}/4, \epsilon_{34}^2/16\right\}\right),
    \end{align*}
    where $\epsilon_{33}=\frac{\min\{\epsilon_1, \epsilon_2\}}{1+\min\{\epsilon_1, \epsilon_2\}}, \epsilon_{34}=\tilde{\delta}\left(1,1,\epsilon_{33}\right).$  By Lemma \ref{le-4} and Lemma \ref{exp-C2}, for $K>\max\{\hat{K}_7, \hat{K}_8, \bar{K}, \frac{1}{\gamma\delta_{21}}\}$ with $\hat{K}_7=\left(\frac{12M_0^2\bar{\alpha}^{-1}K_k}{\delta_{23}}\right)^{4}$ and $\hat{K}_8=\left(\frac{12\sqrt{2M_0^2\bar{\alpha}^{-1}\max\{K_g, K_h\}}}{\delta_{26}}\right)^{8}$,
\begin{align*}
&\P\left(\left| \frac{C_2|z_2[1]|\|B(g_1)\|\|D\|}{\sqrt{K}}\right|\geq \epsilon_1\right)\\&\leq\P\left(\left| \frac{C_2|z_2[1]|(1+\sqrt{1/\gamma}+\epsilon_0)}{\sqrt{K}}\right|\geq \epsilon_1\right)\\
&\leq \P\left(\left| C_2-\bar{C}_2\right|\geq \epsilon_{35}\right)+\P\left(\left| \frac{|z_2[1]|}{\sqrt{K}}\right|\geq \epsilon_{35}\right)\\
&\leq \P\left(\left| C_2^2-\bar{C}_2^2\right|\geq \bar{C}_2\epsilon_{35}\right)+\P\left(\left| \frac{|z_2[1]|}{\sqrt{K}}\right|\geq \epsilon_{35}\right)
\\& \leq 52e^{-\frac{\gamma K}{32}} +  144 \exp\left(- K \min\left\{\frac{\hat{\delta}(\bar{\alpha}\eta_{12})^2}{C_1'}, \frac{\hat{\delta}(\bar{\alpha}\eta_{12})}{C_2'}\right\}\right)+\dfrac{144 M_1^2}{K \hat{\delta}(\bar{\alpha}\eta_{12})^2} + 18 \exp\left( -\frac{\sqrt{K}\gamma \delta_{28}}{2304 } \right)+\frac{37440M_0^2}{\gamma K\delta_{29}^2}\\
&\quad+4 \exp\left(-\frac{1}{2} (\gamma K-1) \min\left\{\delta_{25}^2/16, \delta_{25}/4, \delta_{26}^2/16, \delta_{26}/4\right\}\right)+\exp(-\epsilon_{35}^2K),
\end{align*}
where  $$K_k = \frac{2}{\sqrt{\pi}} (N_\ell^g + N_\ell^h) + \left(1 + \frac{1}{\sqrt{\pi}}\right) (N_r^g + N_r^h)),
C_1' = \max\left\{ \frac{c_{\max}^2}{2}, 32, 32 M_1^2 \right\}, \quad C_2' = \max\left\{ 8, 8 M_1 \right\}, $$$$ \eta_{10} = \frac{\delta_{23} \tau}{192}, \quad \tau = \frac{1}{\sqrt[4]{K}},
\hat{\delta}(\bar{\alpha}\eta_{12}) = \min\left\{ \frac{\gamma \bar{\alpha} \eta_{12} }{2L}, \frac{1}{2} \right\}, \quad L = 4(1 + \sigma_s^2)(\mathbb{E}[f^2(d)] + 1) + 2\sigma_s^2(1 + \mathbb{E}[f^2(d)]).
$$
$$\delta_{23}=\tilde{\delta}(1,\overline{C}_2^2, \overline{C}_2\epsilon_{35}), \delta_{24}=\tilde{\delta}(1,|\mathbb{E}[Z^\dagger q(\bar{\alpha} Z)]|^2,
\delta_{23}), \delta_{25}=\frac{\delta_{24}}{1+\delta_{24}}, \delta_{26}=\min\left\{\sqrt{\frac{\delta_{24}}{2}}, \frac{\delta_{24}}{2|\mathbb{E}[Z^\dagger q(\bar{\alpha} Z)]|}\right\}, $$$$\delta_{27}=\frac{1}{2}(\sqrt{1+\delta_{23}}-1),\delta_{28}=\min\left\{\frac{\delta_{23}^2}{2}, 2\delta_{26}^2\right\}, \delta_{29}=\min\left\{\delta_{23}, \delta_{26}\right\}, \eta_{11}=\frac{\delta_{26}}{96\sqrt[4]{K}}, \eta_{12}=\min\{\eta_{10}, \eta_{11}\},$$
$\epsilon_{35}=\tilde{\delta}(\bar{C}_2, 0, \epsilon_1), \epsilon_0=1/2(1-\sqrt{1/\gamma}).$

Consequently,  combining Step 1, Step 2 and Step 3, we have the coclusion.
\end{proof}

\subsection{Proof of Lemma~\ref{Lem:T_s-convg-barT_s}}
\label{sec:proof of Lem:T_s-convg-barT_s}

A complete version of Lemma~\ref{Lem:T_s-convg-barT_s}
is stated as follows.

\begin{lemma}[Lemma~\ref{Lem:T_s-convg-barT_s}]
\label{Lem:T_s-convg-barT_s-proof}
Let  Assumption \ref{Assu:H-n-s}-\ref{Assu:ratio-K-N} hold.
For any small positive number $\epsilon>0$, 
there exists $\tilde{K}>0$ depending on $\epsilon$ 
such that   
\begin{equation}
\label{eq:T_s-barT_s-diff}
\mathbb{P}\left(\left| T_s-\bar{T}_s\right|\geq \epsilon\right)\leq \tilde{\mathfrak{R}}(\epsilon,K),
\end{equation}
for all $K>\tilde{K}$, where $\tilde{\mathfrak{R}}(\epsilon,K)\to 0$ as $K\to\infty$ and 
\begin{equation}
\tilde{\mathfrak{R}}(\epsilon,K) =\mathfrak{R}(\epsilon_{38},K)+
255\exp(-K\tilde{\epsilon_4}) +  26\exp\left( -\sqrt{K}\tilde{\epsilon_5} \right)+\frac{6\tilde{\epsilon_6}}{K}+6 \exp\left(-\frac{1}{2} (\gamma K-1) \delta'\right),
\end{equation}
where 
$$
\tilde{K}=\max\left\{\hat{K}, \hat{K}_{9}, \hat{K}_{10},\hat{K}_{11},\frac{1}{\gamma\delta_{37}}\right\}
$$ 
with $\hat{K}$ being defined in Lemma \ref{Lem:T_g-convg-barT_g}, depending on $\epsilon_{38}$,
\bgeq \hat{K}_{9}=\left(\frac{12\sqrt{2M_0^2\bar{\alpha}^{-1}\max\{K_g, K_h\}}}{\delta_{30}}\right)^{8}, 
\hat{K}_{10}=\left(\frac{12M_0^2\bar{\alpha}^{-1}K_k}{\delta_{34}}\right)^{4}, \hat{K}_{11}=\left(\frac{12\sqrt{2M_0^2\bar{\alpha}^{-1}\max\{K_g, K_h\}}}{\delta_{37}}\right)^{8},
\edeq 
\begin{align*} \tilde{\epsilon}_4 = &\min\left\{\frac{\epsilon_{39}^2}{32M_1^2}, \frac{\epsilon_{39}}{8M_1}, \frac{\epsilon_{41}^2}{2M_1^2},\frac{\epsilon_{47}}{2},\frac{\epsilon_{47}^2}{2},\frac{2(\sigma_s\epsilon_{46})^2}{c_{max}^2}, \frac{\hat{\delta}(\bar{\alpha}\eta_{13})^2}{C_1'},\frac{\hat{\delta}(\bar{\alpha}\eta_{13})}{C_2'},\frac{\hat{\delta}(\bar{\alpha}\delta_{32})^2}{C_1'},\frac{\hat{\delta}(\bar{\alpha}\delta_{32})}{C_2'},\frac{\gamma \delta_{33}^2}{2}, \frac{\gamma \delta_{33}}{2}, \frac{\hat{\delta}(\bar{\alpha}\eta_{16})}{C_2'},\right.\\&\left.\frac{\hat{\delta}(\bar{\alpha}\delta_{16})^2}{C_1'},\epsilon_{41}^2,\epsilon_{44}^2
\frac{\gamma}{2},
\right\},
\\
\tilde{\epsilon}_5 =&
 \min\left\{\frac{\gamma\delta_{30}^2}{2304}, \frac{\gamma\delta_{39}}{2304}\right\}.
\\
\tilde{\epsilon}_6 =& \max\left\{\frac{8M_1^2}{\epsilon_{39}^2}, \frac{64M_1^2}{\hat{\delta}(\bar{\alpha}\eta_{13})^2}, \frac{18432M_0^2}{\gamma\delta_{30}^2},\frac{8M_1^2}{\hat{\delta}(\bar{\alpha}\delta_{32})^2}, \frac{144M_1^2}{\hat{\delta}(\bar{\alpha}\eta_{16})^2},\frac{37440M_0^2}{\gamma\delta_{40}^2}\right\},\\
\delta'=&\min\left\{\delta_{33}^2/16, \delta_{33}/4, \delta_{36}^2/16, \delta_{36}/4,\delta_{37}^2/16, \delta_{37}/4
\right\}
\end{align*} with
\begin{align*}
\epsilon_{36} &= \tilde{\delta}\left(\mathbb{E}[df(d)], \bar{C}_1, \frac{\epsilon}{3}\right), 
& \epsilon_{37} &= \tilde{\delta}\left(\bar{C}_2, 0, \frac{\epsilon}{3}\right), 
&  \epsilon_{38} &=\tilde{\delta}\left(\bar{T}_g, 0, \frac{\epsilon}{3}\right), \\
\epsilon_{39} &= \tilde{\delta}(\mathbb{E}[df(d)], 1, \epsilon_{36}), 
& \epsilon_{40} &= \frac{\epsilon_{39}}{1+\epsilon_{39}}, 
& \epsilon_{41} &= \tilde{\delta}\left(0, \frac{1}{\sigma_s}, \frac{\epsilon_{37}}{2}\right), \\
\epsilon_{42} &= \tilde{\delta}\left(1,  \frac{1}{\sigma_s}, \epsilon_{41}\right), 
& \epsilon_{43} &= \tilde{\delta}\left(0, \frac{1}{\sigma_s}, \frac{\epsilon_{37}}{2}\right), 
& \epsilon_{44} &= \frac{\epsilon_{42}}{1+\epsilon_{42}}, \\
\epsilon_{45} &= \min\{\epsilon_{41},\epsilon_{42},\epsilon_{43}\}, 
& \epsilon_{46} &= \frac{\epsilon_{45}\sigma_s}{1+\epsilon_{45}\sigma_s}, 
& 
\end{align*}
and
\begin{align*}
\delta_{30} &= \tilde{\delta}\left(\mathbb{E}\left[Z^\dagger q(\alpha Z)\right],\frac{1}{\bar{\alpha}^2}, \epsilon_{36}\right), &\delta_{31}&=\tilde{\delta}(1,\frac{1}{\bar{\alpha}^2}, \delta_{30}), &\delta_{32}&=\frac{\delta_{31}\bar{\alpha}}{1+\delta_{31}\bar{\alpha}},\\\delta_{33}&=\frac{\delta_{31}}{1+\delta_{31}},&
\delta_{34}&=\tilde{\delta}(1,\overline{C}_2^2, \delta), &\delta_{35}&=\tilde{\delta}(1,|\mathbb{E}[Z^\dagger q(\bar{\alpha} Z)]|^2,
\delta_{34}), \\
\delta_{36}&=\frac{\delta_{35}}{1+\delta_{35}},&\delta_{37}&=\min\left\{\sqrt{\frac{\delta_{35}}{2}}, \frac{\delta_{35}}{2|\mathbb{E}[Z^\dagger q(\bar{\alpha} Z)]|}\right\},&\delta_{38}&=\frac{1}{2}(\sqrt{1+\delta_{34}}-1),\\\delta_{39}&=\min\left\{\frac{\delta_{34}^2}{2}, 2\delta_{37}^2\right\}, &\delta_{40}&=\min\left\{\delta_{34}, \delta_{37}\right\} 
\end{align*}
and
\begin{align*}
\eta_{13} &=\frac{\delta_{30}}{96\sqrt[4]{K}}, 
&\eta_{14}& = \frac{\delta_5 \tau}{192}\frac{\epsilon_4}{4}, 
&\eta_{15}&=\frac{\delta_{37}}{96\sqrt[4]{K}}, &\eta_{16}&=\min\{\eta_{14}, \eta_{15}\},  \\
C_1' &= \max\left\{\frac{c_{\max}^2}{2}, 2, \dfrac{2}{\gamma}, 32M_1^2\right\}, 
&C_2' &= \max\left\{ 2, \dfrac{2}{\gamma}, 8M_1\right\},
&\hat{\delta}(\cdot) &= \min\left\{\dfrac{\gamma\cdot}{2L}, \dfrac{1}{2}\right\}, &
\end{align*}
$$L = 4(1 + \sigma_s^2)(\mathbb{E}[f^2(d)] + 1) + 2\sigma_s^2(1 + \mathbb{E}[f^2(d)]). $$
\end{lemma}

\begin{proof}
Observe that 
\begin{align*}
    T_{\mathrm{s}} &= \frac{\mathbf{g}_1^\mathsf{H}[C_1\mathbf{D}\hat{\mathbf{s}}_1 + C_2\mathbf{D}\mathbf{B}(\hat{\mathbf{s}}_1)\mathbf{z}_2[2:N]]}{\|\mathbf{g}_1\|\|\mathbf{s}\|} - T_{\mathrm{g}}\frac{(\mathbf{R}(\mathbf{s})^{-1}\mathbf{g}_2)[1]}{\|\mathbf{s}\|} \\
    &= C_{1} \frac{g_{1}^{H} D f(D)^{T} g_{1}}{\| g_{1} \|^{2}} + C_{2} \frac{g_{1}^{H} D B(\hat{s}_{1}) z_{2}[2 : N]}{\| g_{1} \| \| s \|} - \frac{T_{g} (R(s)^{-1} g_{2})[1]}{\| s \|}.
\end{align*}By Lemma \ref{le-1}- Lemma \ref{le-4},  for $\epsilon>0$, we have
\begin{align*}
 &\mathbb{P}\left(\left| T_s-\bar{T}_s\right|\geq \epsilon\right)\\[12pt]
&= \mathbb{P}\left(\left |  C_{1} \frac{g_{1}^{H} D f(D)^{T} g_{1}}{\| g_{1} \|^{2}}-\overline{C}_1\mathbb{E}[df(d)]\right|>\frac{\epsilon}{3}\right)+\mathbb{P}\left(\left |C_{2} \frac{g_{1}^{H} D B(\hat{s}_{1}) z_{2}[2 : N]}{\| g_{1} \| \| s \|}\right|>\frac{\epsilon}{3}\right) + \mathbb{P}\left(\left |\frac{T_{g} (R(s)^{-1} g_{2})[1]}{\| s \|}\right | \geq \frac{\epsilon}{3}\right)\\[12pt]
&\leq 
\mathbb{P}\left(\left |  C_{1} -\overline{C}_1\right|>\epsilon_{36}\right)+\mathbb{P}\left(\left |  \frac{g_{1}^{H} D f(D)^{T} g_{1}}{\| g_{1} \|^{2}}-\mathbb{E}[df(d)]\right|>\epsilon_{36}\right)+\mathbb{P}\left(\left |  C_{2} -\overline{C}_2\right|>\epsilon_{37}\right)\\&+\mathbb{P}\left(\left | \frac{g_{1}^{H} D B(\hat{s}_{1}) z_{2}[2 : N]}{\| g_{1} \| \| s \|}\right|>\epsilon_{37}\right) +\mathbb{P}\left(\left| T_g-\bar{T}_g\right|\geq \epsilon_{38}\right)+\mathbb{P}\left(\left |\frac{g_{2}[1]}{\| s \|}\right | \geq \epsilon_{38}\right)\\&
\leq 
\mathbb{P}\left(\left |  C_{1} -\overline{C}_1\right|>\epsilon_{36}\right)+\mathbb{P}\left(\left |  \frac{g_{1}^{H} D f(D)^{T} g_{1}}{K^2}-\mathbb{E}[df(d)]\right|>\epsilon_{39}\right)+P\left(\left| \frac{\| g_{1} \|^{2}}{K}-1  \right|\geq \epsilon_{40}\right)
\\&+\mathbb{P}\left(\left |  C_{2} -\overline{C}_2\right|>\epsilon_{37}\right)+\mathbb{P}\left(\left | \frac{g_{1}^{H} D R(\hat{s}_{1}) z_{2}}{\| g_{1} \| \| s \|}\right|>\frac{\epsilon_{37}}{2}\right) +\mathbb{P}\left(\left |\frac{z_{2}[1]\|D \|}{\| s \|}  \right|>\frac{\epsilon_{37}}{2}\right)\\&+\mathbb{P}\left(\left| T_g-\bar{T}_g\right|\geq \epsilon_{38}\right)+\mathbb{P}\left(\left |\frac{g_{2}[1]}{\sqrt{K}}\right | \geq \epsilon_{41}\right)+\mathbb{P}\left(\left |\frac{\sqrt{K}}{\| s \|}-\frac{1}{\sigma_s}\right | \geq \epsilon_{41}\right)\\&
\leq 
\mathbb{P}\left(\left |  C_{1} -\overline{C}_1\right|>\epsilon_{36}\right)+\mathbb{P}\left(\left |  \frac{g_{1}^{H} D f(D)^{T} g_{1}}{K^2}-\mathbb{E}[df(d)]\right|>\epsilon_{39}\right)+P\left(\left| \frac{\| g_{1} \|^{2}}{K}-1  \right|\geq \epsilon_{40}\right)
\\&+\mathbb{P}\left(\left |  C_{2} -\overline{C}_2\right|>\epsilon_{37}\right)+\mathbb{P}\left(\left | \frac{g_{1}^{H} D z_{2}}{K}\right|>\epsilon_{41}\right)+ \mathbb{P}\left(\left | \frac{\sqrt{K}}{\| g_{1} \|}-1\right|>\epsilon_{42}\right)+\mathbb{P}\left(\left | \frac{\sqrt{K}}{ \| s \|}-\frac{1}{\sigma_s}\right|>\epsilon_{42}\right)\\&+\mathbb{P}\left(\left |\frac{z_{2}[1]}{\sqrt{K}}  \right|>\epsilon_{43}\right)+\mathbb{P}\left(\left |\frac{\sqrt{K}}{\| s \|}  -\frac{1}{\sigma_s}\right|>\epsilon_{43}\right)+\mathbb{P}\left(\left| T_g-\bar{T}_g\right|\geq \epsilon_{38}\right)+\mathbb{P}\left(\left |\frac{g_{2}[1]}{\sqrt{K}}\right | \geq \epsilon_{41}\right)+\mathbb{P}\left(\left |\frac{\sqrt{K}}{\| s \|}-\frac{1}{\sigma_s}\right | \geq \epsilon_{41}\right)\\&
\leq 
\mathbb{P}\left(\left |  C_{1} -\overline{C}_1\right|>\epsilon_{36}\right)+\mathbb{P}\left(\left |  \frac{g_{1}^{H} D f(D)^{T} g_{1}}{K^2}-\mathbb{E}[df(d)]\right|>\epsilon_{39}\right)+P\left(\left| \frac{\| g_{1} \|^{2}}{K}-1  \right|\geq \epsilon_{40}\right)
\\&+\mathbb{P}\left(\left |  C_{2} -\overline{C}_2\right|>\epsilon_{37}\right)+\mathbb{P}\left(\left | \frac{g_{1}^{H} D z_{2}}{K}\right|>\epsilon_{41}\right)+ \mathbb{P}\left(\left | \frac{\| g_{1} \|}{\sqrt{K}}-1\right|>\epsilon_{44}\right)+3\mathbb{P}\left(\left | \frac{ \| s \|}{\sqrt{K}}-\sigma_s\right|>\epsilon_{46}\right)\\&+\mathbb{P}\left(\left |\frac{z_{2}[1]}{\sqrt{K}}  \right|>\epsilon_{43}\right)+\mathbb{P}\left(\left| T_g-\bar{T}_g\right|\geq \epsilon_{38}\right)+\mathbb{P}\left(\left |\frac{g_{2}[1]}{\sqrt{K}}\right | \geq \epsilon_{41}\right),
    \end{align*}
    where $\epsilon_{36}= \tilde{\delta}\left(\mathbb{E}[df(d)], \bar{C}_1, \frac{\epsilon}{3}\right), 
 \epsilon_{37} = \tilde{\delta}\left(\bar{C}_2, 0, \frac{\epsilon}{3}\right), 
  \epsilon_{38} =\tilde{\delta}\left(\bar{T}_g, 0, \frac{\epsilon}{3}\right), 
\epsilon_{39} = \tilde{\delta}(\mathbb{E}[df(d)], 1, \epsilon_{36}), 
 \epsilon_{40} = \frac{\epsilon_{39}}{1+\epsilon_{39}}, 
 \epsilon_{41}= \tilde{\delta}\left(0, \frac{1}{\sigma_s}, \frac{\epsilon_{37}}{2}\right), 
\epsilon_{42} = \tilde{\delta}\left(1,  \frac{1}{\sigma_s}, \epsilon_{41}\right), 
 \epsilon_{43} = \tilde{\delta}\left(0, \frac{1}{\sigma_s}, \frac{\epsilon_{37}}{2}\right), 
 \epsilon_{44} = \frac{\epsilon_{42}}{1+\epsilon_{42}}, 
\epsilon_{45} = \min\{\epsilon_{41},\epsilon_{42},\epsilon_{43}\}, 
 \epsilon_{46} = \frac{\epsilon_{45}\sigma_s}{1+\epsilon_{45}\sigma_s},$
By Lemma \ref{thm:explicit_quad}, we have for $K>\bar{K}$,
$$
\begin{aligned}
&\P\left(\left| \frac{\| g_1^H Df(D)^T g_1\|}{K} -\mathbb{E}[d f(d)] \right| > \epsilon_{39}\right)
\leq  2 \exp\left( -\dfrac{1}{2} K \min\left\{ \dfrac{\epsilon_{39}^2}{16M_1^2}, \dfrac{\epsilon_{39}}{4M_1} \right\} \right)+\frac{8M_1^2}{K \epsilon_{39}^2}
\end{aligned}
$$
and
$$\P\left(\left|\frac{ g_1^H \mathbf{D}z_2}{K}\right| > \epsilon_{41}\right)\leq 4\exp\!\bigl(-\tfrac{K\epsilon_{41}^{2}}{2M_1^{2}}\bigr).
$$
By Proposition \ref{pro-exp-bound-lambda1}, we have
$$
\begin{aligned}
&\P\left(\left|  \frac{\|g_1\|^2}{K} -1\right| > \epsilon_{40}\right)+\P\left(\left|  \frac{\|g_1\|}{K} -1\right| > \epsilon_{45}\right)\\&\leq2\P\left(\left|  \frac{\|g_1\|^2}{K} -1\right| > \epsilon_{47}\right)\leq4\exp(-\frac{1}{2}K\min\{\epsilon_{47}/4, \epsilon_{47}^2/16\}),
\end{aligned}
$$ where
$\epsilon_{47}=\min\{\epsilon_{40}, \epsilon_{45}\}$.
By Proposition \ref{pro-hoeff-ineq}, we have
$$
\begin{aligned}
\P\left(\left|  \frac{\|s\|}{K} -\sigma_s \right| > \epsilon_{46}\right)\leq\P\left(\left|  \frac{\|s\|^2}{K} -\sigma_s^2  \right| > \sigma_s\epsilon_{46}\right)\leq2 \exp\left(-\frac{2K (\sigma_s\epsilon_{46})^2}{c_{\max}^2}\right),
\end{aligned}
$$
  By Lemma \ref{exp-C1}, we have for  $K>\max\{\hat{K}_9, \bar{K}\}$ with $\hat{K}_9=\left(\frac{12\sqrt{2M_0^2\bar{\alpha}^{-1}\max\{K_g, K_h\}}}{\delta_{30}}\right)^{8}$, 
\begin{equation}
\begin{aligned}
&\mathbb{P}\left( \left| \left| C_1 \right| - \left| \overline{C}_1 \right| \right| > \epsilon_{36} \right)
\\
& \leq  24e^{-\frac{\gamma K}{2}} +  64 \exp\left(- K \min\left\{\frac{\hat{\delta}(\bar{\alpha}\eta_{13})^2}{C_1'}, \frac{\hat{\delta}(\bar{\alpha}\eta_{13})}{C_2'}\right\}\right)+\dfrac{64M_1^2}{K \hat{\delta}(\bar{\alpha}\eta_{13})^2} + 8 \exp\left( -\frac{\sqrt{K}\gamma \delta_{30}^2}{1152 } \right)+\frac{18432M_0^2}{\gamma K\delta_{30}^2}\\
&\quad+8 \exp\left(- K \min\left\{\frac{\hat{\delta}(\bar{\alpha}\delta_{32})^2}{C_1'}, \frac{\hat{\delta}(\bar{\alpha}\delta_{32})}{C_2'}\right\}\right)+\dfrac{8M_1^2}{K \hat{\delta}(\bar{\alpha}\delta_{32})^2}+2 \exp\left(-\frac{1}{2} \gamma K \min\left(\delta_{33}^2/16, \delta_{33}/4\right)\right),
\end{aligned}
\end{equation}
where $$  C_1' = \max\left\{\frac{c_{\max}^2}{2}, 32, 32M_1^2\right\}, \quad C_2' = \max\left\{ 8, 8M_1\right\}, \eta_{13}=\frac{\delta_{30}}{96\sqrt[4]{K}},$$$$
\delta_{30}=\tilde{\delta}\left(\mathbb{E}\left[Z^\dagger q(\alpha Z)\right],\frac{1}{\bar{\alpha}^2}, \epsilon_{36}\right), \delta_{31}=\tilde{\delta}(1,\frac{1}{\bar{\alpha}^2}, \delta_{30}), \delta_{32}=\frac{\delta_{31}\bar{\alpha}}{1+\delta_{31}\bar{\alpha}},\delta_{33}=\frac{\delta_{31}}{1+\delta_{31}},
$$$$
L = 4(1 + \sigma_s^2)(\E[f^2(d)] + 1) + 2\sigma_s^2(1 + \E[f^2(d)]),\quad
\hat{\delta}(\bar{\alpha}\delta_{32}) = \min\left\{\dfrac{\gamma\bar{\alpha}\delta_{32}}{2L}, \dfrac{1}{2}\right\}.$$  
By Lemma \ref{exp-C2}, for   $K>\max\{\hat{K}_{10}, \hat{K}_{11}, \bar{K}, \frac{1}{\gamma\delta_{38}}\}$ with $\hat{K}_{10}=\left(\frac{12M_0^2\bar{\alpha}^{-1}K_k}{\delta_{34}}\right)^{4}$ and $\hat{K}_{11}=\left(\frac{12\sqrt{2M_0^2\bar{\alpha}^{-1}\max\{K_g, K_h\}}}{\delta_{37}}\right)^{8}$, we have  \begin{equation}
\begin{aligned}
&\P\left( \left| C_2 - \overline{C}_2 \right| > \epsilon_{37} \right) \leq \P\left( \left| C_2^2 - \overline{C}_2^2 \right| > \epsilon_{37}\overline{C}_2\right)\\
& \leq  52e^{-\frac{\gamma K}{32}} +  144 \exp\left(- K \min\left\{\frac{\hat{\delta}(\bar{\alpha}\eta_{16})^2}{C_1'}, \frac{\hat{\delta}(\bar{\alpha}\eta_{16})}{C_2'}\right\}\right)+\dfrac{144 M_1^2}{K \hat{\delta}(\bar{\alpha}\eta_{16})^2} + 18 \exp\left( -\frac{\sqrt{K}\gamma \delta_{39}}{2304 } \right)+\frac{37440M_0^2}{\gamma K\delta_{40}^2}\\
&\quad+4 \exp\left(-\frac{1}{2} (\gamma K-1) \min\left(\delta_{36}^2/16, \delta_{36}/4, \delta_{37}^2/16, \delta_{37}/4\right)\right),
\end{aligned}
\end{equation}
where  $$K_k = \frac{2}{\sqrt{\pi}} (N_\ell^g + N_\ell^h) + \left(1 + \frac{1}{\sqrt{\pi}}\right) (N_r^g + N_r^h)),
C_1' = \max\left\{ \frac{c_{\max}^2}{2}, 32, 32 M_1^2 \right\}, \quad C_2' = \max\left\{ 8, 8 M_1 \right\}, $$$$ \eta_{14} = \frac{\delta_5 \tau}{192}, \quad \tau = \frac{1}{\sqrt[4]{K}},
\hat{\delta}(\bar{\alpha}\eta_{16}) = \min\left\{ \frac{\gamma \bar{\alpha} \eta_{16} }{2L}, \frac{1}{2} \right\}, \quad L = 4(1 + \sigma_s^2)(\mathbb{E}[f^2(d)] + 1) + 2\sigma_s^2(1 + \mathbb{E}[f^2(d)]).
$$

$$\delta_{34}=\tilde{\delta}(1,\overline{C}_2^2, \delta), \delta_{35}=\tilde{\delta}(1,|\mathbb{E}[Z^\dagger q(\bar{\alpha} Z)]|^2,
\delta_{34}), \delta_{36}=\frac{\delta_{35}}{1+\delta_{35}},\delta_{37}=\min\left\{\sqrt{\frac{\delta_{35}}{2}}, \frac{\delta_{35}}{2|\mathbb{E}[Z^\dagger q(\bar{\alpha} Z)]|}\right\}, $$

$$\delta_{38}=\frac{1}{2}(\sqrt{1+\delta_{34}}-1),\delta_{39}=\min\left\{\frac{\delta_{34}^2}{2}, 2\delta_{37}^2\right\}, \delta_{40}=\min\left\{\delta_{34}, \delta_{37}\right\}, \eta_{15}=\frac{\delta_{37}}{96\sqrt[4]{K}}, \eta_{16}=\min\{\eta_{14}, \eta_{15}\}.$$
By Remark \ref{re-exp},     $\left|z_{2}[1]  \right|^2$ and $\left |g_{2}[1]\right |^2$ follow an exponential distribution with rate parameter $1$, so we have
$$
\mathbb{P}\left(\left |\frac{z_{2}[1]}{\sqrt{K}}  \right|>\epsilon_{44}\right)+\mathbb{P}\left(\left |\frac{g_{2}[1]}{\sqrt{K}}\right | \geq \epsilon_{41}\right)\leq2\exp(-K\min\{\epsilon_{41},\epsilon_{44}\}^2).$$
Then similarly to the proof of Lemma \ref{Lem:T_g-convg-barT_g}, we obtain the conclusion.
\end{proof}

\end{appendices}

\appendix

\end{document}